\documentclass[11pt]{amsart}
\usepackage{amssymb}
\usepackage{amsmath}
\usepackage{amsthm}
\usepackage{amsfonts}
\usepackage{amssymb,amssymb}
\usepackage{graphicx}
\usepackage{color}
\usepackage{epsfig}

\DeclareGraphicsExtensions{.pdf,.jpg,.png,.ps,.eps,.gif}

\theoremstyle{plain}
\newtheorem{theorem}{Theorem}\numberwithin{theorem}{section}
\newtheorem{main}{Main~Theorem}{}
\newtheorem{lemma}{Lemma}\numberwithin{lemma}{section}
\newtheorem{proposition}{Proposition}\numberwithin{proposition}{section}
\newtheorem{corollary}{Corollary}\numberwithin{corollary}{section}
\theoremstyle{definition}
\newtheorem{definition}{Definition}\numberwithin{definition}{section}
\theoremstyle{remark}
\theoremstyle{exam} \theoremstyle{ob}
\newtheorem{remark}{Remark}\numberwithin{remark}{section}
\newtheorem{exam}{Example}\numberwithin{exam}{section}
\numberwithin{ob}{section}
\numberwithin{equation}{section}

\begin{document}

\title[Multiple
Hilbert Transforms] {Multiple-Hilbert transforms associated\\ with
 Polynomials}

\author
{Joonil Kim}

\address{Department of Mathematics\\
            Yonsei University  \\
            Seoul 121, Korea}
   \email{jikim7030@yonsei.ac.kr}


\begin{abstract}
Let
$\Lambda=(\Lambda_1,\cdots,\Lambda_d)$ with $\,\Lambda_\nu\subset\mathbb{Z}_+^n\,$, and set
$\mathcal{P}_{\Lambda}$ the family of all vector polynomials,
\begin{eqnarray*}\label{vecpol}
\mathcal{P}_{\Lambda}=\left\{P_{\Lambda}:
P_{\Lambda}(t)=\left(\sum_{\mathfrak{m}\,\in
\Lambda_{1}}\,c_{\mathfrak{m}}^{1}\,t^{\mathfrak{m}},\cdots,\sum_{\mathfrak{m}\,\in
\Lambda_{d}}\,c_{\mathfrak{m}}^{d}\,t^{\mathfrak{m}}\right)\   \text{with}\ \ t\in\mathbb{R}^n\right\}.
\end{eqnarray*}
   Given   $P_\Lambda\in \mathcal{P}_{\Lambda}$,
we consider a class of multi-parameter oscillatory singular integrals,
\begin{eqnarray*}
\mathcal{I}(P_\Lambda,\xi,r)=\text{p.v.}\int_{\prod
[-r_j,r_j]}e^{i\langle\xi,P_{\Lambda}(t)\rangle}
\frac{dt_1}{t_1}\cdots\frac{dt_n}{t_n}\ \ \text{where}\ \ \xi \in\mathbb{R}^d,
 r\in\mathbb{R}_+^n.
\end{eqnarray*}
When $n=1$, the integral $\mathcal{I}(P_\Lambda,\xi,r)$ for any
$P_\Lambda\in\mathcal{P}_\Lambda$ is bounded uniformly in $\xi$ and
$r$. However, when $n\ge 2$, the uniform boundedness depends on each
indivisual polynomial $P_\Lambda$. In this paper, we fix $\Lambda$
and  find a necessary and sufficient condition on $\Lambda$  such that
\begin{eqnarray}\label{pgo}
\text{for all $P_\Lambda\in\mathcal{P}_\Lambda$}, \ \
\sup_{\xi, \, r}
|\mathcal{I}(P_\Lambda,\xi,r)|\le C_{P_\Lambda}<\infty.
\end{eqnarray}
The condition is described by faces and their cones of
polyhedrons associated with $\Lambda_\nu$'s.
\end{abstract}

\maketitle

\setcounter{tocdepth}{1}

\tableofcontents

 \section{Introduction}
Let $\mathbb{Z}_+$ denote the set of all nonnegative
integers and let $\,\Lambda_\nu\subset\mathbb{Z}_+^n\,$ be the
finite set of multi-indices  for each $\nu=1,\cdots,d$.
 Given
$\Lambda=(\Lambda_1,\cdots,\Lambda_d)$, we set
$\mathcal{P}_{\Lambda}$ the family of all vector polynomials
$P_{\Lambda}$  of the following form:
\begin{eqnarray}\label{vecpol}
\mathcal{P}_{\Lambda}=\left\{P_{\Lambda}:
P_{\Lambda}(t)=\left(\sum_{\mathfrak{m}\,\in
\Lambda_{1}}\,c_{\mathfrak{m}}^{1}\,t^{\mathfrak{m}},\cdots,\sum_{\mathfrak{m}\,\in
\Lambda_{d}}\,c_{\mathfrak{m}}^{d}\,t^{\mathfrak{m}}\right)\   \text{with}\ \ t\in\mathbb{R}^n\right\}
\end{eqnarray}
where $c^{\nu}_{\mathfrak{m}}$'s are nonzero real numbers. Given   $P_\Lambda\in \mathcal{P}_{\Lambda}$,
$\xi=(\xi_1,\cdots,\xi_d)\in\mathbb{R}^d$
and $r=(r_1,\cdots,r_n)\in\mathbb{R}_+^n$,
we define a multi-parameter oscillatory singular integral:
\begin{eqnarray*}
\mathcal{I}(P_\Lambda,\xi,r)=\text{p.v.}\int_{\prod
[-r_j,r_j]}e^{i\langle\xi,P_{\Lambda}(t)\rangle}
\frac{dt_1}{t_1}\cdots\frac{dt_n}{t_n}
\end{eqnarray*}
 where the principal value
integral is defined by
\begin{eqnarray*}
 \lim_{\epsilon\rightarrow 0}\int_{\prod
\{\epsilon_j<|t_j|<r_j\}}e^{i\langle\xi,P_{\Lambda}(t)\rangle}
\frac{dt_1}{t_1}\cdots\frac{dt_n}{t_n}
\end{eqnarray*}
where $\epsilon=(\epsilon_1,\cdots,\epsilon_n)$ with $\epsilon_j>0$.
The existence of this limit follows by the Taylor expansion of
$t\rightarrow e^{i\langle\xi,P_{\Lambda}(t)\rangle}$ and the cancelation
property  $\int dt_\nu/t_\nu=0$ with $\nu=1,\cdots,n$.

We see that whether
$\sup_{\xi}\left|\mathcal{I}(P_\Lambda,\xi,r)\right|$ is finite or
not depends on
\begin{itemize}
 \item[(1)] Sets $\Lambda_\nu$ of exponents of monomials in
 $P_{\Lambda}(t)$.
\item[(2)] Coefficients of polynomial  $P_{\Lambda}(t)$.
\item[(3)] Domain of integral $\prod [-r_j,r_j]$.
\end{itemize}
 (1) The dependence on  {\it set $\Lambda_\nu$ of exponents} is observed in the following simple cases:
\begin{eqnarray*}
\quad \sup_{\xi\in\mathbb{R}}|\mathcal{I}(P_\Lambda,\xi,(1,1))|=\begin{cases}
\sup_{\xi\in\mathbb{R}} \left|\int_{-1}^1\int_{-1}^1\sin (\xi t_1t_2)\frac{dt_1}{t_1}\frac{dt_2}{t_2}\right|= \infty\ \ \text{if\ \
$\Lambda =\{(1,1)\}$} \\
\sup_{\xi\in\mathbb{R}}\left|\int_{-1}^1\int_{-1}^1
\sin (\xi t_1^2t_2)\frac{dt_1}{t_1}\frac{dt_2}{t_2}\right|=  0\  \  \ \, \text{if\ \ $\Lambda =\{(2,1)\}$} \end{cases}
\end{eqnarray*}
(2) The dependence on  {\it coefficients of polynomials} $P_\Lambda$
first appeared in \cite{NW}, later in \cite{CWW3}   and \cite{P1}. There exist two different polynomials $P_\Lambda$ and $Q_\Lambda$ in $\mathcal{P}_\Lambda$
 having the same exponent set $\Lambda$, with $\sup_{\xi}\mathcal{I}(P_\Lambda,\xi,r)$ finite but
$\sup_{\xi}\mathcal{I}(Q_\Lambda,\xi,r)$ infinite.
We can check this for $P_\Lambda(t)=t_1^1t_2^3-t_1^3t_2^1$ and
$Q_\Lambda(t)=t_1^1t_2^3+t_1^3t_2^1$.  However, in this paper, we do
not concern with this coefficient dependence. We rather search for a
condition of $\Lambda$ valid for universal $P_\Lambda\in
\mathcal{P}_\Lambda$ that
\begin{eqnarray}\label{pgo}
\text{for all $P_\Lambda\in\mathcal{P}_\Lambda$}, \ \
\sup_{\xi\in\mathbb{R}^d}
|\mathcal{I}(P_\Lambda,\xi,r)|\le C_{P_\Lambda}<\infty.
\end{eqnarray}
(3) The dependence on the domain $\prod [-r_j,r_j]$ is observed for
the case $\Lambda =\{(2,2),(3,3)\}$,
\begin{eqnarray*}
\sup_{\xi\in\mathbb{R} \, ,0<r_1,r_2<1}\left|\int_{-r_2}^{r_2}
\int_{-r_1}^{r_1}e^{i \xi(
t_1^2t_2^2+t_1^3t_2^3)}\frac{dt_1}{t_1}\frac{dt_2}{t_2}\right|\,
 &<&\infty,\\
 \sup_{\xi\in\mathbb{R} \, ,0<r_1,r_2<\infty}\left|\int_{-r_2}^{r_2}
 \int_{-r_1}^{r_1}e^{i \xi( t_1^2t_2^2+t_1^3t_2^3)}\frac{dt_1}{t_1}\frac{dt_2}{t_2}\right|\,
  &=&\infty .
  \end{eqnarray*}
 In the former integral, a monomial $t_1^2t_2^2$  dominating $t_1^3t_2^3$ with
small $t_1,t_2$, makes the vanishing property  $\int
\frac{dt_i}{t_i}=0$ effective. But in the latter integral,
  a monomial $t_1^3t_2^3$, dominating  $t_1^2t_2^2$ with large $t_1,t_2$,
  weakens  the cancellation effect of the integral $\int\frac{dt_i}{t_i}$.
  Knowing this dependence on whether $r_j$ is taken from a finite interval $(0,1)$ or an infinite interval $(0,\infty)$,
   we set up our problem by first fixing the range of $r$ according to $S\subset
    N_n=\{1,\cdots,n\}$:
\begin{eqnarray}\label{an50}
\qquad r\in I(S)=\prod_{j=1}^n I_j \   \text{  where $I_j=(0,1)$ for
$j\in S$ and $I_j=(0,\infty)$ for  $j\in N_n\setminus S$. }
\end{eqnarray}
Instead of (\ref{pgo}), we shall find the necessary and sufficient
condition on $\Lambda$ and $S$ that
\begin{eqnarray}\label{pgob}
\text{for all $P_\Lambda\in\mathcal{P}_\Lambda$}, \ \
\sup_{\xi\in\mathbb{R}^d, \, r\in I(S) }
|\mathcal{I}(P_\Lambda,\xi,r)|\le C_{P_\Lambda}<\infty.
\end{eqnarray}

 For each Schwartz function $f$ on $\mathbb{R}^d$ and  a vector
polynomial $\,P_{\Lambda}\in\mathcal{P}_\Lambda$, the
 multiple Hilbert transform of $f$ associated to   $\,P_{\Lambda}$ is defined to be
\begin{eqnarray*}
\left(\mathcal{H}^{P_\Lambda}_{r}\, f\right)(x)&=&\text{p.v.}
\int_{\prod_{j=1}^n\,[-r_j,r_j]} f\bigl(x-P_\Lambda(t)\bigr)\,
\frac{dt_1}{t_1}\cdots\frac{dt_n}{t_n}.
\end{eqnarray*}
Here $r_j=1$ with $j\in S$ corresponds to a local Hilbert transform,
and $r_j=\infty$ with $j\in N_n\setminus S$  corresponds to  a
global Hilbert transform.
 Since $\mathcal{I}(P_\Lambda,\xi,r)$ is the Fourier multiplier of the Hilbert transform
$ \mathcal{H}^{P_\Lambda}_{r}$,  the boundedness (\ref{pgob})   is
equivalent to that
\begin{eqnarray}\label{pgo5}
\text{for all $P_\Lambda\in\mathcal{P}_\Lambda$}, \ \ \sup_{ r\in
I(S)} \left\|
\mathcal{H}^{P_\Lambda}_{r}\right\|_{L^p(\mathbb{R}^d)\rightarrow
L^p(\mathbb{R}^d)} \le C_{P_\Lambda} \ \text{where}\ p=2.
\end{eqnarray}
In this paper, we show (\ref{pgob}) and (\ref{pgo5}) with
$1<p<\infty$ for all $n$ and $d$ when $S\subset N_n$. To seek and manifest the
condition to determine (\ref{pgob})  and (\ref{pgo5}), we study the concept of faces
and their cones   of the Newton Polyhedron associated with $\Lambda$
and $S\subset N_n$.
It is noteworthy in advance  that the necessary
and sufficient condition of (\ref{pgob}) is not determined by only  faces  but also by
 cones of the Newton polyhedron, which has not appeared explicitly in the graph case  $\Lambda=({\bf e}_1,\cdots,{\bf e}_n,\Lambda_{n+1})$
 or  low dimensional case  $n\le 2$.
\medskip
\\
\noindent {\bf Scheme  and Organization}.\ \
  As a motive  for this problem, we  remark the result
  of A. Carbery, S. Wainger and J. Wright in \cite{CWW}:
  Given a polynomial $P_\Lambda\in\mathcal{P}_\Lambda$
  with $\Lambda=(\{{\bf e}_1\}, \{{\bf e}_2\},\Lambda_3)$ with $n=2,d=3$ and $S=\{1,2\}$, a
necessary and sufficient condition for
  $$\left\|\mathcal{H}_{r}^{P_\Lambda}
  \right\|_{L^p(\mathbb{R}^3)\rightarrow
  L^p(\mathbb{R}^3)}\le C \ \ \text{where $r=(1,1)$}$$   is that  every vertex $\mathfrak{m}$ in a Newton polyhedron $
{\bf N}(\Lambda_3)=\text{Ch}(\Lambda_3+\mathbb{R}^2_+) $ has at least one even component.
 The idea of the proof in  \cite{CWW} is to split the sum of dyadic pieces
   $\mathcal{H}^{P_\Lambda}_{r}=\sum_{J\in\mathbb{Z}_+^2} \mathcal{H}_J^{P_\Lambda}$ into finite sums of
   cones $\{J\in\mathfrak{m}^*\}$ associated with vertices $\mathfrak{m}$ of ${\bf N}(\Lambda_3)$:
\begin{equation}\label{ide1}
 \sum_{\mathfrak{m}\ \text{is a vertex of ${\bf N}(\Lambda_3) $}}
 \left(\sum_{J\in \mathfrak{m}^*} \mathcal{H}^{P_\Lambda}_J\right)\ \
 \text{with}\ \ \mathfrak{m}^*=\{\alpha_1\mathfrak{q}_1+\alpha_2\mathfrak{q}_2:\alpha_1,\alpha_2\ge 0\}
 \end{equation}
  where $\mathfrak{q}_j$ is a normal vector of the supporting
  line $\pi_{\mathfrak{q}_j}$ of an edge $\mathbb{F}_j$ of ${\bf N}(\Lambda_3)$
  such that $\mathfrak{m}=\bigcap_{j=1}^2\mathbb{F}_j$.
They proved that for  $\Lambda'=(\{{\bf e}_1\}, \{{\bf
e}_2\},\{\mathfrak{m}\})$,
\begin{equation}\label{ide2}
\left\| \sum_{J\in \mathfrak{m}^*}
\left(\mathcal{H}^{P_\Lambda}_J- \mathcal{H}^{P_{\Lambda'}}_J\right)
\right\|_{L^p(\mathbb{R}^3)\rightarrow L^p(\mathbb{R}^3)}+\left\| \sum_{J\in \mathfrak{m}^*}
 \mathcal{H}^{P_{\Lambda'}}_J \right\|_{L^p(\mathbb{R}^3)\rightarrow L^p(\mathbb{R}^3)}\le C\
 \end{equation}
   by using the vertex dominating property (1) and   the vanishing property (2):
 \begin{itemize}
\item[(1)]  vertex dominating property: $J\in \mathfrak{m}^* \Rightarrow 2^{-J\cdot \mathfrak{m}}\ge 2^{-J\cdot \mathfrak{n}
}$ for $\mathfrak{n}\in {\bf
N}(\Lambda_3)\setminus\{\mathfrak{m}\}$,
 \item[(2)] vanishing property:  at least one component of
$\mathfrak{m}$ is even, that implies
$\mathcal{H}^{P_{(\emptyset,\emptyset,\mathfrak{m})  }}_J\equiv 0$.
\end{itemize}

For the case $n\ge 3$, we shall establish   the corresponding cone
type decomposition (\ref{ide1}) and the reduction estimate
(\ref{ide2}) together with
 (1) and (2). As an analogue of  (\ref{ide1}),  we split
$\mathcal{H}^{P_\Lambda}_{r}=\sum_{J\in\mathbb{Z}_+^n}
\mathcal{H}_J^{P_\Lambda}$  with $r=(1,\cdots,1)$ into
\begin{equation}\label{ide11}
 \sum_{(\mathbb{F}_\nu);\ \mathbb{F}_\nu \ \text{is a face of ${\bf N}(\Lambda_\nu) $}}
 \left(\sum_{J\in \bigcap_{\nu=1}^d\mathbb{F}_\nu^*} \mathcal{H}^{P_\Lambda}_J\right)\
 \ \text{with}\ \ \mathbb{F}_\nu^*=\left\{\sum_{j=1}^{N_\nu}\alpha_j\mathfrak{q}_j :\alpha_j\ge 0\right\}.
 \end{equation}
 Here $\mathfrak{q}_j$ is normal vector of the supporting plane $\pi_{\mathfrak{q}_j}$
   of a face $\mathbb{F}_\nu$ in the Newton polyhedron ${\bf N}(\Lambda_\nu)$, where
   $\mathbb{F}_\nu=\bigcap_{j=1}^{N_\nu}\pi_{\mathfrak{q}_j}$. For this purpose,
   we
introduce in Section \ref{ptt} the concept of a face  $\mathbb{F}$ and its cone $\mathbb{F}^*$ in a
Polyhedron. In Section \ref{evenset}, we state our main results and
some background  for this problem.   In Sections \ref{secfar}, we
provide  properties of faces and their
 cones  related with their representations.
  In Sections  \ref{sec6}, we   give   a few basic $L^p$ estimation tools.
   In Section \ref{condec}, we make (\ref{ide11}). As an analogue of (\ref{ide2}), we prove in Section
\ref{suffes}  that
\begin{equation}\label{ide22}
\left\| \sum_{J\in \bigcap_{\nu=1}^d\mathbb{F}_\nu^*}
\left(\mathcal{H}^{P_\Lambda}_J- \mathcal{H}^{P_{\Lambda'}}_J\right) \right\|_{L^p(\mathbb{R}^d)
\rightarrow L^p(\mathbb{R}^d)}+\left\| \sum_{J\in \bigcap_{\nu=1}^d\mathbb{F}_\nu^*}
 \mathcal{H}^{P_{\Lambda'}}_J \right\|_{L^p(\mathbb{R}^d)\rightarrow L^p(\mathbb{R}^d)}\le C\
 \end{equation}
 where
$\Lambda'=(\mathbb{F}_\nu\cap\Lambda_\nu)_{\nu=1}^d$. To show (\ref{ide22}),  we
use the dominating  and  vanishing properties:
\begin{itemize}
\item[(1)] If $J\in \bigcap_{\nu=1}^d\mathbb{F}_\nu^*$,
then  $2^{-J\cdot \mathfrak{m}}\ge 2^{-J\cdot \mathfrak{n} }$  where
$\mathfrak{m}\in\mathbb{F}_\nu$ and $\mathfrak{n}\in {\bf
N}(\Lambda_\nu)\setminus \mathbb{F}_\nu$,
\item[(2)] If sum of elements in
$\bigcup_{\nu=1}^d\mathbb{F}_\nu\cap\Lambda_\nu$ has at
least one even component, $\mathcal{H}^{P_{\Lambda'}}_J\equiv 0$.
\end{itemize}
The main feature emerging in the general case $n\ge 3$ is that
 the evenness  hypothesis of
(2) needs to be satisfied only if the following overlapping and low
rank conditions hold
 $$\bigcap_{\nu=1}^d (\mathbb{F}_\nu^*)^{\circ}\ne\emptyset\
 \ \text{and} \ \ \text{rank}\left(\bigcup_{\nu=1}^d\mathbb{F}_\nu\right)\le n-1.$$
 Note that the cones $\mathbb{F}_\nu^*$ as well as faces $\mathbb{F}_\nu$ of the Newton polyhedra associated with $\Lambda_\nu$
 are involved in determining (\ref{pgob}).
Thus, a difficulty in  showing (\ref{ide22}) is to keep the above
cone overlapping condition until the low ranked  faces occurs. For
this purpose, we construct in Section \ref{sec9} a sequence of faces
and cones such that
\begin{eqnarray}\label{pathg}
{\bf N}(\Lambda_\nu)&=&\mathbb{F}_\nu(0) \supset
 \cdots\supset\mathbb{F}_\nu(s) \supset\cdots \supset\mathbb{F}_\nu(N)=\mathbb{F}_\nu,\nonumber\\
  {\bf N}^*(\Lambda_\nu)&=&\mathbb{F}_\nu^*(0) \subset
 \cdots\subset\mathbb{F}_\nu^*(s) \subset\cdots
 \subset\mathbb{F}_\nu^*(N)=\mathbb{F}_\nu^*.
\end{eqnarray}
This sequence plays crucial roles to keep $\bigcap_{\nu=1}^d
(\mathbb{F}_\nu^*(s))^{\circ}\ne\emptyset$ with $s=1,\cdots,N$ and
give an efficient size control of $J\cdot\mathfrak{m}$ with $J\in
\bigcap_{\nu=1}^d\mathbb{F}_\nu^*$ and $\mathfrak{m}\in
\mathbb{F}_\nu(s)$. In Sections \ref{sec10}-\ref{sec13}, we prove necessity parts of
main theorems. In Section \ref{sec15}, we finish the proof for  general situations.

\noindent {\bf Notations.}\hskip.2cm
 For the sake of distinction, we shall use the
notations
$$\imath\cdot\jmath = \imath_1\jmath_1 + \cdots +
\imath_n\jmath_n\,,\quad \langle x,y\rangle=x_1y_1 + \cdots + x_d
y_d$$ for the inner products on $\,\mathbb{Z}^n,\,\mathbb{R}^d,\,$
respectively. Note that a constant  $C$ may be different on each line. As usual, the notation $\,A\lesssim B\,$ for two
scalar expressions $\,A, B\,$ will mean $\,A\leq C B\,$ for some
positive constant $C$ independent of $\,A, B\,$ and $\,A\approx B\,$
will mean $\,A\lesssim B\,$ and $\,B\lesssim A\,.$

\section{Polyhedra, Their Faces and Cones}\label{ptt}
 Throughout this paper, we show detailed proof for basic properties about faces and cones of polyhedra by using an easy
 tool such as elementary linear algebra. For further study, we refer readers to
\cite{F}.

\subsection{Polyhedron}
\begin{definition}
Let $U\subset \mathbb{R}^n$ be a subspace endowed with an inner
product $\langle,\rangle $ in $\mathbb{R}^n$. Then $V$ is called an
affine subspace in $\mathbb{R}^n$ if $V=\mathfrak{p}+ U$ for some
$\mathfrak{p}\in\mathbb{R}^n$.
\end{definition}
\begin{definition}\label{d11}
Let $V$ be an affine subspace in $\mathbb{R}^n$.  A hyperplane in
$V$ is a set
$$\pi_{\mathfrak{q},r}=\{{\bf y}\in  V:   \langle \mathfrak{q},
{\bf y}  \rangle =r\}\ \ \text{where  $\mathfrak{q}\in \mathbb{R}^n$
and $r\in\mathbb{R}$.}$$
  The
corresponding closed upper half-space and lower half-space are
$$ \pi^{+}_{\mathfrak{q},r} =\{{\bf y}\in V: \langle \mathfrak{q}, {\bf y}  \rangle \ge r\}\ \ \text{and}\
\ \pi^{-}_{\mathfrak{q},r} =\{{\bf y}\in V: \langle \mathfrak{q},
{\bf y}  \rangle \le r\}.$$ The open upper half-space and lower half
space are
$$(\pi^{+}_{\mathfrak{q} ,r})^{\circ}=\{{\bf y}\in V: \langle \mathfrak{q}, {\bf y}  \rangle >r\}\ \
\text{and}\ \ (\pi^{-}_{\mathfrak{q} ,r})^{\circ}=\{{\bf y}\in V:
\langle \mathfrak{q}, {\bf y}  \rangle<r\}.$$
\end{definition}

\begin{definition}[Polyhedron in $V$]\label{d12}
Let $V$ be an affine subspace in $\mathbb{R}^n$ and let
$\Pi=\{\pi_{\mathfrak{q}_j,r_j}\}_{j=1}^N $ be a collection of
hyperplanes in $V$. A   polyhedron $\mathbb{P}$ in $V$ is defined to
be an intersection of  closed upper half-spaces $
\pi^+_{\mathfrak{q}_j,r_j} $:
$$\mathbb{P}=\bigcap_{j=1}^N  \pi^+_{\mathfrak{q}_j,r_j} =\bigcap_{j=1}^N\{{\bf y}\in V:
\langle\mathfrak{q}_j, {\bf y}\rangle \ge r_j \ \ \text{for}\ 1\le
j\le N\}.$$  We call the above collection  $\Pi=\Pi(\mathbb{P})$
  the generator of $\mathbb{P} $.
     We  denote the
 polyhedron $\mathbb{P}$ by $\mathbb{P}(\Pi)$ indicating its
generator $\Pi$. Sometimes, we mean also the generator $\Pi$ of
$\mathbb{P}$ to be the collection of   normal vectors
$\{\mathfrak{q}_j\}_{j=1}^N$ instead of   hyperplanes
$\{\pi_{\mathfrak{q}_j,r_j}\}_{j=1}^N $.
\end{definition}
\begin{definition}\label{d13}
Let $B=\{\mathfrak{q}_1,\cdots,\mathfrak{q}_M\}\subset \mathbb{R}^n$
be a finite number of vectors. Then the span of $B$ is the set
$$\text{Sp}(B)=\left\{\sum_{j=1}^M
c_j\mathfrak{q}_j:c_j\in\mathbb{R}\right\}.$$ The convex span of $B$
and its interior are defined by
$$\text{CoSp}(B)=\left\{\sum_{j=1}^M c_j\mathfrak{q}_j:c_j\ge 0\right\}\
\text{ and }\
 \text{CoSp}^{\circ}(B)=\left\{\sum_{j=1}^M c_j\mathfrak{q}_j:c_j>
 0\right\} $$
respectively. Finally the convex hull of $B$ is the set
$$\text{Ch}(B)=\left\{\sum_{j=1}^M c_j\mathfrak{q}_j:c_j\ge 0\
\text{and}\ \sum_{j=1}^M c_j=1\right\}.$$ If $B\subset \mathbb{R}^n$
is not a finite set, then the span of $B$ is defined by the
collection of all finite linear combinations of vectors in $B$.
\end{definition}

\begin{definition}[Ambient Space of  Polyhedron]\label{df35}
Let $\mathbb{P}\subset \mathbb{R}^n$ and
$\mathfrak{p},\mathfrak{q}\in \mathbb{P}$. Then
\begin{eqnarray*}
\text{Sp}(\mathbb{P}-\mathfrak{p})=\text{Sp}(\mathbb{P}-\mathfrak{q})\
\ \text{for all $\mathfrak{p},\mathfrak{q}\in\mathbb{P}$}.
\end{eqnarray*}
We denote the vector space $\text{Sp}(\mathbb{P}-\mathfrak{p})$ by
$V(\mathbb{P})$. The dimension of $\mathbb{P}$ is defined by
$$\text{dim}(\mathbb{P})=
\text{dim}(V(\mathbb{P})).$$
  From the fact $\mathfrak{p}-\mathfrak{q}\in
V(\mathbb{P})$,
$$V(\mathbb{P})+\mathfrak{p}=V(\mathbb{P})+\mathfrak{q}.  $$
We call  $V(\mathbb{P})+\mathfrak{p} $
 the ambient affine space of $\mathbb{P}$ in $\mathbb{R}^n$ and denote it by $V_{am}(\mathbb{P}):$
\begin{eqnarray}\label{jf2}
V_{am}(\mathbb{P})= V(\mathbb{P})+\mathfrak{p},
\end{eqnarray}
which is the smallest affine space containing $\mathbb{P}$.
\end{definition}

\begin{definition}\label{d14}
Let $B\subset \mathbb{R}^n$. Then the rank of a set $B$ is the
number of linearly independent vectors in $B$:
\begin{eqnarray*}
\text{rank}(B)=\text{dim}(\text{Sp}(B)).
\end{eqnarray*}
\end{definition}

\subsection{Faces of  Polyhedron}
\begin{definition}[Faces]\label{dfac}
Let $V$ be an affine subspace in $\mathbb{R}^n$. Given a class $\Pi$
of hyperplane in $V$, let $\mathbb{P}=\mathbb{P}(\Pi)$ be a
polyhedron in $V$. A subset $\mathbb{F}\subset \mathbb{P}$ is a face
if there exists a hyperplane $ \pi_{ \mathfrak{q},r } $ in $ V$
(which does not have to be in $\Pi$) such that
\begin{eqnarray}\label{4g}
\mathbb{F}= \pi_{ \mathfrak{q},r } \cap \mathbb{P}  \ \text{and}\ \
\mathbb{P} \setminus \mathbb{F}\subset  \pi_{ \mathfrak{q},r }^+.
\end{eqnarray}
We may replace $\mathbb{P} \setminus \mathbb{F}$ by $\mathbb{P}$, or
  $\pi_{ \mathfrak{q},r }^+$ by $(\pi_{ \mathfrak{q},r
}^+)^{\circ}$ in  (\ref{4g}). Thus $\mathbb{F}$ is
a face of $\mathbb{P} $ if and only if there exists a   vector $
\mathfrak{q}\in \mathbb{R}^n$ and $r\in\mathbb{R}$
 satisfying
\begin{eqnarray}\label{4g1}
\langle\mathfrak{q},{\bf u}\rangle=r<\langle\mathfrak{q},{\bf y}\rangle\ \ \text{ for all
${\bf u}\in \mathbb{F}$ and ${\bf y}\in
\mathbb{P}\setminus \mathbb{F}$}.
\end{eqnarray}
When $\mathbb{F}$ is a face of $\mathbb{P}$, it is denoted by
 $\mathbb{F}\preceq\mathbb{P}.$ The above hyperplane $\pi_{
\mathfrak{q},r }$ is called the supporting hyperplane of the face
$\mathbb{F}$.
 The dimension of a face $\mathbb{F}$ of $\mathbb{P}$ is the dimension of an ambient affine
 space $V_{am}(\mathbb{F})$ of $\mathbb{F}$ where $V_{am}(\mathbb{F})$ is defined  in (\ref{jf2}).
  We denote the set of all $k$-dimensional faces of
$\mathbb{P}$ by $\mathcal{F}^k(\mathbb{P} )$, and
$\bigcup \mathcal{F}^k(\mathbb{P} )$  by $\mathcal{F}(\mathbb{P})$. By convention, an empty set is
$-1$ dimensional face. Let $\text{dim}(\mathbb{P})=m$. Then we call,
a face $\mathbb{F}$ whose dimension is less than $m$, a proper face
of $\mathbb{P}$ and denote it by $\mathbb{F}\precneqq\mathbb{P}$.
\end{definition}
\begin{lemma}\label{lemss} Let $\mathbb{P}=\mathbb{P}(\Pi)$ be a polyhedron in an affine space $V$.
Then \begin{itemize} \item[(1)] If $\mathbb{F},\mathbb{G}\preceq
\mathbb{P}$ and $\mathbb{G}\subset\mathbb{F}$, then
$\mathbb{G}\preceq \mathbb{F}$.
 \item[(2)] Let $\pi_{ \mathfrak{q},r }\in \Pi$. Then
$\mathbb{F} =\pi_{ \mathfrak{q},r }\cap\mathbb{P} $  is a face of
$\mathbb{P}$.
 \item[(3)] Let $\Delta\subset\Pi$.
  Then $
\mathbb{F}=\bigcap_{\pi_{\mathfrak{q},r}\in \Delta}
\mathbb{F}_{\mathfrak{q},r}$ with $\mathbb{F}_{\mathfrak{q},r}=\pi_{
\mathfrak{q},r }\cap\mathbb{P} $, is a face of $\mathbb{P} $.
\end{itemize}
\end{lemma}
\begin{proof}
Since $\mathbb{G}\preceq \mathbb{P}$, there exist $\mathfrak{q},r$
satisfying (\ref{4g1}) where $\mathbb{F}$ replaced by $\mathbb{G}$.
Next we can also replace $\mathbb{P} $  by $\mathbb{F}$. This proves
(1). By Definitions \ref{d11} and \ref{d12} for
$\pi_{\mathfrak{q},r}$ and $\mathbb{P} $,
\begin{eqnarray*}
 \langle \mathfrak{q}, {\bf
u}\rangle =r<  \langle \mathfrak{q}, {\bf y}\rangle\ \ \text{for all
${\bf u}\in \mathbb{F}=\pi_{\mathfrak{q},r}\cap\mathbb{P} $ and
${\bf y}\in \mathbb{P} \setminus \mathbb{F}$}
\end{eqnarray*}
which shows  (\ref{4g1}). Thus (2) is proved.  Let
$\Delta=\{\pi_{\mathfrak{q}_j,r_j}:1\le j\le M\}\subset\Pi$. Then by
(2), for every $j=1,\cdots,M$
\begin{eqnarray}\label{451g}
\quad \langle \mathfrak{q}_j, {\bf u}\rangle =r_j\le  \langle
\mathfrak{q}_j, {\bf y}\rangle\ \ \text{for all ${\bf u}\in
\mathbb{F}=\bigcap_{j=1}^M\mathbb{F}_{\mathfrak{q}_j,r_j}$ and ${\bf
y}\in \mathbb{P} \setminus \mathbb{F} $.}
\end{eqnarray}
 For ${\bf y}\in \mathbb{P} \setminus
\mathbb{F}=\bigcup_{j=1}^M
 (\mathbb{P}\setminus \mathbb{F}_{\mathfrak{q}_j,r_j})$ above,
 there exists  $j=\ell$ such that
$${\bf y}\in \mathbb{P} \setminus
 \mathbb{F}_{\mathfrak{q}_\ell,r_\ell}.$$ Thus  $\le $
in (\ref{451g}) is replaced by $<$  for $j=\ell$.
   Hence we sum (\ref{451g}) in $j$ to
obtain that
\begin{eqnarray}\label{452g}
\qquad \left\langle \sum_{j=1}^Mc_j\mathfrak{q}_j, {\bf
u}\right\rangle =\sum_{j=1}^Mc_jr_j< \left\langle
\sum_{j=1}^Mc_j\mathfrak{q}_j, {\bf y}\right\rangle\ \ \text{for all
${\bf u}\in \mathbb{F} $ and ${\bf y}\in \mathbb{P} \setminus
\mathbb{F} $}
\end{eqnarray}
where $\mathfrak{q}=\sum_{j=1}^Mc_j\mathfrak{q}_j$ and
$r=\sum_{j=1}^Mc_jr_j$. Hence this with (\ref{4g1}) yields  (3).
\end{proof}

\begin{definition}\label{dfu5}
 Let $\mathbb{F}$ be a face of a convex polyhedron $ \mathbb{P}$.
 Then the boundary $\partial \mathbb{F}$ of $\mathbb{F}$  is
 defined to be $ \bigcup \mathbb{G}$,
 where the union is over all face   $\mathbb{G} \precneqq
 \mathbb{F}$. When $\text{dim}(\mathbb{F})=k$,
 $$\partial\mathbb{F}=\bigcup_{\text{dim}\mathbb{G}=k-1,
  \mathbb{G}\preceq\mathbb{F}}\mathbb{G},$$
 since   faces whose dimensions $< k-1$ are contained on
  $k-1$ dimensional faces of $\mathbb{F}$.
 Note that $\partial \mathbb{F}$ is the boundary
 of
 $\mathbb{F}$ with respect to the usual topology of $V_{am}(\mathbb{F})$ in (\ref{jf2}).
 \end{definition}

\begin{lemma}\label{2323s}
Let $\mathbb{P}=\mathbb{P}(\Pi)$ be a polyhedron.  Then $
\partial\mathbb{P} \subset\bigcup_{\pi\in\Pi}\pi$.
\end{lemma}
\begin{proof}
Let ${\bf x}\in\partial\mathbb{P}$. Assume ${\bf x}\in
\bigcap_{\pi\in\Pi}(\pi_+)^{\circ}$. Then a ball $B({\bf
x},\epsilon)$ with some $\epsilon>0$ is contained in
$\bigcap_{\pi\in\Pi}(\pi_+)^{\circ}\subset
\bigcap_{\pi\in\Pi}(\pi_+)=\mathbb{P}$ in $V_{am}(\mathbb{P})$. Thus,
 $x\notin \partial \mathbb{P}$ because $\partial \mathbb{P}$ is
 a boundary of $\mathbb{P}$ with respect to the usual topology of $V_{am}(\mathbb{P})$.
 Hence  ${\bf x}\notin
 \bigcap_{\pi\in\Pi}(\pi_+)^{\circ}$.   Combined with
 ${\bf x}\in\partial \mathbb{P}\subset\mathbb{P}=\bigcap_{\pi\in\Pi}\pi_+$,
  we have ${\bf x}\in
\bigcup_{\pi\in\Pi}\pi$.
\end{proof}

\begin{definition}\label{dfu4}
 Let $\mathbb{F}$ be a face of a convex polyhedron $ \mathbb{P}$.
 Then the interior $\mathbb{F}^{\circ}$ of $\mathbb{P}$  is
 defined to be   $\mathbb{F}^{\circ}=\mathbb{F}\setminus \partial
 \mathbb{F}$.  Note also that $ \mathbb{F}^{\circ}$ is the interior of
 $\mathbb{F}$ with respect to the usual topology defined on
 $V_{am}(\mathbb{F})$ in (\ref{jf2}).
 \end{definition}

\begin{exam}\label{exam11}
 Observe that
 $\text{CoSp}(\mathfrak{p}_1,\cdots,\mathfrak{p}_N)^{\circ}=
 \left\{\sum_{j=1}^N\alpha_j\mathfrak{p}_j:\alpha_j>0\right\}$.
 \end{exam}

\begin{lemma}\label{lem411d}
 Let $\mathbb{P}$ be a   polyhedron and $\mathbb{F}\preceq\mathbb{P}$ with $\text{dim}(\mathbb{F})=k$.
  Suppose that
   $\mathbb{B}\subset\partial\mathbb{F} $ is a convex
 set.
Then there is a $k-1$ dimensional face $ \mathbb{G}$ such that $
\mathbb{B}\subset \mathbb{G}\preceq\mathbb{F}$.
\end{lemma}
\begin{proof}
Assume contrary. Then $\mathbb{B}$ is not contained in one proper
face of $\mathbb{F}$, that is, there exists
$\mathfrak{p}_1,\mathfrak{p}_2\in\mathbb{B}$ such that
$\text{Ch}(\mathfrak{p}_1,\mathfrak{p}_2)\nsubseteq \mathbb{G}$ for
any $\mathbb{G} \precneqq \mathbb{F}$. We shall find a
contradiction.  Given a plane $\pi$ and a line segment
$\text{Ch}(\mathfrak{p}_1,\mathfrak{p}_2)$ with
$\frac{\mathfrak{p}_1+\mathfrak{p}_2}{2} \in\pi$, we  have only two
cases:
\begin{equation}\label{jegy}
\text{(1) $\text{Ch}(\mathfrak{p}_1,\mathfrak{p}_2)\subset\pi$,\ or\
 (2) $
\mathfrak{p}_1\in(\pi_+)^{\circ}\ \text{and}\
\mathfrak{p}_2\in(\pi_-)^{\circ}$}
\end{equation}
where $\mathfrak{p}_1,\mathfrak{p}_2$ may be switched. By Definition
\ref{dfu5},
\begin{eqnarray}\label{0770}
 \bigcup_{\mathbb{G}\precneqq \mathbb{F}}\mathbb{G}=\partial\mathbb{F} \
\  \text{where  $\text{dim}(\mathbb{G})=k-1$.}
 \end{eqnarray}
It suffices to show that
\begin{eqnarray*}
 \mathfrak{p}_1,\mathfrak{p}_2 \subset\mathbb{B}\ \ \text{implies
that} \ \ \text{Ch}(\mathfrak{p}_1,\mathfrak{p}_2)\subset\mathbb{G}\
\ \text{for some face $\mathbb{G}$ in (\ref{0770}).}
\end{eqnarray*}
  By
$\text{Ch}(\mathfrak{p}_1,\mathfrak{p}_2)\subset\mathbb{B}\subset\partial\mathbb{F}$
and (\ref{0770}), we have  $(\mathfrak{p}_1+\mathfrak{p}_2)/2\in
\mathbb{G}$ for some $ \mathbb{G}$ in (\ref{0770}). Let $\pi $ be a
supporting plane of $\mathbb{G}$ such that $\mathbb{G}=
\mathbb{P}\cap\pi$ and $\mathbb{P}\subset\pi^+$. Then
$\mathfrak{p}_1,\mathfrak{p}_2\in\mathbb{B}\subset\mathbb{F}\subset\mathbb{P}\subset\pi^+$.
This implies that (2) in (\ref{jegy}) is impossible. So
  we have (1) in (\ref{jegy}), that is, $
\text{Ch}(\mathfrak{p}_1,\mathfrak{p}_2)\subset \pi$. Thus $
\text{Ch}(\mathfrak{p}_1,\mathfrak{p}_2)\subset
\pi\cap\mathbb{P}=\mathbb{G}$.
\end{proof}

\subsection{A Cone of Face}
\begin{definition}[Cones, Dual Face]\label{dualface}
Let $\mathbb{F} $ be a face of a polyhedron $\mathbb{P}$ in
$\mathbb{R}^n$.   Then the cone $\mathbb{F}^* $ of $\mathbb{F} $ is
defined by
\begin{eqnarray}\label{brg}
 \quad \mathbb{F}^*|\mathbb{P}  & =&
 \{\mathfrak{q}\in \mathbb{R}^n : \exists\, r\in\mathbb{R}  \ \
\text{such that}\ \ \mathbb{F}\subset \pi_{ \mathfrak{q},r } \cap
\mathbb{P} \ \text{and}\ \ \mathbb{P}\setminus \mathbb{F}\subset
\pi_{ \mathfrak{q},r }^+ \}\nonumber\\
&=& \{\mathfrak{q} \in \mathbb{R}^n : \exists\, r\in\mathbb{R}\
\text{such that}\ \langle\mathfrak{q},{\bf u}\rangle=r\le
\langle\mathfrak{q},{\bf y}\rangle\ \text{for all ${\bf u}\in
\mathbb{F}$,\ ${\bf y}\in \mathbb{P}\setminus \mathbb{F}$}\}.
\end{eqnarray}
The interior of a cone $ \mathbb{F}^*$  is
 the set of all nonzero  normal vectors   $\mathfrak{q}$ satisfying (\ref{4g}):
 \begin{eqnarray}\label{brg2}
 (\mathbb{F}^*)^{\circ}|\mathbb{P}  & =&\{\mathfrak{q}\in \mathbb{R}^n
 : \exists\, r\in\mathbb{R}\ \
\text{such that}\  \  \mathbb{F}= \pi_{ \mathfrak{q},r } \cap
\mathbb{P} \  \text{and}\ \ \mathbb{P}\setminus \mathbb{F}\subset
 \pi_{ \mathfrak{q},r }^+ \}\nonumber\\
 &=&\{\mathfrak{q}\in \mathbb{R}^n  : \exists\, r\in\mathbb{R}\ \
\text{such that}\  \  \mathbb{F}= \pi_{ \mathfrak{q},r } \cap
\mathbb{P} \  \text{and}\ \ \mathbb{P}\setminus \mathbb{F}\subset
 (\pi_{ \mathfrak{q},r }^+)^{\circ} \}\label{brr2}\\
&=& \{\mathfrak{q} \in \mathbb{R}^n : \exists\, r\in\mathbb{R}\
\text{such that}\ \langle\mathfrak{q},{\bf u}\rangle=r<
\langle\mathfrak{q},{\bf y}\rangle\ \text{for all ${\bf u}\in
\mathbb{F}$,\ ${\bf y}\in \mathbb{P}\setminus
\mathbb{F}$}\}.\nonumber
\end{eqnarray}
We use the notation $ \mathbb{F}^*|(\mathbb{P},V) $ when we restrict
$\mathfrak{q}$ in a given vector space $V$. Thus
$\mathbb{F}^*|\mathbb{P}=\mathbb{F}^*|(\mathbb{P},\mathbb{R}^n)$ in
(\ref{brg}). If not confused, we write just
 $\mathbb{F}^*$  instead of  $\mathbb{F}^*|\mathbb{P}$ or $\mathbb{F}^*|(\mathbb{P},\mathbb{R}^n)$.
  We note that $\mathbb{F}^*$ itself is a polyhedron in
$\mathbb{R}^n$ and $(\mathbb{F}^*)^{\circ}$ is an interior of
$\mathbb{F}^*$.
\end{definition}
\begin{remark}
To understand a cone $\mathbb{F}^*$ as a  dual face of $\mathbb{F}$, it is likely
that a cone of $\mathbb{F}$  is to be defined by
  the collection of all normal vectors $\mathfrak{q}$ satisfying (\ref{4g})
 as in (\ref{brg2}).
If so, the collection (\ref{brg2}) is an open set,   not a polyhedron anymore.
 To make $\mathbb{F}^*$ itself a polyhedron, we define a cone of $\mathbb{F}$ by
  (\ref{brg}) instead of  its interior  (\ref{brg2}).
\end{remark}

\begin{lemma}\label{lem2525}
Let $\mathbb{P}$ be a  polyhedron and $\mathbb{F},\mathbb{G}\in
\mathcal{F}(\mathbb{P})$. Then  $\mathbb{F}\preceq \mathbb{G} $ if
and only if $\mathbb{G}^*\preceq \mathbb{F}^*$.
\end{lemma}
\begin{proof}
We first show $\mathbb{F}\preceq \mathbb{G} $ implies that
$\mathbb{G}^*\preceq \mathbb{F}^*$. If $\mathbb{F}=\mathbb{G} $, we
are done. Let $\mathbb{F} \precneqq \mathbb{G} $. It suffices to
show that there exists $\mathfrak{q}\in \mathbb{R}^n$ and
$r\in\mathbb{R}$ such that
\begin{eqnarray}\label{kkll}
\langle\mathfrak{q}, {\bf u}\rangle=r< \langle\mathfrak{q}, {\bf v}\rangle\ \ \text{for
all ${\bf u}\in \mathbb{G}^*$ and ${\bf v}\in\mathbb{F}^*\setminus
\mathbb{G}^*$},
\end{eqnarray}
 which means that $\mathbb{G}^*\preceq
\mathbb{F}^*
 $ by (\ref{4g1}) in Definition \ref{dfac}. Choose $\mathfrak{q}=\mathfrak{n}-\mathfrak{m}$ with
$\mathfrak{n}\in \mathbb{G}\setminus \mathbb{F}$ and
$\mathfrak{m}\in\mathbb{F}$. Then $\langle\mathfrak{q},{\bf u}\rangle=0$
because $\mathfrak{m},\mathfrak{n}\in \mathbb{G}$ and ${\bf
u}\in\mathbb{G}^*$. By ${\bf v}\in
\mathbb{F}^*\setminus\mathbb{G}^*$ with $\mathfrak{m}\in\mathbb{F}$
and $\mathfrak{n}\in\mathbb{G}\setminus\mathbb{F}$,
$\langle\mathfrak{q}, {\bf v}\rangle> 0$ in view of Definition \ref{dualface}.
Therefore (\ref{kkll}) is proved. We next show that
$\mathbb{G}^*\preceq \mathbb{F}^*$  implies that $\mathbb{F}\preceq
\mathbb{G} $. Observe that if $\mathfrak{q}\in \mathbb{G}^*$, then
there exists unique $\rho=\inf\{\langle{\bf x}, \mathfrak{q}\rangle:{\bf
x}\in\mathbb{P}\}$ such that $\pi_{\mathfrak{q},\rho}$ is a
supporting plane of a face containing $\mathbb{G}$. Since
$\mathbb{G}$ is a face, there exists
$\mathfrak{q}\in(\mathbb{G}^*)^{\circ}\subset \mathbb{G}^*$.  By
Definition \ref{dualface},
$
\pi_{\mathfrak{q},\rho}\cap\mathbb{P}=\mathbb{G}.
$
From $\mathfrak{q}\in\mathbb{G}^*\subset\mathbb{F}^*$, it follows that
  $\mathbb{F}\subset
\pi_{\mathfrak{q},\rho}\cap\mathbb{P}=\mathbb{G}$, which yields
$\mathbb{F}\preceq\mathbb{G}$ by (1) of Lemma \ref{lemss}.
\end{proof}
\subsection{Generalized Newton Polyhedron}
For each  $S\subset  N_n=\{1,\cdots,n\}$, we define
\begin{eqnarray*}
\mathbb{R}_+^{S}&=&\{ (u_1,\cdots,u_n):u_j\ge 0\ \ \text{for}\ \
j\in S\ \ \text{and}\ \ \ u_j=0\
 \ \text{for}\ j\in N_n\setminus S\}.
 \end{eqnarray*}
\begin{definition}\label{ded28}
Let $\Omega $ be a finite subset of $\mathbb{Z}_+^n$ and $S\subset
N_n=\{1,\cdots,n\}$. We define a Newton polyhedron  ${\bf N}(\Omega,
 S)$ associated with  $\Omega$ and  $S$   by the convex hull containing
 $(\Omega+\mathbb{R}_+^{S}) $ in $\mathbb{R}^n$:
 $${\bf N}(\Omega,
S)=\text{Ch}\left( \Omega+\mathbb{R}_+^{S}  \right).$$ By
$\mathbb{R}_+^{\emptyset}=\{0\}$ and $\mathbb{R}_+^{N_n}=\mathbb{R}_+^n$,
we see that   ${\bf N}(\Omega,
\emptyset)=\rm{Ch}(\Omega)$, and
${\bf N}(\Omega, N_n)=\text{Ch}\left( \Omega+\mathbb{R}_+^{n}
\right)$ that is the usual Newton Polyhedron denoted by  ${\bf N}(\Omega)$.
Note  that ${\bf N}(\Omega, S)$  is a polyhedron in the sense of
Definition \ref{d12}. See Figure \ref{graph5}.
\end{definition}

  \begin{figure}
   \centerline{\includegraphics[width=10cm,height=7cm]{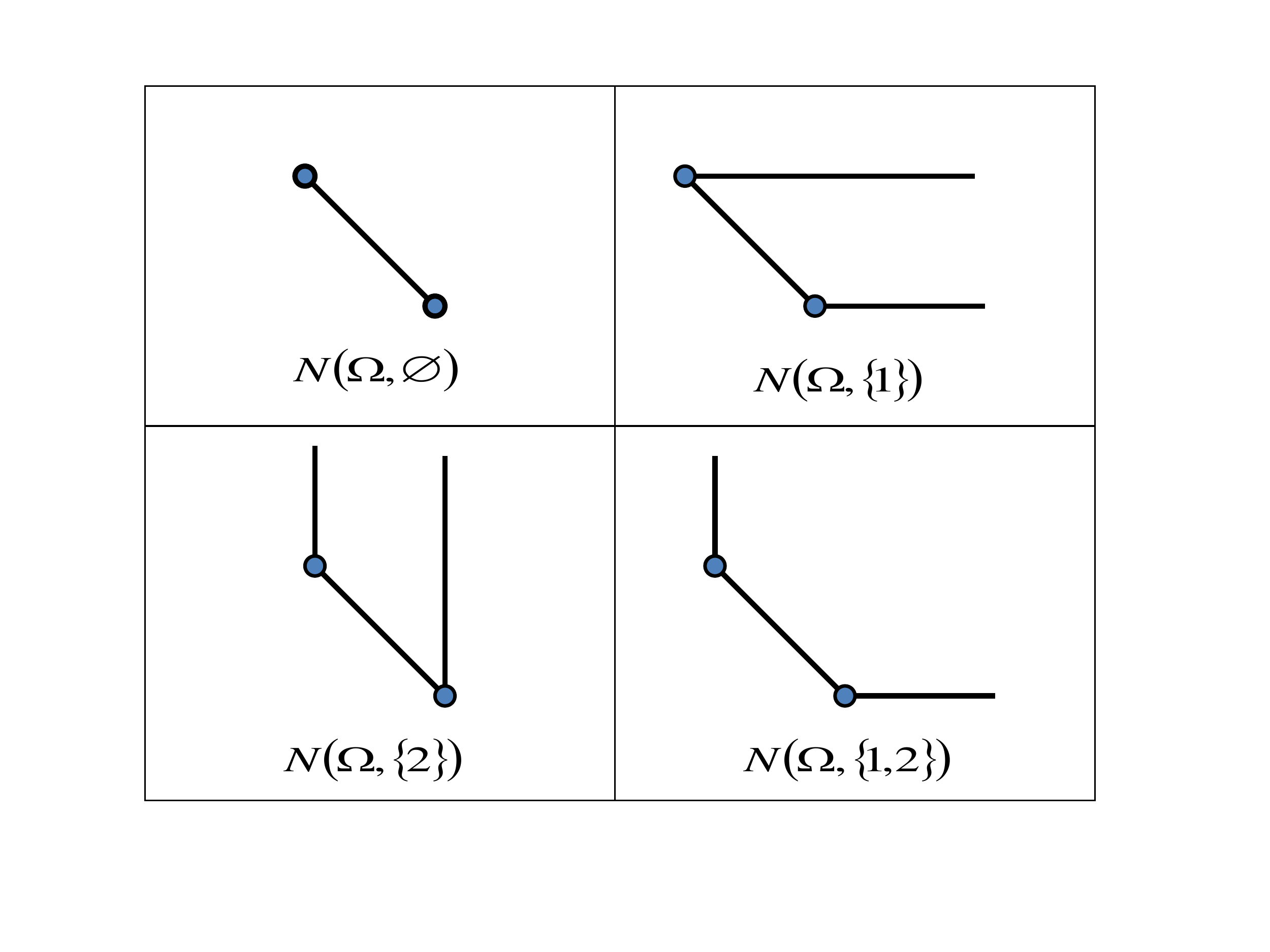}}
 \caption{Newton Polyhedra  ${\bf N}(\Omega,S)$ for $n=2$.} \label{graph5}
  \end{figure}

\begin{definition}
Let $\Lambda=(\Lambda_\nu)$ with $\Lambda_\nu\subset \mathbb{Z}_+^n$
and  $S\subset \{1,\cdots,n\} $. Then, the ordered $d$-tuple of
Newton polyhedra ${\bf N}(\Lambda_\nu,S)$'s is defined by
\begin{eqnarray*}
 \vec{{\bf N}}(\Lambda,S) =\left( {\bf N}(\Lambda_\nu,S)
 \right)_{\nu=1}^d.
\end{eqnarray*}
   To indicate a given polynomial $P=(P_\nu)\in\mathcal{P}_\Lambda$,
   we also
denote $   \vec{{\bf N}} (\Lambda,S)$  by $  \vec{{\bf N}}(P,S)$.
\end{definition}
\begin{definition}
 Let
$\Lambda=(\Lambda_\nu)$ with $\Lambda_\nu\subset \mathbb{Z}_+^n$ and
$S\subset \{1,\cdots,n\} $. We define the collection of
 $d$-tuples  of faces $\mathbb{F}_\nu\in\mathcal{F}({\bf N}(\Lambda_\nu,S))$
by
\begin{eqnarray*}
\mathcal{F}(  \vec{{\bf
N}}(\Lambda,S))=\{\mathbb{F}=\left(\mathbb{F}_1,\cdots,\mathbb{F}_d\right):\mathbb{F}_\nu\in\mathcal{F}({\bf
N}(\Lambda_\nu,S) )\}.
\end{eqnarray*}
For each $\mathbb{F}\in \mathcal{F}( \vec{{\bf N}}(\Lambda,S))$, we
denote $d$-tuple  of cones   by $\mathbb{F}^*=(\mathbb{F}_\nu^*)$.
\end{definition}

\subsection{Basic Decompositions According to Faces and Cones}
Choose $\,\psi\in C^\infty_c([-2,2])\,$ such that $\,0\leq \psi\leq
1\,$ and $\,\psi(u) =1\,$ for $\,|u|\leq 1/2\,.$ Put $\,\eta(u) =
\psi(u) -\psi(2u)\,$ and $\,h(u)=\eta(u)/u\,$ for $\,u\neq 0\,.$
Let $\,\Lambda = \left(\Lambda_1, \cdots, \Lambda_{d}\right)\,$ and
$P_\Lambda\in\mathcal{P}_{\Lambda}$.
For each $\mathbb{F}=(\mathbb{F}_\nu)\in\mathcal{F}( \vec{{\bf N}} (\Lambda,S))$
and $\,J\in\mathbb{Z}^n\,,$ define
\begin{eqnarray}\label{4.0}
  \mathcal{I}_J( P_{\mathbb{F}} ,\xi)&=& \int_{\mathbb{R}^n} \exp\left(i\,\sum_{\nu=1}^d
\left(\sum_{\mathfrak{m}\in\mathbb{F}_\nu\cap\Lambda_{\nu}}
c_{\mathfrak{m}}^\nu2^{-J\cdot\mathfrak{m}}
t^{\mathfrak{m}}\right)\xi_\nu
\right) \,\prod_{\ell=1}^nh(t_\ell)\,dt\\
&=& \int_{\mathbb{R}^n} \exp\left(i\, \sum_{\mathfrak{m}\in
\bigcup\mathbb{F}_\nu\cap\Lambda_\nu}2^{-J\cdot\mathfrak{m}}\langle\xi,
c_{\mathfrak{m}}\rangle t^{\mathfrak{m}} \right)
\,\prod_{\ell=1}^nh(t_\ell)\,dt\nonumber
\end{eqnarray}
where    $c_{\mathfrak{m}}=(c^{\nu}_{\mathfrak{m}})$ defined in
(\ref{vecpol}). We shall write $\mathcal{I}_J( P_{\Lambda} ,\xi)$ instead of
$\mathcal{I}_J(P_{{\bf
N}(\Lambda,S))} ,\xi) $.
\begin{definition}\label{dine}
 Given $S\subset N_n=\{1,\cdots,n\}$, we define
 $$1_S=(r_j)\ \text{where $r_i=1$ for $i\in S$ and $r_i=\infty$ for $i\in
N_n\setminus S$}$$ and
$$ Z(S)=\prod_{i=1}^n Z_i\ \ \text{where}\ \  Z_i=\mathbb{R}_+  \ \ \text{if}\
\ i\in S\ \ \text{and}\ \ Z_i=\mathbb{R}\ \ \text{if}\ \ i\in
 N_n \setminus S.$$
 \end{definition}
Then by using $Z_i$ above, we write
$$  \prod_{i\in S}\{-1<t_i<1\}\prod_{i\in N_n\setminus S}
\{-\infty<t_i<\infty\}=\prod_{i=1}^n\left(\bigcup_{k_i\in
Z_i\cap\,\mathbb{Z}}\{ |t_i|\approx 2^{-k_i}\}\right),$$ and make
the following dyadic decomposition:
\begin{eqnarray}\label{8765}
\mathcal{I}(P_{\Lambda} ,\xi,1_S) = \sum_{J\in
Z(S)\cap\mathbb{Z}^n}\mathcal{I}_J( P_{\Lambda} ,\xi).
\end{eqnarray}
As the name (dual face) tells, each $
J\in\mathbb{F}^*_\nu\cap\mathbb{Z}^n$ can be understood as a linear
functional
 mapping $\mathfrak{n}\in\mathbb{R}^n$ to $J\cdot \mathfrak{n}\in\mathbb{R}$
 satisfying the following dominating property:
\begin{equation}\label{dar}
2^{-  J\cdot\mathfrak{m} }=2^{-r(J)}\ge 2^{-  J\cdot\mathfrak{n}}\ \
\text{for all}\ \mathfrak{m}\in\mathbb{F}_\nu\ \text{and}\
\mathfrak{n}\in\mathbb{P}_\nu\setminus\mathbb{F}_\nu.
\end{equation}
Thus,  for $J\in\bigcap_{\nu=1}^d\mathbb{F}^*_\nu $ in (\ref{8765})
with the property  (\ref{dar}) in (\ref{4.0}),
\begin{equation}\label{dary}
\exists\, \alpha\in\mathbb{Z}_+^n\,,\ \ \left(\frac{\partial}{\partial
t}\right)^{\alpha}\left(\sum_{\nu=1}^d \left(\sum_{\mathfrak{m}\in
\Lambda_{\nu}} c_{\mathfrak{m}}^\nu2^{-J\cdot\mathfrak{m}}
t^{\mathfrak{m}}\right)\xi_\nu\right)\approx
2^{-J\cdot\mathfrak{m}_{\nu}} \xi_\nu \ \ \text{for all
$\mathfrak{m}_{\nu}\in \mathbb{F}_\nu\cap\Lambda_{\nu}$}.
\end{equation}
This combined with $Z(S)=\bigcup_{\mathbb{F}=(\mathbb{F}_\nu)\in
\mathcal{F}( \vec{{\bf N}} (\Lambda,S)) }\left(\bigcap_{\nu=1}^d \mathbb{F}^*_\nu\right)$
suggests us  to decompose in Section \ref{condec}
\begin{eqnarray*}
\sum_{J\in Z(S)\cap\mathbb{Z}^n}\mathcal{I}_J( P_{\Lambda} ,\xi)=
\sum_{ \mathbb{F} \in  \mathcal{F}( \vec{{\bf N}} (\Lambda,S)) }\sum_{J\in
 \bigcap_{\nu=1}^d\mathbb{F}^*_\nu \cap\mathbb{Z}^n}\mathcal{I}_J( P_{\Lambda}
,\xi),
\end{eqnarray*}
and  next prove in Sections \ref{sec9} and \ref{suffes} that for
each $\mathbb{F}=(\mathbb{F}_\nu)\in  \mathcal{F}( \vec{{\bf N}}
(\Lambda,S))$,
\begin{eqnarray}
\qquad \sum_{s=1}^N\sum_{J\in
 \bigcap_{\nu=1}^d\mathbb{F}^*_\nu \cap\mathbb{Z}^n}\left|\mathcal{I}_J( P_{\mathbb{F}(s-1)}
,\xi)- \mathcal{I}_J( P_{\mathbb{F}(s)} ,\xi) \right|+\sum_{J\in
 \bigcap_{\nu=1}^d\mathbb{F}^*_\nu \cap\mathbb{Z}^n}  \left|\mathcal{I}_J(
 P_{\mathbb{F}}
,\xi) \right|\le C. \label{danbi5}
\end{eqnarray}
Here  $\mathbb{F}(s)=(\mathbb{F}_\nu(s))$ will be chosen in an
suitable way so that $\mathbb{F}_\nu(s-1) \succeq
\mathbb{F}_\nu(s)$ with $\nu=1,\cdots,d$ where  $\vec{{\bf
N}}(\Lambda,S)=\mathbb{F}(0)$ and $\mathbb{F}(N)=\mathbb{F}$ as in (\ref{pathg}).

\section{Main Theorem and Background}\label{evenset}
In order to state main results, we first try to find an appropriate
condition
  on an exponent set
$\bigcup_{\nu=1}^d\mathbb{F}_\nu\cap\Lambda_\nu$ which guarantees
   $ \mathcal{I}_J( P_{\mathbb{F}} ,\xi)\equiv 0$.
\subsection{Even Sets} Let
 $\bigcup_{\nu=1}^d\mathbb{F}_\nu\cap\Lambda_\nu=\left\{\mathfrak{m}_1,\cdots,\mathfrak{m}_N \right\}.$
 Suppose
every vector $\mathfrak{m}$ of the form $$
 \alpha_1\mathfrak{m}_1+\cdots+\alpha_N\mathfrak{m}_N\ \text{with}\ \alpha_j=0\
\text{or}\ \ 1$$ has at least one even component. Then,  the Taylor
expansion of the exponential function in (\ref{4.0}) yields that
\begin{eqnarray}\label{m1.1}
\mathcal{I}_J( \mathbb{F},\xi)&=&
\sum_{k=0}^{\infty}\int_{\mathbb{R}^n} \, \frac{
\left(i\sum_{\mathfrak{m}\in\bigcup\mathbb{F}_\nu\cap\Lambda_{\nu}}2^{-J\cdot\mathfrak{m}}\langle\xi,
c_{\mathfrak{m}}\rangle
t^{\mathfrak{m}}\right)^k}{k!}\,\prod_{\ell=1}^nh(t_\ell)dt\, \nonumber\\
&=&\sum_{k=0}^{\infty}\int_{\mathbb{R}^n} \,
\sum_{\alpha_1+\cdots+\alpha_N=k}C(J,\mathfrak{m},\alpha,\xi)\frac{
t^{\alpha_{1}\mathfrak{m}_1+\cdots+\alpha_{1}\mathfrak{m}_N
}}{k!}\,\prod_{\ell=1}^nh(t_\ell)dt=0
\end{eqnarray}
since $h(t_\ell)$ is an odd function for each $\ell=1,\cdots,n$.
This observation leads to the following notions of even and odd sets
in $\mathbb{Z}_+^n$. Let  $\Omega= \{\,\mathfrak{m}_1, \cdots,
\mathfrak{m}_N\,\}\subset\mathbb{Z}^n_+$ and let the class of sum of  vectors in $\Omega$ be
\begin{equation*}\label{411}
\Sigma(\Omega)=\left\{\alpha_1\mathfrak{m}_1+\cdots+\alpha_N\mathfrak{m}_N:\alpha_j=0\
\text{or}\ \ 1\right\}.
\end{equation*}
\begin{definition}\label{def1.3}
A finite subset $\Omega= \{\,\mathfrak{m}_1, \cdots,
\mathfrak{m}_N\,\} $ of $\mathbb{Z}_+^n$ is said to be   {\bf
odd} iff there exists at least one vector
$\mathfrak{m} \in\Sigma(\Omega)$
all of whose components are odd numbers such that
$$
\mathfrak{m}= (\mbox{odd},\cdots,\mbox{odd}) .$$
\end{definition}
\begin{definition}\label{def1.4}
A finite subset $\Omega$ of $\mathbb{Z}_+^n$ is said to be    {\bf
even
 }  iff $\Omega$ is not odd, that is,  every  $\mathfrak{m}=(m_1,\cdots,m_n) \in
\Sigma(\Omega)$ has at least one even numbered component $m_j$.
\end{definition}
\begin{exam}
In $\mathbb{Z}_+^3$,  let  $A =\{(1,1,0), (3,2,1)\}$,
 and $B=\{ (1,1,0),(0,0,3)  \} $. Then $A$ is an even set and
 $B$ an odd set.
 Notice that $A$ is an even set,  though there is no $k\in\{1,2,3\}$
 such that  $k^{th}$ component of every vector in $A$ is even.
\end{exam}
In (\ref{m1.1}), we have proved the following proposition:
\begin{proposition}\label{pr5g}
Suppose that $\bigcup_{\nu=1}^d(\mathbb{F}_\nu\cap\Lambda_\nu)$ is
an even set. Then $ \mathcal{I}_J( \mathbb{F},\xi)\equiv 0$
\end{proposition}
We shall perform  the estimates  (\ref{danbi5}) by using
 a full rank condition of
$\bigcup\mathbb{F}_\nu(s)$ (formulated in Proposition \ref{propyy})
or vanishing property in Propositions \ref{pr5g}. Thus, the evenness
condition in Propositions \ref{pr5g} shall be imposed on
 the only faces  contained in the subclass
  $\mathcal{A}$ of $\mathcal{F}(\vec{{\bf N}}(\Lambda,S))$ satisfying
  the following two conditions:
\begin{eqnarray}\label{4.1vo}
&& \text{ {\bf Low Rank Condition}: \ \
$\text{rank}\left(\bigcup_{\nu=1}^d \mathbb{F}_\nu \right)\le n-1$
for
  $\mathbb{F}\in\mathcal{A}$,}
\\
&&  \text{ {\bf Overlapping Cone  Condition}: \ \ $
 \bigcap_{\nu=1}^d(\mathbb{F}_\nu^*)^{\circ}  \ne \emptyset\ $  for
$\mathbb{F}\in\mathcal{A}$}\label{4.1dd}
\end{eqnarray}
where the overlapping cone condition comes  from the decompositions
in $J$ and the dominating condition (\ref{dar}).
\subsection{Statement of Main Results}
We start with the simplest case $d=1$. Observe for this case that
$\bigcap_{\nu=1}^d \left(\mathbb{F}^*_{\nu}\right)^{\circ}
\ne\emptyset$ always   holds whenever $\text{rank}\left(
\bigcup_{\nu=1}^d \mathbb{F}_\nu \right)\le n-1.$
\begin{main}\label{main18}
Let $\Lambda \subset\mathbb{Z}_+^n$ and $S\subset N_n$. Suppose that $d=1$ in (\ref{vecpol}).  Let
$1<p<\infty$.
 Then
  $$ \text{for all $P \in\mathcal{P}_\Lambda$,}\ \ \exists \, C_{P}>0  \
  \text{such that}\ \sup_{r\in I(S)}\left\|
  \mathcal{H}^{P}_{r}\right\|_{L^p(\mathbb{R}^1)\rightarrow L^p(\mathbb{R}^1)}\le
  C_{P} $$
 if
and only if
$\mathbb{F} \cap\Lambda$
 is an even set for $\mathbb{F}\in\mathcal{F}({\bf N}(\Lambda,S))$
 whenever $\text{rank}\left( \mathbb{F} \right)\le
 n-1$.
\end{main}
For $d>1$, the overlapping condition (\ref{4.1dd}) is crucial as
well as the  rank condition (\ref{4.1vo}).
\begin{definition}\label{ird}
 Given $\vec{{\bf
N}}(\Lambda,S)=\left({\bf N}(\Lambda_\nu,S)\right)_{\nu=1}^d$, we
set the collection of all $d$-tuples of faces satisfying both
  low rank condition  (\ref{4.1vo}) and overlapping  (\ref{4.1dd}) by
\begin{eqnarray*}\label{336t}
\mathcal{F}_{\rm{lo}}( \vec{{\bf N}}(\Lambda,S))=\left\{
(\mathbb{F}_\nu) \in \mathcal{F}\left(\vec{ {\bf
N}}(\Lambda,S)\right):
 \text{rank}\left(  \bigcup_{\nu=1}^d \mathbb{F}_\nu  \right)\le n-1\
\text{and}  \  \bigcap_{\nu=1}^d
\left(\mathbb{F}^*_{\nu}\right)^{\circ}  \ne\emptyset \right\}.
\end{eqnarray*}
\end{definition}
We assume first that $\Lambda_{\nu}$'s are mutually disjoint such
that $\Lambda_{\mu}\cap\Lambda_{\nu}=\emptyset$ for any $\mu\ne
\nu$.
\begin{main}\label{main3}
Let $\Lambda=(\Lambda_1,\cdots,\Lambda_d)$ with
$\Lambda_\nu\subset\mathbb{Z}_+^n$ and $ S\subset N_n$. Suppose that
$\Lambda_{\nu}$'s are mutually disjoint. Let $1<p<\infty$.
 Then
  $$ \text{for all $P \in\mathcal{P}_\Lambda$,}\ \ \exists \, C_{P}>0
  \ \text{such that}\ \sup_{r\in I(S)}\left\|
  \mathcal{H}^{P}_{r}\right\|_{L^p(\mathbb{R}^d)\rightarrow L^p(\mathbb{R}^d)}\le
  C_{P} $$
 if
and only if\ \
$
\bigcup_{\nu=1}^d (\mathbb{F}_\nu \cap\Lambda_\nu)$
 is an even set for $\mathbb{F}=(\mathbb{F}_\nu)
\in\mathcal{F}_{\rm{lo}}( \vec{{\bf N}}(\Lambda,S))$.
\end{main}

\begin{remark}
Main Theorems 1 and 2 do  not give a criteria  for the boundedness with a
given individual polynomial $P_\Lambda$,
but enables us to determine the boundedness
 for universal polynomials $P_\Lambda$ with a set $\Lambda$ of exponents fixed.
Also, Main Theorems 1 and 2 do not give a condition for  the boundedness of
$\left\|
  \mathcal{H}^{P}_{r}\right\|_{L^p(\mathbb{R}^d)\rightarrow
  L^p(\mathbb{R}^d)}$ with fixed $r$, but for the uniform boundedness
 $\sup_{r\in I(S)}\left\|
  \mathcal{H}^{P}_{r}\right\|_{L^p(\mathbb{R}^d)\rightarrow
  L^p(\mathbb{R}^d)}$.
  It is interesting to know if  $\sup_{r\in I(S)}\left\|
  \mathcal{H}^{P}_{r}\right\|_{L^p(\mathbb{R}^d)\rightarrow
  L^p(\mathbb{R}^d)}$
   can be replaced by $ \left\|
  \mathcal{H}^{P}_{1_S}\right\|_{L^p(\mathbb{R}^d)\rightarrow
  L^p(\mathbb{R}^d)}$ in the above theorems where $1_S$ is defined in Definition \ref{dine}.
\end{remark}
Let  $P_\Lambda$ be a form of  graph  $(t_1,\cdots,t_n,P_{n+1}(t))$
so that
 $\Lambda=(\{{\bf e}_1\},\cdots,\{{\bf e}_n\},\Lambda_{n+1})$.
For this case, we are able to show that the $L^p$ boundedness of
$\mathcal{H}^{P_\Lambda}_{1_S}$ and the uniform $L^p$ boundedness of
$\mathcal{H}^{P_\Lambda}_{r}$ in $r\in I(S)$ are equivalent. Moreover,
we do not need the  overlapping condition (\ref{4.1dd}), since we
can make the  condition $ \bigcap_{\nu=1}^d
\left(\mathbb{F}^*_{\nu}\right)^{\circ}\ne\emptyset$ always hold.
\begin{corollary}\label{ccoo}
Let $1<p<\infty$ and let $\Lambda=(\{{\bf e}_1\},\cdots,\{{\bf e}_n\},\Lambda_{n+1})$ and $S\subset N_n$.
 Then
 \begin{eqnarray*}\label{4dd4}
 \text{for all $P \in\mathcal{P}_\Lambda$,}\ \ \exists \, C_{P}>0  \ \text{such that}\ \ \left\|
  \mathcal{H}^{P}_{1_S}\right\|_{L^p(\mathbb{R}^d)\rightarrow L^p(\mathbb{R}^d)}\le
  C_{P}
  \end{eqnarray*}
 if and only if $
  (\mathbb{F}_{n+1} \cap\Lambda_{n+1})\cup A$, for
 $\mathbb{F}_{n+1} \in\mathcal{F}({\bf N}(\Lambda_{n+1},S))$ and $A\subset\{{\bf e}_1,\cdots,{\bf e}_n\}$,
 is an even set
whenever   $\text{rank}\left( \mathbb{F}_{n+1} \cup A\right)\le n-1$.
\end{corollary}
\begin{remark}
The above evenness condition in Corollary \ref{ccoo}  is equivalent to
\begin{eqnarray*}
\bigcup_{\nu=1}^d (\mathbb{F}_\nu \cap\Lambda_\nu)\ \text{
 is an even set whenever    $\text{rank}\left(\bigcup_{\nu=1}^d\mathbb{F}_\nu\right)\le
 n-1$  for $\mathbb{F}\in\mathcal{F}(\vec{{\bf N}}(\Lambda,S))$.}
\end{eqnarray*}
\end{remark}
  We   exclude the assumption of mutually disjointness of $\Lambda_\nu$'s in
  the hypotheses of Main Theorem \ref{main3}. Let $P=(P_\nu)_{\nu=1}^d$ be a vector polynomial.
  For each $\nu=1,\cdots,d$, we define a set $\Lambda(P_\nu)$ to be a set of all
   exponents  of the monomials in $P_\nu$:
  $$\Lambda(P_{\nu})=\left\{ \mathfrak{m} \in \mathbb{Z}^n_+
:c^{\nu}_{\mathfrak{m} } \ne 0\ \text{in $P_{\nu}(t)=\sum
c^{\nu}_{\mathfrak{m} }t^{\mathfrak{m}}$}\right\}.$$
 Moreover, we denote  a $d$-tuple $(\Lambda(P_\nu))_{\nu=1}^d$   by $\Lambda(P)$.
Denote  the set of $d\times d$ invertible matrices by $GL(d)$.
 For $A\in GL(d)$ and $P \in\mathcal{P}_\Lambda$ with $P(t)=(P_{1}(t),\cdots, P_{d}(t))$, we
 let  $AP$   be a vector polynomial given by  the matrix multiplication
$$AP(t)= \left(\sum_{\mathfrak{m}\in
\Lambda((AP)_\nu)} a^{\nu}_{\mathfrak{m}}
t^{\mathfrak{m}}\right)_{\nu=1}^d\ \ \text{for some
$a_{\mathfrak{m}}^\nu\ne 0$} $$ where we regard  $P(t)$ and $AP(t)$
above as column vectors. Then
 $AP \in \mathcal{P}_{\Lambda'}$ where
 $\Lambda'= \Lambda(AP )$.
If $A=I$ an identity matrix,
$\Lambda'=\left(\Lambda((AP)_{\nu})\right)_{\nu=1}^d=
\left(\Lambda(P_{\nu})\right)_{\nu=1}^d=(\Lambda_\nu)_{\nu=1}^d=\Lambda$.
\begin{definition}\label{de0303}
Let $P \in\mathcal{P}_\Lambda$ where $\Lambda=(\Lambda_\nu)$ with
$\Lambda_\nu\subset \mathbb{Z}_+^n$ and  $S\subset \{1,\cdots,n\} $.
Let $A\in GL(d)$.  Given a vector polynomial  $AP$,
 we consider the $d$-tuple of Newton polyhedrons
  \begin{eqnarray*}
  \vec{{\bf N}}(AP,S)=  \left({\bf
 N}((AP)_\nu,S) \right)_{\nu=1}^d
 \end{eqnarray*}
   and $d$-tuple of their faces $$\mathcal{F}(\vec{{\bf
 N}}(AP,S))=\left\{\mathbb{F}_A=((\mathbb{F}_A)_1,\cdots,(\mathbb{F}_A)_d):
 (\mathbb{F}_A)_\nu\in  \mathcal{F}\left({\bf
 N}((AP)_\nu,S)\right)\right\}.$$
 \end{definition}
\begin{main}\label{main4}
Let $\Lambda=(\Lambda_1,\cdots,\Lambda_d)$ with
$\Lambda_\nu\subset\mathbb{Z}_+^n$ and $S\subset N_n$.    Let
$1<p<\infty$.
$$\text{For  all $P \in\mathcal{P}_\Lambda$}\ \ \exists\ C_{P}>0  \ \text{such that}\
  \sup_{r\in I(S)}\left\|\mathcal{H}^{P}_{r}\right\|_{L^p(\mathbb{R}^d)\rightarrow L^p(\mathbb{R}^d)}\le
  C_{P}\ \  $$
 if
and only if  for all $A\in GL_d$ and $P\in\mathcal{P}_\Lambda,$
\begin{eqnarray}
\qquad \bigcup_{\nu=1}^d(\mathbb{F}_A)_\nu\cap \Lambda((AP)_\nu) \
\text{
 is an even set  whenever $\mathbb{F}_A= ((\mathbb{F}_A)_\nu)\in
 \mathcal{F}_{\rm{lo}}( \vec{{\bf N}}(AP,S))$ }\label{enm1}
\end{eqnarray}
where  the class $\mathcal{F}_{\rm{lo}}( \vec{{\bf N}}(AP,S))$ is
defined as in Definition \ref{ird}.
\end{main}
\begin{remark}
For the case  $n= 2$ in Main Theorems \ref{main3} and \ref{main4},   the overlapping
condition of cones in $\mathcal{F}_{\rm{lo}}( \vec{{\bf N}}(\Lambda,S))$
does not have to appear explicitly. By omitting the overlapping cone condition
 in $\mathcal{F}_{\rm{lo}}( \vec{{\bf N}}(\Lambda,S))$, we let
  $$\mathcal{F}_{\rm{l}}( \vec{{\bf N}}(\Lambda,S))=\left\{
(\mathbb{F}_\nu) \in \mathcal{F}\left(\vec{ {\bf
N}}(\Lambda,S)\right):
 \text{rank}\left(  \bigcup_{\nu=1}^d \mathbb{F}_\nu  \right)\le n-1  \right\}
 \supset  \mathcal{F}_{\rm{lo}}( \vec{{\bf N}}(\Lambda,S)).
$$ Then, for the case $n=2$, the evenness condition for
$ \mathcal{F}_{\rm{lo}}( \vec{{\bf N}}(\Lambda,S))$ is equivalent
 to the condition for $ \mathcal{F}_{\rm{l}}( \vec{{\bf N}}(\Lambda,S))$.
  It suffices to show $\Rightarrow$. Suppose that $ \bigcup_{\nu=1}^d \mathbb{F}_\nu \cap\Lambda_\nu$ is an odd set with
$\text{rank}\left(  \bigcup_{\nu=1}^d \mathbb{F}_\nu  \right)\le 1$. Then
there exists $\mu$ such that  $\mathbb{F}_\mu\cap\Lambda_\mu$ has a
point $(odd,odd)$, because both of two points $(even,odd),(odd,even)$
can not lie in the one line passing through the origin. Therefore,
  $\mathbb{G}=(\mathbb{G}_\nu) $ defined by $\mathbb{G}_\mu=\mathbb{F}_\mu$
   and $\mathbb{G}_\nu=\emptyset$ for $\nu\ne \mu$ satisfies that
    $\mathbb{G}\in \mathcal{F}_{\rm{lo}}( \vec{{\bf N}}(\Lambda,S))$
     and $ \bigcup_{\nu=1}^d \mathbb{G}_\nu \cap\Lambda_\nu$ is an odd set.
\end{remark}
\begin{remark}\label{r32}
For the case $n\ge 3$, the overlapping condition is crucial in
 Main Theorems \ref{main3} and \ref{main4}. Moreover, we
 note that it is not  just cones $\bigcap_{\nu=1}^d\mathbb{F}_\nu^*$,
 but  their interiors $\bigcap_{\nu=1}^d (\mathbb{F}_\nu^*)^{\circ}$
 that satisfy the overlapping condition  (\ref{4.1dd}).
The   example \ref{ex32} in Section \ref{secfar} shows that
  there exists $\mathbb{F} \in \mathcal{F}_{\rm{l}}( \vec{{\bf N}}(\Lambda,S))$
  such that  $\bigcap_{\nu=1}^d\mathbb{F}_\nu^*  \ne\emptyset$ and
$\bigcup_{\nu=1}^d (\mathbb{F}_\nu\cap\Lambda_\nu)$ is an odd set, but
$$\text{ for all $P \in\mathcal{P}_\Lambda$ }\ \ \sup_{r\in I(S)}\left\|
  \mathcal{H}^{P}_{r}\right\|_{L^p(\mathbb{R}^d)\rightarrow L^p(\mathbb{R}^d)}\le
  C_{P}. $$
 \end{remark}

\subsection{Background} In the one parameter case ($n=1$), the operator
 $\mathcal{H}^{P_\Lambda}_{r}$ with $r=(1,\cdots,1)$ can be regarded as a particular
instance of singular integrals along curves satisfying finite type
condition in E. M. Stein and S. Wainger \cite{SW}. The $L^p$ theory
of those singular integrals has been developed quite well.  For
example, see  M. Christ, A. Nagel, E. M. Stein and S. Wainger
\cite{CNSW} for singular Radon transforms with the curvature
conditions in a very general setting. See also M. Folch-Gabayet and
J. Wright \cite{FW} for the case that phase functions $P_\Lambda$ are given by
rational functions.

In the multi-parameter case ($\,n\geq 2\,$), it is A. Nagel and S.
Wainger \cite{NW} who introduced the (global) multiple Hilbert
transforms along surfaces having certain dilation invariance
properties and obtained their $L^2$ boundedness. In \cite{RS}, F.
Ricci and E. M. Stein established an $L^p$ theorem for
multi-parameter singular integrals whose kernels satisfy more
general dilation structure.  A special case of their results implies
that if $\Lambda=\left(\{{\bf e}_1\},\cdots, \{{\bf
e}_{n}\},\{\mathfrak{m}\}\right)$ where at least $n-1$ coordinates
of $\mathfrak{m}$ are even,    then
 $ \|\mathcal{H}^{P_\Lambda}_{1_S} \|_{L^p(\mathbb{R}^{n+1})\rightarrow L^p(\mathbb{R}^{n+1})}$
    are bounded   for $1<p<\infty$.
In \cite{CWW}, A. Carbery, S. Wainger and J. Wright obtained a
necessary and sufficient condition for $L^p(\mathbb{R}^3)$
boundedness of $\mathcal{H}_{1_S}^{\Lambda}$ with $S=\{1,2\}$,
 $\Lambda=(\{{\bf e}_1\}, \{{\bf e}_2\},\Lambda_3)$
 where $d=3$ and $n=2$. Their theorem states that
\begin{theorem}[Double Hilbert transform \cite{CWW}]\label{T1}
 Let $\,\Lambda=\left(\{{\bf e}_{1}\},\{{\bf e}_2\},
 \Lambda_{3}\right)$ and $S=\{1,2\}$ with $n=2$ and $d=3$.
  For $1<p<\infty$, the local double
Hilbert transform $\mathcal{H}^{P_\Lambda}_{1_S}$  is bounded in
$L^p(\mathbb{R}^3)$ if and only if every vertex $\mathfrak{m}$ in $
{\bf N}(\Lambda_3,S) $ has at least one even component.
\end{theorem}
 S. Patel \cite{P2} extends this result to $S=\emptyset$ corresponding
 to the global Hilbert transform.
\begin{theorem}[Double Hilbert transform \cite{P2}]\label{T2}
 Let $\,\Lambda=\left(\{{\bf e}_{1}\},\{{\bf e}_2\},
 \Lambda_{3}\right)$ and $S=\{\emptyset\}$ with $n=2$ and $d=3$.
  For $1<p<\infty$, the global double
Hilbert transform $\mathcal{H}^{P_\Lambda}_{1_S}$  is bounded in
$L^p(\mathbb{R}^3)$ if and only if  every vertex $\mathfrak{m}$ in $
{\bf Ch}(\Lambda_3) $ and every edge $E$ in $ {\bf Ch}(\Lambda_3) $
passing through the origin has at least one even component.
\end{theorem}
S. Patel \cite{P1} also studies the case $n=2$ and $d=1$. He has
shown that the necessary and sufficient condition   for the $L^p$
boundedness of $\mathcal{H}^{P_\Lambda}_{1_S}$ cannot be determined
by only the geometry of ${\bf N}(\Lambda)$ but by coefficients of
the given polynomial $P_\Lambda(t)$. More precisely, the condition
is described in terms of not  a single vertex $\mathfrak{m}$ and its
coefficient $c_{\mathfrak{m}}$ in $P_\Lambda$, but   the sum of quantities
associated with many vertices and their corresponding coefficients:
\begin{theorem}[Double Hilbert transform \cite{P1}]\label{T33}
 Let $\,\Lambda=\left(
 \Lambda_{1}\right) $ and $S=\{1,2\}$ with $n=2$ and $d=1$.
Then, the local double
Hilbert transform $\mathcal{H}^{P_\Lambda}_{1_S}$  is bounded in
$L^p(\mathbb{R}^1)$  for $1<p<\infty$ if and only if
\begin{eqnarray}\label{sppt}
\sum_{(m_j,n_j)\in \mathcal{F}^0_{odd} }\frac{\text{sgn}(a_{m_jn_j})}{m_jn_j}\left(\begin{matrix}
m_j&0\\
0&n_j\end{matrix}
\right)(\mathfrak{n}_j^+-\mathfrak{n}_j^-)=0.
\end{eqnarray}
where $\mathcal{F}^0_{odd}=\left\{(m_j,n_j)\in\mathcal{F}^0({\bf
N}(\Lambda)): (m_j,n_j)=(odd,odd)\right\}$, and
$(m_j,n_j)\in\mathcal{F}^0_{odd}$ is  an intersection
$\mathbb{F}_{j}^{-}\cap \mathbb{F}_j^{+}$ of two facets
$\mathbb{F}_j^{\pm}=\pi_{\mathfrak{n}^{\pm}_j,1}\cap{\bf
N}(\Lambda)\in\mathcal{F}^1({\bf N}(\Lambda))$.
\end{theorem}
  A. Carbery, S. Wainger and J. Wright  \cite{CWW-elescorial}
 obtain the asymptotic behaviors of the oscillatory singular
 integrals
associated with   analytic phase functions $P(t_1,t_2)$, which
extends  Theorem \ref{T1}   to the class of analytic functions. They
\cite{CWW-elescorial} also find an example of finite type surface
$(t_1,t_2,P(t_1,t_2))$ with its formal Taylor series satisfying
evenness hypothesis, however $\mathcal{H}_{r}^{P}$ not bounded in
$L^2(\mathbb{R}^3)$. We also refer to \cite{CHKY2} dealing with a
certain class of flat surfaces $(t_1,t_2,P(t_1,t_2))$  without any
curvature. In the general setting of
polynomial surfaces defining the Double Hilbert transform,
  M. Pramanik and C. W. Yang   \cite{MY} obtain the $L^p$
theorem:
  \begin{theorem}[Double Hilbert transform \cite{MY}] \label{T4}
  Let $\,\Lambda=\left(\{{\bf e}_1\},\{{\bf e}_2\},
 \Lambda_{1},\cdots,\Lambda_{k}\right) $ and $S=\{1,2\}$ with $n=2$ and $d=k+2$.
 Suppose that $P(t_1,t_2)=(P_{\Lambda_1}(t),\cdots,P_{\Lambda_k}(t))$ and
  $P_{\Lambda}=(t_1,t_2,P(t_1,t_2))$.
   For $1<p<\infty$, the local double
Hilbert transform $\mathcal{H}^{P_\Lambda}_{1_S}$  is bounded in
$L^p(\mathbb{R}^{d})$ if and only if for every $A\in GL(k)$, every
vertex $\mathfrak{m}\in {\bf N}((AP)_\nu,S)$ with $\nu=1,\cdots,k$ has at least one
even numbered component.
  \end{theorem}
  \begin{remark}
 The results of Main Theorems \ref{main3} and \ref{main4} for $n=2$ and $S=\{1,2\}$
 follows from Theorem \ref{T4} by a slight modification of its necessary proof.
  \end{remark}
The triple Hilbert transforms $\mathcal{H}^{P_\Lambda}_{1_S}$ with
$S=\{1,2,3\}$ and $\Lambda=(\{{\bf e}_1\},\{{\bf e}_2\},\{{\bf
e}_3\},\Lambda_4)$ were studied in the two papers \cite{CWW}
\cite{CHKY} published in 2009. In \cite{CWW}, A. Carbery, S. Wainger
and J. Wright have discovered a remarkable differences between the
triple and the double Hilbert transforms. The $L^2$ boundedness of
the triple Hilbert transform $\mathcal{H}^{P_\Lambda}_{1_S}$ depends
on
 the coefficients of $P_{\Lambda}$ as well as the Newton polyhedron ${\bf
N}(\Lambda_4)$,
  whereas that of the double Hilbert transform depends only on the Newton polygon ${\bf
N}(\Lambda_3)$. They establish two types of theorems.  First one
gives the necessary and sufficient condition that  the operators
$\mathcal{H}^{P_\Lambda}_{1_S}$ are bounded in $L^2$ for all class
of polynomials $P_\Lambda\in\mathcal{P}(\Lambda)$ when   $\Lambda$
is given. This theorem is called the universal theorem. The second
theorem is to inform the necessary and sufficient condition that the
one individual operator $\mathcal{H}^{P_\Lambda}_{1_S}$ is bounded
in $L^2$ when a polynomial $P_\Lambda$ is given. This theorem is
called the individual theorem. The condition of the first theorem is
expressed solely in terms of ${\bf N}(\Lambda_4)$ but that of the
second in terms of individual coefficients of given polynomial
$P_\Lambda$ in question. Here we only state their universal theorem.
\medskip
\begin{theorem}[Triple Hilbert transform \cite{CWW}]\label{T5}  Let
$\,1<p<\infty\,.$
 Given $S=\{1,2,3\}$ and $\,\Lambda=\left(\{{\bf e}_{1}\},\{{\bf e}_2\},\{{\bf
e}_3\},\Lambda_{4}\right)$, suppose that
\begin{itemize}
\item[(H1)] Every entry of a vertex in ${\bf N}(\Lambda_{4},S)$ is positive.
\item[(H2)]
\begin{itemize}
\item[(a)] Each edge ${\bf N}(\Lambda_{4},S)$ is not contained on any hyperplane parallel to a
coordinate plane.
\item[(b)] The projection of the line containing an edge ${\bf N}(\Lambda_{4},S)$ onto a
coordinate plane does not pass through the origin.
\end{itemize}
\item[(H3)] The plane determined by any three vertices in ${\bf N}(\Lambda_{4},S)$ does not contain the
origin.
\end{itemize}
Then the triple Hilbert transform $\mathcal{H}^{P_\Lambda}_{1_S}$ is
bounded in $L^2(\mathbb{R}^4)$   for all
$P_\Lambda\in\mathcal{P}_{\Lambda}$ if and only if every vertex in
${\bf N}(\Lambda_4,S)$   has at least two even entries,  and every
edges $E$ of ${\bf N}(\Lambda_{4},S)$ satisfies that there exists a
one component such that the entry of that component of every vector
in $E\cap\Lambda_4$ is even.
\end{theorem}
\begin{remark}
They found a vector polynomial
$P_\Lambda(t)=(t_1,t_2,t_3,P_{\Lambda_4}(t))$   such that the
corresponding triple Hilbert transform
$\mathcal{H}^{P_\Lambda}_{1_S}$ is bounded on $L^2(\mathbb{R}^4)$
although ${\bf N}(\Lambda_{4},S)$ breaks the above evenness
condition.
\end{remark}
In \cite{CHKY},  Y.K. Cho, H. Hong,  C.W. Yang and the author proved
the theorem without assuming the three hypotheses H1-H3 so that
\begin{theorem}[Triple Hilbert transform \cite{CHKY}]\label{T6}  Let
$\,1<p<\infty\,.$
 Given $S=\{1,2,3\}$ and $\,\Lambda=\left(\{{\bf e}_{1}\},\{{\bf e}_2\},\{{\bf
e}_3\},\Lambda_{4}\right), $   the triple Hilbert transform
$\mathcal{H}^{P_\Lambda}_{1_S}$  is bounded in $L^p(\mathbb{R}^4)$
for all $P_\Lambda\in\mathcal{P}_{\Lambda}$ if and only if every
$\mathbb{F}\in\mathcal{F}({\bf N}(\Lambda_4))$ for
$\text{rank}\left((\mathbb{F}\cap \Lambda_4)\cup A\right)\le 2$ with
$A\subset\{{\bf e}_1,{\bf e}_2,{\bf e}_3\}$  satisfies that there
exists a one component such that the entry of that component of
every vector in $\mathbb{F}\cap\Lambda_4$ is even.
\end{theorem}

\begin{remark}
The hypotheses in Theorems \ref{T1}, \ref{T2}, \ref{T5} and \ref{T6}
are same as those of Corollary \ref{ccoo} for $n=2,3$. It is
interesting to find  for $n\ge 3$  an asymptotic behavior of
$\mathcal{I}(P_\Lambda,\xi,1_S)$ with a coefficient of logarithm in
$\xi$
  having a similar form  to (\ref{sppt}) in Theorem \ref{T33}.
\end{remark}
As a variable coefficient version, we define
$\mathcal{H}^{P}(f)(x)$ by
\begin{eqnarray*}
\int_{\prod_{j=1}^n\,[-r_j,r_j]} f\bigl(x_1-t_1,\cdots,x_n-t_n,x_{n+1}-P(x_1,\cdots,x_n,t_1,\cdots,t_n)\bigr)\,
\frac{dt_1}{t_1}\cdots\frac{dt_n}{t_n}
\end{eqnarray*}
whose corresponding oscillatory singular integral operator is given by
 \begin{eqnarray*}
\mathcal{T}^{P}_{\lambda} \, (f)(x)&=&\text{p.v.}
\int_{\{y:|x_j-y_j|<r_j\}} \frac{e^{i \lambda
P(x,y)}}{(x_1-y_1)\cdots (x_n-y_n)}\, f(y)dy_1\cdots
dy_n.
\end{eqnarray*}
In view of the analogy between the integral operators of D. H. Phong
and E. M. Stein \cite{PS} and the scalar valued integral of Varchenko \cite{V},
one may find the criteria for determining the uniform $L^2$
boundedness $\mathcal{T}^{P}_{\lambda}$ in $\lambda$ in terms of the
Newton polyhedron  associated with the polynomial $P(x,y)$. More
generalized  version   is the multi-parameter singular Radon transform:
\begin{eqnarray*} \mathcal{R}^{P_\Lambda} \, (f)(x)&=&\text{p.v.}
\int_{\prod_{j=1}^n\,[-r_j,r_j]} f\bigl(P_\Lambda(x,t)\bigr)\,
\frac{dt_1}{t_1}\cdots\frac{dt_n}{t_n}.
\end{eqnarray*}
Recently, E. M. Stein and B. Street in \cite{SS1,SS2,SS3} obtain the
$L^p$ boundedness for a certain class of multi-parameter singular
Radon transform. A. Nagel, F. Ricci, E. M. Stein and S. Wainger in
\cite{NRSW1} study the singular integral operators on a homogeneous
nilpotent Lie group that are given by convolution with flag kernels.
Here flag kernels are product type singular kernels that are
generalized version of our kernel
$\frac{1}{t_1}\cdots\frac{1}{t_n}$. This result was preceded by
\cite{NRS} that proves the $L^p$ boundedness of convolution
operators with some special type of flag kernels. This result
applies  to obtain   $L^p$ regularity for the solutions of
Cauchy-Riemann equations on CR manifolds.

\section{Representation of   faces   and their cones}\label{secfar}
We study  representations of  faces $\mathbb{F}$ and their cones $\mathbb{F}^*$ of
a polyhedron $\mathbb{P}=\mathbb{P}(\Pi)$ in $\mathbb{R}^n$.
 It is well known that every face
$\mathbb{F}$ has an expression
$\bigcap_{j=1}^N\pi_{\mathfrak{q}_j,r_j}\cap\mathbb{P}$ with some
generators $\pi_{\mathfrak{q}_j,r_j}\in \Pi$ and its cone
$\mathbb{F}^*$ expressed as $\text{CoSp}(\{\mathfrak{q}_j
\}_{j=1}^N)$.  We shall prove this representation formula and give
 some detailed description of generators for the case  $
 \text{dim}(\mathbb{P})<n$.

\subsection{Low  Dimensional Polyhedron in $\mathbb{R}^n$}
A polyhedron $\mathbb{P}=\mathbb{P}(\Pi)$  in $\mathbb{R}^n$ with
$\text{dim}(\mathbb{P})=m\le n $ is regarded as a $m$ dimensional
polyhedron in the affine space $V_{am}(\mathbb{P})$ of dimension  $m$
defined  in (\ref{jf2}).  Since   $V_{am}(\mathbb{P})$ itself is a polyhedron in
$\mathbb{R}^n$, we  choose the generator $\Pi$ of $\mathbb{P}$ and split it
into two parts $\Pi=\Pi_a\cup\Pi_b$: $\Pi_b$ a generator of
$V_{am}(\mathbb{P})$ and $\Pi_a$  a generator of $\mathbb{P}$ in
$V_{am}(\mathbb{P})$.  See the left picture in Figure \ref{graph1}.

  \begin{figure}
  \centerline{\includegraphics[width=14cm,height=10cm]{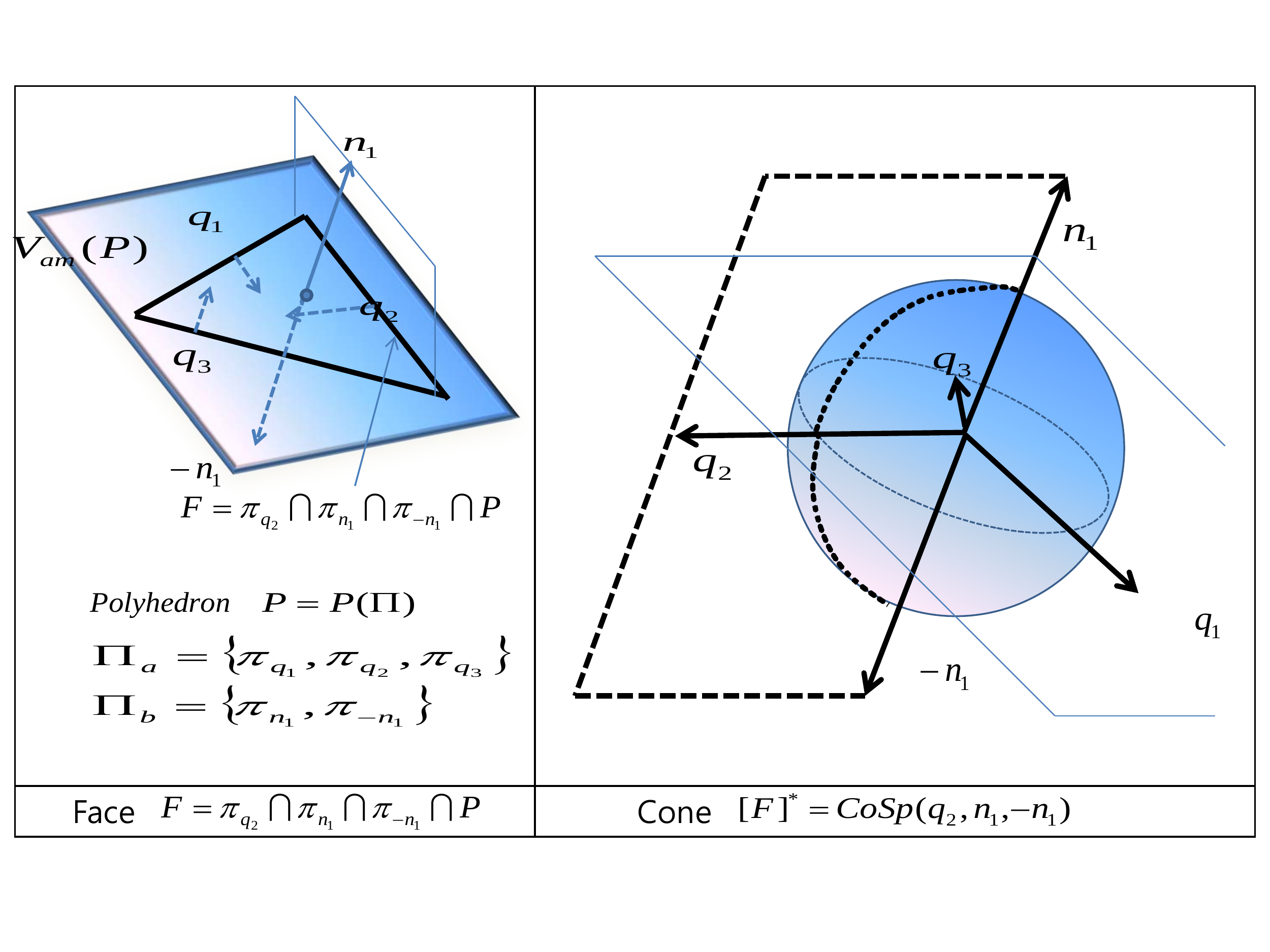}}
  \caption{Low Dimensional Polyhedron, Face and Cone.} \label{graph1}
 \end{figure}

 \begin{lemma}\label{lemj77}
Let $\mathbb{P}\subset \mathbb{R}^n$ be a   polyhedron with
$\text{dim}(\mathbb{P})=\text{dim}(V_{am}(\mathbb{P}))=m\le n $.
Then $ \mathbb{P}=\mathbb{P}(\Pi_{a}\cup\Pi_{b})$ such that
\begin{eqnarray}\label{tn19}
\qquad  V_{am}(\mathbb{P})&=&\mathbb{P}(\Pi_{b})\ \text{in $\mathbb{R}^n$ with}\
  \Pi_{b}=\{\pi_{\pm\mathfrak{n}_i,\pm s_i}\}_{i=1}^{n-m} \   \text{with}\
  \mathfrak{n}_i\perp \mathfrak{n}_j \ \text{for}\ i\ne j\,,\label{7118} \\
   \mathbb{P}&=&\mathbb{P}(\Pi_{a}') \ \text{in $V_{am}(\mathbb{P})$ with}\
 \Pi_{a}=\{\pi_{\mathfrak{q}_j,r_j}\}_{j=1}^M  \ \text{and}\
  \Pi_{a}'=\{ \pi_{\mathfrak{q}_j,r_j}\cap
V_{am}(\mathbb{P})\}_{j=1}^M,  \label{71182}
 \end{eqnarray}
where  $\mathfrak{q}_j\in V(\mathbb{P})=\text{Sp}^{\perp}(\{\mathfrak{n}_i\}_{i=1}^{n-m})$.
\end{lemma}
\begin{proof}
If $n=m$, then let $\Pi_a=\Pi$  and $\Pi_b=\emptyset$ so that
$\mathbb{P}(\Pi_b)=\mathbb{R}^n$. Let $m<n$. There exist $n-m$
orthonormal vectors $\mathfrak{n}_j$'s and some constants $s_j$'s
such that
\begin{eqnarray*}
V(\mathbb{P})=\bigcap_{i=1}^{n-m}\pi_{\mathfrak{n}_i,0}\ \ \text{so
that}\ \
V_{am}(\mathbb{P})=\bigcap_{i=1}^{n-m}\pi_{\mathfrak{n}_i,s_i}.
\end{eqnarray*}
where
$V(\mathbb{P})=\text{Sp}^{\perp}(\{\mathfrak{n}_i\}_{i=1}^{n-m})$.
  By (\ref{jf2}),
$V_{am}(\mathbb{P})=V(\mathbb{P})+\mathfrak{r}$ with
$\mathfrak{r}\in\mathbb{P}$ and
$s_i=\mathfrak{r}\cdot\mathfrak{n}_i$. This combined with
$\pi_{\mathfrak{n}_i,s_i}^+\cap
\pi_{-\mathfrak{n}_i,-s_i}^{+}=\pi_{\mathfrak{n}_i,s_i} $ implies
\begin{eqnarray*}
 V_{am}(\mathbb{P})=\mathbb{P}(\Pi^b)\ \ \text{with}\ \ \Pi_{b}=\{\pi_{\pm\mathfrak{n}_i,\pm s_i}:i=1,\cdots,n-m\},
 \end{eqnarray*}
which yields (\ref{7118}). By Definition \ref{d12},  there are
$\mathfrak{p}_j\in\mathbb{R}^n$ such that
\begin{eqnarray}\label{hhpa2}
\mathbb{P}=\bigcap_j\{{\bf x}\in V_{am}(\mathbb{P}):\langle
\mathfrak{p}_j,{\bf x}\rangle\ge \rho_j\}.
\end{eqnarray}
  Let ${\bf
x}\in\mathbb{P}$ and $P_{V(\mathbb{P})}$ be a projection map to the
vector space $V(\mathbb{P})$. Then from ${\bf x}-\mathfrak{r}\in
V(\mathbb{P})$,
$$\langle \mathfrak{p}_j,{\bf
x}\rangle=\langle P_{V(\mathbb{P})}(\mathfrak{p}_j) ,{\bf
x}\rangle+\langle P_{V^{\perp}(\mathbb{P})}(\mathfrak{p}_j),{\bf
x}\rangle=\langle P_{V(\mathbb{P})}(\mathfrak{p}_j) ,{\bf
x}\rangle+\langle Proj_{V^{\perp}(\mathbb{P})}(\mathfrak{p}_j),{\bf
\mathfrak{r}}\rangle.
$$
We put $Proj_{V(\mathbb{P})}(\mathfrak{p}_j)=\mathfrak{q}_j$ and
$r_j=\rho_j-\langle Proj_{V^{\perp}(\mathbb{P})}(\mathfrak{p}_j),\mathfrak{r}\rangle$ and rewrite (\ref{hhpa2}) as
$$ \mathbb{P}=\bigcap_{j=1}^M\{{\bf x}\in
V_{am}(\mathbb{P}):\langle \mathfrak{q}_j,{\bf x}\rangle\ge r_j\} \
\text{where}\ \mathfrak{q}_j=Proj_{V(\mathbb{P})}(\mathfrak{p}_j)\in
V(\mathbb{P}) =\text{Sp}^{\perp}(\{\mathfrak{n}_i\}_{i=1}^{n-m}).
$$
This proves (\ref{71182}). Finally,
$\mathbb{P}=\mathbb{P}(\Pi_{a}\cup\Pi_{b})$ follows from
(\ref{7118}) and (\ref{71182}).
\end{proof}

\subsection{Representation  of Face }
\begin{lemma}\label{lemjqc}
Let $\mathbb{P}=\mathbb{P}(\Pi)$ with
$\text{dim}(V_{am}(\mathbb{P}))=m$ and $\mathbb{F}$ be a proper face
of $\mathbb{P}$. Then
\begin{eqnarray*}
\exists\, \pi\in\Pi\ \ \text{such that}\ \ \mathbb{F}\subset \pi.
\end{eqnarray*}
\end{lemma}
\begin{proof}
By Definition \ref{dfu5} and by Lemma \ref{2323s},
\begin{eqnarray}\label{4b4b}
\mathbb{F} \subset \partial \mathbb{P}\subset
\bigcup_{\pi\in\Pi}\pi\cap\mathbb{P}.
\end{eqnarray}
By (\ref{71182}) of Lemma \ref{lemj77}, we may take $\Pi=\Pi_a'$ in
(\ref{4b4b}). Thus from (2) of Lemma \ref{lemss},   each
$\pi\cap\mathbb{P}$ with $\pi\in\Pi$ is a face of $\mathbb{P}$ with
dimension $m-1$. Therefore the second $\subset$ in (\ref{4b4b}) is
replaced by $=$. We next use the same argument as in the proof of
Lemma \ref{lem411d} to obtain that $\mathbb{F}$ in (\ref{4b4b}) is
contained in one face $\pi\cap\mathbb{P}$ in (\ref{4b4b}).
\end{proof}

\begin{definition}[Facet]
Let $\mathbb{P}=\mathbb{P}(\Pi)$ be a polyhedron in an  affine space $V$
such that $\text{dim}(\mathbb{P})=\text{dim}(V_{am}(\mathbb{P}))=m$. Then $m-1$
dimensional face $\mathbb{F} $ of $\mathbb{P}$ is called a
facet of $\mathbb{P}$.
\end{definition}
\begin{lemma}\label{0941}
Let $\mathbb{P}=\mathbb{P}(\Pi)$ be a polyhedron in an  affine space
$V$ such that $ \text{dim}(
\mathbb{P})=\text{dim}(V_{am}(\mathbb{P}))=m$ where $\Pi=\Pi_a\cup
\Pi_b$ as in Lemma \ref{lemj77}. Then every facet $\mathbb{F}$ of
$\mathbb{P}$    is expressed as
$$\mathbb{F}=\pi\cap  \mathbb{P}\ \ \text{and}\ \
\mathbb{P}\setminus\mathbb{F}\subset(\pi^{+})^{\circ} \ \text{for some $\pi\in\Pi_{a}\subset\Pi$}. $$
\end{lemma}
\begin{proof}
By Lemma \ref{lemj77}, we regard $ \mathbb{P}=\mathbb{P}(\Pi'_a)  $
 as a polyhedron in the $m$ dimensional  affine space
$V_{am}(\mathbb{P})$. Here $\pi'=\pi\cap V_{am}(\mathbb{P})\in\Pi_a'
$ is a $m-1$ dimensional  hyperplane in $V_{am}(\mathbb{P})$.  By
Lemma \ref{lemjqc},
\begin{eqnarray}\label{jqc44}
\exists\, \pi'\in\Pi'_a\ \ \text{such that}\ \ \mathbb{F}\subset
\pi'=\pi\cap V_{am}(\mathbb{P}) \ \ \text{where}\ \pi \in\Pi_a.
\end{eqnarray}
On the other hand, by Definition \ref{dfac},
there exists an $m-1$ dimensional hyperplane $\pi_{\mathfrak{q},r}$ in $V_{am}(\mathbb{P})$
such that \begin{eqnarray}\label{jqc444}
 \text{ $\mathbb{F}=\pi_{\mathfrak{q},r}\cap\mathbb{P}$
and $\mathbb{P}\setminus\mathbb{F}\subset
(\pi_{\mathfrak{q},r}^+)^{\circ}$. }
\end{eqnarray}
In view of (\ref{jqc44}) and (\ref{jqc444}), both $m-1$ dimensional
hyperplanes $\pi' $ and $\pi_{\mathfrak{q},r}$ in
$V_{am}(\mathbb{P})$ contain the  $m-1$ dimensional polyhedron
$\mathbb{F}$.   Thus
 $ \pi'=\pi_{\mathfrak{q},r}$. By this and (\ref{jqc444}),
\begin{equation*}\label{44gs}
\mathbb{F}=\pi_{\mathfrak{q},r}\cap\mathbb{P}=\pi'\cap\mathbb{P}=\left(\pi\cap
V_{am}(\mathbb{P})\right)\cap\mathbb{P}=\pi\cap\mathbb{P}
      \ \text{where $\pi\in\Pi_a$}
      \end{equation*}
      and
      $\mathbb{P}\setminus\mathbb{F}\subset
(\pi_{\mathfrak{q},r}^+)^{\circ}=((\pi')^+)^{\circ}\subset(\pi^+)^{\circ}.
$
\end{proof}

\begin{proposition}[Face Representation]\label{facerep}
Let $\mathbb{P}=\mathbb{P}(\Pi)$ be a  polyhedron in $\mathbb{R}^n$
where $\Pi=\Pi_a\cup \Pi_b$ as in Lemma \ref{lemj77}. Let
$\text{dim}(\mathbb{P}) =m\le n.$  Then   every face $\mathbb{F}$ of
$\mathbb{P}$ with $\text{dim}(\mathbb{F})\le n-1$ has the expression
\begin{eqnarray}\label{456g}
 \mathbb{F}=\bigcap_{\pi \in \Pi(\mathbb{F})} \mathbb{F}_{\pi}  \ \
 \text{where $\mathbb{F}_{\pi}=
 \pi\cap \mathbb{P}$ }\ \text{where}\ \  \Pi(\mathbb{F})=
 \{\pi_{\mathfrak{p}_j}\}_{j=1}^N\subset\Pi.
\end{eqnarray}
  \end{proposition}
\begin{remark}\label{remfp}
 We split the generator $\Pi(\mathbb{F})=\{\pi_{\mathfrak{p}_j}\}_{j=1}^N=
 \Pi_a(\mathbb{F})\cup\Pi_b(\mathbb{F})$ in (\ref{456g}) so that
\begin{eqnarray}\label{4g00}
 \quad \Pi_a(\mathbb{F})=\Pi(\mathbb{F})\cap\Pi_a=\{\pi_{\mathfrak{q}_j}\}_{j=1}^\ell  \ \ \text{and}\ \
  \Pi_b(\mathbb{F})=\Pi(\mathbb{F})\cap\Pi_b=\Pi_b=\{\pi_{\pm\mathfrak{n}_i}\}_{i=1}^{n-m}.
\end{eqnarray}
  See the left side of Figure \ref{graph1}. Denote only normal vectors
   $\{\mathfrak{p}_j\}_{j=1}^N$ in $\Pi(\mathbb{F})$ by   $\Pi(\mathbb{F})$ also.
  \end{remark}
\begin{proof}[Proof of Proposition \ref{facerep}]
Let $\text{dim}(\mathbb{F})=m\le n-1$. An improper face
$\mathbb{F}=\mathbb{P}$ has an expression
\begin{eqnarray*}
\mathbb{P}=\bigcap_{\pi\in\Pi_b}\pi\cap\mathbb{P}\ \ \text{where\
$\pi \in \Pi_b$}.
\end{eqnarray*}
It suffices to show that
each proper face $\mathbb{F}$ of $\mathbb{P}(\Pi)$ is expressed as
\begin{eqnarray}\label{45g}
 \mathbb{F}=\bigcap_{j=1}^M\mathbb{F}_j\ \ \text{ where   $\mathbb{F}_j=\pi_j\cap  \mathbb{P}$ with
$\pi_j\in\Pi_{a}\subset\Pi$ are facets of $\mathbb{P}$}.
\end{eqnarray}
 To show (\ref{45g}), we first let
   $\mathbb{F}$ be a face of codimension 1 of the
$m$-dimensional ambient affine space $V_{am}(\mathbb{P})$. Then
$\mathbb{F}$ itself is a facet of $\mathbb{P}$ such that
$\mathbb{F}=\pi\cap  \mathbb{P}$ with $\pi\in\Pi_{a}\subset\Pi$ from
Lemma \ref{0941}. Let $\mathbb{F}$ be a face of codimension 2 of the
$m$-dimensional ambient affine space $V_{am}(\mathbb{P})$.   By
 Lemma \ref{lemjqc},
\begin{eqnarray*}
\text{$\pi_{\mathfrak{q},r}\in \Pi_a$ such that $\mathbb{F}\subset
\pi_{\mathfrak{q},r}$.}
\end{eqnarray*}
 By  (2) of Lemma \ref{lemss},
 \begin{eqnarray}
\text{$\mathbb{P}'=\pi_{\mathfrak{q},r}\cap\mathbb{P}$ is a facet of
$\mathbb{P}$ such that $\text{dim}(\mathbb{P}')=m-1$.}
\end{eqnarray}
Moreover, observe that $\mathbb{P}'$ itself is an $m-1$ dimensional
polyhedron with
\begin{equation}\label{70}
\Pi_a(\mathbb{P}')\subset\{\pi_{\mathfrak{q},r}\cap\pi:\pi\in\Pi_a(\mathbb{P})\}.
\end{equation}
  By $\mathbb{F}\subset \mathbb{P}'$ and (1) of Lemma \ref{lemss},
   $m-2$ dimensional face $\mathbb{F}$ of $\mathbb{P}$ is
a facet of an $m-1$ dimensional polyhedron $\mathbb{P}'$. Hence by
Lemma \ref{0941}, there exists $\pi'\in\Pi_a(\mathbb{P}') $ in
(\ref{70}) such that $\mathbb{F}=\pi'\cap\mathbb{P}'$. Thus, by
(\ref{70}) there exists $\pi\in\Pi_a(\mathbb{P})$ such that
$\pi'=\pi_{\mathfrak{q},r}\cap\pi$ and
$$\mathbb{F}=\pi'\cap\mathbb{P}'=(\pi_{\mathfrak{q},r}\cap\pi)\cap\mathbb{P}'=
 \left(\pi_{\mathfrak{q},r}\cap\mathbb{P}\right)\cap
 \left(\pi\cap\mathbb{P}\right)=
\mathbb{F}_{\pi_{\mathfrak{q},r}}\cap \mathbb{F}_{\pi}$$
 where $\mathbb{F}_{\pi}$ and $\mathbb{F}_{\pi_{\mathfrak{q},r}}$ are facets of $\mathbb{P}$.
We finish the proof of (\ref{45g}) inductively.
\end{proof}

\subsection{Representation of Cone}
\begin{proposition}[Cone representation]
\label{pr50}
Every  proper face $\mathbb{F}\preceq \mathbb{P}=\mathbb{P}(\Pi)$
  having a generator $\Pi(\mathbb{F})=\{\mathfrak{p}_j\}_{j=1}^N$ with expression
  (\ref{456g})
has its cone of the form:
\begin{eqnarray*}
(\mathbb{F}^*)^{\circ}| \mathbb{P}&=&(\mathbb{F}^*)^{\circ}|(\mathbb{P},\mathbb{R}^n)=
\rm{CoSp}^{\circ}(\{\mathfrak{p}_i:i=1,\cdots, N\}).
\end{eqnarray*}
Here $ \mathbb{F}^*| \mathbb{P}=\mathbb{F}^*|(\mathbb{P},\mathbb{R}^n)=
\rm{CoSp}(\{\mathfrak{p}_i:i=1,\cdots, N\})$ similarly.
\end{proposition}
\begin{remark}
See the
right side of Figure \ref{graph1}, which elucidate the relation between
 faces and their cones.
In the above, $\mathbb{F}^*|(\mathbb{P},\mathbb{R}^n)=
\mathbb{F}^*(\mathbb{P},V(\mathbb{P}))\oplus V^{\perp}(\mathbb{P})$ where
$\mathbb{F}^*|(\mathbb{P},V(\mathbb{P}))=
\rm{CoSp}(\{\mathfrak{p}_i:\mathfrak{p}_i\in \Pi_a(\mathbb{F})\})$ with $\Pi_a(\mathbb{F})$ in (\ref{4g00}).
From this, we also
obtain that $\text{dim}(\mathbb{F})+\text{dim}(\mathbb{F}^*|(\mathbb{P},\mathbb{R}^n))=n$ whereas
 $\text{dim}(\mathbb{F})+\text{dim}(\mathbb{F}^*|(\mathbb{P},V(\mathbb{P}))=\text{dim}(V(\mathbb{P}))$.
 \end{remark}
\begin{lemma}\label{lem3399}
Let $\mathbb{P}=\mathbb{P}(\Pi)$ be a  polyhedron in an inner
product space $V$ with $\text{dim}(\mathbb{P})=\text{dim}(V)=n$.
Let  $\mathbb{F}
\in\mathcal{F}(\mathbb{P})$ be a  facet  expressed as
\begin{equation}
\mathbb{F}= \pi_{\mathfrak{q},r}\cap \mathbb{P}\  \text{where}\
\pi_{\mathfrak{q},r}\in\Pi\ \ \text{and}\ \ \mathbb{P}\setminus
\mathbb{F}\subset (\pi_{ \mathfrak{q},r
}^+)^{\circ}.\label{jnn}\end{equation} Then
$$ (\mathbb{F}^*)^{\circ}|(\mathbb{P},V)=\rm{CoSp}^{\circ}(\mathfrak{q}).$$
\end{lemma}
\begin{proof}
 Let $\mathfrak{q}'=c\,\mathfrak{q}\subset \text{CoSp}^{\circ}(\mathfrak{q})$
  with $\mathfrak{q}$ in (\ref{jnn}) and $c>0$. Then
 $\mathfrak{q}'$   satisfies (\ref{brg2})  in Definition
\ref{dualface}. So $\mathfrak{q}'\in (\mathbb{F}^*)^{\circ}$. Thus
$\text{CoSp}^{\circ}(\mathfrak{q})\subset (\mathbb{F}^*)^{\circ}$.
 Let
$\mathfrak{p}\in (\mathbb{F}^*)^{\circ}$. Then by (\ref{brr2}),
$\pi_{ \mathfrak{p},r}\cap \mathbb{P}=\mathbb{F}$. This combined
with (\ref{jnn}) implies that $\mathbb{F}=(\pi_{ \mathfrak{p},r}\cap
\mathbb{P})\cap \mathbb{F}=\pi_{ \mathfrak{p},r}\cap
(\pi_{\mathfrak{q},r}\cap \mathbb{P})$. So, $\text{dim}(\mathbb{F})<n-1$ if $\mathfrak{p}\notin
\text{CoSp}^{\circ}(\mathfrak{q})$. Thus $\mathfrak{p}\in
\text{CoSp}^{\circ}(\mathfrak{q})$, which proves $
(\mathbb{F}^*)^{\circ}\subset \text{CoSp}^{\circ}(\mathfrak{q})$.
\end{proof}

\begin{lemma}\label{conerep}
Let $\mathbb{P}=\mathbb{P}(\Pi)$ be a   polyhedron in an inner
product space $V$ with $\text{dim}(\mathbb{P})=\text{dim}(V)=n$.
Let $\mathbb{F} \in\mathcal{F}(\mathbb{P})$ be a $k$ dimensional
face of $\mathbb{P}$ with a generator $\Pi(\mathbb{F})=\{\mathfrak{q}_j\}_{j=1}^M$, that is,
$$\mathbb{F}=\bigcap_{j=1}^M\mathbb{F}_{j}$$ where
$\mathbb{F}_{j}=\pi_{\mathfrak{q}_j,r_j}\cap \mathbb{P}$ is a facet
of $\mathbb{P}$ so that $
(\mathbb{F}_j^*)^{\circ}=\rm{CoSp}^{\circ}(\{\mathfrak{q}_j\})$ for
$j=1,\cdots,M$. Then,
\begin{eqnarray}\label{sj3}
  (\mathbb{F}^*)^{\circ}|(\mathbb{P},V)=\rm{CoSp}^{\circ}(\{\mathfrak{q}_j\}_{j=1}^M ).
\end{eqnarray}
\end{lemma}
\begin{proof}
 Let $ \mathfrak{q} \in
\rm{CoSp}^{\circ}(\{\mathfrak{q}_j\}_{j=1}^M )$.  Then (\ref{452g})
with $\mathfrak{q}=\sum_{j=1}^Mc_j\mathfrak{q}_j$   yields
(\ref{brr2}).  Thus $\mathfrak{q}\in (\mathbb{F}^*)^{\circ}$, which
proves  $\supset$ of (\ref{sj3}). We next show  $\subset$ of
(\ref{sj3}). Let $U=\text{Sp}\{\mathfrak{q}_j:j=1,\cdots,M\}$.
Subtract a vector $\mathfrak{r}\in \mathbb{F} =\bigcap_{j=1}^M\pi_{\mathfrak{q}_j,r_j}\cap \mathbb{P}$,
$$V\left( \bigcap_{j=1}^M\pi_{\mathfrak{q}_j,0} \right)=V\left( \bigcap_{j=1}^M\pi_{\mathfrak{q}_j,0}\cap
(\mathbb{P}-\mathfrak{r})\right)=V(\mathbb{F}-\mathfrak{r})=V(\mathbb{F})\ \ \text{and}\ \ \text{dim}(V(\mathbb{F}))=k.$$
 Thus
 $\text{dim}\left(U^{\perp} \right)=k$ and
  $\text{dim}\left(U \right)=n-k$. Let $\mathfrak{q}\in  (\mathbb{F}^*)^{\circ}|(\mathbb{P},V)\subset V=U\oplus
U^{\perp}$. Then,
\begin{eqnarray}\label{u3}
\mathfrak{q}=\sum_{j=1}^N s_j\mathfrak{q}_j+{\bf u}\ \ \ \text{for
some $s_j\in\mathbb{R}^n$  and}\ \ {\bf u}\in U^{\perp}.
\end{eqnarray}
Since $\mathbb{F}$ is a $k$ dimensional face, we can choose $k$
linearly independent vectors
$$\{{\bf u}_{2}-{\bf u}_1,\cdots,{\bf u}_{k+1}-{\bf u}_1:{\bf u}_{\ell}\in
\mathbb{F}\}.$$  Since $\mathfrak{q} \in  (\mathbb{F}^*)^{\circ}$
and $\mathfrak{q}_j \in  (\mathbb{F}_j^*)^{\circ}$ where $(\mathbb{F}^*)^{\circ}=
 (\mathbb{F}^*)^{\circ}|(\mathbb{P},V)$,  by (\ref{brr2})
and ${\bf u}_\ell,{\bf u}_1\in \mathbb{F}=\bigcap \mathbb{F}_j$,
\begin{eqnarray}\label{u2}
 \langle\mathfrak{q},  {\bf u}_\ell-{\bf u}_1 \rangle=\langle\mathfrak{q}_j,
 {\bf u}_\ell-{\bf u}_1 \rangle=0\ \ \text{for
 all $\ell=2,\cdots,k+1$. }
\end{eqnarray}
This implies that $
 \{{\bf u}_{2}-{\bf u}_1,\cdots,{\bf u}_{k+1}-{\bf u}_1 \}\subset
 U^{\perp}
$
and forms a basis of  $U^{\perp}$  because
 $\text{dim}(U^{\perp})=k$. Hence   ${\bf u}\in   U^{\perp}$ is
 expressed as
\begin{eqnarray*}
{\bf u}=\sum_{\ell=2}^{k+1} c_{\ell}({\bf u}_\ell-{\bf u}_1).
\end{eqnarray*}
Thus by (\ref{u2}), we have $ \langle \mathfrak{q},  {\bf
u}\rangle=0\   \text{and}\   \langle \mathfrak{q}_j, {\bf
u}\rangle=0\   \text{for}\  j=1,\cdots,M. $ Therefore, in (\ref{u3})
\begin{eqnarray*}
0=\langle \mathfrak{q},  {\bf u}\rangle =\sum_{j=1}^N s_j \langle
\mathfrak{q}_j\cdot  {\bf u}\rangle+\langle {\bf u}, {\bf
u}\rangle=|{\bf u}|^2 ,
\end{eqnarray*}
which implies that
\begin{eqnarray}\label{mmp}
  \mathfrak{q}=\sum_{j=1}^M s_j \mathfrak{q}_j \ \ \text{for $\mathfrak{q}
  \in  (\mathbb{F}^*)^{\circ}$.}
\end{eqnarray}
We now fix $\ell$ and show $s_\ell>0$. Since $\mathbb{F}_j$'s are
facets of one polyhedron, we can choose
$$ {\bf y}\in \left(\bigcap_{1\le j\le M, j\ne l}
\mathbb{F}_{j}\right)\setminus \mathbb{F}_l\subset
\mathbb{P}\setminus \mathbb{F}_{l}\subset \mathbb{P}\setminus
\mathbb{F}\ \ \text{and}\ \ {\bf u}\in
\mathbb{F}=\bigcap\mathbb{F}_j.
$$  Thus for $\mathfrak{q}\in (\mathbb{F}^*)^{\circ}$ and $\mathfrak{q}_l\in (\mathbb{F}_l^*)^{\circ}$,
\begin{eqnarray}\label{lg1}
\langle\mathfrak{q}, {\bf y}-{\bf u}\rangle>0 \ \ \text{and}\ \
\langle\mathfrak{q}_l,{\bf y}-{\bf u}\rangle >0.
\end{eqnarray}
Since $ {\bf y},{\bf u}  \in  \mathbb{F}_{j}$ and $\mathfrak{q}_j\in
(\mathbb{F}_j^*)^{\circ} $ for all $j\in\{1,\cdots,M\}\setminus\{
l\}$,
\begin{eqnarray}\label{lg2}
\langle\mathfrak{q}_j,{\bf y}-{\bf u}\rangle=0.
\end{eqnarray}
By   (\ref{lg1})-(\ref{lg2})  in (\ref{mmp}), we obtain that
$s_l>0$. Similarly $s_{j}>0$ for all $1\le j\le M$. Therefore
$\mathfrak{q}\in \rm{CoSp}^{\circ}(\{\mathfrak{q}_j\}_{j=1}^M )$.
Thus $(\mathbb{F}^*)^{\circ}
\subset\rm{CoSp}^{\circ}(\{\mathfrak{q}_j\}_{j=1}^M ).$
\end{proof}

We note that $\mathbb{F}^*$ is translation-invariant in the
following sense.
\begin{lemma}\label{lem37u}
Let $\mathfrak{m}\in V$. Then
$[(\mathbb{F}+\mathfrak{m})^*]^{\circ}|(\mathbb{P}+\mathfrak{m},V)=(\mathbb{F}^*)^{\circ}|(\mathbb{P},V)$.
\end{lemma}
\begin{proof}
Note that $\mathfrak{q}\in
[(\mathbb{F}+\mathfrak{m})^*]^{\circ}|(\mathbb{P}+\mathfrak{m},V)$ if and only
if there exists $\rho$ such that
$$  \langle\mathfrak{q},{\bf u}+\mathfrak{m}\rangle=\rho< \langle\mathfrak{q},
{\bf y}+\mathfrak{m}\rangle\
 \ \text{for $ {\bf u}+\mathfrak{m} \in \mathbb{F}+\mathfrak{m}$ and $ {\bf y}
 +\mathfrak{m} \in  (\mathbb{P}+\mathfrak{m})\setminus (\mathbb{F}+\mathfrak{m})$},$$
that is equivalent to  the following:
$$\exists\, \rho'=\rho-\langle\mathfrak{q}, \mathfrak{m}\rangle\ \text{such that}\
\langle\mathfrak{q},{\bf u} \rangle=\rho' < \langle\mathfrak{q}, {\bf y}\rangle \
 \ \text{for $ {\bf u}  \in \mathbb{F} $ and $ {\bf y}  \in  \mathbb{P}\setminus\mathbb{F} $}$$
which means that $\mathfrak{q}\in (\mathbb{F}^*)^{\circ}|(\mathbb{P},V)$.
\end{proof}

\begin{proof}[Proof of Proposition \ref{pr50}]
We rewrite (\ref{456g}) and (\ref{4g00}) as
\begin{eqnarray} \label{gel3}
\mathbb{F}=\left(\bigcap_{\pi\in  \Pi_a(\mathbb{F})} \mathbb{F}_{\pi}\right)\bigcap
 \left(\bigcap_{\pi\in\Pi_b}\mathbb{F}_\pi\right)=
 \bigcap_{\pi\in  \Pi_a(\mathbb{F})} \mathbb{F}_{\pi}
\end{eqnarray}
where
\begin{itemize}
\item[(1)] $\Pi_a(\mathbb{F})=\{\pi_{\mathfrak{q}_j}\}_{j=1}^\ell\subset\Pi_a$,
$\mathbb{F}_{\pi} =\pi  \cap \mathbb{P}=\pi'\cap\, \mathbb{P}$
 a facet of $\mathbb{P}$ with $\pi'=\pi\cap  V_{am}(\mathbb{P})\in \Pi_a'$
\item[(2)]  $\Pi_b=\{\pi_{\pm\mathfrak{n}_i,\pm
s_i}\}_{i=1}^{n-m}$, $\mathbb{F}_{\pi}=\pi  \cap \mathbb{P}=\mathbb{P}$
 where $V(\mathbb{P})=\text{Sp}^{\perp}(
\{\mathfrak{n}_i\}_{i=1}^{n-m} ).$
\end{itemize}
We claim that $\mathbb{F}$ has a cone of the following form:
\begin{eqnarray*}
(\mathbb{F}^*)^{\circ}|
\mathbb{P}&=&\rm{CoSp}^{\circ}(\{\mathfrak{q}_j:j=1,\cdots,\ell
\})\oplus
V(\mathbb{P})^{\perp}\nonumber\\
&=&\text{CoSp}^{\circ}(\{\mathfrak{q}_j:j=1,\cdots,\ell\})\oplus
\text{CoSp}^{\circ}(\{\mathfrak{n}_i,-\mathfrak{n}_i\}_{i=1}^{n-m}
).
\end{eqnarray*}
By (\ref{jf2}),
\begin{eqnarray*}
 \exists \ \mathfrak{m}\in
V \ \   \text{such that}\ \ V_{am}(\mathbb{P})=\mathfrak{m}+
V(\mathbb{P}).
\end{eqnarray*}
 We first work with $\mathfrak{m}=0$. By of (\ref{71182}) of Lemma \ref{lemj77},
  we regard $\mathbb{P}$
 as
a  polyhedron $\mathbb{P}(\Pi_a')$ defined in  $
V_{am}(\mathbb{P})$. Thus by  (1) of (\ref{gel3}) and Lemma
\ref{conerep},
$$(\mathbb{F}^*)^{\circ}|\mathbb{P},V(\mathbb{P}))=
\rm{CoSp}^{\circ}(\{\mathfrak{q}_j:j=1,\cdots,\ell\}).$$
This means that  $\mathfrak{q}\in
\rm{CoSp}^{\circ}(\{\mathfrak{q}_j:j=1,\cdots,\ell\})$ if and only
if $\mathfrak{q}\in
(\mathbb{F}^*)^{\circ}|(\mathbb{P},V(\mathbb{P}))$, that is,
$$  \mathfrak{q}\in V(\mathbb{P})\ \text{and}\ r\
\text{such that}\ \langle\mathfrak{q}, {\bf u}\rangle=r<\langle\mathfrak{q}, {\bf
y}\rangle\ \text{for all ${\bf u}\in \mathbb{F}$,\ ${\bf y}\in
\mathbb{P}\setminus \mathbb{F}$}.$$   By this combined with $\langle
\mathfrak{n},{\bf u}\rangle= \langle \mathfrak{n},{\bf y}\rangle=0\
\text{ for all $\mathfrak{n}\in V (\mathbb{P})^{\perp}$ and ${\bf
u},{\bf y}\in V (\mathbb{P})$},$ we see that
$$\text{ $\mathfrak{q}\in
\text{CoSp}^{\circ}(\{\mathfrak{q}_j:j=1,\cdots,\ell \})\oplus
V(\mathbb{P})^\perp$ }$$ if and only if
$$\ \exists\, \mathfrak{q}\in  V(\mathbb{P})\oplus V(\mathbb{P})^\perp
=\mathbb{R}^n \ \text{and}\ r\ \text{such that}\
\langle\mathfrak{q}, {\bf u}\rangle=r<\langle\mathfrak{q}, {\bf y}\rangle\ \text{for all
${\bf u}\in \mathbb{F}$,\ ${\bf y}\in \mathbb{P}\setminus
\mathbb{F}$}. $$  Hence we have for a proper face $\mathbb{F}$,
\begin{eqnarray}\label{snn2}
(\mathbb{F}^*)^{\circ}|
(\mathbb{P},\mathbb{R}^n)=\rm{CoSp}^{\circ}(\{\mathfrak{q}_j:j=1,\cdots,\ell
\})\oplus V^{\perp}(\mathbb{P}).
\end{eqnarray}
The case  $\mathfrak{m}\ne 0$ follows from the case $\mathfrak{m}=0$
in (\ref{snn2}) and Lemma \ref{lem37u}. Similarly,
\begin{eqnarray}\label{snn26}
 \mathbb{F}^* |(\mathbb{P},\mathbb{R}^n)=\rm{CoSp}(\{\mathfrak{q}_j:j=1,\cdots,\ell
\})\oplus V^{\perp}(\mathbb{P}).
\end{eqnarray}
We finished the proof of Proposition  \ref{pr50}.
\end{proof}
\begin{remark}\label{pitse}
By (2) of (\ref{gel3}), an improper face $\mathbb{P}$ has the
expression that
$$ \mathbb{P}=
  \bigcap_{\pi\in\Pi_b}\mathbb{F}_\pi =
  \bigcap_{\pi\in\Pi_b}\pi\cap\mathbb{P}\
  \text{where $\Pi_b=\{\pi_{\pm \mathfrak{n}_i}\}_{i=1}^{n-m}$.} $$
 Then we see that
\begin{eqnarray}\label{e8v}
\quad \mathbb{P}^* |
\mathbb{P}=V^{\perp}(\mathbb{P})=\text{CoSp}(\{\pm \mathfrak{n}_i
\}_{i=1}^{n-m})\ \  \text{and}\  \ (\mathbb{P}^*)^{\circ}|
\mathbb{P}=V^{\perp}(\mathbb{P})\setminus\{0\}.
\end{eqnarray}
 This accords with
  Definition \ref{dualface} together with  (\ref{snn2}) and (\ref{snn26}).
  Finally, when $\mathbb{P}=\mathbb{P}(\Pi)$ with $\Pi=\{\mathfrak{p}_j\}_{j=1}^N$,
 we take $\mathbb{F}^*=\text{CoSp}(\{\mathfrak{p}_j\}_{j=1}^N)$ if $\mathbb{F}=\emptyset$.
\end{remark}
In  Example \ref{ex32}, we construct the faces and cones for
the Newton Polyhedrons ${\bf N}(\Lambda_1)$ and ${\bf N}(\Lambda_2)$
 associated with a polynomial  $P_{\Lambda}(t)=(P_{\Lambda_1}(t_1,t_2,t_3),
 P_{\Lambda_1}(t_1,t_2,t_3))$
and check the hypotheses of Main Theorem \ref{main3} for $n=3$ and
$d=2$ with $S=\{1,2,3\}$.

 \begin{figure}
 \centerline{\includegraphics[width=14cm,height=10cm]{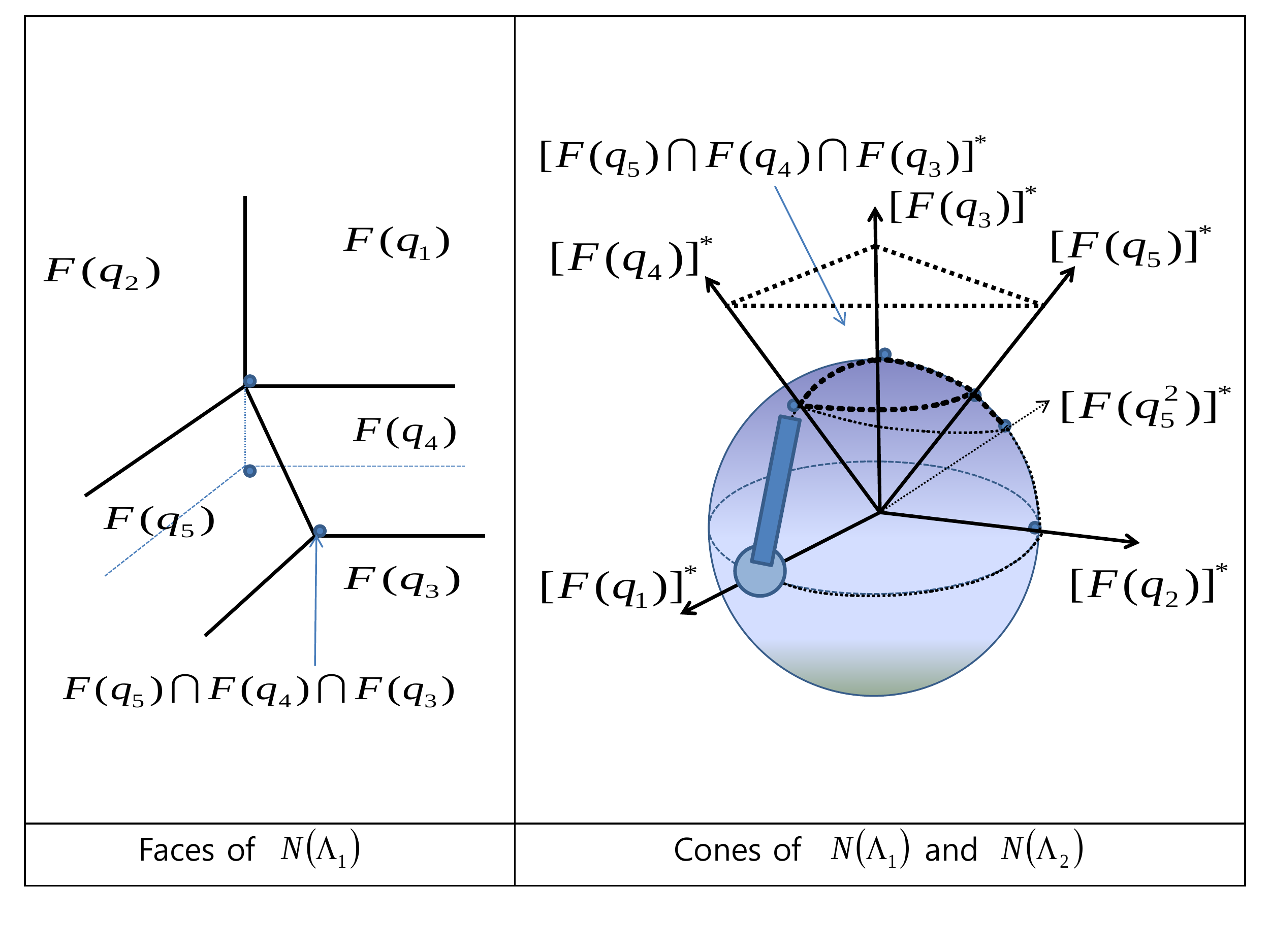}}
 \caption{Faces and their cones of ${\bf N}(\Lambda_1)$ and ${\bf N}(\Lambda_2)$.} \label{graph2}
 \end{figure}

\begin{exam}\label{ex32}
Consider the polynomial
$
P_\Lambda(t)=(c_{\mathfrak{m}_1}^1t^{\mathfrak{m}_1}+c_{\mathfrak{n}_1}^1t^{\mathfrak{n}_1},
 c_{\mathfrak{m}_2}^2t^{\mathfrak{m}_2}+ c_{\mathfrak{n}_2}^2t^{\mathfrak{n}_2})$
 where
\begin{eqnarray*}
\Lambda_1&=&\{\mathfrak{m}_1=(0,0,2),\mathfrak{n}_1=(3,3,0)\},
 \\
\Lambda_2&=&\{\mathfrak{m}_2=(0,0,3),\mathfrak{n}_2=(3,2,1)\}.
 \end{eqnarray*}
Normal vectors $\{\mathfrak{q}^\nu_j\}_{j=1}^5$ of facets of ${\bf N}(\Lambda_\nu)$ for $\nu=1,2$ are
$$  \text{$\mathfrak{q}_j^\nu={\bf e}_j$ for $j=1,2,3,\mathfrak{q}_4^\nu=\frac{(2,0,3)}{\sqrt{13}} ,
\mathfrak{q}_5^1=\frac{(0,2,3)}{\sqrt{13}}$, and
$\mathfrak{q}_5^2=\frac{(0,1,1)}{\sqrt{2}}$.}$$ See Figure
\ref{graph2}, where normal vectors $\mathfrak{q}_j^{\nu}$ are
written without the superscript $\nu=1$ for simplicity.  All the
faces of ${\bf N}(\Lambda_\nu)$ for $\nu=1,2$ are written as
\begin{eqnarray*}
\mathcal{F}^2({\bf
N}(\Lambda_\nu))&=&\left\{\mathbb{F}(\mathfrak{q}_j^\nu)=\pi_{\mathfrak{q}_j^\nu}\cap
{\bf N}(\Lambda_\nu):j=1,\cdots,5\right\},  \\
\mathcal{F}^1({\bf N}(\Lambda_\nu))&=&\left\{
\mathbb{F}(\mathfrak{q}_i^\nu)\cap
\mathbb{F}(\mathfrak{q}_j^\nu):(i,j)=(1,2),(1,4),(2,5),(3,4),(3,5),(4,5)
 \right\}, \\
\mathcal{F}^0({\bf N}(\Lambda_\nu))&=&\left\{
\mathbb{F}(\mathfrak{q}_1^\nu)\cap
\mathbb{F}(\mathfrak{q}_2^\nu)\cap
\mathbb{F}(\mathfrak{q}_4^\nu)\cap \mathbb{F}(\mathfrak{q}_5^\nu),\
\mathbb{F}(\mathfrak{q}_3^\nu)\cap
\mathbb{F}(\mathfrak{q}_4^\nu)\cap
\mathbb{F}(\mathfrak{q}_5^\nu)\right\}=\left\{\mathfrak{m}_\nu,\mathfrak{n}_\nu\right\}.
\end{eqnarray*}
  Cones of 0-faces
(vertices) are
\begin{eqnarray*}
\left(\mathcal{F}^0({\bf N}(\Lambda_\nu))\right)^*=\left\{
\mathfrak{m}_\nu^*=\text{CoSp}
\left(\mathfrak{q}_1^\nu,\mathfrak{q}_2^\nu,\mathfrak{q}_4^\nu,\mathfrak{q}_5^\nu
\right) \ \text{and}\ \ \mathfrak{n}_\nu^*=\text{CoSp}
\left(\mathfrak{q}_3^\nu,
\mathfrak{q}_4^\nu,\mathfrak{q}_5^\nu\right)\right\}.
\end{eqnarray*}
Cones of 1-faces (edges) are
\begin{eqnarray*}
\left(\mathcal{F}^1({\bf
N}(\Lambda_\nu))\right)^*=\left\{[\mathbb{F}(\mathfrak{q}_i^\nu)\cap
\mathbb{F}(\mathfrak{q}_j^\nu)]^*=\text{CoSp}
\left(\mathfrak{q}_i^\nu,\mathfrak{q}_j^\nu\right) \right\}.
\end{eqnarray*}
Cones of 2-faces  are
\begin{eqnarray*}
 \left(\mathcal{F}^2({\bf N}(\Lambda_\nu))\right)^* =\left\{
(\mathbb{F}(\mathfrak{q}_j^\nu))^*= \text{CoSp}
\left(\mathfrak{q}^\nu_j \right) \right\}.
\end{eqnarray*}
All possible combinations $\bigcup_\nu\mathbb{F}_\nu\cap\Lambda_\nu$
with $\text{rank}\left( \bigcup_\nu\mathbb{F}_\nu  \right)\le 2$ are
even sets except the following two odd set:
\begin{itemize}
  \item[(1)] odd set $
\{\mathfrak{n}_1,\mathfrak{m}_2\}$, \item[(2)]
 odd set $\{
\mathfrak{m}_1,\mathfrak{n}_1, \mathfrak{m}_2\}$.
\end{itemize}
We can check the following in view of Figure \ref{graph2},
 where we add the cone $(\mathbb{F}(\mathfrak{q}_5^2))^*=\text{CoSp}(\mathfrak{q}^2_5)$
 for the face $\mathbb{F}(\mathfrak{q}_5^2)$ of
${\bf N}(\Lambda_2)$.\\ From $\{\mathfrak{n}_1,\mathfrak{m}_2\}$
where $ \mathfrak{n}_1^*=\text{CoSp} \left(\mathfrak{q}_3^1,
\mathfrak{q}_4^1,\mathfrak{q}_5^1\right)$ and $
\mathfrak{m}_2^*=\text{CoSp}
\left(\mathfrak{q}_1^2,\mathfrak{q}_2^2,\mathfrak{q}_4^2,\mathfrak{q}_5^2
\right) $,
$$(\mathfrak{n}_1^*)^{\circ}\cap (\mathfrak{m}_2^*)^{\circ}=\emptyset\ \text{and}\ \
\mathfrak{n}_1^*\cap
\mathfrak{m}_2^*=\text{CoSp}\left(\frac{(2,0,3)}{\sqrt{13}}\right).$$
 From $\{
\overline{\mathfrak{m}_1\mathfrak{n}_1},\mathfrak{m}_2\}$ where $
\overline{\mathfrak{m}_1\mathfrak{n}_1}^*=\text{CoSp}\left(
 \mathfrak{q}_4^1,\mathfrak{q}_5^1\right)$ and $
\mathfrak{m}_2^*=\text{CoSp}
\left(\mathfrak{q}_1^2,\mathfrak{q}_2^2,\mathfrak{q}_4^2,\mathfrak{q}_5^2
\right) $,
 $$(\overline{\mathfrak{m}_1\mathfrak{n}_1}^*)^{\circ}
 \cap (\mathfrak{m}_2^*)^{\circ}=\emptyset\ \ \text{and}\ \
 \overline{\mathfrak{m}_1\mathfrak{n}_1}^* \cap  \mathfrak{m}_2^*=
 \text{CoSp}\left(\frac{(2,0,3)}{\sqrt{13}}\right).$$
 As we point out in Remark \ref{r32}, it is not  just cones $\bigcap\mathbb{F}_\nu^*$,
 but  their interiors $\bigcap (\mathbb{F}_\nu^*)^{\circ}$
 that satisfy the overlapping condition  (\ref{4.1dd}). Thus even if   $
\{\mathfrak{n}_1,\mathfrak{m}_2\}$ and $\{
\mathfrak{m}_1,\mathfrak{n}_1, \mathfrak{m}_2\}$ are odd sets, it
does not prevent the uniform boundedness of the integrals:
$$  \sup_{r_j\in (0,1),\xi\in\mathbb{R}^2}\left|\int_{\prod (-r_j,r_j)}
e^{i\xi_1(c_{\mathfrak{m}_1}^1t^{\mathfrak{m}_1}+c_{\mathfrak{n}_1}^1t^{\mathfrak{n}_1})
+\xi_2(
 c_{\mathfrak{m}_2}^2t^{\mathfrak{m}_2}+
  c_{\mathfrak{n}_2}^2t^{\mathfrak{n}_2})}
  \frac{dt_1}{t_1}\frac{dt_2}{t_2}\frac{dt_3}{t_3}\right| \le C. $$
\end{exam}
\medskip

\subsection{Representations of Unbounded Faces}

\begin{lemma}\label{dg23}
Let $\Lambda\subset\mathbb{Z}_+^n$ and $S_0\subset S\subset N_n$.
Suppose that
 $\mathbb{F}\in\mathcal{F}\left({\bf N}(\Lambda,S)\right)$
  such that $\mathfrak{q}=(q_j)\in  (\mathbb{F}^*)^{\circ}|{\bf
N}(\Lambda,S)$
 where $q_j=0$ if $j\in S_0$ and $q_j>0$ if $j\in S\setminus S_0$.
 Then,
\begin{eqnarray}\label{dss}
\mathbb{F}=\mathbb{F}+\mathbb{R}_+^{S_0},
\end{eqnarray}
and
\begin{eqnarray}\label{ds41}
\mathbb{F}={\bf N}(\Lambda\cap \mathbb{F}, S_0).
\end{eqnarray}
Here $S_0$ can be an empty set.
\end{lemma}
\begin{proof}[Proof of (\ref{dss})]
Since $0\in
 \mathbb{R}^{S_0}_+ $,
$\mathbb{F}\subset\mathbb{F}+ \mathbb{R}_+^{S_0}$.  Let
$\mathfrak{m}+\sum_{j\in S_0}a_j{\bf e}_j\in \mathbb{F}+
\mathbb{R}_+^{S_0} $ where $\mathfrak{m}\in\mathbb{F}$. Assume that
$\mathfrak{m}+\sum_{j\in S_0}a_j{\bf e}_j\in {\bf
N}(\Lambda,S)\setminus \mathbb{F}$. By Definition \ref{dualface},
 \begin{eqnarray*}
\langle\mathfrak{m},\mathfrak{q}\rangle< \left\langle \mathfrak{m}+\sum_{j\in
S_0}a_j{\bf e}_j,\mathfrak{q}\right\rangle\ \ \text{for
$\mathfrak{q}\in (\mathbb{F}^*)^{\circ}|{\bf N}(\Lambda,S)$},
\end{eqnarray*}
which is impossible because $q_j=0$ for $j\in S_0$ in the
hypothesis.
   Thus
\begin{eqnarray*}
\mathfrak{m}+\sum_{j\in S_0}a_j{\bf e}_j\in\mathbb{F}.
\end{eqnarray*}
This implies that $\mathbb{F}+ \mathbb{R}_+^{S_0}
\subset\mathbb{F}$.
\end{proof}
\begin{proof}[Proof of (\ref{ds41})]
By definition, ${\bf N}(\Lambda\cap\mathbb{F}, S_0)$ is the smallest
convex set containing $(\Lambda\cap\mathbb{F})+ \mathbb{R}^{S_0}_+
$. In view of (\ref{dss}),  $\mathbb{F}$ contains the set set
$(\Lambda\cap\mathbb{F})+ \mathbb{R}^{S_0}_+ $. Thus,
$${\bf
N}(\Lambda\cap\mathbb{F}, S_0)\subset \mathbb{F}. $$ We next show
that  $\mathbb{F}\subset {\bf N}(\Lambda\cap\mathbb{F}, S_0).$ Let
${\bf x}\in \mathbb{F}\subset{\bf N}(\Lambda,
S)=\text{Ch}(\Lambda+\mathbb{R}_+^{S})$. Then,
$${\bf x}=\sum_{\mathfrak{m}\in  \Omega}c_{\mathfrak{m}}\mathfrak{m}\
\ \text{with $\Omega$ is a finite subset of $
\Lambda+\mathbb{R}_+^{S}$}$$  where $\sum_{\mathfrak{m}\in
\Omega}c_{\mathfrak{m}}=1$ and $c_\mathfrak{m}>0$. Assume that
$\mathfrak{m}\in\Omega\cap \mathbb{F}^c\ne\emptyset$. Then by
Definition \ref{dualface},  for $\mathfrak{q}=(q_j)\in
(\mathbb{F}^*)^{\circ}|{\bf N}(\Lambda,S)$,
$\langle\mathfrak{m},\mathfrak{q}\rangle>\langle{\bf x},\mathfrak{q}\rangle$. Thus
$$\langle{\bf x},\mathfrak{q}\rangle=\sum_{\mathfrak{m}\in
\mathbb{F}\cap\Omega }c_{\mathfrak{m}}\langle\mathfrak{m},\mathfrak{q}\rangle+
\sum_{\mathfrak{m}\in  \mathbb{F}^c\cap\Omega
}c_{\mathfrak{m}}\langle\mathfrak{m},\mathfrak{q}\rangle>
\left(\sum_{\mathfrak{m}\in  \Omega}c_{\mathfrak{m}}\right)\langle{\bf
x},\mathfrak{q}\rangle=\langle{\bf x},\mathfrak{q}\rangle $$ which is a
contradiction. So  $\Omega\cap \mathbb{F}^c=\emptyset$.
 Hence
\begin{eqnarray}\label{essbb}
{\bf x}=\sum_{\mathfrak{m}\in \Omega}c_{\mathfrak{m}}\mathfrak{m}\
\text{where}\ \Omega\subset
\mathbb{F}\cap\left(\Lambda+\mathbb{R}_+^{S}\right).
\end{eqnarray}
Here each $\mathfrak{m}\in \Omega\subset\mathbb{F}
\cap\left(\Lambda+\mathbb{R}_+^{S}\right)$ above is expressed as
\begin{eqnarray}\label{oee}
\mathfrak{m}= {\bf z}+\sum_{j\in S}a_j{\bf e}_j\in\mathbb{F}\
\text{where}\ {\bf z}\in\Lambda\ \text{and}\ a_j\ge 0.
\end{eqnarray}
By  Definition \ref{dualface},  for $\mathfrak{q}=(q_j)\in
(\mathbb{F}^*)^{\circ}|{\bf N}(\Lambda,S)$,
$\mathfrak{m}\in\mathbb{F}$ and ${\bf z}\in {\bf N}(\Lambda,S)$,
 \begin{eqnarray}\label{oeey}
\left\langle {\bf z}+\sum_{j\in S}a_j{\bf e}_j,
\mathfrak{q}\right\rangle\le \langle {\bf z}, \mathfrak{q}\rangle.
\end{eqnarray}
If ${\bf z} \in {\bf N}(\Lambda,S)\setminus\mathbb{F}$, then the
inequality in (\ref{oeey}) is strict. This is impossible because
$q_j$ with $j\in S$ in $\mathfrak{q}$ is nonnegative in the above
hypothesis.
 Thus  ${\bf z}\in \mathbb{F}$ in (\ref{oee}).
  Moreover $a_j=0$ for $j\in S\setminus S_0$ in (\ref{oeey})
 because $q_j>0$ for $j\in  S\setminus S_0$.
 Therefore ${\bf z}\in \mathbb{F}\cap \Lambda$ and $j\in S_0$ in (\ref{oee}). Hence,
 in (\ref{essbb}),
\begin{eqnarray*}
{\bf x}=\sum_{\mathfrak{m}\in \Omega}c_{\mathfrak{m}}\mathfrak{m}\ \
\text{where}\ \ \Omega\subset (\mathbb{F}\cap
\Lambda)+\mathbb{R}_+^{S_0}\ \  \text{and}\ \ \sum_{\mathfrak{m}\in
\Omega}c_{\mathfrak{m}}=1,
\end{eqnarray*}
that is ${\bf x}\in
\text{Ch}((\mathbb{F}\cap\Lambda)+\mathbb{R}_+^{S_0})={\bf
N}(\mathbb{F}\cap \Lambda,S_0)$, which implies $ \mathbb{F}\subset {\bf
N}(\Lambda\cap\mathbb{F}, S_0) .$
\end{proof}

\subsection{Essential Faces}
In constructing a sequence $\{\mathbb{F}^*_\nu(s)\}_{s=0}^N$ in
(\ref{pathg}), we need the following concept of faces.
\begin{definition}\label{dfu1}
 Let $\mathbb{P}$ be a  polyhedron such that $\mathbb{B}\subset\mathbb{P}$.
 Then a set $F(\mathbb{B}|\mathbb{P})$ is defined to be the smallest face of $\mathbb{P}$ containing $\mathbb{B}$
 in the sense that $$\mathbb{B}\subset
 F(\mathbb{B}|\mathbb{P})\preceq \mathbb{P}\ \ \text{and}\ \  \mathbb{B} \nsubseteq
 \mathbb{G}\   \text{ for any}\
 \mathbb{G} \precneqq F(\mathbb{B}|\mathbb{P}).$$
 We call $F(\mathbb{B}|\mathbb{P})$ the essential face of $\mathbb{P}$ containing $\mathbb{B}$.
 See the first and second pictures in Figure \ref{graph3}.
 \end{definition}

  \begin{figure}
 \centerline{\includegraphics[width=14cm,height=10cm]{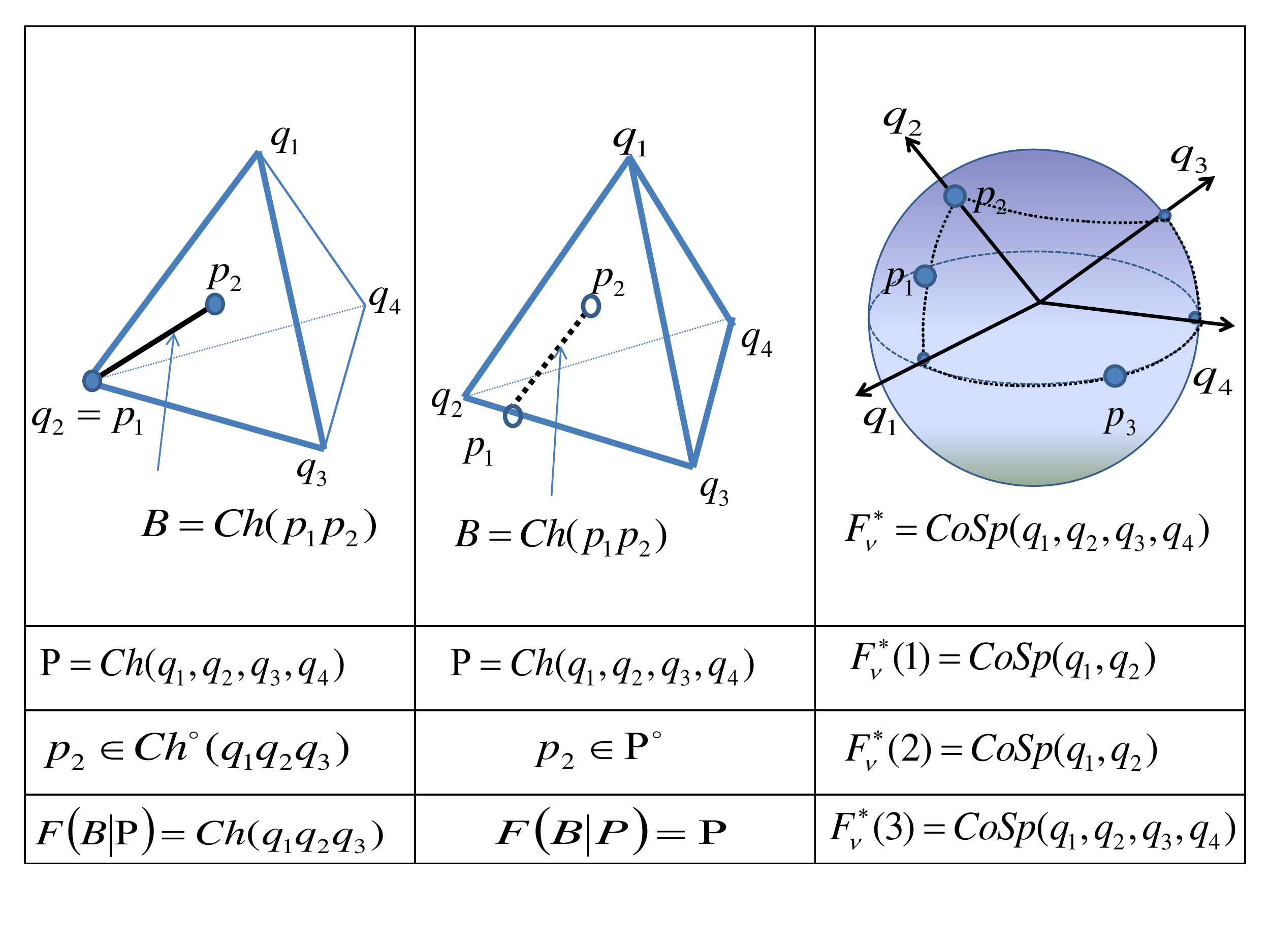}}
  \caption{Essential Faces.} \label{graph3}
 \end{figure}

\begin{lemma}\label{lem7k}
Let $\mathbb{P}$ be a  polyhedron such that  $\mathbb{B}\cap
\mathbb{P}^{\circ}\ne \emptyset$. Then
$$F(\mathbb{B}|\mathbb{P})=\mathbb{P}. $$
\end{lemma}
\begin{proof}
We see that  $F(\mathbb{B}|\mathbb{P}) \preceq\mathbb{P}$. Assume
that $F(\mathbb{B}|\mathbb{P}) \precneqq\mathbb{P}$. Then $
\mathbb{B}\subset F(\mathbb{B}|\mathbb{P})
\subset\partial\mathbb{P}$. This is a contradiction to the
hypothesis $\mathbb{B}\cap  \mathbb{P}^{\circ}\ne \emptyset$.
\end{proof}
\begin{lemma}\label{lem77k}
Let $\mathbb{P}$ be a  polyhedron such that $\mathbb{B}\subset\mathbb{P}$. Then
$$(F(\mathbb{B}|\mathbb{P}))^{\circ}\cap \text{Ch}(\mathbb{B})\ne \emptyset. $$
\end{lemma}
\begin{proof}
If not, $\text{Ch}(\mathbb{B})\subset \partial
F(\mathbb{B}|\mathbb{P})$. By Lemma \ref{lem411d},
$\text{Ch}(\mathbb{B}) \subset\mathbb{G}\precneqq
F(\mathbb{B}|\mathbb{P})$, which is impossible by Definition
\ref{dfu1}.
\end{proof}
\begin{lemma}\label{lemj7}
 Let $\mathbb{P}$ be a  polyhedron in $\mathbb{R}^n$ and let $\mathbb{B}\subset\mathbb{P}$ be a convex set.
 Then
\begin{eqnarray*}
\mathbb{B}^{\circ}\subset  F(\mathbb{B}|\mathbb{P})^{\circ}.
\end{eqnarray*}
\end{lemma}
\begin{proof}
We need the following observation:
 If two affine spaces $V_1$ and $V_2$
 meet  at $z\in V_1\cap V_2$ with $V_1\nsubseteq  V_2$ (transversally), then
\begin{eqnarray}\label{4bb}
\text{ $B_{V_1}(z,\epsilon)\cap (V_2^{+})^{\circ}\ne \emptyset$ and
$B_{V_1}(z,\epsilon)\cap (V_2^{-})^{\circ}\ne \emptyset$ for any
$\epsilon>0$ }
\end{eqnarray}
where $B_{V_1}(z,\epsilon)=\{v\in V_1:|v-z|<\epsilon\}$ is an
$\epsilon$-neighborhood of $z$ in $V_1$. Let
$z\in\mathbb{B}^{\circ}$. Then we show that $z\in
(F(\mathbb{B}|\mathbb{P}))^{\circ}$. Since $z\in
\mathbb{B}^{\circ}\subset \mathbb{B}\subset
F(\mathbb{B}|\mathbb{P})$,  it suffices to prove that $z\in\partial
F(\mathbb{B}|\mathbb{P})$ leads to a contradiction that $z\notin
\mathbb{B}^{\circ}$. If $ z\in \partial F(\mathbb{B}|\mathbb{P})$,
then by Definition \ref{dfu5},
     $z\in \mathbb{G}\precneqq F(\mathbb{B}|\mathbb{P})$
     where $\mathbb{G}\subset \partial F(\mathbb{B}|\mathbb{P})$.
     Let $V_{am}(\mathbb{G})$ be the plane
     containing $\mathbb{G}$ with
 \begin{eqnarray}\label{e77}
     \text{dim}(V_{am}(\mathbb{G}))=k-1\le k=\text{dim}(F(\mathbb{B}|\mathbb{P}))\
     \text{and}\  F(\mathbb{B}|\mathbb{P})\subset V_{am}^+(\mathbb{G}).
     \end{eqnarray}
By Definition \ref{dfu1} and Lemma \ref{lem77k},  $\mathbb{B}\cap
F(\mathbb{B}|\mathbb{P})^{\circ}\ne\emptyset$,
  that is, $V_{am}(\mathbb{B})\nsubseteq  V_{am}(\mathbb{G})$. From $z\in  \mathbb{B}^{\circ}$ and $z\in
  \mathbb{G}$, it follows that
  $z\in V_{am}(\mathbb{B})\cap V_{am}(\mathbb{G})$. By (\ref{4bb}) and (\ref{e77})
  with $V_1=V_{am}(\mathbb{B}) $ and $V_2=V_{am}(\mathbb{G})$,
  \begin{eqnarray*}
 B_{V_{am}(\mathbb{B})}(z,\epsilon) \nsubseteq F(\mathbb{B}|\mathbb{P}),\ \
  \text{which implies that}\ \ B_{V_{am}(\mathbb{B})}(z,\epsilon) \nsubseteq \mathbb{B}\ \
  \text{for any $\epsilon>0$.}
\end{eqnarray*}
This means that  $z\notin  \mathbb{B}^{\circ}$.
\end{proof}

\section{Preliminaries Estimates}\label{sec6}
In this section, we prove Proposition \ref{propyy}, which is an
elementary tool for the $L^p$ estimation driven by the finite type
conditions in the same spirit of \cite{CNSW} and  \cite{SW}.
Proposition \ref{propyy} and Proposition \ref{pr5g} are  two basic $L^p$ estimation tools
used for the proof of  sufficiency parts of Main Theorems 1-3.
\subsection{Preliminary
Inequalities}

Under the same setting as in the definition of multiple Hilbert
transforms (1.1), we consider the multi-parameter maximal function
\begin{equation}
\mathcal{M}_\Lambda f(x) = \sup_{r_1,\cdots,
r_n>0}\,\frac{1}{r_1\cdots
r_n}\,\int_{-r_1}^{r_1}\cdots\int_{-r_n}^{r_n}\, \left|f\left(x
-P_\Lambda(t)\right)\right|\,dt\end{equation} defined for each
locally integrable function $f$ on $\mathbb{R}^d$.

\begin{theorem} For $\,1<p\leq\infty\,,$
$\mathcal{M}_\Lambda$ is a bounded operator from $L^p(\mathbb{R}^d)$
into itself and there exists a bound $C_p$ depending only on $\,p,
n, d\,$ and the maximal degree of the polynomials $P_\nu$ such that
$$\left\|\,\mathcal{M}_\Lambda
f\,\right\|_{L^p(\mathbb{R}^d)}\,\leq\,C_p\,
\|f\|_{L^p(\mathbb{R}^d)}\,.$$
\end{theorem}

\begin{remark} This result can be proved by combining a theorem of Ricci and Stein
(\cite{RS}, Theorem 7.1) and the so-called lifting argument (see
Chapter 11 of \cite{S}).  \end{remark}
\begin{remark}
B. Street  in \cite{SS3}
showed the $L^p$ boundedness for a variable coefficient version of
$\mathcal{M}_\Lambda$ associated with analytic functions. Furthermore, A. Nagel and M.
Pramanik in \cite{NP} obtain the $L^p$ boundedness for a different
kind of multi-parameter maximal operators, that were motivated by the
study of several complex variables. This maximal average is taken
over family of sets (balls) that are defined by finite number of
monomial inequalities. In particular, to establish the $L^p$ theory
in \cite{NP}, the geometric properties of the associated polyhedra
are also systematically studied.
\end{remark}

Take a function $\,\psi\in C^\infty_c([-2,2])\,$ such that $\,0\leq
\psi\leq 1\,$ and $\,\psi(u) =1\,$ for $\,|u|\leq 1/2\,.$ Put
$\,\eta(u) = \psi(u) -\psi(2u)\,.$ Given an integer $k\in\mathbb{Z}$ and
$\,\alpha, \beta, \gamma\in\{1,\cdots,n\}\,,$ we consider the
measures $A_k^{\,\alpha,\beta}$ and $P_k^\gamma$ defined in terms of
Fourier transforms
\begin{equation}\label{2.2}
\left(A_k^{\,\alpha, \beta}\right)\widehat{}\,(\xi) =
\psi\left(\frac{\xi_\alpha}{2^k\,\xi_\beta}\right)\,,\quad
\left(P_k^\gamma\right)\widehat{}\,(\xi) =
\eta\left(2^k\,\xi_\gamma\right)\,.
\end{equation}

\begin{lemma}\label{lem2.1}
Suppose that $\{\mathfrak{m}_k\}_{k=1}^M, \{\mathfrak{q}_j\}_{j=1}^N \subset\mathbb{Z}^n\,$
 where
$\,{\rm{rank}}\,\left[ \{\mathfrak{q}_j\}_{j=1}^N \right]=n\,.$ Given $\,\alpha_k,\beta_k, \gamma_j \in
\mathbb{R}\,,$ define
\begin{equation}\label{2.3}
A_J = A_{J\cdot\mathfrak{m}_1}^{\,\alpha_1, \beta_1}\ast\cdots\ast
A_{J\cdot\mathfrak{m}_M}^{\,\alpha_M, \beta_M}\quad\text{and}\quad
P_J=P_{J\cdot\mathfrak{q}_1}^{\gamma_1}\ast\cdots \ast
P_{J\cdot\mathfrak{q}_N}^{\gamma_N}
\end{equation}
for each $\,J\in\mathbb{Z}^n\,.$ Then for $1<p<\infty$,
\begin{equation}\label{2.3m}
 \biggl\|\,\left(\sum_{J\in\mathbb{Z}^n}
\left| P_J\ast
f\right|^2\right)^{1/2}\,\biggr\|_{L^p(\mathbb{R}^d)} \,\lesssim\,
\left\| f\right\|_{L^p(\mathbb{R}^d)}\,,
\end{equation}
and
\begin{equation}\label{2.3n}
 \biggl\|\,\left(\sum_{J\in\mathbb{Z}^n}
\left|A_J\ast P_J\ast
f\right|^2\right)^{1/2}\,\biggr\|_{L^p(\mathbb{R}^d)} \,\lesssim\,
\left\| f\right\|_{L^p(\mathbb{R}^d)}\,.
\end{equation}
\end{lemma}
\begin{proof}
It suffices to deal  with the sum over $\mathbb{Z}^n_+$. We show (\ref{2.3n}).  With
$\left(r_J(t)\right)$ denoting the Rademacher functions of product
form,
$$\biggl\|\left(\sum_{J\in\mathbb{Z}^n_+}
\left|A_J\ast P_J\ast
f\right|^2\right)^{1/2}\biggr\|_{L^p(\mathbb{R}^d)}^p \,\approx\,
\int_U\bigl\|\sum_{J\in\mathbb{Z}^n_+}\,r_J(t)\,A_J\ast P_J\ast f
\bigr\|_{L^p(\mathbb{R}^d)}^p dt$$ where $\,U=[0,1]^n\,.$ Consider
the symbol
$$ m(\xi) = \biggl(\sum_{J\in\mathbb{Z}^n_+}\,r_J(t)\,A_J\ast P_J\biggr)
\,\,\,\widehat{}\,(\xi)\,.$$ Using the full rank condition for the
$\mathfrak{q}_j$ and the support conditions, it can be shown that
$m$ satisfies
$$\left|\frac{\partial^\ell }{\partial_{\nu_1}\cdots\partial_{\nu_\ell}}\,
m(\xi_1, \cdots, \xi_d)\,\right|
\,\leq\,\frac{C_\ell}{\left|\xi_{\nu_1}\right|
\cdots\left|\xi_{\nu_\ell}\right|}$$ for every $\,\ell
=1,\cdots,d\,,$ where $\,1\leq \nu_1<\cdots<\nu_\ell\leq d\,.$ Thus
the desired conclusion follows from the multi-parameter
Marcinkiewicz multiplier theorem. (\ref{2.3m}) follows similarly.
\end{proof}

\begin{lemma}\label{lem2.2}
 Let $\,(\sigma_J)_{J\in\mathbb{Z}^n}\,$
be a sequence of positive measures on $\mathbb{R}^d$ with the
following properties :
\begin{align*}
&{\rm(i)}\quad\bigl\|\,\sigma_J\ast
f\,\bigr\|_{L^1(\mathbb{R}^d)}\,\lesssim\,
\|f\|_{L^1(\mathbb{R}^d)}\quad (J\in\mathbb{Z}^n)\\
&{\rm(ii)}\quad \left\|\,\sup_{J\in\mathbb{Z}^n}\,\left|\sigma_J\ast
f\right|\,\right\|_{L^{p_0}(\mathbb{R}^d)}\,\lesssim\,
\|f\|_{L^{p_0}(\mathbb{R}^d)}
\end{align*}
for some $\,1<p_0\leq 2\,$. Then
$$\biggl\|\left(\sum_{J\in\mathbb{Z}^n}
\left|\sigma_J\ast
f_J\right|^2\right)^{1/2}\biggr\|_{L^{p_1}(\mathbb{R}^d)}
\,\lesssim\, \biggl\|\left(\sum_{J\in\mathbb{Z}^n} \left|
f_J\right|^2\right)^{1/2}\biggr\|_{L^{p_1}(\mathbb{R}^d)}$$ for
$p_1$ determined by $\,1/p_1 \leq 1/2\,(1+1/p_0)\,.$
\end{lemma}

\begin{proof} For $\,1\leq p, \,q\leq\infty\,,$ consider
the operator $T$ defined by $\,T[(f_J)] = (\sigma_J\ast f_J)\,$ on
the mixed-norm spaces $L^p(\ell^q)$. The condition (i) implies that
$T$ maps $L^1(\ell^1)$ boundedly into itself. The condition (ii) and
the positivity of each $\sigma_J$ imply that $T$ maps
$L^{p_0}(\ell^\infty)$ boundedly into itself. It follows from the
vector-valued Riesz-Thorin interpolation theorem that $T$ maps
$L^{p_1}(\ell^2)$ boundedly into itself.
\end{proof}

\subsection{Basic $L^p$ estimates}
\begin{proposition}\label{propyy}
Let $\{H_J\}_{J\in\mathbb{Z}^d}$ be a class of measures such that
$\widehat{H}_J$ be the Fourier multiplier of $H_J$. Suppose that
\begin{eqnarray}\label{jr1v}
\left|\widehat{H}_J(\xi)\right|&\le& C\min\left\{
\left|2^{-J\cdot\mathfrak{q}_i}\xi_{\nu_i}\right|^{-\delta_1},
\left|2^{-J\cdot\mathfrak{q}_i}\xi_{\nu_i}\right|^{\delta_2}:i=1,\cdots,N\right\}
\end{eqnarray}
where \begin{eqnarray}\label{jr1v1}
\text{rank}\{\mathfrak{q}_i:i=1,\cdots,r\}=n.
\end{eqnarray}
 Then  for   $C_2=C/(1-2^{-\min\{\delta_1,\delta_2\}/N})^N$ with $C,\delta_1,\delta_2,N$ in (\ref{jr1v}),
\begin{eqnarray}\label{jr14}
 \sum_{J\in Z}\left|\widehat{H}_J(\xi)\right| \le
C_2\ \ \text{where} \ Z\subset \mathbb{Z}^n
\end{eqnarray}
which implies that for $A_J$  of the form defined as in (\ref{2.3}) and for any
$Z\subset\mathbb{Z}^n$,
\begin{eqnarray}\label{jr14anir}
&&\left\|\sum_{J\in Z} H_J*A_J*f\right\|_{L^2(\mathbb{R}^d)}\le
C_2\left\| f\right\|_{L^2(\mathbb{R}^d)}. \nonumber
\end{eqnarray}
Moreover, suppose that $1<p\le \infty$
\begin{eqnarray}\label{jq}
\left\|\sup_{J\in Z} |H_J|*f\right\|_{L^p(\mathbb{R}^d)}\le C
\left\| f\right\|_{L^p(\mathbb{R}^d)}.
\end{eqnarray}
Then,  for $1<p\le \infty$
\begin{eqnarray}\label{jr1x}
\left\|\sum_{J\in Z} H_J*A_J*f\right\|_{L^p(\mathbb{R}^d)}\le C_p
\left\| f\right\|_{L^p(\mathbb{R}^d)}.
\end{eqnarray}
\end{proposition}
\begin{proof}[Proof of (\ref{jr14})]
Let
\begin{eqnarray*}
\big(P_{-J\cdot \mathfrak{q}_{i}- \ell_{i}}\big)^{\wedge}(\xi)&=&
\eta\left( 2^{-J\cdot \mathfrak{q}_{i}-
\ell_{i}}\,\xi_{\nu_i}\right)\quad\text{for}\quad i=1,\cdots,
N\,\label{kdd}
\end{eqnarray*}
with a view to restricting frequency variables as
\begin{eqnarray}
  2^{-J\cdot \mathfrak{q}_{i}}\,\left|\xi_{\nu_i}\right|
\,&\approx &\, 2^{\ell_{i}}\quad\text{for}\quad
i=1,\cdots,N\,.\label{heh}
\end{eqnarray}
 We use  for (\ref{jr1v}) and (\ref{heh}) to obtain that
\begin{eqnarray}
\qquad \prod_{i=1}^N\eta\left( 2^{-J\cdot \mathfrak{q}_{i}-
\ell_{i}}\,\xi_{\nu_i}\right) \left|\widehat{H}_{J}(\xi)\right|\,\le \,C2^{- b\,|L|}\ \text{for
 }  \ b=\frac{\min\{\delta_1,\delta_2\}}{N} \, \ \text{and}\ C\
\text{ in (\ref{jr1v})}.\label{hr6}
\end{eqnarray}
where $|L|=\sum_{i=1}^{N} |\ell_i|$.
 Then by using positivity of $\eta$ and $\sum_{\ell_i\in\mathbb{Z}}\eta\left( 2^{-J\cdot \mathfrak{q}_{i}-
\ell_{i}}\,\xi_{\nu_i}\right)=1$,
\begin{eqnarray*}
\sum_{|J|\le R}\left|\widehat{H}_J(\xi)\right| &=&\sum_{L\in \mathbb{Z}^N}\prod_{i=1}^N\eta\left( 2^{-J\cdot \mathfrak{q}_{i}-
\ell_{i}}\,\xi_{\nu_i}\right)\sum_{|J|\le R}\left|\widehat{H}_J(\xi)\right|\\
&=&\sum_{L\in \mathbb{Z}^N}\sum_{|J|\le R}\prod_{i=1}^N\eta\left( 2^{-J\cdot \mathfrak{q}_{i}-
\ell_{i}}\,\xi_{\nu_i}\right)\left|\widehat{H}_J(\xi)\right|\\
&\le&\sum_{L\in \mathbb{Z}^N}C2^{- b\,|L|}\le C/(1-2^{-\min\{\delta_1,\delta_2\}/N})^N
\end{eqnarray*}
 where the first inequality follows from (\ref{hr6}) and the observation that for each fixed $\xi$, there exists finitely many
$J$ such that $\prod_{i=1}^N\eta\left( 2^{-J\cdot \mathfrak{q}_{i}-
\ell_{i}}\,\xi_{\nu_i}\right) \ne 0$. We proved (\ref{jr14}).
\end{proof}
\begin{proof}[Proof of (\ref{jr1x})]
Define
$$
\mathcal{P}_{J,L}^{\mathfrak{q}} =P_{-J\cdot \mathfrak{q}_{1}-
\ell_{1}}*\cdots*P_{-J\cdot \mathfrak{q}_{N}- \ell_{N}}\,, \
\text{where}\ \ L=(\ell_i)_{i=1}^N\in\mathbb{Z}^{N}.
$$
We use the Littlewood-Paley decomposition  for each $J\in\mathbb{Z}$:
\begin{equation*}
\sum_{ L\in\mathbb{Z}^{N}} \mathcal{P}_{J,L}^{\mathfrak{q}}\ast f
=f.
\end{equation*}
Define $\tilde{\mathcal{P}}_{J,L}^{\mathfrak{q}}$  by replacing
$\eta(\cdot)$ with $\eta(\cdot/2)$ in (\ref{kdd}). Then
$\tilde{\mathcal{P}}_{J,L}^{\mathfrak{q}}*\mathcal{P}_{J,L}^{\mathfrak{q}}=\mathcal{P}_{J,L}^{\mathfrak{q}}$.
Thus
$$  \sum_{J\in Z} H_J*A_J*\mathcal{P}_{J,L}^{\mathfrak{q}}*f=\sum_{J\in Z}\tilde{\mathcal{P}}_{J,L}^{\mathfrak{q}}*H_J*A_J*\mathcal{P}_{J,L}^{\mathfrak{q}}*f.$$
By Applying  the dual inequality  of (\ref{2.3m}) in Lemma \ref{lem2.1},
$$ \left\|\sum_{J\in Z}\tilde{\mathcal{P}}_{J,L}^{\mathfrak{q}}*H_J*A_J*
\mathcal{P}_{J,L}^{\mathfrak{q}}*f\right\|_{L^p(\mathbb{R}^{d})}\le
\biggl\| \biggl(\sum_{J\in Z}
\left|H_J*A_J*\mathcal{P}_{J,L}^{\mathfrak{q}}*f\right|^{2}\biggr)^{1/2}
\biggl\|_{L^p(\mathbb{R}^{d})}.  $$ Thus, it is sufficient to find a
constant $\,b>0\,$ independent of $ L \in\mathbb{Z}^N\,$ such that
\begin{equation}\label{6.4g}
 \biggl\| \biggl(\sum_{J\in Z}
\left|H_J*A_{J}
*
\mathcal{P}_{J,L}^{\mathfrak{q}}*f\right|^{2}\biggr)^{1/2}
\biggl\|_{L^p(\mathbb{R}^{d})} \,
  \le  \,
\tilde{C}2^{-b\, |L| }\| f\|_{L^p(\mathbb{R}^{d})}\,,
\end{equation}
where $\tilde{C}$ is a
multiple of $C$ in (\ref{jr1v}).
 By the rank condition   (\ref{jr1v1}) and (\ref{2.3n}) in Lemma
\ref{lem2.1},
\begin{eqnarray}\label{6.10b}
\biggl\| \Bigl(\sum_{J\in Z}|A^{\sigma}_{J}
*
\mathcal{P}_{J,L}^{\mathfrak{q}}*f|^2\Bigl)^{\frac{1}{2}}
\biggl\|_{L^p(\mathbb{R}^d)} &\lesssim&\parallel f
\parallel_{L^p(\mathbb{R}^d)}.
\end{eqnarray}
For $p=2$, we use (\ref{jr14}),(\ref{hr6}) and (\ref{6.10b})  to obtain (\ref{6.4g}).
  Applying a standard bootstrap
argument combined with (\ref{jq}), Lemmas \ref{lem2.1} and
\ref{lem2.2}, we obtain (\ref{6.4g}) for the other values of $p\ne 2$.
The proof of (\ref{jr1x}) is now complete.
\end{proof}
\begin{remark}\label{rem00}
The decay condition in (\ref{jr1v}) always holds for the case that
$\Lambda_\nu$'s are mutually disjoint. Given
$\mathbb{G}=(\mathbb{G}_\nu)\in\mathcal{F}({\bf N}(\Lambda,S))$, we
have from the multi-dimensional Van der Corput Lemma,
\begin{eqnarray}\label{nniiii}
|\mathcal{I}_J(P_{\mathbb{G}},\xi)|\le
C\min\left\{|2^{-J\cdot\mathfrak{m}_\nu}\xi_\nu|^{-\delta},1:\mathfrak{m}_\nu\in\mathbb{G}_\nu\cap\Lambda_\nu\
\text{for}\ \nu=1,\cdots,d\right\}.
\end{eqnarray}
\end{remark}

\section{Cone Type Decompositions}\label{condec}
\subsection{Cone  Decompositions}
Recall that $Z(S)=\prod Z_i$ with $Z_i=\mathbb{R}_+$ for $i\in S$
and $Z_i=\mathbb{R}$ for $i\in N_n\setminus S$ as in Definition
\ref{dine}. We decompose $Z(S)$ into finite number of different
cones that appears in (\ref{8765}) and (\ref{danbi5}) as follows:
\begin{proposition}\label{74sj}
 Let
  $\Lambda=(\Lambda_\nu)$ with
$\Lambda_\nu\subset \mathbb{Z}_+^n$ and  $S\subset \{1,\cdots,n\} $.
Then,
$$\bigcup_{\mathbb{F} \in\mathcal{F}\left( \vec{ {\bf
N}}(\Lambda,S)\right)}   {\rm Cap}(\mathbb{F}^*) = Z(S) \
\text{where}\
\rm{Cap}(\mathbb{F}^*)=\bigcap_{\nu=1}^d\mathbb{F}_\nu^*\
\text{for}\ \mathbb{F}=(\mathbb{F}_\nu)\in\mathcal{F}({\bf
N}(\Lambda,S)).$$ Moreover,
$$\bigcup_{\mathbb{F} \in\mathcal{F}\left( \vec{ {\bf
N}}(\Lambda,S)\right)}   {\rm Cap}((\mathbb{F}^*)^{\circ})  =
Z(S)\setminus\{0\}\ \text{where}\
\rm{Cap}((\mathbb{F}^*)^{\circ})=\bigcap_{\nu=1}^d(\mathbb{F}_\nu^*)^{\circ}
.$$
\end{proposition}

\begin{lemma}\label{lem713}
Let $\mathbb{P}=\mathbb{P}(\Pi)$ with
$\Pi=\{\pi_{\mathfrak{q}_j,r_j}:j=1,\cdots,N\}$ be a  polyhedron.
Then $$\inf\{\langle{\bf x}, {\bf e}\rangle:{\bf x}\in\mathbb{P}\}>-\infty\ \
\text{if and only if}\ \ {\bf e}\in
\rm{CoSp}(\{\mathfrak{q}_j:j=1,\cdots,N\}).$$
\end{lemma}
\begin{proof}
 Let  $\rho=\inf\{\langle{\bf x}, {\bf e}\rangle:{\bf x}\in\mathbb{P}\}>-\infty$ and set
 the plane $\pi_{{\bf e},\rho}=\{{\bf x}\in V:\langle {\bf x}, {\bf e}\rangle=\rho\}$.
 Since $\mathbb{P}$ is a closed set, $\mathbb{F}$ defined by
 $\pi_{{\bf e},\rho}\cap\mathbb{P}$ is a non-empty closed set.
 From $\rho\le {\bf x}\cdot {\bf e}$ for all ${\bf x}\in \mathbb{P}$
  and $\mathbb{F}=\pi_{{\bf e},\rho}\cap\mathbb{P}$,
\begin{equation}\label{lp12}
\mathbb{P}\setminus \mathbb{F}\subset  \left(\pi_{{\bf e},\rho}^{+}\right)^{\circ}.
\end{equation}
 Thus $\mathbb{F}\preceq \mathbb{P}$ and ${\bf e}\in\mathbb{F}^*
 \subset\rm{CoSp}(\{\mathfrak{q}_j:j=1,\cdots,N\})$ by using Propositions \ref{facerep} and \ref{pr50}.
 To  show the other direction, let
 ${\bf e}=\sum_{j=1}^N c_j\mathfrak{q}_j\in \rm{CoSp}(\{\mathfrak{q}_j:j=1,\cdots,N\})$. Then
\begin{equation*}
\langle {\bf e}, {\bf x}\rangle=\sum_{j=1}^N c_j\langle \mathfrak{q}_j, {\bf x}\rangle\ge
\sum_{j=1}^Nc_jr_j>-\infty
\end{equation*}
 for all ${\bf x}\in\mathbb{P}=\bigcap_{j=1}^N\{\langle\mathfrak{q}_j, {\bf x}\rangle\ge r_j:j=1,\cdots,N\}$.
\end{proof}

\begin{lemma}\label{b32}
 Let $\mathbb{P}=\mathbb{P}(\Pi)$ with $\Pi=\{\pi_{\mathfrak{q}_j,r_j}:j=1,\cdots,N\}$
  be a  polyhedron.
 Then
 $$ \rm{CoSp}(\{\mathfrak{q}_j:j=1,\cdots,N\}) =\bigcup_{\mathbb{F}\preceq\mathbb{P}}\mathbb{F}^*.$$
\end{lemma}
\begin{proof}
We first show $\subset$. Let   ${\bf e}=\sum_{j=1}^N
c_j\mathfrak{q}_j \in \rm{CoSp}(\{\mathfrak{q}_j:j=1,\cdots,N\})
 $ and let
$$\rho=\inf\left\{ \langle{\bf e}, {\bf x}\rangle=\sum_{j=1}^N c_j\langle\mathfrak{q}_j,
{\bf x}\rangle:{\bf x}\in\mathbb{P}\right\} $$ that exits from Lemma
\ref{lem713}. Set $\mathbb{F}=\pi_{{\bf e},\rho}\cap\mathbb{P}$. By
(\ref{lp12}),
  $\mathbb{F}$ is a face of $\mathbb{P}$ with a supporting
  plane $\pi_{{\bf e},\rho}$. Thus
  ${\bf e}\in \mathbb{F}^*$. The other direction $\supset$
  follows from Propositions \ref{facerep} and  \ref{pr50}.
\end{proof}

\begin{lemma}\label{lem132}
 Let ${\bf u}=(u_1,\cdots,u_n)\in  \mathbb{F}^* $ and $
\mathbb{F}\in\mathcal{F}({\bf N}(\Lambda,S))$. Then $u_j\ge 0$ for
all $j\in S$.
\end{lemma}
\begin{proof}
Let $\mathfrak{m}\in \mathbb{F}$ and $j\in S$. Then
 $\mathfrak{m}+r{\bf e}_j\in {\bf N}(\Lambda,S)$ for all $r\ge 0$.
By Definition \ref{brg} and ${\bf u}\in\mathbb{F}^*$,  $u_jr=\langle{\bf
u},  \mathfrak{m}+r{\bf e}_j-\mathfrak{m} \rangle\ge 0$.
 Thus $u_j\ge 0$. See Figure \ref{graph5}.
 \end{proof}

\begin{lemma}\label{lem715}
 Let $S\subset N_n$ and $\Omega\subset\mathbb{Z}_+^n$ be a finite set.
  Suppose that $\mathbb{P}=\mathbb{P}(\Pi)$ with $\Pi=\{\pi_{\mathfrak{q}_j,r_j}:j=1,\cdots,N\}$
  is a   polyhedron given by ${\bf N}(\Omega,
S)$. Then
   $$\bigcup_{\mathbb{F} \in\mathcal{F}\left({\bf
N}(\Omega,S)\right)} \mathbb{F}^* = Z(S). $$
Moreover, $\bigcup_{\mathbb{F} \in\mathcal{F}\left({\bf
N}(\Omega,S)\right)} (\mathbb{F}^*)^{\circ} = Z(S)\setminus\{0\} $.
\end{lemma}
\begin{proof}
It follows $\subset$ from Lemma \ref{lem132}.   We next show
$\supset$. Put
\begin{eqnarray*}
m_k&=&\min\{u_k:{\bf u}=(u_1,\cdots,u_n)\in \Omega\},\\
M_k &=&\max\{u_k:{\bf u}=(u_1,\cdots,u_n)\in \Omega\}
\end{eqnarray*}
By ${\bf N}(\Omega,S)=\rm{Ch}\{{\bf u}+\mathbb{R}_+^S:{\bf
u}\in\Omega\}$,
\begin{itemize}
\item[(1)] if $k\in N_n\setminus S$, then $m_k\le x_k\le M_k$ for all ${\bf x}=(x_1,\cdots,x_n)\in {\bf
N}(\Omega,S)$,
    \item[(2)] if $k \in S$, then $m_k\le x_k $ for all ${\bf x}=(x_1,\cdots,x_n)\in {\bf
    N}(\Omega,S)$.
\end{itemize}
Thus,
  if $k\in N_n\setminus S$,   then ${\bf e}_k,-{\bf e}_k\in\rm{CoSp}(\{\mathfrak{q}_j:j=1,\cdots,N\}) $ by Lemma \ref{lem713}
  and (1) above.
 If $k \in S$,  then ${\bf e}_k\in\rm{CoSp}(\{\mathfrak{q}_j:j=1,\cdots,N\}) $ by
 Lemma \ref{lem713}
 and (2) above.
  Hence $$Z(S)=\rm{CoSp}(\{\pm{\bf e}_k: k\in N_n\setminus S\}\cup \{{\bf e}_k: k\in  S\})
\subset  \rm{CoSp}(\{\mathfrak{q}_j:j=1,\cdots,N\}).$$ By Lemma
\ref{b32}, $Z(S) \subset \bigcup_{\mathbb{F}
\in\mathcal{F}\left({\bf N}(\Omega,S)\right)} \mathbb{F}^*$. By
Definitions \ref{dfu5}-\ref{dualface} together with  (\ref{e8v}),
\begin{eqnarray}\label{o550}
 \bigcup_{\mathbb{F} \in\mathcal{F}\left({\bf
N}(\Omega,S)\right)} (\mathbb{F}^*)^{\circ}=\bigcup_{\mathbb{F}
\in\mathcal{F}\left({\bf N}(\Omega,S)\right)}
\mathbb{F}^*\setminus\{0\}.
\end{eqnarray}
This implies the last statement.
 \end{proof}
\begin{proof}[Proof of Proposition \ref{74sj}]
 By    Lemma \ref{lem715},
 $$ \bigcup_{\mathbb{F}_\nu\in\mathcal{F}({\bf N}(\Lambda_\nu,S))}\mathbb{F}_\nu^*= Z(S) \ \ \text{for every $\nu=1,\cdots,d$}.$$
 Hence, by taking an intersection for $\nu=1,\cdots,d$,
 $$ \bigcup_{\mathbb{F} \in\mathcal{F}\left( \vec{ {\bf
N}}(\Lambda,S)\right)}   {\rm Cap}(\mathbb{F}^*) =
\bigcap_{\nu=1}^d\bigcup_{\mathbb{F}_\nu\in\mathcal{F}({\bf
N}(\Lambda_\nu,S))}\mathbb{F}_\nu^*=Z(S) .$$ The last
statement follows from (\ref{o550}).
\end{proof}

\subsection{Projection to Sphere; Boundary Deleted Neighborhood}
We show that
\begin{proposition}\label{l50}
Suppose that  $\text{rank}\left(\bigcup_{\nu=1}^d {\bf N}(\Lambda_\nu,S) \right)
\le n-1$ and the hypothesis of Main
Theorems \ref{main3} holds. Then for $1<p<\infty$,
\begin{equation*}
\left\|\sum_{J\in Z(S)}
H^{P_{\Lambda}}_J*f\right\|_{L^p(\mathbb{R}^d)}\,\lesssim\,\left\|
f\right\|_{L^p(\mathbb{R}^d) }
\end{equation*}
where $\mathcal{I}_J(P_{\Lambda},\xi)=\left(H^{P_{\Lambda}}_J\right)^{\wedge}(\xi)$.
\end{proposition}
To show Proposition \ref{l50}, we consider the projective cone of $\mathbb{F}^*$ to the sphere $\mathbb{S}^{n-1}$.
\begin{definition}
In stead of working with the cone $ \mathbb{F}^* $ of a face
$\mathbb{F}$, it is sometimes convenient to work with its
intersection $\mathbb{F}^*\cap\mathbb{S}^{n-1}$ with the sphere. We
denote it and its boundary  by
$$ \mathbb{S}[\mathbb{F}^*]=\mathbb{F}^*\cap\mathbb{S}^{n-1}\ \ \text{and}\
\
\partial\mathbb{S}[\mathbb{F}^*]=(\partial \mathbb{F}^*)\cap\mathbb{S}^{n-1}.$$
\end{definition}
Let $K\in \mathbb{S}^{n-1}$, then we define the
$\epsilon$-neighborhood of $K$ by
\begin{eqnarray*}
N_{\epsilon}(K)=\{x\in \mathbb{S}^{n-1}: |x-y|<\epsilon\ \text{for
some $y\in K$}\}.
\end{eqnarray*}
\begin{definition}\label{deef}[Boundary Deleted $\epsilon$-neighborhood of
$\mathbb{S} [\mathbb{F}^*] $ Let $\mathbb{P}$ be a polyhedron in
$\mathbb{R}^n$ of  $\text{dim} (\mathbb{P})=m< n$ and let
$\mathbb{S}[\mathbb{F}^*] \in\mathbb{S}^{n-1}$ with
$\mathbb{F}\in\mathcal{F}^{m-k}(\mathbb{P})$ where $k=0, \cdots,m$.
To give some width to $\mathbb{S}[\mathbb{F}^*] \setminus
N_{\epsilon}(\partial \mathbb{S}[\mathbb{F}^*] ))$, we define a
boundary deleted $\epsilon$-neighborhood
$\mathbb{S}_{\epsilon}[\mathbb{F}^*]$ by
\begin{eqnarray*}
\mathbb{S}_{\epsilon}[\mathbb{F}^*]&=&N_{\epsilon/M^k}\left(\mathbb{S}[\mathbb{F}^*])
\right)\ \
\text{for $k=0$ where $\mathbb{F}^*=\text{CoSp}(\Pi_b)=V^{\perp}(\mathbb{P})$}\,,\label{ua}\\
\mathbb{S}_{\epsilon}[\mathbb{F}^*]&=&N_{\epsilon/M^{k+1/3}}\left(\mathbb{S}[\mathbb{F}^*]\setminus
N_{\epsilon/M^{k-1/3}}(\partial \mathbb{S}[\mathbb{F}^*])\right)\ \
\text{for $1\le k\le m$ }\nonumber
\end{eqnarray*}
where $\mathbb{F}^*=\text{CoSp}(\Pi_a(\mathbb{F}))\oplus
V^{\perp}(\mathbb{P})=\text{CoSp}(\{\mathfrak{q}_j\}_{j=1}^{\ell})\oplus
V^{\perp}(\mathbb{P})$ as in (\ref{snn26}). Here $M$ will be chosen
to be a large positive number. For the case that   $\text{dim}
(\mathbb{P})=n$ with $\mathbb{F}\in\mathcal{F}^{m-k}(\mathbb{P})$
where $k=1, \cdots,n$,  we define a boundary deleted
$\epsilon$-neighborhood $\mathbb{S}_{\epsilon}[\mathbb{F}^*]$ by
\begin{eqnarray*}
\mathbb{S}_{\epsilon}[\mathbb{F}^*]&=&N_{\epsilon/M^k}\left(\mathbb{S}[\mathbb{F}^*])
\right)\ \
\text{for $k=1$ where $\mathbb{F}^*=\text{CoSp}(\mathfrak{q}_j)$}\,,\label{ua}\\
\mathbb{S}_{\epsilon}[\mathbb{F}^*]&=&N_{\epsilon/M^{k+1/3}}\left(\mathbb{S}[\mathbb{F}^*]\setminus
N_{\epsilon/M^{k-1/3}}(\partial \mathbb{S}[\mathbb{F}^*])\right)\ \
\text{for $2\le k\le n$ }\nonumber
\end{eqnarray*}
where $\mathbb{F}^*=\text{CoSp}(\{\mathfrak{q}_j\}_{j=1}^{\ell})$.
See $\mathbb{S}_{\epsilon}[\mathbb{F}^*(\mathfrak{q}_1)]$ and
$\mathbb{S}_{\epsilon}[\mathbb{F}^*(\mathfrak{q}_1,\mathfrak{q}_4)]$
 in the right side of Figure \ref{graph2}.
 \end{definition}

\begin{lemma}\label{lem840}
Let $\Omega\subset\mathbb{Z}_+^n$ and $\mathbb{P}={\bf
N}(\Omega,S)\subset\mathbb{R}^n$ be a polyhedron.   Then
$$ \bigcup_{\mathbb{F}\in\mathcal{F}({\bf
N}(\Omega,S))}
\mathbb{S}[\mathbb{F}^*]\subset\bigcup_{\mathbb{F}\in\mathcal{F}({\bf
N}(\Omega,S))} \mathbb{S}_{\epsilon}[\mathbb{F}^*].
$$
\end{lemma}
\begin{proof}
We prove the case $\text{dim}(\mathbb{P})=m<n$.  By Definition \ref{deef},
$$\bigcup_{\mathbb{F}\in\mathcal{F}^{m}({\bf N}(\Omega,S))}
\mathbb{S}[\mathbb{F}^*]\subset\bigcup_{\mathbb{F}\in\mathcal{F}^{m}({\bf
N}(\Omega,S))} \mathbb{S}_{\epsilon}[\mathbb{F}^*]=
\bigcup_{\mathbb{F}\in\mathcal{F}^{m}({\bf N}(\Omega,S))}
N_{\epsilon/M}\left(\mathbb{S}[\mathbb{F}^*] \right).$$ Using this
and Definition \ref{deef},
$$\bigcup_{\mathbb{F}\in\mathcal{F}^{m-1}({\bf N}(\Omega,S))}
\mathbb{S}[\mathbb{F}^*]\subset
\left(\bigcup_{\mathbb{F}\in\mathcal{F}^{m-1}({\bf N}(\Omega,S))}
\mathbb{S}_{\epsilon}[\mathbb{F}^*]\right)\cup
\left(\bigcup_{\mathbb{F}\in\mathcal{F}^{m}({\bf N}(\Omega,S))}
\mathbb{S}_{\epsilon}[\mathbb{F}^*]\right).$$ Inductive application
of this inclusion completes the proof.
\end{proof}
Note that by Proposition  \ref{74sj},
\begin{eqnarray}\label{j1}
 \bigcup_{\mathbb{F}=(\mathbb{F}_\nu) \in\mathcal{F}\left( \vec{ {\bf N}}(P,S)\right)}
\bigcap_{\nu=1}^d\mathbb{S} [\mathbb{F}^*_\nu] =
Z(S)\cap\mathbb{S}^{n-1}.
 \end{eqnarray}
 By Lemma \ref{lem840} together with (\ref{j1}), we have
\begin{lemma}
Let $\Lambda=(\Lambda_\nu)$ with $\Lambda_\nu\subset\mathbb{Z}_+^n$
and $\vec{{\bf N}}(\Lambda,S)=\left({\bf N}(\Lambda_\nu,S)\right)$.
Then
$$Z(S)\cap\mathbb{S}^{n-1}\subset\bigcup_{\mathbb{F}=(\mathbb{F}_\nu)\in\mathcal{F}(\vec{{\bf
N}}(\Lambda,S))} \bigcap_{\nu=1}^d
 \mathbb{S}_{\epsilon}[\mathbb{F}^*_\nu]. $$
\end{lemma}
Using this we can decompose for sufficiently small $\epsilon>0$,
\begin{eqnarray}\label{j2}
\sum_{J\in Z(S)} H^{P_{\Lambda}}_J=\sum_{\mathbb{F}
\in\mathcal{F}(\vec{{\bf N}}(\Lambda,S))} \sum_{J/|J|\in
Z\subset\bigcap_{\nu=1}^d
 \mathbb{S}_{\epsilon}[\mathbb{F}^*_\nu] } H^{P_{\Lambda}}_J.
\end{eqnarray}
In order to check the overlapping condition (\ref{4.1dd}), we need
the following lemma.
\begin{lemma}\label{des}
 Let
$\mathbb{F}_\nu\in\mathcal{F}({\bf N}(\Lambda_\nu,S))$ for
$\nu=1,\cdots,d$. Then for some sufficiently small $\epsilon>0$, we
have the property that $ \bigcap_{\nu=1}^d
 \mathbb{S}_{\epsilon}[\mathbb{F}^*_\nu] \ne \emptyset$
implies that $
\bigcap_{\nu=1}^d\mathbb{S}[(\mathbb{F}^*_\nu)^{\circ}]\ne
\emptyset$.
\end{lemma}
\begin{proof}[Proof of lemma \ref{des}]
We prove the case $\text{dim}(\mathbb{P})=m<n$. It suffices to find
an $\epsilon>0$ such that
$$\bigcap_{\nu=1}^d\mathbb{S}[(\mathbb{F}^*_\nu)^{\circ}]= \emptyset\ \ \text{implies that}\ \
\bigcap_{\nu=1}^d  \mathbb{S}_{\epsilon}[\mathbb{F}^*_\nu] =
\emptyset.$$ Suppose that $d$-tuple $(\mathbb{F}_\nu)$ of faces are given so that
\begin{eqnarray}\label{m4}
\bigcap_{\nu=1}^d\mathbb{S}[(\mathbb{F}^*_\nu)^{\circ}]= \emptyset.
 \end{eqnarray}
Note that $\mathbb{S}[\mathbb{F}^*_\nu]\setminus N_{
\epsilon}(\partial \mathbb{S}[\mathbb{F}^*_\nu])\subset
\mathbb{S}[(\mathbb{F}^*_\nu)^{\circ}]$ for any positive number
$\epsilon>0$ and $\mathbb{S}[
\mathbb{F}^*_\nu]=\mathbb{S}[(\mathbb{F}^*_\nu)^{\circ}]$ for
$\text{dim}(\mathbb{F}_\nu)= m $. From this, we
splits  (\ref{m4}) into two smaller parts:
\begin{eqnarray}\label{m6}
\bigcap_{\nu;\,\text{dim}(\mathbb{F}_\nu)\le
m-1}\left(\mathbb{S}[\mathbb{F}^*_\nu]\setminus N_{
\epsilon/M^{k(\nu)-1/3}}(\partial
\mathbb{S}[\mathbb{F}^*_\nu])\right)\bigcap
\bigcap_{\nu;\,\text{dim}(\mathbb{F}_\nu)= m} \mathbb{S}[\mathbb{F}^*_\nu]
\subset \bigcap_{\nu=1}^d\mathbb{S}[(\mathbb{F}^*_\nu)^{\circ}]=
\emptyset.
 \end{eqnarray}
Since $\mathbb{S}[\mathbb{F}^*_\nu]$ and
$\mathbb{S}[\mathbb{F}^*_\nu]\setminus N_{
\epsilon/M^{k(\nu)-1/3}}(\partial \mathbb{S}[\mathbb{F}^*_\nu])$ are
 closed sets in $\mathbb{S}^{n-1}$ in (\ref{m6}), we take a little bit thicker
  intersection in $\nu=1,\cdots,d$ with some large $M$
and small $\epsilon$ to obtain  that
\begin{eqnarray*}
\bigcap_{\nu;\,\text{dim}(\mathbb{F}_\nu)\le m-1}
N_{\epsilon/M^{k(\nu)+1/3}}\left(\mathbb{S}[\mathbb{F}^*_\nu]\setminus
N_{\epsilon/M^{k(\nu)-1/3}}(\partial
\mathbb{S}[\mathbb{F}^*_\nu])\right)\bigcap
\bigcap_{\nu;\,\text{dim}(\mathbb{F}_\nu)= m}
\mathbb{S}_{\epsilon/M}[\mathbb{F}^*_\nu]= \emptyset.
 \end{eqnarray*}
By Definition \ref{deef}, we have
\begin{eqnarray*}
\bigcap_\nu  \mathbb{S}_{\epsilon}[\mathbb{F}^*_\nu] = \emptyset.
 \end{eqnarray*}
This proves Lemma \ref{des}. The case $\text{dim}(\mathbb{P})=n$
follows similarly.
\end{proof}
 \begin{lemma}\label{lem17778}
Let   $\mathbb{F}$ be a face of    $\mathbb{P}={\bf N}(\Omega,S)$
with $\text{dim}(\mathbb{P})=m<n$.  Suppose that
$\tilde{\mathfrak{m}}\in \mathbb{F}\cap\Omega$ and $\mathfrak{m}\in
\Omega\setminus \mathbb{F}$. Then
 for all $\mathfrak{p}\in
 \mathbb{S}_{\epsilon}[\mathbb{F}^*]$ with $\text{dim}(\mathbb{F})=m-k$ where $k=1,\cdots,m$,
\begin{eqnarray}
 \mathfrak{p}\cdot(\mathfrak{m}-\tilde{\mathfrak{m}})\ge c
>0 \ \ \text{where} \ \ c \ \text{is independent of $\mathfrak{p}$}.\label{gar}
\end{eqnarray}
\end{lemma}
\begin{remark}\label{rem6.1}
We shall use Lemma \ref{lem17778} for the estimate of the difference   $\mathcal{I}_J(P_{
\Omega},\xi)-\mathcal{I}_J(P_{\mathbb{F}},\xi)$ where
$J/|J|=\mathfrak{p}$. We do not need
Lemma \ref{lem17778} if $\text{dim}(\mathbb{F})=m-k$ with $k=0$, since  $\mathbb{F}={\bf N}(\Omega,S)$ for the case $\text{dim}(\mathbb{F})=m$ so that $\mathcal{I}_J(P_{
\Omega},\xi)-\mathcal{I}_J(P_{\mathbb{F}},\xi)\equiv 0$.
 For the case $\text{dim}(\mathbb{P})=n$, (\ref{gar}) also holds for all $
 \mathbb{S}_{\epsilon}[\mathbb{F}^*]$ with $\text{dim}(\mathbb{F})=n-k$ where
 $k=1,\cdots,n$.
\end{remark}
\begin{proof}[Proof of Lemma \ref{rem6.1}]
 By Proposition \ref{pr50},
 \begin{eqnarray*}
\mathbb{F}^*=\mathbb{F}^*|\mathbb{P} =\text{CoSp}(\{\mathfrak{q}_j\}_{j=1}^\ell
\}\cup\{\pm\mathfrak{n}_i\}_{i=1}^{n-m} )
\end{eqnarray*}
  where  $\{\mathfrak{q}_j\}_{j=1}^\ell
$ and $\{\pm\mathfrak{n}_i\}_{i=1}^{n-m} $ is defined as in
(\ref{4g00}). Here we can take  $\mathfrak{q}_j\in
\mathbb{S}^{n-1}$. Then
$$ \mathbb{S}[\mathbb{F}^*]=\text{CoSp}(\{\mathfrak{q}_j\}_{j=1}^\ell
\}\cup\{\pm\mathfrak{n}_i\}_{i=1}^{n-m}
)\cap\mathbb{S}^{n-1}=\left\{\mathfrak{q}\in\mathbb{S}^{n-1}:\mathfrak{q}=\sum_{j}c_j\mathfrak{q}_j+\mathfrak{r}\
\text{with}\ c_j>0\right\}$$ where $\mathfrak{r}=\sum_{i=1}^{n-m}
c_{i,\pm}\, (\pm \mathfrak{n}_i) \in V(\mathbb{P})^{\perp}$. Thus,
for sufficiently large $M$,
\begin{equation}\label{9011}
\mathbb{S}[\mathbb{F}^*]\setminus N_{\epsilon/M^{k -1/3}}(\partial
\mathbb{S}[\mathbb{F}^*])\subset
\left\{\mathfrak{q}\in\mathbb{S}^{n-1}:\mathfrak{q}=\sum_{j}c_j\mathfrak{q}_j+\mathfrak{r}\
\text{with}\ c_j> \frac{\epsilon}{ M^{k -1/4}}\right\}
\end{equation}
where $\mathfrak{r}\in V^{\perp}(\mathbb{P})$. By (\ref{gel3}),
\begin{eqnarray*}
\mathbb{F}= \bigcap_{j}^{\ell}\mathbb{F}_{j}\  \ \text{with}\ \
\mathbb{F}_{j}=\pi_{\mathfrak{q}_j}\cap\mathbb{P}.
\end{eqnarray*}
From $\tilde{\mathfrak{m}}\in \mathbb{F}$ and $\mathfrak{m}\in
\Omega\setminus \mathbb{F}$,
\begin{eqnarray*}
  \mathfrak{m} \in \mathbb{P}\setminus \mathbb{F}_{k} \ \ \text{ for
some $k\in\{1,\cdots,\ell\}$  and}\ \
 \tilde{\mathfrak{m}}\in\mathbb{F}\subset\mathbb{F}_{k}.
\end{eqnarray*}
Thus, by   Definition \ref{dualface},
\begin{eqnarray} \label{4ggvvb}
\mathfrak{q}_k\cdot (\mathfrak{m}-\tilde{\mathfrak{m}})>\eta_k>0\
\ \text{and}\ \
 \mathfrak{q}_j\cdot (\mathfrak{m}-\tilde{\mathfrak{m}})\ge 0\ \ \text{for
 $j=1,\cdots,\ell$}
 \end{eqnarray}
 where $\eta_k$ depends on $\Omega$. Let
$\mathfrak{q}\in \mathbb{S}[\mathbb{F}^*]\setminus N_{\epsilon/M^{k
-1/3}}(\partial \mathbb{S}[\mathbb{F}^*])$. Then by (\ref{9011}),
$$\text{$\mathfrak{q} =\sum_{j=1}^\ell c_j\mathfrak{q}_j+\mathfrak{r}  $
where $c_j\ge  \epsilon/M^{k -1/4} $ and $\mathfrak{r}\in
V^{\perp}(\mathbb{P})$.}$$ Thus, we use  (\ref{4ggvvb}) and the fact
$\mathfrak{r}\cdot(\mathfrak{m}-\tilde{\mathfrak{m}})=0$ (which
follows from $\mathfrak{r}\in V^{\perp}(\mathbb{P})$) to have
\begin{eqnarray}
\mathfrak{q}\cdot(\mathfrak{m}-\tilde{\mathfrak{m}})&=&\sum_{j=1}^\ell
c_j\mathfrak{q}_j\cdot(\mathfrak{m}-\tilde{\mathfrak{m}})
\nonumber\\
&=&c_k\mathfrak{q}_k\cdot
(\mathfrak{m}-\tilde{\mathfrak{m}})+\sum_{j=1,j\ne
k}^\ell c_j\mathfrak{q}_j\cdot (\mathfrak{m}-\tilde{\mathfrak{m}})\label{2882}\\
&\ge& c_k\mathfrak{q}_k\cdot
(\mathfrak{m}-\tilde{\mathfrak{m}})+0\ge (\epsilon/M^{k
-1/4})\eta_k\ge \epsilon\,\eta/M^{k -1/4}>0\nonumber
\end{eqnarray}
where $\eta=\min\{\eta_k:k=1,\cdots,\ell\}$. Finally, let
$$\mathfrak{p}\in
\mathbb{S}_{\epsilon}[\mathbb{F}^*]=N_{\epsilon/M^{k+1/3}}
\left(\mathbb{S}[\mathbb{F}^*]\setminus
N_{\epsilon/M^{k-1/3}}(\partial \mathbb{S}[\mathbb{F}^*])\right).$$
Then there exists $\mathfrak{q}\in \mathbb{S}[\mathbb{F}^*]\setminus
N_{\epsilon/M^{k-1/3}}(\partial \mathbb{S}[\mathbb{F}^*])$
satisfying (\ref{2882}) and
$|\mathfrak{p}-\mathfrak{q}|<\epsilon/M^{k+1/3}$. For sufficiently
large $M>0$, we have $
\mathfrak{p}\cdot(\mathfrak{m}-\tilde{\mathfrak{m}}) \ge \epsilon\,
\eta/(2M^{k -1/4}) $, which proves (\ref{gar}).
\end{proof}
In view of (\ref{j2}),
\begin{eqnarray}\label{ttjj}
\left\|\sum_{J\in Z(S)}
H^{P_{\Lambda}}_J*f\right\|_{L^p(\mathbb{R}^d)}\le \sum_{\mathbb{F}
\in\mathcal{F}(\vec{{\bf N}}(\Lambda,S))} \left\|\sum_{J/|J|\in
\bigcap_{\nu=1}^d
 \mathbb{S}_{\epsilon}[\mathbb{F}^*_\nu]}
H^{P_{\Lambda}}_J*f\right\|_{L^p(\mathbb{R}^d)}.
\end{eqnarray}
Furthermore,
\begin{lemma}\label{lemb5.2}
Let $\mathbb{P}_\nu={\bf N}(\Lambda_\nu,S)$ be a polyhedron in $\mathbb{R}^n$ with $\text{dim}(\mathbb{P}_\nu)=m_\nu$ and let $\mathbb{F}_\nu\preceq\mathbb{P}_\nu$ for each $\nu=1,\cdots,d$. Suppose that there exists $\nu\in\{1,\cdots,d\}$ such that
$\mathbb{F}_\nu\in\mathcal{F}^{m_\nu-k_\nu}(\mathbb{P}_\nu)$ with $k_\nu\ge 1$, that is, $\mathbb{F}_\nu \precneqq \mathbb{P}_\nu$.
 Then,
\begin{equation}\label{4.3g}
\left\|\left(H^{P_{\Lambda}}_J-H_J^{P_\mathbb{F}}\right)*f\right\|_{L^p(\mathbb{R}^d)
}\le 2^{-c|J|}\left\| f\right\|_{L^p(\mathbb{R}^d) }\ \ \text{for
$J/|J|\in \bigcap_{\nu=1}^d\mathbb{S}_{\epsilon}[\mathbb{F}_\nu^*]$}.
\end{equation}
\end{lemma}
\begin{proof}
Let $$B=\left\{\nu:\text{ $\mathbb{F}_{\nu}\precneqq {\bf
N}(\Lambda_\nu,S)$, that is, $
\mathbb{F}_\nu\in\mathcal{F}^{m_\nu-k_\nu}(\mathbb{P}_\nu)$ where $k_\nu\ge
1$}\right\}.$$  For each $\nu\in B$, choose
$\tilde{\mathfrak{m}}_{\nu}\in \mathbb{F}_{\nu}\cap\Lambda_{\nu}$
and $\mathfrak{m}\in \Lambda_{\nu}\setminus \mathbb{F}_{\nu}$. By
Lemma \ref{lem17778}, observe that for  there exists $\beta>0$ such
that
\begin{eqnarray}\label{kensi}
J/|J|\cdot(\mathfrak{m}-\tilde{\mathfrak{m}}_{\nu})>\beta\
\text{for all}\ \    J/|J|\in
 \mathbb{S}_{\epsilon}[\mathbb{F}_\nu^*]\  \ \text{with $\nu\in B$}
 \end{eqnarray}
  where $c$ is
independent of   $ J/|J|$.   By (\ref{nniiii}), the Fourier
multipliers of $H^{P_\Lambda}_J$ ($=H^{P_{{\bf N}(\Lambda,S)}}_J $)
and $H^{P_\mathbb{F}}_J$ are
\begin{equation}\label{4.3a}
\left|\mathcal{I}_J(P_{\Lambda},\xi) \right|, \
\left|\mathcal{I}_J(P_{\mathbb{F}},\xi) \right|\,\lesssim\,
\min\,\left\{\,\left|\,2^{-J\cdot \tilde{\mathfrak{m}}_\nu}\,
\xi_{\nu} a_{\tilde{\mathfrak{m}}_\nu}^{\nu} \right|^{-\delta}\,:\,
\tilde{\mathfrak{m}}_\nu\in \mathbb{F}_{\nu},\,\
\nu=1,\cdots,d\right\}\,.
\end{equation}
By the mean value theorem,
\begin{eqnarray}\label{4.3b}
\left|\mathcal{I}_J(P_{\Lambda},\xi)-
\mathcal{I}_J(P_{\mathbb{F}},\xi)\right| &\lesssim& \sum_{\nu\in
B}\sum_{\mathfrak{m}\in\Lambda_{\nu}\setminus \mathbb{F}_{\nu}}
\left|2^{-J\cdot \mathfrak{m}}\, \xi_{\nu}
a_{\mathfrak{m}}^{\nu}\,\right|^{\delta}.  \nonumber
\end{eqnarray}
  By (\ref{kensi})-(\ref{4.3b}),
\begin{eqnarray}\label{4.3b}
\sup_{\xi}\left|\mathcal{I}_J(P_{\Lambda},\xi)-
\mathcal{I}_J(P_{\mathbb{F}},\xi)\right|&\lesssim&
\sum_{\mathfrak{m}\in\Lambda_{\nu}\setminus \mathbb{F}_{\nu}}
\left|2^{-J\cdot (\mathfrak{m}-\tilde{\mathfrak{m}}_{\nu})}
 \,\right|^{\delta/2}\, \lesssim\,   2^{-\beta\delta|J|/2}.\nonumber
\end{eqnarray}
This implies that (\ref{4.3g}) holds for $p=2$. Interpolation with
$p=1$ or $p=\infty$ yields the range  $1<p<\infty$.
\end{proof}
We  sum up (\ref{4.3g}) of Lemma (\ref{lemb5.2}) to obtain the
following lemma.
\begin{lemma}\label{lemgool}
\begin{equation}\label{4.3orang}
\left\|\sum_{J\in Z(S)}
H^{P_{\Lambda}}_J*f\right\|_{L^p(\mathbb{R}^d)}\,\lesssim\,\left\|
f\right\|_{L^p(\mathbb{R}^d) }+\sum_{\mathbb{F} \in\mathcal{F}({\bf
N}(\Lambda,S))}\left\|\sum_{J/|J|\in
\bigcap_{\nu=1}^d\mathbb{S}_{\epsilon}[\mathbb{F}_\nu^*]} H_J^{P_\mathbb{F}}
*f\right\|_{L^p(\mathbb{R}^d) }.
\end{equation}
\end{lemma}
  By using Lemma \ref{lemgool}, we are now able to obtain Proposition \ref{l50}: Under the assumption
\begin{eqnarray}\label{09000}
\text{rank}\left(\bigcup_{\nu=1}^d {\bf N}(\Lambda_\nu,S) \right)
\le n-1\end{eqnarray}
and the hypothesis of Main
Theorems \ref{main3},
 we have
 \begin{equation*}
\left\|\sum_{J\in Z(S)}
H^{P_{\Lambda}}_J*f\right\|_{L^p(\mathbb{R}^d)}\,\lesssim\,\left\|
f\right\|_{L^p(\mathbb{R}^d) }.
\end{equation*}
\begin{proof}[Proof of Proposition \ref{l50}]
Since there are finitely many
$\mathbb{F}=(\mathbb{F}_\nu)\in\mathcal{F}(\vec{{\bf
N}}(\Lambda,S))$ in (\ref{4.3orang}), it suffice to work with one fixed $\mathbb{F}$ on the right hand side.  By
(\ref{4.3orang}) and Lemma  \ref{des}, it suffices to show that
\begin{equation*}
\left\|\sum_{J/|J|\in
\bigcap_{\nu=1}^d\mathbb{S}_{\epsilon}[\mathbb{F}_\nu^*]}
H_J^{P_\mathbb{F}}*f\right\|_{L^p(\mathbb{R}^d) }\lesssim \left\|
f\right\|_{L^p(\mathbb{R}^d) }\ \ \text{ only if}\ \
 \bigcap\mathbb{S}[(\mathbb{F}_\nu^*)^{\circ}]\ne \emptyset.
\end{equation*}
Note that
\begin{itemize}
\item[(1)]
$\text{rank}\left(\bigcup_{\nu=1}^d \mathbb{F}_\nu
 \right)\le \text{rank}\left(\bigcup_{\nu=1}^d  {\bf N}(\Lambda_\nu,S)
 \right)\le  n-1$,
 \item[(2)]
  $\bigcap_{\nu=1}^d (\mathbb{F}_\nu^*)^{\circ}\supset
 \bigcap_{\nu=1}^d\mathbb{S}[(\mathbb{F}_\nu^*)^{\circ}]\ne\emptyset $.
 \end{itemize}
  From this and the evenness hypothesis of Main Theorem \ref{main3},
  it follows that $\bigcup_{\nu=1}^d
\mathbb{F}_\nu\cap\Lambda_\nu
 $ is an even set. Thus
 $\mathcal{I}_J(P_{\mathbb{F}},\xi)\equiv 0$ for all $J$.
\end{proof}
\subsection{Sufficiency Theorem}
We shall prove the sufficient part of Main Theorems \ref{main18} and
\ref{main3} by showing Theorem \ref{th60} below. Let
$\Lambda=(\Lambda_\nu)_{\nu=1}^d$ with $\Lambda_\nu\subset
\mathbb{Z}_+^n$ and $S\subset N_n$. To each
$\mathbb{F}\in\mathcal{F}(\vec{ {\bf N}}(\Lambda,S))$ and
$J\in\mathbb{Z}^n$, we recall (\ref{4.0}):
\begin{equation}\label{oocc}
\mathcal{I}_J(P_{\mathbb{F}},\xi)=\int e^{i\left(\xi_1
\sum_{\mathfrak{m}
  \in \mathbb{F}_1\cap\Lambda_1}c_{\mathfrak{m}}^12^{-J\cdot \mathfrak{m}}t^{\mathfrak{m}}+  \cdots+ \xi_d \sum_{\mathfrak{m}
  \in \mathbb{F}_d\cap\Lambda_d}c_{\mathfrak{m}}^d2^{-J\cdot \mathfrak{m}}t^{\mathfrak{m}}\right)}\prod h(t_\nu)dt_1\cdots
  dt_n
\end{equation}
where $\mathcal{I}_J(P_{\vec{ {\bf N}}(\Lambda,S)},\xi)=\mathcal{I}_J(P_{\Lambda},\xi)$. Then $\mathcal{I}_J(P_{\mathbb{F}},\xi)$
  is the Fourier multiplier of the operator
\begin{equation*}
 f\rightarrow H^{\mathbb{F}}_J*f.
\end{equation*}
\begin{theorem}\label{th60}
 Let
$\Lambda=(\Lambda_\nu)_{\nu=1}^d$ with $\Lambda_\nu\subset
\mathbb{Z}_+^n$ and  $S\subset \{1,\cdots,n\} $.  Suppose that for
$\mathbb{G}=(\mathbb{G}_\nu)\in\mathcal{F}(\vec{{\bf
N}}(\Lambda,S))$,
\begin{eqnarray}\label{nnii}
|\mathcal{I}_J(P_{\mathbb{G}},\xi)|\le
C\min\left\{|2^{-J\cdot\mathfrak{m}_\nu}\xi_\nu|^{-\delta}:\mathfrak{m}_\nu\in\mathbb{G}_\nu\cap\Lambda_\nu\
\text{for}\ \nu=1,\cdots,d\right\}.
\end{eqnarray}
Suppose that
 \begin{eqnarray*}
 \bigcup_{\nu=1}^d(\mathbb{F}_\nu\cap\Lambda_\nu)\ \text{
 is an even set for $\mathbb{F}\in \mathcal{F}_{\rm{lo}}(\vec{ {\bf
N}}(\Lambda,S)) $ }
\end{eqnarray*}
where $\mathcal{F}_{\rm{lo}}(\vec{ {\bf N}}(\Lambda,S))$ is defined
in Definition \ref{ird}.
 Then for any $\mathbb{F}\in\mathcal{F}(\vec{{\bf
N}}(\Lambda,S))$,
 \begin{eqnarray}\label{6760hha}
 \sum_{J  \in
 \rm{Cap}(\mathbb{F}^*)}\left|\mathcal{I}_J(P_\Lambda,\xi) \right|\le C_2\  \ \text{where \  $\rm{Cap}(\mathbb{F}^*)=\bigcap_{\nu=1}^d\mathbb{F}_\nu^*$},
 \end{eqnarray}
 and for $1<p<\infty$,
\begin{eqnarray}\label{6760}
 \left\|\sum_{J \in Z\subset
 \rm{Cap}(\mathbb{F}^*)}H^{P_\Lambda}_J*f\,
  \right\|_{L^p(\mathbb{R}^d) }\le   C_p \left\| f
  \right\|_{L^p(\mathbb{R}^d) }.
 \end{eqnarray}
\end{theorem}
  By (\ref{nniiii}), the condition (\ref{nnii}) is satisfied
if $\Lambda_\nu$'s are mutually disjoint. Thus, Theorem \ref{th60} together with Proposition
\ref{74sj}  immediately leads the sufficient
part of Main Theorems \ref{main18} and  \ref{main3}.
\begin{remark}\label{rem91}
In the above, $C_2$ in (\ref{6760hha}) is majorized by
\begin{eqnarray}\label{1300}
C_R\prod_{\nu}\prod_{\mathfrak{m}\in\Lambda_\nu}
(|c^\nu_{\mathfrak{m}} |+1/|c^\nu_{\mathfrak{m}} |)^{1/R}\ \
\text{for some large $R$}.
\end{eqnarray}
\end{remark}
\medskip
\section{Descending  Faces v.s. Ascending Cones}\label{sec9}
  Suppose
that we are given
$\mathbb{F}=(\mathbb{F}_\nu)\in\mathcal{F}(\mathbb{P})$ with
$\mathbb{P}=(\mathbb{P}_\nu)$ where $\mathbb{P}_\nu={\bf N}(\Lambda_\nu,S)$.
To establish (\ref{6760hha}), as we have planed in (\ref{ide22}), (\ref{pathg}) and (\ref{danbi5}), we
shall choose an appropriate descending chain $\{\mathbb{F}(s):s=0,\cdots,N\}$ in $\mathcal{F}(\mathbb{P})$ such that
\begin{eqnarray}\label{path}
\quad\mathbb{P}=\mathbb{F}(0) \succeq
 \cdots\succeq\mathbb{F}(s) \succeq\cdots \succeq\mathbb{F}(N)=\mathbb{F} \ \
 \text{($\mathbb{F}_\nu(s-1)\succeq \mathbb{F}_\nu(s)\ \ \text{for each $\nu$}$)}.
\end{eqnarray}
We shall make the  estimates:
 \begin{eqnarray}\label{ab12}
 \sum_{J \in
 \rm{Cap}(\mathbb{F}^*)}\left|\mathcal{I}_J(P_{\mathbb{F}(s-1)},\xi)- \mathcal{I}_J(P_{\mathbb{F}(s)},\xi)\right|\le
 C \ \text{for}\ s=1,\cdots,N.
 \end{eqnarray}
 To perform this estimates successfully, we need to have the full rank condition for applying Proposition \ref{propyy}:
 \begin{eqnarray}\label{dif1}
\text{rank}\left( \bigcup_{\nu=1}^d\mathbb{F}_\nu(s-1)\right)=n.
 \end{eqnarray}
 Without the full rank condition, we need to have the overlapping condition for applying Proposition \ref{pr5g}:
   \begin{eqnarray}
 \text{Cap}(\mathbb{F}^*(s)^{\circ})=\bigcap_{\nu=1}^d(\mathbb{F}_\nu^*(s))^{\circ}\ne\emptyset.\label{micl}
 \end{eqnarray}
The following technical difficulty arises for each  (\ref{dif1}) and (\ref{micl}). \\
 {\bf Difficulty  satisfying overlapping property (\ref{micl})}.
By Lemma \ref{lem2525},  we see that
$$\text{Cap}(\mathbb{F}^*(s-1))\ne \emptyset
\Rightarrow\text{Cap}(\mathbb{F}^*(s))\ne \emptyset\ \ \text{
whenever}\ \   \mathbb{F}_\nu(s-1)\succeq \mathbb{F}_\nu(s) \
\text{for all $\nu$.}$$ However,
$\text{Cap}(\mathbb{F}^*(s-1)^{\circ})\ne \emptyset
\Rightarrow\text{Cap}(\mathbb{F}^*(s)^{\circ})\ne \emptyset$ is not
always true even if $\mathbb{F}_\nu(s-1)\succeq \mathbb{F}_\nu(s) $
for all $\nu$. To keep (\ref{micl}), we construct (\ref{path}) in
Definition \ref{dfu2} so that $\text{Cap}(\mathbb{F}^*(s)^{\circ})$
with every $s=1,\cdots,N$  contains some common portion of
$\text{Cap}(\mathbb{F}^*)$ in (\ref{cce}). For this, we use the
concept of the essential faces as defined  in Definition
\ref{dfu1}.\\
 {\bf Difficulty satisfying the full rank condition (\ref{dif1})}.
  Even if we have (\ref{dif1}), we might have $$\text{rank}\left( \bigcup_{\nu=1}^d\mathbb{F}_\nu(s-1)\cap\Lambda_\nu\right)\le n-1.$$
For this case, in order to satisfy (\ref{jr1v})  in Proposition
\ref{propyy},
 $\left|\mathcal{I}_J(P_{\mathbb{F}(s-1)},\xi)-
\mathcal{I}_J(P_{\mathbb{F}(s)},\xi)\right|$ in (\ref{ab12}) must be dominated by
  $ |2^{-J\cdot\mathfrak{m}}\xi_\nu|^c $ not only with
  $\mathfrak{m}\in \mathbb{F}_\nu(s-1)\cap\Lambda_\nu$
  exponents of polynomial $P_\Lambda$, but also with
  $\mathfrak{m}\in \mathbb{F}_\nu(s-1)$ not exponents of that polynomial.
  To fulfill this requirement, we shall make an efficient size control tool for
  $$\left\{2^{-J\cdot\mathfrak{m}}:\mathfrak{m}\in\mathbb{F}_\nu(s)\right\}_{s=1}^N\
  \ \text{with $J\in \rm{Cap}(\mathbb{F}^*) $ fixed},$$
  in Proposition \ref{prop66aa}.

 \subsection{Construction of Descending  Faces and Ascending Cones}
 Given a  a face
$\mathbb{F} =(\mathbb{F}_\nu)\in\mathcal{F}(\mathbb{P})$,   an
intersection $\bigcap_{\nu=1}^d\mathbb{F}_\nu^{*} $  of  cones
   is itself   a
cone type polyhedron. Thus there exist
$\mathfrak{p}_1,\cdots,\mathfrak{p}_N$ in
$\bigcap_{\nu=1}^d\mathbb{F}_\nu^{*}$:
\begin{equation}\label{cce}
\rm{Cap}(\mathbb{F}^{*})=\bigcap_{\nu=1}^d\mathbb{F}_\nu^{*}  =
\text{CoSp}(\mathfrak{p}_1,\cdots,\mathfrak{p}_N).
\end{equation}    In
order to show (\ref{6760hha}) and (\ref{6760}), we first split
$\rm{Cap}(\mathbb{F}^*)$ as
$$\rm{Cap}(\mathbb{F}^*)=\bigcup
\rm{Cap}(\mathbb{F}^*)(\sigma)$$ where union is over all permutations
$\sigma:\{1,\cdots,N\}\rightarrow \{1,\cdots,N\}$ and
$$\rm{Cap}(\mathbb{F}^*)(\sigma)=\{ \alpha_1\mathfrak{p}_{1}+\cdots+
\alpha_N\mathfrak{p}_{N}\in\rm{Cap}(\mathbb{F}^*)
:\alpha_{\sigma(1)} \ge \alpha_{\sigma(2)} \ge\cdots\ge
\alpha_{\sigma(N)} \ge 0\}.$$ To prove (\ref{6760hha}),  it suffices
to show for each  $\sigma $,
$$ \sum_{J \in
\rm{Cap}(\mathbb{F}^*)(\sigma)}\left|\mathcal{I}_J(P_\Lambda,\xi) \right|\le C_2.$$
 Since the order of $\mathfrak{p}_1,\cdots,\mathfrak{p}_N$ is random, it
suffices to work with only $\sigma=id$ where
\begin{eqnarray}\label{g58}
\quad \rm{Cap}(\mathbb{F}^*)(id)=\{
\alpha_1\mathfrak{p}_{1}+\cdots+\alpha_N\mathfrak{p}_{N}
\in\rm{Cap}(\mathbb{F}^*):\alpha_1\ge \alpha_2\ge\cdots\ge \alpha_N\ge
0\}.
\end{eqnarray}

\begin{definition}\label{ec}[Intersection of Cones]
Let $\mathbb{P}=(\mathbb{P}_\nu)$ so that
$\mathbb{P}_\nu=\mathbb{P}(\Pi^\nu)$ is a polyhedron in
$\mathbb{R}^n$ and
$\text{dim}(\mathbb{P}_\nu)=\text{dim}(V(\mathbb{P}_\nu))=m_\nu\le n
$. Suppose that $\Pi^\nu=\Pi_a^\nu\cup\Pi_b^\nu $ where
$\Pi_{a}^{\nu}=\{ \mathfrak{q}^\nu_j \}_{j=1}^{L_\nu}$ is a
generator for $\mathbb{P}_\nu$ in $V_{am}(\mathbb{P})$, and
$\Pi_b^\nu =\{ \pm\mathfrak{n}_i^\nu \}_{i=1}^{n-m_\nu}$ is a
generator for $V_{am}(\mathbb{P}_\nu)$ in $\mathbb{R}^n$ as in Lemma
\ref{tn19}.
 By
Propositions \ref{facerep}  and \ref{pr50} with Remark \ref{pitse},
a face $\mathbb{F}_\nu$    having an expression:
\begin{eqnarray*}
 \qquad\mathbb{F}_\nu=\bigcap_{j=1}^{N_\nu}
  \pi_{\mathfrak{p}^{\nu}_j}
 \cap \mathbb{P}_\nu \ \text{where $
 \{ \mathfrak{p}_j^{\nu} \}_{j=1}^{N_\nu}=\{ \mathfrak{q}_j^\nu \}_{j=1}^{\ell_\nu}\cup \{ \pm\mathfrak{n}_i^\nu \}_{i=1}^{n-m_\nu}
  $}
 \end{eqnarray*}
  has its cone of the form:
\begin{eqnarray}\label{ggnn}
\mathbb{F}^*_\nu &=& \rm{CoSp}\left(
\{\mathfrak{p}_j^\nu\}_{j=1}^{N_\nu}\right)\ \ \text{where}\  \
\Pi(\mathbb{F}_\nu)=\{\mathfrak{p}^{\nu}_j\}_{j=1}^{N_\nu}.
\end{eqnarray}
 Here we remind that
\begin{eqnarray}\label{nji}
\text{CoSp}
 \left(\{ \pm\mathfrak{n}_i^\nu
 \}_{i=1}^{n-m_\nu}\right)=V^{\perp}(\mathbb{P}_\nu)\ \ \text{and}\ \ \{ \mathfrak{q}_j^\nu \}_{j=1}^{\ell_\nu} \subset V(\mathbb{P}_\nu).
 \end{eqnarray}
\end{definition}
\begin{lemma}\label{lemk44}
 In proving (\ref{6760hha}), we may
assume  that
\begin{eqnarray*}
\rm{Cap}(\mathbb{F}^{*}) \cap (\mathbb{F}_\nu^*)^{\circ}\ne
\emptyset \ \ \text{for all $\nu$}.
\end{eqnarray*}
\end{lemma}
\begin{proof}
If the cone $\rm{Cap}(\mathbb{F}^{*})$ is given by $\bigcap_{\nu=1}^d\mathbb{F}_\nu^{*}=\{0\}$,
 the proof of (\ref{6760hha}) is
done since there is only one term $J=0$ in the summation.
Thus we assume that the cone $\rm{Cap}(\mathbb{F}^{*})$  is not $\{0\}$, that is,
\begin{eqnarray}
\left(\bigcap_{\nu=1}^d\mathbb{F}_\nu^{*}
\right)\cap\mathbb{S}^{n-1}\ne \emptyset. \label{m89}
\end{eqnarray}
Let $\left(\bigcap_{\nu=1}^d\mathbb{F}_\nu^{*} \right) \cap
(\mathbb{F}_\nu^*)^{\circ}= \emptyset$, say $\nu=1$. Then from  $
\bigcap_{\nu=1}^d\mathbb{F}_\nu^{*} \subset (\mathbb{F}_1^*)$ and
Definitions \ref{dfu5} and \ref{dfu4}, we have $
\bigcap_{\nu=1}^d\mathbb{F}_\nu^{*}
 \subset\partial\mathbb{F}_1^*$. By Lemma \ref{lem411d}
there exists $ \mathbb{F}_{1,1}^* \precneqq \mathbb{F}_1^*$ (so $
\mathbb{F}_1
  \precneqq \mathbb{F}_{1,1} $) such that $ \bigcap_{\nu=1}^d\mathbb{F}_\nu^{*}  \subset
\mathbb{F}_{1,1}^*$. So we  replace $\mathbb{F}_1^*$ in $
\bigcap_{\nu=1}^d\mathbb{F}_\nu^{*} $
  by $\mathbb{F}_{1,1}^*$ with keeping (\ref{m89}). If
$\left(\bigcap_{\nu=1}^d\mathbb{F}_\nu^{*} \right) \cap
(\mathbb{F}_1^*)^{\circ}\ne\emptyset$ where
$\mathbb{F}_1=\mathbb{F}_{1,1}$,  we stop. Otherwise
$\left(\bigcap_{\nu=1}^d\mathbb{F}_\nu^{*} \right) \cap
(\mathbb{F}_1^*)^{\circ}=\emptyset$ where
$\mathbb{F}_1=\mathbb{F}_{1,1}$, we repeat this process until we
have $\left(\bigcap_{\nu=1}^d\mathbb{F}_\nu^{*} \right) \cap
(\mathbb{F}_1^*)^{\circ}\ne\emptyset$ satisfying (\ref{m89}) where
$\mathbb{F}_1$ is taken as new $\mathbb{F}_{1,k}$ such that $$
\mathbb{F}_{1,k}^* \precneqq
\cdots\precneqq\mathbb{F}_{1,1}^*\precneqq\mathbb{F}_1^*\,\ (
\mathbb{F}_1 \precneqq
\mathbb{F}_{1,2}\precneqq\cdots\precneqq\mathbb{F}_{1,k}).$$ Assume
we arrive at the final round with $\mathbb{F}_{1,k}=\mathbb{P}_1$ in
(\ref{m89}). By (\ref{m89}) and Remark \ref{pitse} that tells
$(\mathbb{P}^*_1)^{\circ}= \mathbb{P}^*_1\setminus\{0\}$,
$$\left(\mathbb{P}^*_1\cap\left(\bigcap_{\nu=2}^d\mathbb{F}_\nu^{*}\right)
\right)\cap
(\mathbb{P}^*_1)^{\circ}=\left(\mathbb{P}^*_1\cap\left(\bigcap_{\nu=2}^d\mathbb{F}_\nu^{*}\right)
\right)\cap  \mathbb{P}^*_1 \setminus\{0\} \ne \emptyset.$$ Hence our process  ends
up with  $\left(\bigcap_{\nu=1}^d\mathbb{F}_\nu^{*} \right) \cap
(\mathbb{F}_1^*)^{\circ}\ne\emptyset$ satisfying (\ref{m89}) where
 $\mathbb{F}_1$ is taken possibly as a  face between the original
  $\mathbb{F}_1$ and the entire $\mathbb{P}_1$. By applying the same argument to $\nu=2,\cdots,d$,
we finish the proof.
\end{proof}
\begin{remark}
Lemma \ref{lemk44} combined with Lemma \ref{lem7k} tells us
$\text{Cap}(\mathbb{F}^*)$ is an essential part of
$\mathbb{F}_\nu^*$ in the sense that
$F(\text{Cap}(\mathbb{F}^*)|\mathbb{F}_\nu^*)=\mathbb{F}_\nu^*$.
This
 is used for proving   Lemma \ref{bnnn}.
 \end{remark}

\begin{definition}\label{dfu2}[Essential Cone]
Suppose  that polyhedrons $\mathbb{P}=(\mathbb{P}_\nu)$ and
faces $\mathbb{F}=(\mathbb{F}_\nu)$ are given as in Definition
\ref{ec}. Fix the order of $\{ \mathfrak{p}_j \}_{j=1}^{N}$ in
(\ref{cce}) and note that $\{ \mathfrak{p}_j \}_{j=1}^{N},
\{ \pm\mathfrak{n}_i^\nu \}_{i=1}^{n-m_\nu}\subset \mathbb{F}_\nu^*$.
Then for each
  $\nu=1,\cdots,d$ and $s=0,1,\cdots,N$,
 define a cone:
\begin{eqnarray}\label{45asi}
 \mathcal{C}_\nu(s)=\text{CoSp}
 \left(\{ \pm\mathfrak{n}_i^\nu \}_{i=1}^{n-m_\nu} \cup \{ \mathfrak{p}_j \}_{j=1}^{s}
 \right)\subset \mathbb{F}_\nu^*
 \end{eqnarray}
 where from (\ref{nji}) and (\ref{cce})
 \begin{eqnarray}
 \mathcal{C}_\nu(0)&=&\text{CoSp}
 \left(\{ \pm\mathfrak{n}_i^\nu
 \}_{i=1}^{n-m_\nu}\right)=V^{\perp}(\mathbb{P}_\nu),\label{oja}\\
 \mathcal{C}_\nu(N)&=&\text{CoSp}
 \left(\{ \pm\mathfrak{n}_i^\nu \}_{i=1}^{n-m_\nu} \cup\{  \mathfrak{p}_j \}_{j=1}^{N}
 \right)   \supset\text{CoSp}
 \left(\{  \mathfrak{p}_j \}_{j=1}^{N}
 \right)=\text{Cap}(\mathbb{F}^*) .\label{ee11}
 \end{eqnarray}
In view of Definition \ref{dfu1}, for each $s=0,\cdots,N$ and
$\nu=1,\cdots,d$, define
 the smallest face  $\mathbb{F}_\nu^*(s)$ of $\mathbb{F}^*_\nu$ containing $\mathcal{C}_\nu(s)$
 by \begin{eqnarray}\label{ee111}
 \mathbb{F}_\nu^*(s)=
F\left(\mathcal{C}_\nu(s) |\mathbb{F}^{*}_\nu\right),
\end{eqnarray}
and call $\mathbb{F}_\nu^*(s)$ the essential cone of
  $\mathbb{F}^*_\nu$ containing $\mathcal{C}_\nu(s)$. See the third
  picture of Figure \ref{graph3}.
\end{definition}
\begin{lemma}\label{lemgg}
Suppose that the  $\mathbb{F}_\nu^*(s)$ are defined above.    Then
for $s=1,\cdots,N$,
$$
\bigcap_{\nu=1}^d\left(\mathbb{F}_\nu^*(s)\right)^{\circ} \ne
\emptyset.$$
\end{lemma}
\begin{proof}
Fix $s\in\{1,\cdots,N\}$. By  $\mathcal{C}_\nu(s)=\text{CoSp}
 \left(\{ \pm\mathfrak{n}_i^{\nu} \}_{i=1}^{n-m_{\nu}} \cup \{ \mathfrak{p}_j \}_{j=1}^{s}
 \right)$ in  (\ref{45asi}),
 \begin{eqnarray}\label{hhh5}
 \mathfrak{p}_1+\cdots+\mathfrak{p}_s= \sum_{j=1}^s\mathfrak{p}_j+
  \left(\sum_{i=1}^{n-m_\nu}\mathfrak{n}_i^{\nu}+
  \sum_{i=1}^{n-m_\nu}-\mathfrak{n}_i^{\nu}\right)\in \mathcal{C}_\nu(s)^{\circ}
\end{eqnarray}
for each $\nu=1,\cdots,d$. By Lemma \ref{lemj7},
\begin{eqnarray}\label{hhh56}
 \mathcal{C}_\nu(s)^{\circ}\subset F\left(\mathcal{C}_\nu(s)
|\mathbb{F}^{*}_{\nu}\right)^{\circ}=
 \left(\mathbb{F}_\nu^*(s)\right)^{\circ}.
\end{eqnarray}
  By (\ref{hhh5}) and (\ref{hhh56}),
$
\mathfrak{p}_1+\cdots+\mathfrak{p}_s\in \bigcap_{\nu=1}^d
 \left(\mathbb{F}_\nu^*(s)\right)^{\circ}$ for $s\ge
 1$.
\end{proof}
\begin{remark}
Lemma \ref{lemgg} does not hold for the case $s=0$. However, this
initial case estimate for
$\left(\mathbb{F}_\nu(0)\right)_{\nu=1}^d=\left({\bf
N}(\Lambda_\nu,S)\right)_{\nu=1}^d$ with the low rank condition
(\ref{09000}), is already finished in Proposition \ref{l50}.
\end{remark}

 \begin{lemma}\label{bnnn}
  For each $\nu=1,\cdots,d$,
\begin{eqnarray*}
 \mathbb{F}_\nu^*(N)=\mathbb{F}_\nu^*\ \ \text{and}\ \
 \mathbb{F}_\nu^*(0)=V^{\perp}(\mathbb{P}_\nu).
 \end{eqnarray*}
 \end{lemma}
 \begin{proof}
The first identity follows from Lemmas \ref{lem7k} and  \ref{lemk44}
together with (\ref{ee11}) above.  The second from (\ref{oja})
  with $V^{\perp}(\mathbb{P}_\nu)=\mathbb{P}_\nu^*$  in
Remark \ref{pitse} and $\mathbb{P}_\nu^*\preceq \mathbb{F}_\nu^*$ in
Lemma \ref{lem2525}.
 \end{proof}
Since $\mathbb{F}_\nu^*(s) $  is a face of a cone $\mathbb{F}^*_\nu=
\text{CoSp}(\{ \mathfrak{p}_j^{\nu} \}_{j=1}^{N_\nu} )$, it is also
a cone expressed as:
\begin{eqnarray}\label{blp2}
 \qquad\mathbb{F}_\nu^*(s)=\text{CoSp}(\{ \mathfrak{p}_j^\nu  \}_{j\in B_s^\nu}
) \ \ \text{where}\ \ \{ \mathfrak{p}_j^\nu  \}_{j\in
B_s^\nu}\subset \{ \mathfrak{p}_j^\nu \}_{j=1}^{N_\nu}= \Pi(\mathbb{F}_\nu)\subset\Pi(\mathbb{P}_\nu).
 \end{eqnarray}
Here, by Lemma \ref{bnnn} with (\ref{ggnn}) and (\ref{nji}),
\begin{eqnarray}\label{h7k}
 \mathbb{F}_\nu^*(N)=\text{CoSp}(\{ \mathfrak{p}_j^\nu  \}_{j=1}^{N_\nu}
) \ \ \text{and }\ \
\mathbb{F}_\nu^*(0)=\text{CoSp}(\pm\mathfrak{n}_i^{\nu}\}_{i=1}^{n-m_\nu}).
 \end{eqnarray}
  By   (\ref{blp2})
combined with (3) of Lemma \ref{lemss},   we assign to each
$\mathbb{F}_\nu^*(s) $, a
   face $\mathbb{F}_\nu(s)$ of $\mathbb{P}_\nu $ whose cone   is
   $\mathbb{F}_\nu^*(s)$:
\begin{eqnarray*}
\mathbb{F}_\nu(s)=\bigcap_{j\in B_s^\nu}
\left(\pi_{\mathfrak{p}_j^\nu}\cap \mathbb{P}_\nu\right).
\end{eqnarray*}
By (\ref{h7k}),
 \begin{eqnarray}\label{ennnn24}
\quad\
\mathbb{F}_\nu(0)=\bigcap_{i=1}^{n-m_\nu}\left(\pi_{\pm\mathfrak{n}_i^\nu}\cap
 \mathbb{P}_\nu\right)=\mathbb{P}_\nu\
 \text{and}\ \mathbb{F}_\nu(N)=\bigcap_{j\in B_N^\nu}\left(\pi_{\mathfrak{p}_j^\nu}\cap
 \mathbb{P}_\nu\right)=\mathbb{F}_\nu.
 \end{eqnarray}

\begin{proposition}\label{lemmdd}
For each fixed $\nu=1,\cdots,d$, we have an ascending sequence
$\{\mathbb{F}_\nu^*(s)\}_{s=0}^N$ and a descending sequence
$\{\mathbb{F}_\nu(s)\}_{s=0}^N$.
\begin{eqnarray*}
 V^{\perp}(\mathbb{P}_\nu)&=&\mathbb{F}_\nu^*(0) \preceq \mathbb{F}_\nu^*(1)
 \preceq\cdots\preceq\mathbb{F}_\nu^*(N)=\mathbb{F}^*_\nu,\\
  \mathbb{P}_\nu& =&\mathbb{F}_\nu(0) \succeq \mathbb{F}_\nu(1)
\succeq\cdots\succeq \mathbb{F}_\nu(N)=\mathbb{F}_\nu.
\end{eqnarray*}
\end{proposition}
\begin{proof}
  Since    $\mathcal{C}_\nu(s-1)
\subset \mathcal{C}_\nu(s)$ with $s\ge 1$,
$$\mathbb{F}_\nu^*(s-1)\preceq \mathbb{F}_\nu^*(s). $$ By Lemma
\ref{lem2525}, $ \mathbb{F}_\nu(s)\preceq \mathbb{F}_\nu(s-1)$. The
cases $s=0,N$ are in Lemma \ref{bnnn} and (\ref{ennnn24}).
  \end{proof}

\subsection{Size Control Number}
 Before showing
 $$\sum_{J\in \text{Cap}(\mathbb{F}^*)(id)}|\mathcal{I}_J(P_{\mathbb{F}(s-1)},\xi)-
 \mathcal{I}_J(P_{\mathbb{F}(s)},\xi)|\le
 C$$ in Section 8,  we shall investigate the size of
 $2^{-J\cdot\mathfrak{m}}$ with $\mathfrak{m}\in
 \mathbb{F}_\nu(s-1)\setminus\mathbb{F}_\nu(s)$ and
   $$J\in\text{Cap}(\mathbb{F}^*)(id)=
 \left\{J=\sum_{j=1}^N\alpha_j\mathfrak{p}_j:\alpha_1\ge\cdots\ge\alpha_{s}\ge \cdots \ge\alpha_N\ge0 \right\}.$$
We assert in Proposition
 \ref{prop66aa} that
  $\alpha_s$ ($s=1,\cdots,N$)  is the key number controlling sizes:
  $$ 2^{-C_2\alpha_s}\le \frac{2^{-\mathfrak{m}\cdot J}}{2^{-\tilde{\mathfrak{m}}\cdot J}}
  =\frac{\text{Effect of Mean Value Property}}{\text{Effect of Decay Property}}\le 2^{-C_1\alpha_s} $$ where
  $\mathfrak{m}\in \mathbb{F}_\nu(s-1)\setminus\mathbb{F}_\nu(s)  $ and
  $\tilde{\mathfrak{m}}\in \mathbb{F}_\nu$. Here $C_1,C_2>0$ are independent of $J$.
  \begin{definition}\label{d93}
Given a cone
$\mathbb{F}^*=\text{CoSp}(\mathfrak{p}_1,\cdots,\mathfrak{p}_N)$,
its  $r$-neighborhood  is defined by
$$\mathcal{D}_r\left(\mathbb{F}^*\right)=\left\{\sum_{j=1}^Nc_j\mathfrak{p}_j\in \mathbb{F}^*:
c_j>r>0\right\}.$$
\end{definition}
We shall use the following three lemmas  to prove Proposition
\ref{prop66aa}.
\begin{lemma}\label{lemcc}
Suppose that
$\text{CoSp}(\mathfrak{p}_1,\cdots,\mathfrak{p}_k)^{\circ}
\subset\text{CoSp}(\mathfrak{q}_1,\cdots,\mathfrak{q}_N)^{\circ}$.
 Then there exists $c>0$ depending only on $\mathfrak{p}_i, \mathfrak{q}_j$
 with $1\le i\le k$ and $1\le j\le N$ such that
$$\mathcal{D}_r(\text{CoSp}(\mathfrak{p}_1,\cdots,\mathfrak{p}_k))\subset
\mathcal{D}_{cr}(\text{CoSp}(\mathfrak{q}_1,\cdots,\mathfrak{q}_N)). $$
\end{lemma}
\begin{proof}
From
$\mathfrak{p}_1+\cdots+\mathfrak{p}_k\in\text{CoSp}(\mathfrak{p}_1,\cdots,\mathfrak{p}_k)^{\circ}\subset
\text{CoSp}(\mathfrak{q}_1,\cdots,\mathfrak{q}_N)^{\circ}$, we see
that
\begin{eqnarray}\label{3jjj}
 \mathfrak{p}_1+\cdots+\mathfrak{p}_k=\sum_{j=1}^Nc_j\mathfrak{q}_j\
\ \text{where $c_j>2c$ with $c$ depending on $\mathfrak{p}_i$'s}.
\end{eqnarray}
Let $\mathfrak{p}\in
\mathcal{D}_r(\text{CoSp}(\mathfrak{p}_1,\cdots,\mathfrak{p}_k))$.
By using (\ref{3jjj}), we split $\mathfrak{p}$ into two parts
\begin{eqnarray}\label{g58n}
\mathfrak{p}=\sum_{j=1}^k\alpha_j\mathfrak{p}_j=\sum_{j=1}^k
\left(\alpha_j-\frac{r}{2}\right)\mathfrak{p}_j+\frac{r}{2}\sum_{j=1}^N
c_j\mathfrak{q}_j\ \ \text{where}\ \alpha_j>r.
\end{eqnarray}
Since $\alpha_j-r/2\ge r/2>0$, the first term on the right hand side
of (\ref{g58n})
$$\sum_{j=1}^k\left(\alpha_j-\frac{r}{2}\right)\mathfrak{p}_j\in
\text{CoSp}(\mathfrak{p}_1,\cdots,\mathfrak{p}_k)^{\circ}\subset
\text{CoSp}(\mathfrak{q}_1,\cdots,\mathfrak{q}_N)^{\circ}. $$  Also
 the second term on the right hand side of (\ref{g58n}) is
$$\frac{r}{2}\sum_{j=1}^N c_j\mathfrak{q}_j\in
\mathcal{D}_{cr}(\text{CoSp}(\mathfrak{q}_1,\cdots,\mathfrak{q}_N))$$  because
$(r/2)c_j\ge rc$. So $\mathfrak{p}\in
\mathcal{D}_{cr}(\text{CoSp}(\mathfrak{q}_1,\cdots,\mathfrak{q}_N))$.
\end{proof}

\begin{lemma}\label{lem1777}
Let $\mathbb{P}$ be a polyhedron and let $\mathbb{F}$ be an  proper
face of $ \mathbb{P}$. Suppose that $\tilde{\mathfrak{m}}\in
\mathbb{F}$ and $\mathfrak{m}\in \mathbb{P}\setminus \mathbb{F}$.
Then
 for all $\mathfrak{p}\in
\mathcal{D}_r\left(\mathbb{F}^*\right)$,
\begin{eqnarray*}
 \mathfrak{p}\cdot(\mathfrak{m}-\tilde{\mathfrak{m}})\ge c
>0 \ \ \text{where} \ \ c \ \text{depends on}\  r,\mathfrak{m},\tilde{\mathfrak{m}}.
\end{eqnarray*}
\end{lemma}
\begin{remark}
This  lemma is needed  only when $\mathbb{F} \precneqq\mathbb{P}$
for the same reason in Remark \ref{rem6.1}. The proof of this
lemma is also similar to that of Lemma \ref{lem17778}.
\end{remark}
\begin{proof}
Let $\Pi(\mathbb{F})=\{\mathfrak{p}_j\}_{j=1}^N =\{\mathfrak{q}_j\}_{j=1}^\ell \cup \{\pm\mathfrak{n}_i\}_{i=1}^{n-m}$
  where  $\{\mathfrak{q}_j\}_{j=1}^\ell=\Pi_a(\mathbb{F}) \subset \Pi_a
$ and $\{\pm\mathfrak{n}_i\}_{i=1}^{n-m}=\Pi_b $ so that
 \begin{eqnarray}\label{4aa}
 \mathbb{F}^*| \mathbb{P} =\text{CoSp}(\{\mathfrak{q}_j\}_{j=1}^\ell
\}\cup\{\pm\mathfrak{n}_i\}_{i=1}^{n-m} ).\end{eqnarray}
  By (\ref{gel3}),
\begin{eqnarray*}
\mathbb{F}= \bigcap_{j=1}^{\ell} \mathbb{F}_{j}\  \ \text{with}\ \
\mathbb{F}_{j}=\pi_{\mathfrak{q}_j}\cap\mathbb{P}.
\end{eqnarray*}
Since $\tilde{\mathfrak{m}}\in \mathbb{F}$ and $\mathfrak{m}\in
\mathbb{P}\setminus \mathbb{F}$,
\begin{eqnarray} \label{gbgg}
  \mathfrak{m} \in \mathbb{P}\setminus \mathbb{F}_{k} \ \ \text{ for
some $k\in\{1,\cdots,\ell\}$  and}\ \
 \tilde{\mathfrak{m}}\in\mathbb{F}\subset\mathbb{F}_k.
\end{eqnarray}
Thus by (\ref{gbgg}) and (\ref{brr2}) in Definition \ref{dualface},
\begin{eqnarray} \label{4ggvv}
\mathfrak{q}_k\cdot (\mathfrak{m}-\tilde{\mathfrak{m}})>\eta>0\   \
\text{and}\ \
 \mathfrak{q}_j\cdot (\mathfrak{m}-\tilde{\mathfrak{m}})\ge 0\ \ \text{for
 $j=1,\cdots,\ell$}.\end{eqnarray}
 where $\eta$ depends on $ \mathfrak{m},\tilde{\mathfrak{m}}$. Let
$\mathfrak{p}\in\mathcal{D}_r(\mathbb{F}^*)$. Then $\mathfrak{p}
=\sum_{j=1}^\ell  c_j\mathfrak{q}_j+\mathfrak{r}  $ where $c_j\ge r$
and $\mathfrak{r}\in V^{\perp}(\mathbb{P})$  according to Definition
\ref{d93} and (\ref{4aa}). Thus, by using (\ref{4ggvv}) and
$\mathfrak{r}\cdot(\mathfrak{m}-\tilde{\mathfrak{m}})=0$,
\begin{eqnarray*}
\mathfrak{p}\cdot(\mathfrak{m}-\tilde{\mathfrak{m}})&=&
\sum_{j=1}^\ell c_j\mathfrak{q}_j\cdot(\mathfrak{m}-\tilde{\mathfrak{m}})\\
&=&c_k\mathfrak{q}_k\cdot (\mathfrak{m}-\tilde{{\bf
m}})+\sum_{j=1,j\ne
k}^\ell c_j\mathfrak{q}_j\cdot (\mathfrak{m}-\tilde{\mathfrak{m}})\\
&\ge& c_k\mathfrak{q}_k\cdot
(\mathfrak{m}-\tilde{\mathfrak{m}})+0\ge r \eta >0
\end{eqnarray*}
where $c=r \eta$ is depending on $r,
\mathfrak{m},\tilde{\mathfrak{m}} $ and independent of
$\mathfrak{p}$.
\end{proof}

\begin{lemma}\label{lemn}
Let $\mathbb{F}=(\mathbb{F}_\nu)\in\mathcal{F}(\vec{{\bf
N}}(\Lambda,S))$ and let $\mathbb{F}_\nu^*(s)$ where $s= 1,\cdots,N$
be defined as in Definition \ref{dfu2}. Then for each $s=
1,\cdots,N$, every vector
$\mathfrak{p}=\sum_{j=1}^N\alpha_j\mathfrak{p}_j
\in\rm{Cap}(\mathbb{F}^*)(id)=\left\{\sum_{j=1}^N\alpha_j\mathfrak{p}_j:\alpha_1\ge\cdots\ge\alpha_N\ge0\right\}$
defined in (\ref{g58}) is expressed as
\begin{eqnarray*}
\mathfrak{p}=\mathfrak{r}_1(s)+\mathfrak{r}_2(s)+\mathfrak{r}_3(s)
\end{eqnarray*}
where for each $\nu=1,\cdots,d$,
\begin{itemize}
\item[(1)] $\mathfrak{r}_1(s)\in
 \mathbb{F}_\nu^*(s-1) $,
\item[(2)]  $\mathfrak{r}_2(s)=\alpha_{s}\mathfrak{u}(s)$ where
 $\mathfrak{u}(s)\in \mathcal{D}_r(\mathbb{F}_\nu^*(s)) $ for  $r>0$
  independent of $\alpha_1,\cdots,\alpha_N$,
\item[(3)] $\mathfrak{r}_3(s)=\alpha_{s+1}
\mathfrak{p}_{s+1}+ \cdots+\alpha_N\mathfrak{p}_N\in
\mathbb{F}_\nu^*(N) $.
\end{itemize}
\end{lemma}
\begin{proof}
We express $\mathfrak{p}=\sum_{j=1}^N\alpha_j\mathfrak{p}_j$ as
$\mathfrak{r}_1(s)+\mathfrak{r}_2(s)+\mathfrak{r}_3(s) $ where
\begin{eqnarray}
\qquad \mathfrak{r}_1(s)&=&\left(\alpha_{1}-\alpha_{s}\right)
\mathfrak{p}_{1}+\cdots+ \left(\alpha_{s-1}-\alpha_{s}\right)
\mathfrak{p}_{s-1},\label{1vk}\\
\mathfrak{r}_2(s)&=&\alpha_{s}\left(\mathfrak{p}_1+\cdots+\mathfrak{p}_{s}\right),\label{2vk}
 \\
\mathfrak{r}_3(s)&=&\alpha_{s+1} \mathfrak{p}_{s+1}+
\cdots+\alpha_N\mathfrak{p}_N.\label{3vk}
\end{eqnarray}
  By  (\ref{45asi}) and (\ref{1vk}),  $$\mathfrak{r}_1(s) \in
 \mathbb{F}_\nu^*(s-1) $$ proving (1). We next show (2).  By
 $\mathcal{C}_\nu(s)=\text{CoSp}
 \left(\{ \pm\mathfrak{n}_i^\nu \}_{i=1}^{n-m-\nu} \cup \{ \mathfrak{p}_j \}_{j=1}^{s}
 \right)$ in (\ref{45asi}),
\begin{eqnarray}\label{2g4}
  \quad \mathfrak{p}_1+\cdots+ \mathfrak{p}_{s} &=&\sum_{j=1}^s\mathfrak{p}_j+
  \left(\sum_{i=1}^{n-m_\nu}\mathfrak{n}_i^\nu +
  \sum_{i=1}^{n-m}-\mathfrak{n}_i^\nu\right)\in
\mathcal{D}_t\left(\mathcal{C}_\nu(s) \right)\ \ \text{for $t=1$}.
\end{eqnarray}
By (\ref{blp2}),
\begin{eqnarray}\label{sw49}
\mathbb{F}_\nu^*(s)=F(\mathcal{C}_\nu(s)|\mathbb{F}^*)=
 \text{CoSp}(\{ \mathfrak{p}_j^\nu  \}_{j\in B_s^\nu}).
\end{eqnarray}By Lemmas \ref{lemj7},
\begin{eqnarray}\label{9999}
\quad\mathcal{C}_\nu(s)^{\circ} &\subset&
F(\mathcal{C}_\nu(s)|\mathbb{F}^*_\nu)^{\circ}=\text{CoSp}(\{
\mathfrak{p}_j^\nu \}_{j\in B_s^\nu} )^{\circ}.
\end{eqnarray}
By using (\ref{sw49}),(\ref{9999}) and Lemmas \ref{lemcc},
$$\mathcal{D}_t(\mathcal{C}_\nu(s))
\subset   \mathcal{D}_{ct}(\text{CoSp}(\{ \mathfrak{p}_j^\nu \}_{j\in B_s^\nu}
) ) =\mathcal{D}_{ct}(\mathbb{F}_\nu^*(s))
$$ for some $c>0$ depending on   $\mathfrak{p}_j^\nu$'s.
By this and (\ref{2g4}), put
\begin{eqnarray*}
 \mathfrak{u}(s)=\mathfrak{p}_1+\cdots+\mathfrak{p}_s\in
\mathcal{D}_{ct}(\mathbb{F}_\nu^*(s)).
\end{eqnarray*}
Set $r=ct >0$. Note (2) follows from $\mathfrak{r}_2(s)= \alpha_{s}  \mathfrak{u}(s)$ in
(\ref{2vk}). Finally (3) follows  from (\ref{3vk}),(\ref{45asi}) and
(\ref{ee111}).
\end{proof}
Using Lemmas \ref{lem1777} and \ref{lemn}, we obtain
\begin{proposition}\label{prop66aa}
Let $\mathbb{F}_\nu(s)$ and $\mathbb{F}_\nu^*(s)$ where
$\nu=1,\cdots,d$ and  $s=1,\cdots,N$ be defined as in Definition
\ref{dfu2}.
 Suppose that
  $$\mathfrak{p}=\sum_{j=1}^N\alpha_j\mathfrak{p}_j
   \in \rm{Cap}(\mathbb{F}^*)(id)=
 \left\{\sum_{j=1}^N\alpha_j\mathfrak{p}_j:\alpha_1\ge\cdots\ge\alpha_N\ge0\right\},$$
and $ \tilde{\mathfrak{m}} \in
 \mathbb{F}_\nu(N)=\mathbb{F}_\nu $. Then  for   $s=1,\cdots,N$, there exist
 $C_1,C_2>0$ such that
 \begin{eqnarray}
 \mathfrak{p}\cdot   (\mathfrak{m}-\tilde{\mathfrak{m}}) &\ge& C_1\alpha_{s}\
 \ \text{for}\ \ \mathfrak{m}\in
\mathbb{F}_\nu(s-1) \setminus \mathbb{F}_\nu(s),\label{995aa}\\
\mathfrak{p}\cdot   (\mathfrak{n}-\tilde{\mathfrak{m}})&\le&C_2
\alpha_{s}\ \ \text{for}\ \ \mathfrak{n}\in
 \mathbb{F}_\nu(s-1)\label{996bb}
 \end{eqnarray}
 where   $C_1,C_2>0$ are independent of $\mathfrak{p}\in \rm{Cap}(\mathbb{F}^*)(id)$,
 but may depend on $\mathfrak{m},\tilde{\mathfrak{m}}$ and $\mathfrak{n}$.
 \end{proposition}
\begin{proof}
By  Lemma \ref{lemn}, we have $\mathfrak{p}= \mathfrak{r}_1(s)+
\mathfrak{r}_2(s)+ \mathfrak{r}_3(s)$ satisfying (1),(2) and (3).
Since $\mathfrak{m}\in \mathbb{F}_\nu(s-1)\setminus \mathbb{F}_\nu(s)$
and $ \tilde{\mathfrak{m}} \in
 \mathbb{F}_\nu(N)\subset \mathbb{F}_\nu(s)$, the property (2) of Lemma \ref{lemn}
 combined with Lemma \ref{lem1777}  yields that
 $
 \mathfrak{u}(s)\cdot   (\mathfrak{m}-\tilde{\mathfrak{m}}) >c >0,
 $
that is
 \begin{eqnarray}\label{j0j1a}
 \mathfrak{r}_2(s)\cdot   (\mathfrak{m}-\tilde{\mathfrak{m}})\ge c\alpha_{s}\ \
 \text{where $c$ is independent of  $\mathfrak{p}$}.
 \end{eqnarray}
Since   $\mathfrak{r}_1(s)+\mathfrak{r}_3(s)\in \mathbb{F}_\nu^*(N)$
where $\mathbb{F}_\nu^*(1)\preceq \mathbb{F}_\nu^*(N)$ and $
\tilde{\mathfrak{m}} \in
 \mathbb{F}_\nu(N)$,
  \begin{eqnarray}\label{jjka}
 (\mathfrak{r}_1(s)+\mathfrak{r}_3(s))\cdot   (\mathfrak{m}-\tilde{\mathfrak{m}}) \ge 0.
 \end{eqnarray}
Thus (\ref{995aa}) follows from  (\ref{j0j1a}) and (\ref{jjka}).
Finally,
 the property (1) of Lemma \ref{lemn}  together with the fact  $\mathfrak{n} \in  \mathbb{F}_\nu(s-1)$ and
$\tilde{\mathfrak{m}}\in \mathbb{F}_\nu(N)\subset
\mathbb{F}_\nu(s-1)$  yields
 $$\mathfrak{r}_1(s)\cdot   (\mathfrak{n}-\tilde{\mathfrak{m}})=0.$$
Using this and $\alpha_{s}\ge \alpha_{s+1}\ge \cdots\ge \alpha_N$ in
(\ref{2vk}) and  (\ref{3vk}),
  \begin{eqnarray*}
 \mathfrak{p}\cdot   (\mathfrak{n}-\tilde{\mathfrak{m}})=
 (\mathfrak{r}_2(s)+\mathfrak{r}_3(s))\cdot   (\mathfrak{n}-\tilde{\mathfrak{m}})
 \lesssim \alpha_{s}
 \end{eqnarray*}
which proves (\ref{996bb}).
\end{proof}

\section{Proof of Sufficiency}\label{suffes}
In this section, we shall finish the proof of Theorem \ref{th60}.
Remind (\ref{oocc}) and write the Fourier multiplier for the
operator $
 f\rightarrow H^{\mathbb{F}(s)}_J*f
$ with $\mathbb{F}_\nu(s)\preceq \mathbb{F}_\nu(0)=\mathbb{P}_\nu={\bf N}(\Lambda_\nu,S)$ as
\begin{eqnarray*}
\mathcal{I}_J(P_{\mathbb{F}(s)},\xi) =\int e^{i\sum_{\nu=1}^d\left(
\sum_{\mathfrak{m}
  \in \mathbb{F}(s,\nu)\cap\Lambda_\nu}c_{\mathfrak{m}}^{\nu}
  2^{-J\cdot \mathfrak{m}}t^{\mathfrak{m}} \right)\xi_\nu} \prod h(t_\nu)dt.
\end{eqnarray*}
\begin{lemma}\label{ddcca}
 For   $s=0,1,\cdots,N$
and $J\in \bigcap_{\nu=1}^d\mathbb{F}_\nu^*$,
\begin{eqnarray}
&& \left|  \mathcal{I}_J(P_{\mathbb{F}(s)},\xi) \right| \le
CK\min\left\{ |2^{-J\cdot
   \mathfrak{m}_\nu }\xi_\nu|^{-\delta}: \mathfrak{m}_\nu\in \mathbb{P}_\nu={\bf N}(\Lambda_\nu,S)  \ \text{where}\
   \nu=1,\cdots,d\right\}\label{janna}\\
   &&\qquad\qquad\qquad\text{where}\ \ K=\prod_{\nu}\prod_{\mathfrak{m}\in\Lambda_\nu}
(|c^\nu_{\mathfrak{m}} |+1/|c^\nu_{\mathfrak{m}}
|)^{1/\delta}.\nonumber
 \end{eqnarray}
\end{lemma}
\begin{proof}
By (\ref{nnii}) in Theorem \ref{th60} and $\mathbb{F}_\nu(s)
\succeq\mathbb{F}_\nu(N)=\mathbb{F}_\nu$,
 for all $s=0,1,\cdots,N$,
 \begin{eqnarray*}\label{eland}
 \left|\mathcal{I}_J(P_{\mathbb{F}(s)},\xi)\right| &\le  C& \min\left\{ |c_{\tilde{\mathfrak{m}}}^{\nu}2^{-J\cdot
  \tilde{\mathfrak{m}}_\nu }\xi_\nu|^{-\delta}:\tilde{\mathfrak{m}}_\nu
  \in \mathbb{F}_\nu\cap\Lambda_\nu \right\}\\
  &\le  C& \min\left\{ |c_{\tilde{\mathfrak{m}}}^{\nu}2^{-J\cdot
  \tilde{\mathfrak{m}}_\nu }\xi_\nu|^{-\delta}:\tilde{\mathfrak{m}}_\nu
  \in  \mathbb{F}_\nu \right\}
 \end{eqnarray*}
 where the second inequality follows from (\ref{ds41}) in Lemma \ref{dg23}.
This combined with the fact that for $J\in  \mathbb{F}_\nu^*$ for each $\nu=1,\cdots,d$,
\begin{eqnarray}\label{e550}
 2^{-J\cdot \mathfrak{m}_{\nu}}\le 2^{-J\cdot \tilde{\mathfrak{m}}_{\nu}}\ \text{where
 $\mathfrak{m}_{\nu}\in
\mathbb{F}_\nu(0)=\mathbb{P}_\nu$ and $\tilde{\mathfrak{m}}_\nu\in
\mathbb{F}_\nu(N)=\mathbb{F}_\nu$,}
\end{eqnarray}
 yields (\ref{janna}).
\end{proof}
Choose $d$ vectors
$$ \tilde{\mathfrak{m}}_\nu\in
\mathbb{F}_\nu(N)\cap\Lambda_\nu=\mathbb{F}_\nu\cap\Lambda_\nu\ \
\text{for each}\ \ \nu=1,\cdots,d .$$
 According to (\ref{2.2}),
 define for each $\{\alpha,\beta\}\subset\{1,\cdots,d\}$ with $\alpha>
 \beta$,
 $$\widehat{A_J^{(\alpha,\beta)}}(\xi)=\psi\left(\frac{2^{-J\cdot\tilde{\mathfrak{m}}_{\alpha}}\xi_{\alpha}}
 { 2^{-J\cdot\tilde{\mathfrak{m}}_{\beta}}\xi_{\beta} }\right)\ \ \text{and}\ \
 \widehat{A_J^{(\beta,\alpha)}}(\xi)
 =1-\widehat{A_J^{(\alpha,\beta)}}(\xi).$$
There are $M=\binom d2$ collections of $(\alpha,\beta)$ with $\alpha>\beta$ in $ \{1,\cdots,d\}$.   Then
 $$1=\prod_{(\alpha,\beta)\subset \{1,\cdots,d\},\, \alpha<\beta}
 \left(\widehat{A_J^{(\alpha,\beta)}}(\xi)+\widehat{A_J^{(\beta,\alpha)}}(\xi)\right)
 =\sum_{\gamma}\widehat{A_J^{\gamma}}(\xi)  $$
 where  $ \widehat{A_J^{\gamma}}(\xi)=
\prod_{k=1}^{M}\widehat{A_J^{(\alpha_k,\beta_k)}}(\xi)$ with
$\gamma=((\alpha_k,\beta_k))_{k=1}^M $ and
  the summation above is  over all possible $2^M$ choices of $\gamma$
  having $\alpha_k<\beta_k$ or $\alpha_k>\beta_k$ for each $k\in \{1,\cdots,M\}$.
In order to show  (\ref{6760}), we prove that for each
$\gamma=((\alpha_j,\beta_j))_{j=1}^M$,
\begin{eqnarray}\label{p6761}
 \left\|\sum_{J \in  \rm{Cap}(\mathbb{F}^*)(id)}H^{ {\bf N}(\Lambda,S)}_J*A_J^{\gamma}*f\right\|_{L^p(\mathbb{R}^d)}\lesssim
  \left\| f\right\|_{L^p(\mathbb{R}^d)}.
 \end{eqnarray}
 Note that there exists an $n$-permutation $\sigma$ such that
 $$\text{supp}\left( \widehat{A_J^{\gamma}} \right)
 \subset\left\{\xi\in\mathbb{R}^d:
 \left|2^{-J\cdot\tilde{\mathfrak{m}}_{\sigma(1)}}\xi_{\sigma(1)}\right|
 \lesssim\cdots \lesssim \left|2^{-J\cdot\tilde{\mathfrak{m}}_{\sigma(d)}}\xi_{\sigma(d)}\right|\right\}.$$
Without loss of generality,
\begin{eqnarray}\label{nmm}
\text{supp}\left( \widehat{A_J^{\gamma}} \right)
 \subset\left\{\xi\in\mathbb{R}^d:|2^{-J\cdot\tilde{\mathfrak{m}}_{1}}\xi_{1}|\lesssim\cdots\lesssim
|2^{-J\cdot\tilde{\mathfrak{m}}_{d}}\xi_{d}|\right\}.
\end{eqnarray}
In proving (\ref{p6761}) it suffices to show that for
$A_J=A^{\gamma}_J$ satisfying (\ref{nmm}),
\begin{eqnarray}\label{p676}
\left\|\sum_{J \in \rm{Cap}(\mathbb{F}^*)(id)}H^{\mathbb{F}(0)}_J*
A_J*f\right\|_{L^p(\mathbb{R}^d)}\lesssim  \left\|
f\right\|_{L^p(\mathbb{R}^d)}
 \end{eqnarray}
 where $\mathbb{F}(0)=(\mathbb{F}_\nu(0))_{\nu=1}^d=({\bf N}(\Lambda_\nu,S))={\bf
 N}(\Lambda,S)$.
\begin{remark}
From now on, we write $A\lesssim  B$ when $A\le C B$ where  $C$ is a constant
multiple of the constant $K$ in Lemma \ref{ddcca}.
\end{remark}
By Proposition \ref{l50} and ${\bf N}(\Lambda_\nu,S)=\mathbb{F}_\nu(0)$,
it suffices to assume that
\begin{eqnarray}
\text{rank}\left(\bigcup_{\nu=1}^d \mathbb{F}_\nu(0)
 \right)=n.\label{wlp}
\end{eqnarray}
 To show (\ref{p676}), by Proposition    \ref{lemmdd},
it suffices to prove (\ref{1394}) and (\ref{139b}) below:
\begin{eqnarray}
&& \text{If $\text{rank}\left(\bigcup_{\nu=1}^d
 \mathbb{F}_\nu(s-1) \right) =n$, then}\label{1394}\\
&&\qquad\quad \left\|\sum_{J \in
\rm{Cap}(\mathbb{F}^*)(id)}\left(H^{\mathbb{F}(s-1)}_J-H^{\mathbb{F}(s)}_J\right)
* A_J*f\,\right\|_{L^p(\mathbb{R}^d)} \lesssim \,
 \|f\|_{L^p(\mathbb{R}^d)}. \nonumber
 \end{eqnarray}
 Also,
\begin{eqnarray}
 \left\|\sum_{J \in
\rm{Cap}(\mathbb{F}^*)(id)}H^{\mathbb{F}(N)}_J*
A_J*f\right\|_{L^p(\mathbb{R}^d)\rightarrow
L^p(\mathbb{R}^d)}\lesssim \,  \|f\|_{L^p(\mathbb{R}^d)}
.\label{139b}
 \end{eqnarray}
We claim that (\ref{1394}) and (\ref{139b}) imply (\ref{p676}).
Assume that (\ref{1394}) and (\ref{139b}) are true.
 Let  $\text{rank}\left(\bigcup_{\nu=1}^d
 \mathbb{F}_\nu(s-1)  \right)=n$ for all $s=1,\cdots,N$.
 Then (\ref{1394}) and (\ref{139b}) yield that
 $$ \left\|\sum_{J \in
\rm{Cap}(\mathbb{F}^*)(id)}H^{\mathbb{F}(0)}_J*
A_J*f\right\|_{L^p(\mathbb{R}^d)} \lesssim
\|f\|_{L^p(\mathbb{R}^d)}.$$ Let $\text{rank}\left(\bigcup_{\nu=1}^d
 \mathbb{F}_\nu(s-1) \right)=n$ for $s=1,\cdots,r,$ and
$\text{rank}\left(\bigcup_{\nu=1}^d
 \mathbb{F}_\nu(r) \right)\le n-1$. Then (\ref{1394})  yields
that
\begin{equation} \label{hdssn}
 \left\|\sum_{J \in
\rm{Cap}(\mathbb{F}^*)(id)}H^{\mathbb{F}(0)}_J*
A_J*f\,\right\|_{L^p(\mathbb{R}^d)}  \lesssim\,
\left\|f\right\|_{L^p(\mathbb{R}^d)}+ \left\|\sum_{J \in
\rm{Cap}(\mathbb{F}^*)(id)}H^{\mathbb{F}(r)}_J*
A_J*f\,\right\|_{L^p(\mathbb{R}^d)} .
\end{equation}
By (\ref{wlp}),   $r\ge 1$. Thus, by Lemma \ref{lemgg}, we have an
overlapping condition
  $$\bigcup_{\nu=1}^d\left(\mathbb{F}_\nu^*(r)\right)^{\circ} \ne\emptyset.$$
From the hypothesis of Theorem \ref{th60}    and the rank condition
  $$\text{rank}\left(\bigcup_{\nu=1}^d
 \mathbb{F}_\nu(r) \right)\le n-1,$$ it follows that
$\bigcup_{\nu=1}^d\left(\mathbb{F}_\nu(r)\cap \Lambda_\nu\right)$ is
an even set. Thus the convolution kernel $H^{\mathbb{F}(r)}_J$
vanishes in the right hand side of (\ref{hdssn}) as its Fourier
multiplier
  $\mathcal{I}_J(P_{\mathbb{F}(r)},\xi)\equiv 0$.
\begin{proof}[Proof of (\ref{1394})]
Let $s\in\{1,\cdots,N\}$ fixed. Choose $\mu\in\{1,\cdots,d\}$ such
that
\begin{eqnarray}
\text{rank}\left(\bigcup_{\nu=\mu}^d  \mathbb{F}_\nu(s-1)
\right)&=&n, \label{vbn}
\\  \text{rank}\left(\bigcup_{\nu=\mu+1}^d
 \mathbb{F}_\nu(s-1) \right)&\le&n-1\label{vbn2}
\end{eqnarray}
where
$\bigcup_{\nu=\mu+1}^{d}\mathbb{F}\left(s-1,\nu\right)\cap\Lambda_\nu=\emptyset$
for the case  $\mu=d$. For each $s$, set
\begin{eqnarray*}
\mathbb{F}'(s-1)&=&(\emptyset,\cdots,\emptyset,\mathbb{F}_{\mu+1}(s-1),\cdots,\mathbb{F}_d(s-1))\\
\mathbb{F}'(s)&=&(\emptyset,\cdots,\emptyset,\mathbb{F}_{\mu+1}(s),\cdots,\mathbb{F}_d(s)).
\end{eqnarray*}
In order to show (\ref{1394}), we shall prove
\begin{eqnarray}
  \left\|\sum_{J \in
\rm{Cap}(\mathbb{F}^*)(id)}
\left(H^{\mathbb{F}'(s-1)}_J-H^{\mathbb{F}'(s)}_J\right)
* A_J*f\,\right\|_{L^p(\mathbb{R}^d)} \lesssim \,
 \|f\|_{L^p(\mathbb{R}^d)}, \label{sq33}
 \end{eqnarray}
 and
\begin{equation}
 \left\|\sum_{J \in \rm{Cap}(\mathbb{F}^*)(id)}
\left(H^{\mathbb{F}(s-1)}_J-H^{\mathbb{F}'(s-1)}_J -
 H^{\mathbb{F}(s)}_J+H^{\mathbb{F}'(s)}_J\right)
* A_J*f\,\right\|_{L^p(\mathbb{R}^d)}  \lesssim \,
 \|f\|_{L^p(\mathbb{R}^d)}.\label{eue}
 \end{equation}
 \begin{proof}[Proof of (\ref{sq33})]
Note that from  (\ref{vbn2}) and
$\mathbb{F}_{\nu}\left(s\right)\preceq\mathbb{F}_{\nu}\left(s-1\right)$,
$$\text{rank}\left(
\bigcup_{\nu=\mu+1}^{d}  \mathbb{F}_{\nu}(s-1)\right)  \le n-1,\ \
\text{and}\ \ \text{rank}\left( \bigcup_{\nu=\mu+1}^{d}
 \mathbb{F}_{\nu}(s) \right)\le n-1.
$$
 By Lemma \ref{lemgg}, for $s=2,\cdots,N$,
$$\bigcap_{\nu=\mu+1}^{d}(\mathbb{F}_{\nu}^*(s-1))^{\circ}\ne
\emptyset\ \ \text{and}\ \
\bigcap_{\nu=\mu+1}^{d}(\mathbb{F}_\nu^*(s))^{\circ}\ne \emptyset.$$
Thus
$$\bigcup_{\nu=\mu+1}^{d}\left(\mathbb{F}_{\nu}(s-1)
\cap\Lambda_\nu\right)\ \ \text{and}\ \
\bigcup_{\nu=\mu+1}^{d}\left(\mathbb{F}_{\nu}(s)
\cap\Lambda_\nu\right) \ \text{are even sets.}$$  Thus
$$ \mathcal{I}_J(\mathbb{F}'(s-1),\xi)
  = \mathcal{I}_J(\mathbb{F}'(s ),\xi)
  \equiv 0.$$
  We next consider the case for $s=1$, that is,
  \begin{eqnarray}
  \left\|\sum_{J \in
\rm{Cap}(\mathbb{F}^*)(id)}
\left(H^{\mathbb{F}'(0)}_J-H^{\mathbb{F}'(1)}_J\right)
* A_J*f\,\right\|_{L^p(\mathbb{R}^d)}
 \end{eqnarray}
 where $\mathcal{I}_J(\mathbb{F}'(1),\xi)
  \equiv 0$  by the previous argument.  By applying the Proposition \ref{l50} with
$$\text{rank}\left(
\bigcup_{\nu=\mu+1}^{d}  \mathbb{F}_{\nu}(0) \right)\le n-1 \ \
\text{and}\ \ \mathbb{F}_{\nu}(0)={\bf N}(\Lambda_\nu,S),
 $$
we obtain
 \begin{eqnarray*}
  \left\|\sum_{J \in
\rm{Cap}(\mathbb{F}^*)(id)} \left(H^{\mathbb{F}'(0)}_J \right)
* A_J*f\,\right\|_{L^p(\mathbb{R}^d)}\lesssim \left\|
f\,\right\|_{L^p(\mathbb{R}^d)}.
 \end{eqnarray*}
 Therefore we proved (\ref{sq33}).
 \end{proof}

\begin{proof}[Proof of (\ref{eue})]
 We denote by $\mathcal{I}_J(\mathbb{F}(s-1),\mathbb{F}(s),\xi)$
 the Fourier multiplier of $$\left(H^{\mathbb{F}(s-1)}_J-H^{\mathbb{F}'(s-1)}_J -
 H^{\mathbb{F}(s)}_J+H^{\mathbb{F}'(s)}_J\right)*A_J.$$
Then  $\mathcal{I}_J(\mathbb{F}(s-1),\mathbb{F}(s),\xi)$ is
\begin{equation}
 \biggl(\left(
  \mathcal{I}_J(\mathbb{F}(s-1),\xi)- \mathcal{I}_J(\mathbb{F}'(s-1),\xi)\right)-\left(\mathcal{I}_J(\mathbb{F}(s),\xi)
   -\mathcal{I}_J(\mathbb{F}'(s),\xi)\right)\biggl)\widehat{A_J}(\xi).\label{jkjk1}
\end{equation}
We shall show that for  all
 $J\in\rm{Cap}(\mathbb{F}^*)(id)$,
\begin{eqnarray}
  \left|\mathcal{I}_J(\mathbb{F}(s-1),\mathbb{F}(s),\xi)\right|
 \, \lesssim  \,  \min\left\{|2^{- J\cdot
 \mathfrak{n}_{\nu}}\xi_\nu|^{\pm\epsilon} :\mathfrak{n}_{\nu}\in
  \mathbb{F}_{\nu}(s-1)\right\}_{\nu=\mu}^d.\label{dadan1}
\end{eqnarray}
This combined with the rank condition (\ref{vbn})   and Proposition
\ref{propyy} implies (\ref{eue}). To show (\ref{dadan1}), by Lemma
\ref{ddcca}, it suffices to show that
 \begin{eqnarray}
  \left|\mathcal{I}_J(\mathbb{F}(s-1),\mathbb{F}(s),\xi)\right|
\, \lesssim \,\min\left\{|2^{- J\cdot
 \mathfrak{n}_{\nu}}\xi_\nu|^{\epsilon} :\mathfrak{n}_{\nu}\in
  \mathbb{F}_{\nu}(s-1) \ \text{for}\
  \right\}_{\nu=\mu}^d.\label{2300}
\end{eqnarray}
Thus, the proof of (\ref{eue}) is finished if  (\ref{2300}) is
proved.
 \end{proof}
\begin{proof}[Proof of (\ref{2300})]
We write
$ \mathcal{I}_J(\mathbb{F}(s-1),\mathbb{F}(s),\xi) $ as
\begin{equation}
 \biggl(\left(
  \mathcal{I}_J(\mathbb{F}(s-1),\xi)- \mathcal{I}_J(\mathbb{F}(s),\xi)\right)-\left(\mathcal{I}_J(\mathbb{F}'(s-1),\xi)
   -\mathcal{I}_J(\mathbb{F}'(s),\xi)\right)\biggl)\widehat{A_J}(\xi).\label{jkjk2}
\end{equation}
By using mean value theorem in (\ref{jkjk2}),
\begin{eqnarray}
\qquad\
\left|\mathcal{I}_J(\mathbb{F}(s-1),\mathbb{F}(s),\xi)\right|
\lesssim\, \sum_{\nu=1}^d\sum_{\mathfrak{m}_\nu\in
(\mathbb{F}_{\nu}(s-1)\cap\Lambda_\nu)\setminus
(\mathbb{F}_{\nu}(s)\cap\Lambda_\nu)} |c_{\mathfrak{m}}^\nu2^{- J\cdot
 \mathfrak{m}_\nu}\xi_\nu|\label{23100}.
\end{eqnarray}
By (\ref{995aa}) of Proposition \ref{prop66aa},  for any
$\mathfrak{m}_\nu\in \mathbb{F}_{\nu}(s-1)\setminus \mathbb{F}_{\nu}(s)$
and   $\tilde{\mathfrak{m}}_\nu\in \mathbb{F}_{\nu}(N)$, there exists
a constant $b>0$ independent of $J$ and $\alpha_s$  such that
$$J\cdot   (\mathfrak{m}_{\nu}-\tilde{\mathfrak{m}}_{\nu})
\ge b\, \alpha_{s}\ \ \text{where}\ \
 J=\sum_{j=1}^N\alpha_j\mathfrak{p}_j
   \in \rm{Cap}(\mathbb{F}^*)(id).   $$
 So in (\ref{23100}),
\begin{eqnarray*}
|2^{-J\cdot\mathfrak{m}_\nu}\xi_\nu|\lesssim 2^{-b\alpha_{s}
}|2^{-J\cdot\tilde{\mathfrak{m}}_\nu}\xi_\nu|.
\end{eqnarray*}
This together with Lemma \ref{ddcca}  yields that in (\ref{23100}),
\begin{eqnarray}\label{22mm}
  \left|\mathcal{I}_J(\mathbb{F}(s-1),\mathbb{F}(s),\xi)\right|
   \,\lesssim\, \sum_{\nu=1}^d2^{-b\alpha_{s}
}|2^{-J\cdot\tilde{\mathfrak{m}}_\nu}\xi_\nu|\,\lesssim\,
2^{-c_1\alpha_{s}}.
\end{eqnarray}
By using the mean value theorem in (\ref{jkjk1}) together with
(\ref{e550}) and the support condition of $\widehat{A_J}$ in
(\ref{nmm}),
\begin{eqnarray}
 \left|\mathcal{I}_J(\mathbb{F}(s-1),\mathbb{F}(s),\xi)\right|&\lesssim&  \sum_{\nu=1}^{\mu}
  \sum_{\mathfrak{m}_\nu\in \left(\mathbb{F}_{\nu}(s-1)\cup\,\mathbb{F}_{\nu}(s)\right) \cap\Lambda_\nu}| \xi_{\nu}
  2^{-J\cdot \mathfrak{m}_\nu}|\nonumber\\&\lesssim&   |\xi_{\mu}
  2^{-J\cdot\tilde{\mathfrak{m}}_\mu}| \ \ \text{for any $\tilde{\mathfrak{m}}_\mu\in \mathbb{F}_{\nu}(N)$. }\label{pp44}
 \end{eqnarray}
By (\ref{22mm}),(\ref{pp44}) and (\ref{nmm}),
\begin{eqnarray}
 \left|\mathcal{I}_J(\mathbb{F}(s-1),\mathbb{F}(s),\xi)\right|&\lesssim&
  \min\{|\xi_{\mu}2^{-J\cdot\tilde{\mathfrak{m}}_\mu}|,2^{-c_1\alpha_s}: \tilde{\mathfrak{m}}_\mu\in
  \mathbb{F}_{\mu}(N)\} \nonumber\\
 & \lesssim&
  \min\{|\xi_{\nu}2^{-J\cdot\tilde{\mathfrak{m}}_\nu}|,2^{-c_1\alpha_s}: \tilde{\mathfrak{m}}_\nu\in
  \mathbb{F}_{\nu}(N)\}_{\nu=\mu}^{d}\label{pp2}.
 \end{eqnarray}
  By
(\ref{996bb}) of Proposition \ref{prop66aa},
 $$J\cdot   (\mathfrak{n}_\nu-\tilde{\mathfrak{m}}_\nu)  \lesssim  \alpha_{s} \ \ \text{where $\mathfrak{n}_\nu\in \mathbb{F}_{\nu}(s-1) $
 and $ \tilde{\mathfrak{m}}_\nu \in
 \mathbb{F}_{\nu}(N)$}.$$
  Hence for any $\mathfrak{n}_\nu\in
\mathbb{F}_{\nu}(s-1) $
 and $ \tilde{\mathfrak{m}}_\nu \in
 \mathbb{F}_{\nu}(N)$ with $\nu=\mu,\mu+1,\cdots,d$ in (\ref{pp2}),
\begin{eqnarray}
|2^{-J\cdot\tilde{\mathfrak{m}}_\nu}\xi_\nu|\lesssim
2^{c_2\alpha_{s} } |2^{-J\cdot\mathfrak{n}_\nu}\xi_\nu|.\label{eo09}
 \end{eqnarray}
Then by (\ref{pp2}) and (\ref{eo09}),
\begin{eqnarray*}
 \left|\mathcal{I}_J(\mathbb{F}(s-1),\mathbb{F}(s),\xi)\right|
  & \lesssim & \min\left\{2^{c_2\alpha_{s}
}|\xi_{\nu}2^{-J\cdot
\mathfrak{n}_{\nu}}|,2^{-c_1\alpha_s}:\mathfrak{n}_{\nu}\in
\mathbb{F}_\nu(s-1)\right\}_{\nu=\mu}^{d}
   \\
   &  \lesssim &\min\left\{ |\xi_{\nu}2^{-J\cdot
\mathfrak{n}_{\nu}}|^{\epsilon}:\mathfrak{n}_{\nu}\in
\mathbb{F}_\nu(s-1) \right\}_{\nu=\mu}^{d} .\nonumber
\end{eqnarray*}
 This yields  (\ref{2300}).
\end{proof}
Therefore the proof of (\ref{eue}) is finished.
\end{proof}

\begin{proof}[Proof of (\ref{139b})]
Assume that  $\text{rank}\left(\bigcup_{\nu=1}^d
 \mathbb{F}_{\nu}(N) \right)\le n-1$.  By this and  Lemma \ref{lemgg},  $\bigcup_{\nu=1}^d
\mathbb{F}_{\nu}(N)\cap\Lambda_\nu$ is an even set so that $
\mathcal{I}_J(\mathbb{F}(N),\xi)  \equiv 0$. Thus we suppose that
$$\text{rank}\left(\bigcup_{\nu=1}^d  \mathbb{F}_{\nu}(N) \right)=
n.$$
 As in (\ref{vbn}) and (\ref{vbn2}), we choose
$\mu\in\{1,\cdots,d\}$ such that
\begin{eqnarray} \text{rank}\left(
\bigcup_{\nu=\mu}^{d} \mathbb{F}_{\nu}(N)\right)  \,=\,n  \ \
\text{and}\ \
  \text{rank}\left(
\bigcup_{\nu=\mu+1}^{d} \mathbb{F}_{\nu}(N)\right)
 \,\le\,n-1.\label{lemgg2}
\end{eqnarray}
Set
\begin{eqnarray}
\mathbb{F}'(N)=(\emptyset,\cdots,\emptyset,\mathbb{F}_{\mu+1}(N),\cdots,\mathbb{F}_d(N))\
\text{for $\mu\le d-1$},
\end{eqnarray}
and $\mathbb{F}'(N)=(\emptyset,\cdots,\emptyset)$ for $\mu= d$. By
Lemma \ref{lemgg}
 $$\bigcap_{\nu=\mu+1}^{d}(\mathbb{F}_{\nu}^*(N))^{\circ}\ne
\emptyset.$$  Thus by this and (\ref{lemgg2}),
 $\bigcup_{\nu=\mu+1}^{d}\left(\mathbb{F}_{\nu}(N)
\cap\Lambda_\nu\right)$ is an even set.   So, $
\mathcal{I}_J(\mathbb{F}'(N),\xi)  \equiv 0$. Thus it suffices to
show that
\begin{eqnarray}
 \left\|\sum_{J \in
\rm{Cap}(\mathbb{F}^*)(id)}\left(H^{\mathbb{F}(N)}_J-H^{\mathbb{F}'(N)}_J\right)*
A_J*f\right\|_{L^p(\mathbb{R}^d) } \lesssim\left\|
f\right\|_{L^p(\mathbb{R}^d)} .\label{b45a}
 \end{eqnarray}
Let $J\in \rm{Cap}(\mathbb{F}^*)(id)\subset
\bigcap\mathbb{F}_\nu^*\subset\mathbb{F}_\nu^*=\mathbb{F}_{\nu}^*(N).$
  Then
  for every $\mathfrak{n}_\nu\in \mathbb{F}_{\nu}(N)=\mathbb{F}_\nu$\,,
 \begin{eqnarray*}
\quad J\cdot (\mathfrak{n}_\nu-\tilde{\mathfrak{m}}_\nu)=0  \ \
\text{ where   $ \tilde{\mathfrak{m}}_\nu \in
 \mathbb{F}_{\nu}(N)$.}
 \end{eqnarray*}
By this and the support condition (\ref{nmm}) for
$\widehat{A_J}(\xi)$ such that $
 \left|2^{-J\cdot\tilde{\mathfrak{m}}_{1}}\xi_{1}\right|\lesssim\cdots\lesssim
|2^{-J\cdot\tilde{\mathfrak{m}}_{d}}\xi_{d}|,
$
\begin{eqnarray}
 \left|\left(\mathcal{I}_J(\mathbb{F}(N),\xi)
 -\mathcal{I}_J(\mathbb{F}'(N),\xi)\right)\widehat{A_J}(\xi)\right|
 & \lesssim& \sum_{\nu=1}^\mu \sum_{\mathfrak{m}_\nu\in \mathbb{F}_{\nu}(N)\cap\Lambda_\nu}
 \left|2^{-J\cdot \mathfrak{m}_{\nu}}\xi_\nu\right|
 \approx \left|2^{-J\cdot\tilde{\mathfrak{m}}_{\mu}}\xi_\mu\right|\nonumber\\
 & \lesssim &\min\left\{ \left|2^{-J\cdot\tilde{\mathfrak{m}}_{\nu}}\xi_\nu\right|:\nu=\mu,\cdots,d \right\}
 \nonumber\\
 &  \lesssim&  \min\left\{ \left|2^{-J\cdot \mathfrak{n}_{\nu}}\xi_\nu\right|:
 \mathfrak{n}_\nu \in  \mathbb{F}_{\nu}(N)
 \right\}_{\nu=\mu}^d.\label{40809}
 \end{eqnarray}
Thus Lemma \ref{ddcca} combined with
$\text{rank}\left(\bigcup_{\nu=\mu}^d
 \mathbb{F}_{\nu}(N)\right) = n$ together  and Proposition
\ref{propyy} yields (\ref{b45a}). This completes the proof of
(\ref{139b}). Therefore we finish the proof of (\ref{6760}).
Similarly, we   also obtain (\ref{6760hha}) as in (\ref{jr14}).
\end{proof}

\section{Necessity Theorem}\label{sec10}
To prove the necessity part of Main Theorem, we need more properties
of cones.
\subsection{Transitivity Rule for Cones}
\begin{proposition}\label{lemggu}
Let $\mathbb{P}\subset\mathbb{R}^n$ be a polyhedron and
$\mathbb{F},\mathbb{G}\in\mathcal{F}(\mathbb{P})$ such that
$\mathbb{G}\preceq \mathbb{F}$. Suppose that $\mathfrak{q}\in
(\mathbb{F}^*)^{\circ}| \mathbb{P}$. Suppose that $\mathfrak{p}\in
(\mathbb{G}^*)^{\circ}| \mathbb{F}$. Then there exists
$\epsilon_0>0$ such that
$$0<\epsilon<\epsilon_0\ \ \text{implies that}\ \ \mathfrak{q}+\epsilon\mathfrak{p}\in
(\mathbb{G}^*)^{\circ}| \mathbb{P}.$$
\end{proposition}

 \begin{figure}
 \centerline{\includegraphics[width=10cm,height=7cm]{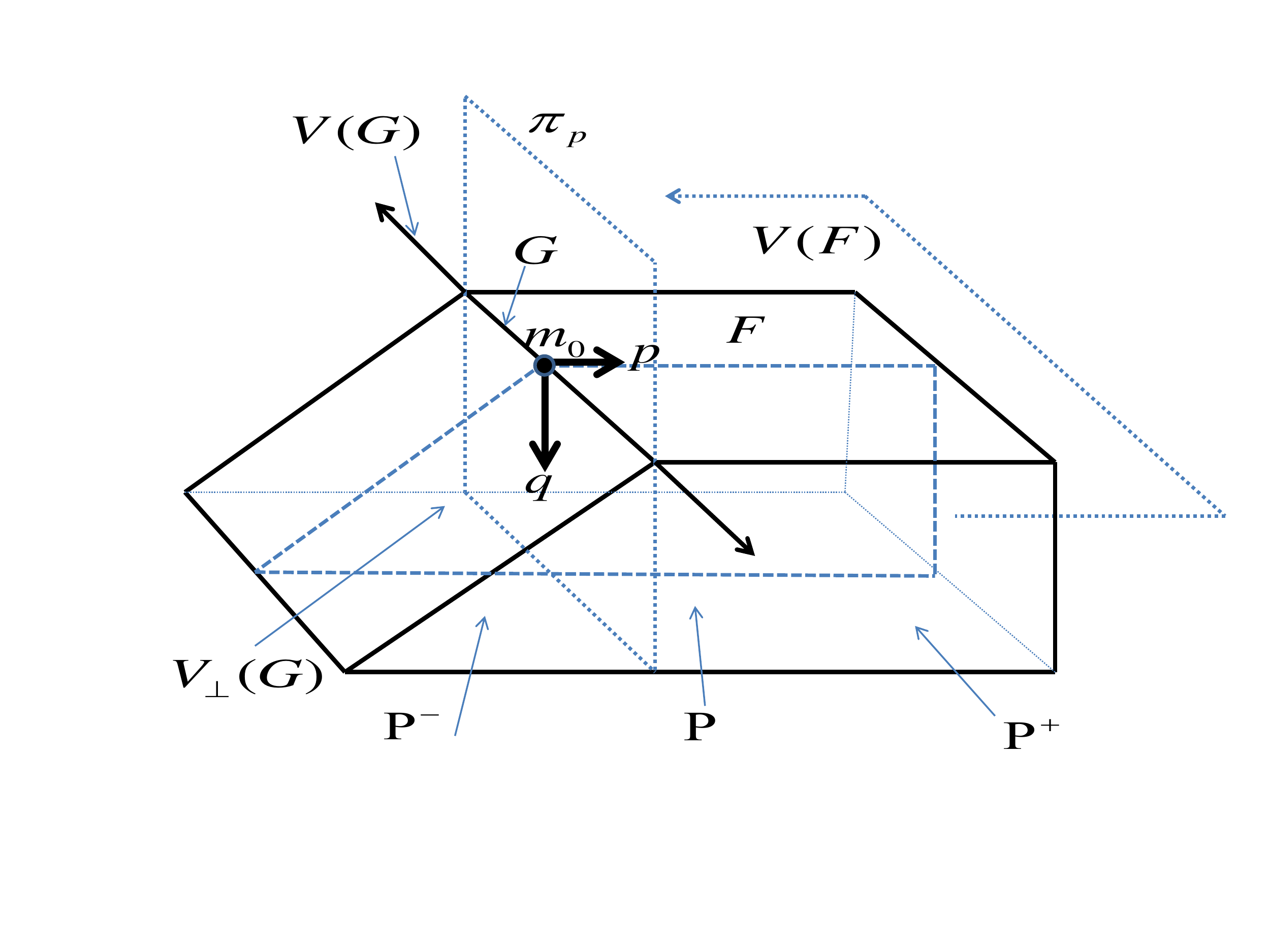}}
 \caption{Transitivity.} \label{graph4}
 \end{figure}

See Figure \ref{graph4} that visualizes  Proposition \ref{lemggu}
and Lemma \ref{lem4114}.
\begin{definition}
Let $V$ be a subspace of $\mathbb{R}^n$. Denote a projection from $\mathbb{R}^n$ to $V$ by $P_V$.
\end{definition}
\begin{lemma}\label{lem4114}
Let $\mathbb{P}$ be a polyhedron in $\mathbb{R}^n$ and
$\mathbb{G}\preceq \mathbb{P}$. Given  $\mathfrak{q}\in
(\mathbb{G}^*)^{\circ}| \mathbb{P}$,  there exists $r>0$ depending
only $\mathfrak{q}$ such that for any $\mathfrak{n}\in
\mathbb{P}\setminus \mathbb{G}$ and $\mathfrak{m}\in \mathbb{G}$,
\begin{eqnarray}
\mathfrak{q} \cdot \frac{{\rm
P}_{V^{\perp}(\mathbb{G})}(\mathfrak{n}-\mathfrak{m})}{|{\rm
P}_{V^{\perp}(\mathbb{G})}(\mathfrak{n}-\mathfrak{m}) |}\ge r>0.
\end{eqnarray}
\end{lemma}
\begin{proof}
We start with the case that  $\mathbb{G}=\{\mathfrak{m}_0\} $ is a
vertex of $\mathbb{P}$. Observe that given a polyhedron
$\mathbb{Q}$, there exists a small positive number  $\epsilon$ such
that $$ \left\{\frac{\mathfrak{n}}{|\mathfrak{n}|}:\mathfrak{n}\in
\mathbb{Q}\setminus\{0\}\ \text{where $0$ is a vertex of
$\mathbb{Q}$}\right\}=\left\{\frac{\mathfrak{n}}{|\mathfrak{n}|}:\mathfrak{n}\in
\mathbb{Q}, |\mathfrak{n}|\ge\epsilon \right\},$$ which   is a
closed set in the sphere $\mathbb{S}^{n-1}$. By this,  we set a closed set in  $\mathbb{S}^{n-1}$:
\begin{eqnarray}\label{s0}
 S(\mathbb{P}-\mathfrak{m}_0)=\left\{\frac{ \mathfrak{n}-\mathfrak{m}_0}{| \mathfrak{n}-\mathfrak{m}_0
 |}:\mathfrak{n}\in \mathbb{P}\setminus \{\mathfrak{m}_0\}\right\}.
\end{eqnarray}
  By $\mathfrak{q}\in (\mathbb{G}^*)^{\circ}| \mathbb{P}$, we have for
all $\mathfrak{n}\in\mathbb{P}\setminus \{\mathfrak{m}_0\},$
\begin{eqnarray*}
\mathfrak{q}\cdot
\frac{\mathfrak{n}-\mathfrak{m}_0}{|\mathfrak{n}-\mathfrak{m}_0|}>0,
\ \text{which with  (\ref{s0}) implies that} \ \
\mathfrak{q}\cdot\mathfrak{s}>0\ \text{for}\ \mathfrak{s}\in
S(\mathbb{P}-\mathfrak{m}_0).
\end{eqnarray*}
A map $\mathfrak{s}\rightarrow \mathfrak{q}\cdot\mathfrak{s}$ is
continuous on the compact set $S(\mathbb{P}-\mathfrak{m}_0)$. So it
has a minimum $r>0$:
\begin{eqnarray}\label{43bona}
\mathfrak{q}\cdot  \mathfrak{s} \ge r \ \text{for all}\ \
\mathfrak{s}\in S(\mathbb{P}-\mathfrak{m}_0).
\end{eqnarray}
From $V(\mathbb{G})=\{0\}$ and $V^{\perp}(\mathbb{G})=\mathbb{R}^n$,
\begin{eqnarray*}
{\rm
P}_{V^{\perp}(\mathbb{G})}(\mathfrak{n}-\mathfrak{m}_0)=\mathfrak{n}-\mathfrak{m}_0.
\end{eqnarray*}
From this together with (\ref{s0}),
\begin{eqnarray*}
 S(\mathbb{P}-\mathfrak{m}_0)=\left\{\frac{ {\rm
P}_{V^{\perp}(\mathbb{G})}(\mathfrak{n}-\mathfrak{m}_0)}{| {\rm
P}_{V^{\perp}(\mathbb{G})}(\mathfrak{n}-\mathfrak{m}_0)
 |}:\mathfrak{n}\in \mathbb{P}\setminus \{\mathfrak{m}_0\}\right\}.
\end{eqnarray*}
By this and (\ref{43bona}),
\begin{eqnarray*}
\mathfrak{q}\cdot  \frac{ {\rm
P}_{V^{\perp}(\mathbb{G})}(\mathfrak{n}-\mathfrak{m}_0)}{| {\rm
P}_{V^{\perp}(\mathbb{G})}(\mathfrak{n}-\mathfrak{m}_0)
 |} \ge r \ \text{for all}\ \
\mathfrak{n}\in \mathbb{P}\setminus \{\mathfrak{m}_0\}.
\end{eqnarray*}
We next consider the general case that $\mathbb{G}$ is a
$k$-dimensional face of $\mathbb{P}$. We shall use the following two
properties:
\begin{eqnarray}\label{635}
\left\{ {\bf x}:\rm{P}_{V^{\perp}(\mathbb{G})}({\bf
x})=0\right\}=V(\mathbb{G})
\end{eqnarray}
and
\begin{eqnarray}\label{636}
\ \text{If $\mathfrak{n}\in \mathbb{P}\setminus\mathbb{G}$ and
$\mathfrak{m}\in\mathbb{G}$,\ then} \ \
\mathfrak{n}-\mathfrak{m}\notin V(\mathbb{G}).
\end{eqnarray}
Choose any $\mathfrak{m}_0\in\mathbb{G}\preceq\mathbb{P}$. Since an image of
polyhedron under any linear transform is also a polyhedron, we see
that $\rm{P}_{V^{\perp}(\mathbb{G})}(\mathbb{P}-\mathfrak{m}_0)$ is
a polyhedron in $V^\perp(\mathbb{G})$. Moreover,
\begin{eqnarray}\label{d22p}
\text{$0$ is a vertex of
$\rm{P}_{V^{\perp}(\mathbb{G})}(\mathbb{P}-\mathfrak{m}_0)$.}
\end{eqnarray}
\begin{proof}[Proof of (\ref{d22p})]
By Definition \ref{dualface},  we see that for $\mathfrak{q}\in
(\mathbb{G}^*)^{\circ}| \mathbb{P}$ and for
$\mathfrak{m}\in\mathbb{G}$ and $\mathfrak{n}\in\mathbb{P}\setminus
\mathbb{G}$,
\begin{eqnarray}\label{sauo}
0<\mathfrak{q}\cdot(\mathfrak{n}-\mathfrak{m}).
\end{eqnarray}
 Let $\rm{P}_{V^{\perp}(\mathbb{G})}(\mathfrak{n}-\mathfrak{m}_0)\in
 \rm{P}_{V^{\perp}(\mathbb{G})}(\mathbb{P}-\mathfrak{m}_0)\setminus\{0\}$,
that is,
$\rm{P}_{V^{\perp}(\mathbb{G})}(\mathfrak{n}-\mathfrak{m}_0)\ne 0$.
Then we have $\mathfrak{n}-\mathfrak{m}_0 \notin V(\mathbb{G})$ by
(\ref{635}), that is, $\mathfrak{n}\notin\mathbb{G}$.
Thus $\mathfrak{n}\in\mathbb{P}\setminus \mathbb{G}$. By
(\ref{sauo}),
$$\mathfrak{q} \cdot 0=0<\mathfrak{q}\cdot(\mathfrak{n}-\mathfrak{m})=\mathfrak{q}\cdot(\mathfrak{n}-\mathfrak{m}_0)=
\mathfrak{q}\cdot
\rm{P}_{V^{\perp}(\mathbb{G})}(\mathfrak{n}-\mathfrak{m}_0)$$ where
$\mathfrak{q}\perp  V(\mathbb{G})$   in the last equality. Thus
   the condition (\ref{4g1}) of Definition \ref{dfac} holds.
\end{proof}
In view of (\ref{s0}) and (\ref{d22p}), we set a compact set
\begin{eqnarray}\label{s088}
K=
S(\rm{P}_{V^{\perp}(\mathbb{G})}(\mathbb{P}-\mathfrak{m}_0)-0)=\left\{\frac{
\mathfrak{n}-0}{|
 \mathfrak{n}-0
 |}:\mathfrak{n}\in \rm{P}_{V^{\perp}(\mathbb{G})}(\mathbb{P}-\mathfrak{m}_0)\setminus
 \{0\}\right\}.
\end{eqnarray}
In the above,
\begin{eqnarray}\label{ashiw}
\rm{P}_{V^{\perp}(\mathbb{G})}(\mathbb{P}-\mathfrak{m}_0)\setminus
\{0\}=\rm{P}_{V^{\perp}(\mathbb{G})}((\mathbb{P}\setminus\mathbb{G})-\mathbb{G}),
\end{eqnarray}
\begin{proof}[Proof of (\ref{ashiw})]
Let ${\bf z}\in
\rm{P}_{V^{\perp}(\mathbb{G})}(\mathbb{P}-\mathfrak{m}_0)\setminus
\{0\}$. Then ${\bf
z}=\rm{P}_{V^{\perp}(\mathbb{G})}(\mathfrak{n}-\mathfrak{m}_0)\ne 0$
with $\mathfrak{n}\in\mathbb{P}$. From (\ref{635})
$\mathfrak{n}-\mathfrak{m}_0\notin V(\mathbb{G})$, which also implies that
 $\mathfrak{n}\notin \mathbb{G}$. Thus ${\bf
z}\in\rm{P}_{V^{\perp}(\mathbb{G})}((\mathbb{P}\setminus\mathbb{G})-\mathbb{G})$.\\
Let ${\bf
z}\in\rm{P}_{V^{\perp}(\mathbb{G})}((\mathbb{P}\setminus\mathbb{G})-\mathbb{G})$.
Then ${\bf
z}=\rm{P}_{V^{\perp}(\mathbb{G})}(\mathfrak{n}-\mathfrak{m})$ with
$\mathfrak{n}\in \mathbb{P}\setminus\mathbb{G}$ and
$\mathfrak{m}\in\mathbb{G}$. Thus
\begin{eqnarray*}
{\bf
z}&=&\rm{P}_{V^{\perp}(\mathbb{G})}(\mathfrak{n}-\mathfrak{m}_0+\mathfrak{m}_0-\mathfrak{m})\\
&=&\rm{P}_{V^{\perp}(\mathbb{G})}(\mathfrak{n}-\mathfrak{m}_0)+
\rm{P}_{V^{\perp}(\mathbb{G})}(\mathfrak{m}_0-\mathfrak{m})
\\
&=&\rm{P}_{V^{\perp}(\mathbb{G})}(\mathfrak{n}-\mathfrak{m}_0)\ne 0
\end{eqnarray*}
where the last inequality follows from (\ref{635}) and (\ref{636}).
Hence ${\bf z}\in
\rm{P}_{V^{\perp}(\mathbb{G})}(\mathbb{P}-\mathfrak{m}_0)\setminus
\{0\}$.
\end{proof}
By (\ref{ashiw}), we rewrite the compact set in (\ref{s088}) as
\begin{eqnarray}\label{kiji}
 K
 =\left\{\frac{\rm{P}_{V^{\perp}(\mathbb{G})}(\mathfrak{n}-\mathfrak{m})  }
 {|\rm{P}_{V^{\perp}(\mathbb{G})}(\mathfrak{n}-\mathfrak{m})
 |}:\mathfrak{n}\in  \mathbb{P}\setminus\mathbb{G}, \mathfrak{m}\in\mathbb{G}\right\}.
\end{eqnarray}
By $\mathfrak{q}\in (\mathbb{G}^*)^{\circ}|\mathbb{P}$, for all
$\mathfrak{n}\in\mathbb{P}\setminus \mathbb{G}$ and
$\mathfrak{m}\in\mathbb{G}$,
\begin{eqnarray*}
\mathfrak{q}\cdot
\rm{P}_{V^{\perp}(\mathbb{G})}(\mathfrak{n}-\mathfrak{m})=\mathfrak{q}\cdot
 (\mathfrak{n}-\mathfrak{m})>0
\end{eqnarray*}
because of $\mathfrak{q}\perp V(\mathbb{G})$ and  Definition
\ref{dualface}.
  Therefore
$$ \mathfrak{q}\cdot\mathfrak{s}>0\ \text{for $\mathfrak{s}\in K$.}$$
From this combined with the compactness of $K$, there exists $r>0$
such that
$$ \mathfrak{q}\cdot\mathfrak{s}\ge r \ \text{for all $\mathfrak{s}\in K$.}$$
By (\ref{kiji}),
 for any $\mathfrak{n}\in
\mathbb{P}\setminus \mathbb{G}$ and $\mathfrak{m}\in \mathbb{G}$,
\begin{eqnarray*}
\mathfrak{q} \cdot \frac{{\rm
P}_{V^{\perp}(\mathbb{G})}(\mathfrak{n}-\mathfrak{m})}{|{\rm
P}_{V^{\perp}(\mathbb{G})}(\mathfrak{n}-\mathfrak{m}) |}\ge r>0.
\end{eqnarray*}
This completes the proof of Lemma \ref{lem4114}.
\end{proof}

\begin{proof}[Proof of Proposition \ref{lemggu}]
 It suffices to show that $0<\epsilon<\epsilon_0$ implies that
for all $\mathfrak{n}\in\mathbb{P}\setminus\mathbb{G}$ and
$\mathfrak{m}\in\mathbb{G}$,
\begin{eqnarray}\label{cc222}
(\mathfrak{q}+\epsilon\mathfrak{p})\cdot(\mathfrak{n}-\mathfrak{m})>
0.
 \end{eqnarray}
We first observe from  $ \mathfrak{q}\in (\mathbb{F}^*)^{\circ}|
\mathbb{P}$,
\begin{eqnarray}\label{nb2}
\mathfrak{q}\cdot(\mathfrak{n}-\mathfrak{m})\ge 0\ \ \text{for all
$\mathfrak{n}\in\mathbb{P}\setminus\mathbb{G}$ and
$\mathfrak{m}\in\mathbb{G}$}.
\end{eqnarray}
Moreover, combined with $\mathfrak{p}\in (\mathbb{G}^*)^{\circ}| \mathbb{F}$,
\begin{eqnarray}\label{nb3}
\mathfrak{q}\cdot(\mathfrak{n}-\mathfrak{m})> 0\ \ \text{for all
$\mathfrak{n}\in\mathbb{P}\setminus\mathbb{F}$ and
$\mathfrak{m}\in\mathbb{G}$},\label{nb3}\\
\mathfrak{p}\cdot(\mathfrak{n}-\mathfrak{m})>0 \ \ \text{for all
$\mathfrak{n}\in\mathbb{F}\setminus\mathbb{G}$ and
$\mathfrak{m}\in\mathbb{G}$}\label{nooni}
\end{eqnarray}
Let $\pi_{\mathfrak{p}}$ be a supporting plane of
$\mathbb{G}\preceq\mathbb{F}$,
\begin{eqnarray*}
\mathbb{G}\subset \pi_{\mathfrak{p}}\ \ \text{and}\ \
\mathbb{F}\setminus\mathbb{G}\subset (\pi_{\mathfrak{p}}^+)^{\circ}.
\end{eqnarray*}
Split
$$\mathbb{P}=\left(\mathbb{P}\cap \pi_{\mathfrak{p}}^-\right)\cup
\left(\mathbb{P}\cap \pi_{\mathfrak{p}}^+\right)=\mathbb{P}^{-}\cup
\mathbb{P}^{+}\ \text{that are visualized in Figure  \ref{graph4}}.$$
 {\bf Case 1}. Suppose $\mathfrak{n}\in\mathbb{P}^+\setminus\mathbb{G}$.
  Then in view that $\mathfrak{n}\in\mathbb{P}^+\subset
\pi_{\mathfrak{p}}^+$,
\begin{eqnarray}\label{dvo}
\mathfrak{p}\cdot(\mathfrak{n}-\mathfrak{m})\ge 0\ \ \text{for all
$\mathfrak{n}\in\mathbb{P}\setminus\mathbb{G}$ and
$\mathfrak{m}\in\mathbb{G}$}.
\end{eqnarray}
By (\ref{nb2}) and (\ref{dvo}), we have $\ge $ in  (\ref{cc222}). Thus
either (\ref{nb3}) or (\ref{nooni}) yields $>$ in (\ref{cc222}).
 \\
 {\bf Case 2}. Suppose that  $\mathfrak{n}\in
\mathbb{P}^-\setminus\mathbb{G}$. Note that $\mathbb{G}\preceq
\mathbb{P}^-$. Consider a  hyperplane
$\pi_{\mathfrak{q}}(\mathbb{F})$ containing $\mathbb{F}$ and whose
normal vector is $\mathfrak{q}$. Then
$\pi_{\mathfrak{q}}(\mathbb{F})$ is a supporting plane of
$\mathbb{G}\preceq \mathbb{P}^-$. Thus
\begin{eqnarray*}
\mathfrak{q}\in (\mathbb{G}^*)^{\circ}|(\mathbb{P}^-,\mathbb{R}^n).
\end{eqnarray*}
Hence by Lemma \ref{lem4114}, for all $\mathfrak{n}\in
\mathbb{P}^-\setminus\mathbb{G}$ and $\mathfrak{m}\in\mathbb{G}$,
\begin{eqnarray} \label{joau}
\mathfrak{q}\cdot\frac{{\rm
P}_{V^{\perp}(\mathbb{G})}(\mathfrak{n}-\mathfrak{m})}{|{\rm
P}_{V^{\perp}(\mathbb{G})}(\mathfrak{n}-\mathfrak{m})| }\ge r>0\ \
\text{and}\ \ {\rm
P}_{V^{\perp}(\mathbb{G})}(\mathfrak{n}-\mathfrak{m})\ne 0
\end{eqnarray}
where the last follows from (\ref{635}) and (\ref{636}). Split
$$\mathfrak{n}-\mathfrak{m}={\rm
P}_{V(\mathbb{G})}(\mathfrak{n}-\mathfrak{m})+ {\rm
P}_{V^{\perp}(\mathbb{G})}(\mathfrak{n}-\mathfrak{m}).$$ Since
$\mathfrak{q}\in (\mathbb{G}^*)^{\circ}|(\mathbb{P}^-,\mathbb{R}^n)$
and $\mathfrak{p}\in (\mathbb{G}^*)^{\circ}| \mathbb{F}$, that is
$\mathfrak{q},\mathfrak{p}\perp V(\mathbb{G})$,
\begin{eqnarray*}
\mathfrak{q}\cdot  {\rm
P}_{V(\mathbb{G})}(\mathfrak{n}-\mathfrak{m}) =\mathfrak{p}\cdot
{\rm P}_{V(\mathbb{G})}(\mathfrak{n}-\mathfrak{m}) =0.
\end{eqnarray*}
So
\begin{eqnarray*}
(\mathfrak{q}+\epsilon\mathfrak{p})\cdot
(\mathfrak{n}-\mathfrak{m})=
(\mathfrak{q}+\epsilon\mathfrak{p})\cdot {\rm
P}_{V^{\perp}(\mathbb{G})}(\mathfrak{n}-\mathfrak{m}) .
\end{eqnarray*}
Choose $\epsilon_0=\frac{r}{100|\mathfrak{p}|}$ and
$0<\epsilon<\epsilon_0$. Then, by (\ref{joau}),
\begin{eqnarray*}
(\mathfrak{q}+\epsilon\mathfrak{p})\cdot(\mathfrak{n}-\mathfrak{m})
&=&(\mathfrak{q}+\epsilon\mathfrak{p})\cdot{\rm
P}_{V^{\perp}(\mathbb{G})}(\mathfrak{n}-\mathfrak{m})\\
&\ge&  r \left|{\rm
P}_{V^{\perp}(\mathbb{G})}(\mathfrak{n}-\mathfrak{m})\right|-
\epsilon |\mathfrak{p}|\left|{\rm
P}_{V^{\perp}(\mathbb{G})}(\mathfrak{n}-\mathfrak{m})\right|\\
&\ge& \frac{99}{100}r\left|{\rm
P}_{V^{\perp}(\mathbb{G})}(\mathfrak{n}-\mathfrak{m})\right|>0.
 \end{eqnarray*}
 This completes the proof of Proposition \ref{lemggu}
\end{proof}

\begin{lemma}\label{lem15}
Let $\mathbb{P}_\nu\subset\mathbb{R}^n$ be a polyhedron and
$\mathbb{F}_\nu,\mathbb{G}_\nu\in\mathcal{F}(\mathbb{P}_\nu)$ such
that $\mathbb{G}_\nu\preceq \mathbb{F}_\nu$. Suppose that
$\mathfrak{q}\in\bigcap_\nu
(\mathbb{F}_\nu^*)^{\circ}|(\mathbb{P}_\nu,\mathbb{R}^n)$. Suppose
$\mathfrak{p}\in \bigcap_\nu
(\mathbb{G}_\nu^*)^{\circ}|(\mathbb{F}_\nu,\mathbb{R}^n)$. Then there exists a
vector ${\bf w}\in \bigcap
(\mathbb{G}_\nu^*)^{\circ}|(\mathbb{P}_\nu,\mathbb{R}^n)$.
\end{lemma}
\begin{proof}
By Proposition \ref{lemggu}, there exists $\epsilon_{\nu}>0$ such
that for $0<\epsilon<\epsilon_\nu$, $
\mathfrak{q}+\epsilon\mathfrak{p}\in
(\mathbb{G_\nu}^*)^{\circ}|(\mathbb{P}_\nu,\mathbb{R}^n). $ Choose
$\epsilon_0=\min \{\epsilon_\nu:\nu=1,\cdots,d\}$. Then for
$0<\epsilon<\epsilon_0$,  $
\mathfrak{q}+\epsilon\mathfrak{p}\in\bigcap_{\nu=1}^d
(\mathbb{G_\nu}^*)^{\circ}|(\mathbb{P}_\nu,\mathbb{R}^n).$
\end{proof}

\subsection{Lemma  for  Necessity}
\begin{lemma}\label{lem92d}
Let $\Lambda=(\Lambda_\nu)$ with $\Lambda_\nu\subset\mathbb{Z}_+^n$
and let $P_\Lambda\in\mathcal{P}_\Lambda$. Fix
$S\subset\{1,\cdots,n\}$. Given
$\mathbb{F}=(\mathbb{F}_\nu)\in\mathcal{F}({\bf N}(\Lambda,S))$,
define
 \begin{eqnarray*}
 \mathcal{I}(P_{\mathbb{F}},\xi,r)&=&\int_{\prod(-r_j,r_j)} e^{i\left(\xi_1 \sum_{\mathfrak{m}
  \in \mathbb{F}_1\cap\Lambda_1}c_{\mathfrak{m}}^1 t^{\mathfrak{m}}+  \cdots+ \xi_d\sum_{\mathfrak{m}
  \in \mathbb{F}_d\cap\Lambda_d}c_{\mathfrak{m}}^d t^{\mathfrak{m}}\right)} \frac{dt_1}{t_1}\cdots
  \frac{dt_n}{t_n},\\
   \mathcal{I}(P_{\mathbb{F}},\xi,a,b)&=&\int_{\prod\{a_j<|t_j|<b_j\}} e^{i\left(\xi_1 \sum_{\mathfrak{m}
  \in \mathbb{F}_1\cap\Lambda_1}c_{\mathfrak{m}}^1 t^{\mathfrak{m}}+  \cdots+ \xi_d\sum_{\mathfrak{m}
  \in \mathbb{F}_d\cap\Lambda_d}c_{\mathfrak{m}}^d t^{\mathfrak{m}}\right)} \frac{dt_1}{t_1}\cdots
  \frac{dt_n}{t_n}.
\end{eqnarray*}
Suppose that
 $$\sup_{r\in I(S),\xi\in\mathbb{R}^d}\left|\mathcal{I}(P_\Lambda,\xi,r) \right|< \infty $$
where $I(S)$ is as defined in (\ref{an50}).   Suppose also  that
 \begin{eqnarray*}
  {\bf u}=(u_1,\cdots,u_n)\in\bigcap_{\nu=1}^d
(\mathbb{F}_{\nu}^{*})^{\circ} \ne \emptyset
\end{eqnarray*}
where
\begin{eqnarray}\label{swgg}
  \text{ $u_{j}>0$  for $j\in S\setminus S_0$ and
$u_j=0$ for $j\in S_0\subset S$ by Lemma \ref{lem132}}.
\end{eqnarray}
 Then
 \begin{eqnarray}\label{bd66}
 \sup_{\xi\in\mathbb{R}^d,  a,b\in I(S_0)  }\left|
 \mathcal{I}(P_{\mathbb{F}},\xi,a,b)\right|&<&\infty.
\end{eqnarray}
\end{lemma}
\begin{proof}[Proof of (\ref{bd66})]
By the definition of $(\mathbb{F}_\nu^*)^{\circ}$, there exists
$\rho_{\nu}$
 such that
\begin{eqnarray}\label{ggf3}
{\bf u}\cdot\mathfrak{m}=\rho_{\nu}< {\bf u}\cdot\mathfrak{n}\ \
\text{for all $\mathfrak{m}\in \mathbb{F}_{\nu}$ and
$\mathfrak{n}\in {\bf N}(\Lambda_{\nu},S)\setminus \mathbb{F}_{\nu}$}.
\end{eqnarray}
Let $a=(a_j),b=(b_j)\in I(S_0)$ where $I(S_0)=\prod_{j=1}^nI_j$ where
\begin{equation}\label{bd2hh}
\text{ $I_j=(0,\infty)$ for $j\in \{1,\cdots,n\}\setminus
S_0$ and $I_j=(0,1)$ for $j\in S_0 $ as  in (\ref{an50}). }
 \end{equation}
  Let $\rho=(\rho_\nu)$ with $\rho_\nu$ in (\ref{ggf3}). Set
$$I(a(\delta),b(\delta) )=\prod_{j=1}^n\{a_j\delta^{u_j}
 <|t_j|<b_j\delta^{u_j}   \}\ \ \text{and}\ \
 \delta^{-\rho}\xi=(\delta^{-\rho_1}\xi_1,\cdots,\delta^{-\rho_d}\xi_d).$$
Then \begin{eqnarray*}
&&\mathcal{I}(P_\Lambda,\delta^{-\rho}\xi,a(\delta),b(\delta)) \\
&&\qquad\qquad=\int_{I(a(\delta),b(\delta) )}
e^{i\left(\delta^{-\rho_1}\xi_1 \sum_{\mathfrak{m}\in
\Lambda_1}c_{\mathfrak{m}}^1t^{\mathfrak{m}}+\cdots+\delta^{-\rho_d}\xi_d
\sum_{\mathfrak{m}\in
\Lambda_{d}}c_{\mathfrak{m}}^dt^{\mathfrak{m}}\right)}\frac{dt_1}{t_1}\cdots\frac{dt_n}{t_n}.
\end{eqnarray*}
By (\ref{swgg}) and (\ref{bd2hh}), we find a sufficiently small
$\delta$ such that $$\delta^{u_j} a_j, \delta^{u_j}b_j<1 \ \
\text{for $j\in S\setminus S_0$\ \ and \ }   \delta^{u_j} a_j=a_j,
\delta^{u_j}b_j=b_j<1
 \ \ \text{ for   $j\in  S_0\subset S$. }$$  Thus $a(\delta),b(\delta)\in I(S)$. Hence, by our hypothesis
\begin{eqnarray}\label{bd2}
\left|\mathcal{I}(P_\Lambda,\delta^{-\rho}\xi,a(\delta),b(\delta))
 \right|\le C  \ \text{uniformly in $\xi$ and $a,b ,\delta$.}
\end{eqnarray}
Consider the difference of two multipliers given
 by
\begin{eqnarray*}
 M(\xi,\delta,a,b)&=& \mathcal{I}(P_\Lambda,\delta^{-\rho}\xi,a(\delta),b(\delta))- \mathcal{I}(P_{\mathbb{F}},\delta^{-\rho}\xi,a(\delta),b(\delta))\\
 &=&\int_{I(a(\delta) ,b(\delta) )}
e^{i\left(\xi_1\delta^{-\rho_1}\sum_{\mathfrak{m}\in
\Lambda_1}c_{\mathfrak{m}}^1t^{\mathfrak{m}}+\cdots+\xi_d\delta^{-\rho_d}\sum_{\mathfrak{m}\in
\Lambda_{d}}c_{\mathfrak{m}}^dt^{\mathfrak{m}}\right)}\frac{dt_1}{t_1}\cdots\frac{dt_n}{t_n}\\
 &-& \int_{I(a(\delta),b(\delta) )}
e^{i\left(\xi_1\delta^{-\rho_1}\sum_{\mathfrak{m}\in
\mathbb{F}_{1}\cap\Lambda_1}c_{\mathfrak{m}}^1t^{\mathfrak{m}}+\cdots+\xi_d\delta^{-\rho_d}\sum_{\mathfrak{m}\in
\mathbb{F}_{d}\cap
\Lambda_d}c_{\mathfrak{m}}^dt^{\mathfrak{m}}\right)}\frac{dt_1}{t_1}\cdots\frac{dt_n}{t_n}.
\end{eqnarray*}
By the mean value theorem and change of variable
$t_j'=\delta^{-u_j}t_j$ in the above two integrals,
\begin{eqnarray*}
&&|M(\xi,\delta,a,b)|\\
&&\quad\le \int_{I(a,b)}\left| \xi_1\sum_{\mathfrak{n}\in
\Lambda_{1}\setminus \mathbb{F}_{1}}\delta^{{\bf u}\cdot
\mathfrak{n}-\rho_1}c_{\mathfrak{n}}^1t^{\mathfrak{n}}+\cdots+\xi_d\sum_{\mathfrak{n}\in
\Lambda_{d}\setminus \mathbb{F}_{d}}\delta^{{\bf u}\cdot
\mathfrak{n}-\rho_d}c_{\mathfrak{n}}^dt^{\mathfrak{n}}\right|\frac{dt_1}{|t_1|}\cdots\frac{dt_n}{|t_n|}\\
&&\quad\le   |\xi_1|\sum_{\mathfrak{n}\in \Lambda_{1}\setminus
\mathbb{F}_{1}}\delta^{{\bf u}\cdot
\mathfrak{n}-\rho_1}|c_{\mathfrak{n}}^1|C^1_{\mathfrak{n}}(a,b)+\cdots+|\xi_d|\sum_{\mathfrak{n}\in
\Lambda_{d}\setminus \mathbb{F}_{d}}\delta^{{\bf u}\cdot
\mathfrak{n}-\rho_d}|c_{\mathfrak{n}}^d|C^d_{\mathfrak{n}}(a,b).
\end{eqnarray*}
The constants $C^1_{\mathfrak{m}}(a,b),\cdots,
C^d_{\mathfrak{m}}(a,b)$ above are absolute value for the integral
of $\frac{t^{\mathfrak{m}}}{|t_1|\cdots|t_n|}$ on the region
$I(a,b)$. From  ${\bf u}\cdot \mathfrak{n}-\rho_\nu>0$ in
(\ref{ggf3}), we can choose $\delta>0$ so that  $ \delta^{{\bf
u}\cdot \mathfrak{n}-\rho_\nu}$
  above is small enough to satisfy
\begin{eqnarray*}
|M(\xi,\delta,a,b)|&\le& 1.
\end{eqnarray*}
By this and (\ref{bd2}),
\begin{eqnarray}\label{bd6}
\left|\mathcal{I}(P_{\mathbb{F}},\delta^{-\rho}\xi,a(\delta),b(\delta))\right|\le
C.
\end{eqnarray}
By (\ref{ggf3}) and  the change of variables $\delta^{-u_j}t_j=t_j'$
for all $j=1,\cdots,n$,
\begin{eqnarray*}\label{bd5}
&&\mathcal{I}(P_{\mathbb{F}},\delta^{-\rho}\xi,a(\delta),b(\delta))\\
&&\quad=\int_{I(a(\delta),b(\delta))}
e^{i\left(\xi_1\delta^{-\rho_1}\sum_{\mathfrak{m}\in
\mathbb{F}_{1}}c_{\mathfrak{m}}^1t^{\mathfrak{m}}+\cdots+\xi_d\delta^{-\rho_d}\sum_{\mathfrak{m}\in
\mathbb{F}_{d}}c_{\mathfrak{m}}^dt^{\mathfrak{m}}\right)}\frac{dt_1}{t_1}\cdots\frac{dt_n}{t_n}\\
&&\quad=\int_{I(a ,b )} e^{i\left(\xi_1 \sum_{\mathfrak{m}\in
\mathbb{F}_{1}}c_{\mathfrak{m}}^1t^{\mathfrak{m}}+\cdots+\xi_d
\sum_{\mathfrak{m}\in
\mathbb{F}_{d}}c_{\mathfrak{m}}^dt^{\mathfrak{m}}\right)}\frac{dt_1}{t_1}\cdots\frac{dt_n}{t_n}=
\mathcal{I}(P_{\mathbb{F}},\xi,a,b).
\end{eqnarray*}
Hence this identity combined with (\ref{bd6})   yields (\ref{bd66}).
\end{proof}

\subsection{Necessity Theorem}
\begin{definition}
To each subset $M
=\{\mathfrak{q}_1,\cdots,\mathfrak{q}_N\}\subset\mathbb{R}^n$, we associate
a  matrix:
$$\rm{Mtr}(M)=\left(
\begin{matrix}
&\mathfrak{q}_1&\\
& \vdots&\\
&\mathfrak{q}_N&
\end{matrix}
\right)$$
 whose rows are vectors in $M$. We define a class of rank $m$-subsets in
 $\mathbb{R}^n$:
\begin{eqnarray}\label{ma1}
\qquad\  \ \mathcal{M}_{m,n}=\biggl\{M\subset \mathbb{R}^n:
\rm{Mtr}(M)\sim\left(
\begin{matrix}
1&0&\cdots&0&a_{1,m+1}&\cdots&a_{1,n}\\
0&1&0&\vdots&a_{2,m+1}&\cdots&a_{2,n}\\
\vdots& 0&1&0&\vdots&\cdots&\vdots\\
0&\cdots&0&1&a_{m,m+1}&\cdots&a_{m,n}
\end{matrix}
\right)  \biggl\}.
\end{eqnarray}
  Here   $\sim$ means row equivalence and
 $(a_{ij})_{1\le i\le m, \, m+1\le j\le n}$  a real $m\times (n-m)$
 matrix.
\end{definition}
\begin{theorem}[Necessity Part of Main Theorems \ref{main18} through \ref{main4}]\label{T3}
Let $\Lambda=(\Lambda_\nu)$ where $\Lambda_\nu\subset\mathbb{Z}_+^n$
is a finite set for $\nu=1,\cdots,d$ and let $S\subset
\{1,\cdots,n\}$.   Suppose that there exist faces $\mathbb{F}
\in\mathcal{F}({\bf N}(\Lambda,S))$ such that
\begin{eqnarray}\label{912}
\bigcup_{\nu=1}^d\left(\mathbb{F}_\nu\cap\Lambda_{\nu}\right)\
\text{ is not an even set\,, }
\end{eqnarray}
and $\mathbb{F}\in\mathcal{F}_{\rm{lo}}({\bf N}(\Lambda,S))$, that is,
\begin{eqnarray}\label{913}
 \text{rank}\left(\bigcup_{\nu}  \mathbb{F}_\nu \right)\le n-1 \ \
 \text{and}\ \
   \bigcap_{\nu} (\mathbb{F}_\nu^{*})^{\circ}|{\bf N}(\Lambda_\nu,S)\ne \emptyset.
\end{eqnarray}
Then there
 exist a vector polynomial
$P_{\Lambda}\in\mathcal{P}_\Lambda$ so that $$\sup_{\xi\in\mathbb{R}^d, r
\in I(S)}\left|\mathcal{I}(P_{\Lambda},\xi,r)\right|=\infty.$$
\end{theorem}
\begin{proof} [Proof of Theorem \ref{T3}]
Choose the  integer $m$ such that
\begin{equation}\label{nec1}
m=\min\left\{ \text{rank}\big(\bigcup_{\nu}  \mathbb{F}_\nu
\big):\exists\ \mathbb{F} \in\mathcal{F}({\bf N}(\Lambda,S))\
\text{satisfying (\ref{912}), (\ref{913})} \right\}.
\end{equation}
Then we have $\mathbb{F}=(\mathbb{F}_\nu)$ with
$\mathbb{F}_\nu\in\mathcal{F}({\bf N}(\Lambda_\nu,S))$ such that
\begin{eqnarray}\label{912-}
\bigcup_{\nu=1}^d\left(\mathbb{F}_{\nu}\cap\Lambda_{\nu}\right)\ \text{ is
not an even set\,, }
\end{eqnarray}
and
\begin{eqnarray}\label{913-}
 \text{rank}\left(\bigcup_{\nu} \mathbb{F}_\nu \right)=m\le n-1 \ \
 \text{and}\ \
   \bigcap_{\nu} (\mathbb{F}_{\nu}^{*})^{\circ}| {\bf N}(\Lambda_\nu,S) \ne \emptyset.
\end{eqnarray}
This with (\ref{ma1}) implies that we can assume that without loss
of generality,
\begin{eqnarray}\label{ma1gg}
\text{Sp}\left(\bigcup_{\nu}  \mathbb{F}_\nu \right)\in\mathcal{M}_{m,n}.
\end{eqnarray}
By (\ref{913-}) and (\ref{swgg}), we have for some  $S_0\subset
S\subset N_n$
\begin{equation*}
 (u_j)\in\bigcap_{\nu}(\mathbb{F}_{\nu}^{*})^{\circ}| {\bf
N}(\Lambda_\nu,S) \ \ \text{and}\ \ u_j=0\ \ \text{for $j\in S_0$
and $u_j>0$ for $j\in S\setminus S_0$}.
\end{equation*}
By Lemma \ref{dg23},
\begin{equation}\label{49k}
 \{{\bf e}_j:j\in S_0\}\subset\text{Sp}\left(\bigcup_{\nu}  \mathbb{F}_\nu \right).
\end{equation}
Moreover by (\ref{nec1}), for any $\mathbb{G}_\nu \preceq \mathbb{F}_\nu$,
\begin{equation}\label{96}
  \bigcup \mathbb{G}_\nu\cap\Lambda_\nu\ \ \text{is   even  }\
\text{whenever}\  \text{rank}\,\big(\bigcup  \mathbb{G}_\nu \big)\le
m-1\ \text{and}\   \bigcap (\mathbb{G}_\nu^*)^{\circ}|{\bf N}(\Lambda_\nu,S) \ne \emptyset.
\end{equation}
In order to show Theorem \ref{T3}, by Lemma \ref{lem92d}, it
suffices to find, under the assumption (\ref{912-})-(\ref{96}), a
polynomial $P_\Lambda(t)=\left(\sum_{\mathfrak{q}\in
\Lambda_\nu}c^{\nu}_{\mathfrak{q}}t^{\mathfrak{q}}\right)\in
\mathcal{P}_\Lambda$ with an appropriate $c^{\nu}_{\mathfrak{q}}$
such that
\begin{eqnarray}\label{781}
\sup_{\xi\in\mathbb{R}^d, a,b\in
I(S_0)}\left|\mathcal{I}(P_\mathbb{F},\xi,a,b)\right|&=&\infty
 \end{eqnarray}
 where
$$\mathcal{I}(P_{\mathbb{F}},\xi,a,b)=\int_{\prod\{a_j<|t_j|<b_j\}} e^{i\left(\xi_1 \sum_{\mathfrak{m}
  \in \mathbb{F}_1\cap\Lambda_1}c_{\mathfrak{m}}^1 t^{\mathfrak{m}}+  \cdots+ \xi_d\sum_{\mathfrak{m}
  \in \mathbb{F}_d\cap\Lambda_d}c_{\mathfrak{m}}^d t^{\mathfrak{m}}\right)} \frac{dt_1}{t_1}\cdots
  \frac{dt_n}{t_n}.
 $$
\begin{definition}\label{defi101}
Let $1\le m<n$. Using $\mathbb{R}^n=X \bigoplus Y$ with
$X=\mathbb{R}^m\times\{0\}$ and $Y=\{ 0\}\times \mathbb{R}^{n-m}$,
we write  $a=(a_1,\cdots,a_n)\in \mathbb{R}^n$ as $a=(a_X,a_Y)$ so
that
\begin{eqnarray*}\label{3542}
a_X=(a_1,\cdots,a_m)\in\mathbb{R}^m\ \ \text{and}\ \ a_Y=(a_{m+1},\cdots,a_n)\in\mathbb{R}^{n-m}.
\end{eqnarray*}
Note that $1_X(S_0)$ is restricted to $\mathbb{R}^m$
 so that $$\text{$1_X(S_0)=(r_j)_{j=1}^m$ with $r_j=1$ for $j\in S_0$ and
$r_j=\infty$ for $j\in N_m\setminus S_0$\,,} $$ and $I_X(S_0)$ is also
restricted to $\mathbb{R}^m$ in view of (\ref{an50}) so that
\begin{eqnarray*}
\qquad I_X(S_0)=\prod_{j=1}^m I_j \   \text{  where $I_j=(0,1)$ for
$j\in S_0$ and $I_j=(0,\infty)$ for  $j\in N_m\setminus S_0$. }
\end{eqnarray*}
\end{definition}
To show (\ref{781}), we prove that there exists   $C(\xi)>0$,
\begin{equation}\label{921}
\left|\lim_{a_X\rightarrow 0,b_X\rightarrow
1_X(S_0)}\mathcal{I}(P_{\mathbb{F}},\xi, a,b) \right|\ge
C(\xi)\prod_{j=m+1}^n\log
  (b_j/ a_{j})\rightarrow \infty\ \ \text{as} \ a_j\rightarrow 0\,,
\end{equation}
where the integral $\mathcal{I}(P_{\mathbb{F}},\xi, a,b)$ is
evaluated so that $\frac{dt_1}{t_1}\cdots\frac{dt_m}{t_m}$ first and
$\frac{dt_{m+1}}{t_{m+1}}\cdots
  \frac{dt_n}{t_n}$ next:
 $$\mathcal{I}(P_{\mathbb{F}},\xi,a,b)=\int_{\prod_{j=m+1}^n\{a_j<|t_j|<b_j\}}
  \left[\int_{\prod_{j=1}^m\{a_j<|t_j|<b_j\}} e^{i \xi\cdot P_{\mathbb{F}}(t)} \frac{dt_1}{t_1}\cdots
  \frac{dt_m}{t_m}\right] \frac{dt_{m+1}}{t_{m+1}}\cdots
  \frac{dt_n}{t_n}.
 $$ As a first step to show (\ref{921}), we  write the integral in (\ref{921}) as
 \begin{equation}\label{921gg}
 \int_{\prod_{j=m+1}^n\{a_j<|t_j|<b_j\}} \left[\lim_{a_X\rightarrow 0,b_X\rightarrow 1_X(S_0)}
 \int_{\prod_{j=1}^m\{a_j<|t_j|<b_j\}} e^{i \xi\cdot P_{\mathbb{F}}(t)} \frac{dt_1}{t_1}\cdots
  \frac{dt_m}{t_m}\right] \frac{dt_{m+1}}{t_{m+1}}\cdots
  \frac{dt_n}{t_n}.
  \end{equation}
  This follows from the dominated convergence theorem. To apply the  convergence theorem,
     we shall prove that  for each $t_Y=(t_{m+1},\cdots,t_n)\in
  \prod_{j=m+1}^n\{a_j<|t_j|<b_j\}$
  \begin{equation}\label{916}
\sup_{\xi\in\mathbb{R}^d, a_X,b_X\in I_X(S_0)}
\left|\int_{\prod_{j=1}^m\{a_j<|t_j|<b_j\}} e^{i \xi\cdot
P_{\mathbb{F}}(t)} \frac{dt_1}{t_1}\cdots
  \frac{dt_m}{t_m}\right|\le
C((a_j)_{j=m+1}^n,(b_j)_{j=m+1}^n),
\end{equation}
and that
\begin{equation}\label{916i}
\lim_{a_X\rightarrow 0,b_X\rightarrow 1_X(S_0)}\int_{\prod_{j=1}^m\{a_j<|t_j|<b_j\}}
 e^{i \xi\cdot P_{\mathbb{F}}(t)} \frac{dt_1}{t_1}\cdots
  \frac{dt_m}{t_m}\ \text{exists}.
\end{equation}
Assume that (\ref{916}) and (\ref{916i})  are proved in the next
section. Then, by writing each integral in the $n$-tuple integral
(\ref{921gg}) as the sum of 2 integrals by separating the variables
into the positive and negative parts, we decompose the above
$n$-tuple integral into the sum of $2^n$ pieces. For this purpose,
we denote $\,O=\{\sigma= (\sigma_j)=
 (\underbrace{\pm 1,\cdots,\pm 1}_{\text{$n$
components}})\}  \,,$ the sign-index set of $2^n$ elements. Next by
using the change of variables $\,t_1 =\sigma_1
t_1',\cdots,\,t_n'=\sigma_n t_n'\,$ in each integration, we write
the  integral (\ref{921gg}) as
\begin{align}\label{958y}
 \int_{\prod_{j=m+1}^n\{a_j<t_j<b_j\}}
 \mathcal{J}(P_{\mathbb{F}},\xi,(t_{m+1},\cdots,t_n))\,\frac{dt_{m+1}}{t_{m+1}}\cdots\frac{dt_n}{t_n}.
\end{align}
Here  the integrand above is
\begin{eqnarray}\label{nec22}
&& \mathcal{J}(P_{\mathbb{F}},\xi,(t_{m+1},\cdots,t_n))\nonumber\\
&&\quad=\lim_{a_X\rightarrow 0,b_X\rightarrow 1_X(S_0)}
 \int_{\prod_{j=1}^m\{a_j<t_j<b_j\}} \sum_{\sigma\in O}
(-1)^{|\sigma|} \exp\left(iP_{\mathbb{F}}(\xi, \sigma t)\right)
\frac{dt_1}{t_1}\cdots\frac{dt_{m}}{t_m}
\end{eqnarray}
where $|\sigma|$ is a number of $-1$ in the components of $\sigma\in
O$ and $P_{\mathbb{F}}(\xi,\sigma t) $
  is
\begin{eqnarray*}\label{9.9}
P_\mathbb{F}(\xi,\sigma_1 t_1, \cdots, \sigma_n
t_n)=\sum_{\nu=1}^d\left(\sum_{\mathfrak{q}\in \mathbb{F}_\nu}
c_{\mathfrak{q}}^{\nu}\sigma^{\mathfrak{q}}t^{\mathfrak{q}}\right)\xi_\nu.
\end{eqnarray*}
 The limit in (\ref{nec22}) exists  by (\ref{916i}). Moreover,
$\mathcal{J}(P_{\mathbb{F}},\xi,(t_{m+1},\cdots,t_n))$ is finite by
(\ref{916}). We shall in the next section prove that
 it is independent of $
t_Y=(t_{m+1},\cdots,t_n) $:
\begin{eqnarray}\label{nec23}
\mathcal{J}(P_{\mathbb{F}},\xi,(t_{m+1},\cdots,t_n))=\mathcal{J}(P_{\mathbb{F}},\xi)
\end{eqnarray}
and non-vanishing:
\begin{eqnarray}\label{nec55g}
\exists \, \text{ $P_\Lambda(t)$ and $\xi $ such that }\
\mathcal{J}(P_{\mathbb{F}},\xi)\ne 0.
\end{eqnarray}
Therefore (\ref{921}) follows from (\ref{nec23}) and (\ref{nec55g})
in (\ref{958y}). We shall prove  (\ref{916}) and (\ref{916i})
 in the first part of Section \ref{sec13}, and  show (\ref{nec23}) and (\ref{nec55g})
 in the last part of
 Section \ref{sec13}.
\end{proof}
\section{ Proof of Necessity}\label{sec13}
\subsection{ Proof of (\ref{916}) and (\ref{916i})}
 Let
$\Omega=(\Omega_{\nu})$ where $\Omega_{\nu}\subset\mathbb{R}^m$ is
given by
\begin{eqnarray}
\Omega_{\nu}=  (\mathbb{F}_{\nu} \cap
 \Lambda_{\nu} )_X=\{(q_{1},\cdots,q_{m}):
(q_1,\cdots,q_n)\in \mathbb{F}_\nu\cap\Lambda_\nu\}.\label{omee}
\end{eqnarray}
For each $t_Y=(t_{m+1},\cdots,t_n)\in
\prod_{j=m+1}^n\{a_j<|t_j|<b_j\}$, define $$
\mathcal{I}(P_{\Omega},\xi,a_X,b_X,t_Y)=
\int_{\prod_{j=1}^m\{a_j<|t_j|<b_j\}} e^{i \xi\cdot
P_{\mathbb{F}}(t_1,\cdots,t_m,t_Y)} \frac{dt_1}{t_1}\cdots
  \frac{dt_m}{t_m}. $$ To show (\ref{916}) and (\ref{916i}), it suffices to show that for all $P_{\Omega}\in
\mathcal{P}_{\Omega}$,
\begin{equation}
\sum_{J\in
Z(S_0)\cap\,\mathbb{Z}^m}|\mathcal{I}_J(P_{\Omega},\xi,a_X,b_X,t_Y)|\le
C_R\prod_{\nu}\prod_{\mathfrak{q}\in\Lambda_\nu}
(|D_{\mathfrak{q}}c^\nu_{\mathfrak{q}}
|+1/|d_{\mathfrak{q}}c^\nu_{\mathfrak{q}} |)^{1/R}\label{151gh}
\end{equation}
where  $Z(S_0)=\prod_{i=1}^m Z_i$ with $Z_i=\mathbb{R}_+$ for $i\in
S_0$ and $Z_i=\mathbb{R}$ as in  Proposition \ref{74sj}. Here
$D_{\mathfrak{q}}=\max\{\prod_{j=m+1}^n
|t_{j}^{q_j}|:a_{j}<|t_j|<b_j\}$ and
$d_{\mathfrak{q}}=\min\{\prod_{j=m+1}^n
|t_{j}^{q_j}|:a_{j}<|t_j|<b_j\}$.
 By Theorem \ref{th60} with Remark
\ref{rem91}, we obtain (\ref{151gh}) with the similar bound in
(\ref{1300})   because the hypotheses of Theorem \ref{th60} are
satisfied as it is checked in the following proposition:
\begin{proposition}\label{lem83}
Suppose (\ref{nec1}),(\ref{912-}),(\ref{913-}) and (\ref{ma1gg})
hold. Then
\begin{eqnarray*}
\bigcup_{\nu=1}^d (\mathbb{K}_\nu\cap\Omega_{\nu}) \ \ \text{ is an
even set}
\end{eqnarray*}
whenever $
\mathbb{K}=(\mathbb{K}_{\nu})\in\mathcal{F}_{\rm{lo}}(\vec{{\bf
N}}(\Omega,S_0))$ where $\mathcal{F}_{\rm{lo}}(\vec{{\bf
N}}(\Omega,S_0))$ is
\begin{eqnarray*}
 \left\{ \mathbb{K} \in\mathcal{F}(\vec{{\bf
N}}(\Omega,S_0)): \rm{rank} \,\big( \bigcup_{\nu=1}^d
\mathbb{K}_\nu \big)\le m-1\ \ \text{
 and}
   \ \ \bigcap_{\nu=1}^d (\mathbb{K}_\nu^*)^{\circ}| {\bf N}(\Omega_{\nu},S_0) \ne
   \emptyset\right\}.
\end{eqnarray*}
\end{proposition}
\noindent To prove Proposition \ref{lem83}, we first observe that
for $\mathbb{F}=(\mathbb{F}_\nu)$ satisfying (\ref{912-}) and
(\ref{913-}),
 \begin{eqnarray}\label{mn12004}
 &&\big\{\mathfrak{q}\in \Sigma\left(\bigcup_{\nu=1}^d \left(\mathbb{F}_\nu\cap\Lambda_\nu\right)\right) :
\mathfrak{q}=(\underbrace{odd,\cdots,odd}_{\text{$m$
components}},*,\cdots,*)
\big\}\nonumber\\&&\qquad\qquad\qquad\subset \left\{\mathfrak{q}\in
\Sigma\left(\bigcup_{\nu=1}^d
\left(\mathbb{F}_\nu\cap\Lambda_\nu\right) \right):
\mathfrak{q}=(\underbrace{odd,\cdots,odd}_{\text{$n$
components}})\right\}
\end{eqnarray}
where $\Sigma(A) $ with $A\subset\mathbb{Z}^n$ is defined below
(\ref{411}).  The proof for (\ref{mn12004}) follows by taking
$\mathbb{U}=\Sigma\left(\bigcup_{\nu=1}^d
\left(\mathbb{F}_\nu\cap\Lambda_\nu\right)\right)$ in the following
lemma:
 \begin{lemma}\label{lem141}
Suppose  $\text{Sp}\left(\mathbb{U}\right)=k\le m$ and
$\mathbb{U}\in \mathcal{M}_{k,n}$ where $\mathcal{M}_{k,n}$ is
defined in (\ref{ma1}). If there exists a vector $
 \mathfrak{p}=(odd,\cdots,odd)\in\mathbb{U}$, then
\begin{eqnarray*}
\qquad \big\{\mathfrak{q}\in \mathbb{U}:
\mathfrak{q}=(\underbrace{odd,\cdots,odd}_{\text{$k$
components}},*,\cdots,*)  \big\} \subset \left\{\mathfrak{q}\in
\mathbb{U}: \mathfrak{q}=(odd,\cdots,odd)\right\}.
\end{eqnarray*}
\end{lemma}
\begin{proof}
  Assume that there exists  $\mu\in\{k+1,\cdots,n\}$ such that
\begin{eqnarray}\label{11332}
 \mathfrak{q}=(\underbrace{ \underbrace{odd, \cdots,odd}_{\text{$k$
components}},*,\cdots,*,even}_{\text{$\mu$ components}},*,\cdots,*)
\in \mathbb{U}
\end{eqnarray}
where $\mathfrak{q}=(q_j)$ with $q_j=\text{odd numbers}$ for $j=1,\cdots,k$ and $q_{\mu}=\text{even
number}$. Thus
 \begin{eqnarray*}
 \mathfrak{r}=\mathfrak{q}+\mathfrak{p}=(\underbrace{ \underbrace{even, \cdots,even}_{\text{$k$
components}},*,\cdots,*,odd}_{\text{$\mu$ components}},*,\cdots,*)
\in \Sigma(\mathbb{U})
\end{eqnarray*}
where $\mathfrak{r}=(r_j)$ with $r_j=\text{even numbers}$ for
$j=1,\cdots,k$ and $r_{\mu}=\text{even number}$.  We add
$\mathfrak{r}$ as the last row to the matrix in (\ref{ma1}) in view
of $\mathbb{U}\in\mathcal{M}_{k,n}$. Then,
\begin{eqnarray}\label{mn3}
\quad \rm{Mtr}\left(\mathbb{U}\right)\sim\left(
\begin{matrix}
1&0&\cdots&0&c_{1,k+1}&\cdots&c_{1,\mu}&\cdots&c_{1,n}\\
0&1&0&\vdots&c_{2,k+1}&\cdots&c_{2,\mu}&\cdots&c_{2,n}\\
\vdots& 0&1&0&\vdots&\cdots&\vdots&\cdots&\vdots\\
0&\cdots&0&1&c_{k,k+1}&\cdots&c_{k,\mu}&\cdots&c_{k,n}\\
even&\cdots&even&even&*&\cdots&odd&\cdots&*
\end{matrix}
\right).
\end{eqnarray}
Consider the $(k+1)\times (k+1)$ submatrix
$\text{Mtr}_\mu(\mathbb{U})$ consisting of the first $k$ columns and
the $\mu^{th}$ column  of the matrix in (\ref{mn3}):
\begin{eqnarray*}
\text{Mtr}_\mu(\mathbb{U}) =\left(
\begin{matrix}
1&0&\cdots&0&c_{1,\mu} \\
0&1&0&\vdots&c_{2,\mu} \\
\vdots& 0&1&0&\vdots \\
0&\cdots&0&1&c_{k,\mu} \\
even&\cdots&even&even&odd
\end{matrix}
\right).
\end{eqnarray*}
Then we expand the entries of the last row multiplied by their minors to compute
 $$\det(\text{Mtr}_\mu(\mathbb{U}))=odd\ne 0.$$ Thus
$\text{rank}(\text{Mtr}_\mu(\mathbb{U}))= k+1$, which is a
contradiction to
  $\mathbb{U}\in\mathcal{M}_{k,n}.$  Therefore there is no such $\mu$ satisfying  (\ref{11332}).
\end{proof}

\begin{proof}[Proof of Proposition \ref{lem83}]
  By (\ref{ma1gg}),
$${\rm Sp}\left(\bigcup \mathbb{F}_\nu \right)\in \mathcal{M}_{mn}.$$
Thus   a projection $P_X:{\rm Sp}\left(\bigcup \mathbb{F}_\nu
\right)\rightarrow X=\mathbb{R}^m$ defined  by
$$P_X(q_1,\cdots,q_m,q_{m+1},\cdots,q_n)=(q_1,\cdots,q_m)$$
  is an isomorphism. We shall denote $P_X(\mathfrak{q})=\mathfrak{q}_X$.
To show Proposition \ref{lem83}, we use the  invariance properties    proved in the following three lemmas:
\begin{lemma}\label{59j}
 Let $\mathcal{T}:V\rightarrow W$ be an isomorphism where $V,W$ be inner product spaces in $\mathbb{R}^n$. Let $\mathbb{P}=\mathbb{P}(\Pi)$ with $\Pi=\{\pi_{\mathfrak{q}_1,r_1},\cdots,
 \pi_{\mathfrak{q}_N,r_N}\}$ be a polyhedron in $V$.
Then
\begin{itemize}
\item[(1)] $\mathcal{T}(\mathbb{P})$ is a
polyhedron $\mathbb{P}(\Pi_{\mathcal{T}})$
 with
$
  \Pi_{\mathcal{T}}=\{\pi_{(\mathcal{T}^{-1})^{t}(\mathfrak{q}_1),r_1},\cdots,
  \pi_{(\mathcal{T}^{-1})^{t}(\mathfrak{q}_N),r_N}\}.
  $
\item[(2)]  If $\mathbb{F}\in \mathcal{F}^k(\mathbb{P})$, then $ \mathcal{T}(\mathbb{F})\in
\mathcal{F}^k(\mathcal{T}(\mathbb{P}))$ for all $k\ge 0$.
\item[(3)]
 $(\mathcal{T}(\mathbb{F})^*)^{\circ}|( \mathcal{T}(\mathbb{P}),W)
=(\mathcal{T}^{-1})^{t}\left((\mathbb{F}^*)^{\circ}| (\mathbb{P},V)
\right)$ where $\mathcal{T}^t$ denotes a transpose of $\mathcal{T}$.
\item[(4)] For any set $B\subset V$, we have $\mathcal{T}(\rm{Ch}(B))=\rm{Ch}(\mathcal{T}(B)).$
\end{itemize}
\end{lemma}
\begin{proof}
Our proof is based on
\begin{equation}
\langle  (\mathcal{T}^{-1})^{t}\mathfrak{q}, \mathcal{T}({\bf
x})\rangle =\langle \mathfrak{q}, \mathcal{T}^{-1}\mathcal{T}({\bf
x})\rangle =\langle\mathfrak{q}, {\bf x}\rangle \ \ \text{ for
$\mathfrak{q},{\bf x}\in V$.}\label{gw4}
\end{equation}
  By (\ref{gw4}),
$$\mathcal{T}(\pi_{\mathfrak{q}_j,r_j})=\left\{\mathcal{T}({\bf x}): \langle\mathfrak{q}_j, {\bf x}\rangle=r_j\right\}
=\pi_{(\mathcal{T}^{-1})^{t}(\mathfrak{q}_j),r_j}\ \ \text{and}\ \
\mathcal{T}(\pi_{\mathfrak{q}_j,r_j}^{+})=\pi_{(\mathcal{T}^{-1})^{t}(\mathfrak{q}_j),r_j}^+.$$
Thus  $\mathcal{T}(\mathbb{P})=\mathbb{P}(\Pi_{\mathcal{T}})$ is  a
polyhedron because of Definition \ref{d12} and
$$ \mathcal{T}(\mathbb{P})=\mathcal{T}\left(\bigcap\pi_{\mathfrak{q}_j,r_j}^{+}\right)=
\bigcap \mathcal{T}(\pi_{\mathfrak{q}_j,r_j}^{+})=\bigcap
\pi_{(\mathcal{T}^{-1})^{t}(\mathfrak{q}_j),r_j}^{+}. $$ Hence (1)
is proved. If $\mathbb{F}\in \mathcal{F}(\mathbb{P})$, by
(\ref{4g}), $\mathbb{F}=\pi_{\mathfrak{q}_j,r_j}\cap \mathbb{P}$ and
$\mathbb{P}\setminus \mathbb{F}\subset
(\pi_{\mathfrak{q}_j,r_j}^+)^{\circ}$. So,
\begin{eqnarray*}
\mathcal{T}(\mathbb{F})=\pi_{(\mathcal{T}^{-1})^{t}(\mathfrak{q}_j),r_j}\cap
\mathcal{T}(\mathbb{P}) \ \ \text{and}\ \
\mathcal{T}(\mathbb{P})\setminus
\mathcal{T}(\mathbb{F})=\mathcal{T}(\mathbb{P}\setminus\mathbb{F}
)\subset
(\pi_{(\mathcal{T}^{-1})^{t}(\mathfrak{q}_j),r_j}^+)^{\circ}.
\label{se3}
\end{eqnarray*}
This means  $\mathcal{T}(\mathbb{F})\in
\mathcal{F}(\mathcal{T}(\mathbb{P}))$. Moreover $V(\mathbb{F})$ and
$V(\mathcal{T}(\mathbb{F}))$ are isomorphic. Hence
  $\mathcal{T}(\mathbb{F})\in \mathcal{F}^k(\mathcal{T}(\mathbb{P}))$. So (2)  is proved.
Next, (\ref{gw4}) yields that
 \begin{eqnarray*}
&&  (\mathcal{T}(\mathbb{F})^*)^{\circ}| (\mathcal{T}(\mathbb{P}),W)\\
&& =\left\{\mathfrak{q}\in W: \exists\, \rho  \ \text{such that}\
\langle\mathfrak{q},\mathcal{T}({\bf u})\rangle=\rho
<\mathfrak{q}\cdot \mathcal{T}({\bf y})\ \text{for all ${\bf u}\in
\mathbb{F}$,\ ${\bf y}\in
\mathbb{P} \setminus \mathbb{F}$}\right\}\\
&& =\{ (\mathcal{T}^{-1})^{t}(\mathfrak{p})\ : \exists\, \rho\
\text{such that}\ \langle\mathfrak{p}, {\bf u}\rangle =\rho< \langle
\mathfrak{p}, {\bf y}\rangle \ \text{for all ${\bf u}\in
\mathbb{F}$,\ ${\bf y}\in
\mathbb{P} \setminus \mathbb{F}$}\}\\
&&= (\mathcal{T}^{-1})^{t}
\left((\mathbb{F}^*)^{\circ}|(\mathbb{P},V)\right).
\end{eqnarray*}
This proves (3). Finally,
\begin{eqnarray*}
T(\rm{Ch}(B))&=&\left\{T(\sum_{j=1}^Nc_j{\bf x}_j):{\bf x}_j\in B\ \text{and}\ \sum_{j=1}^N c_j=1\ \text{with} \ c_j\ge 0\right\}\\
&=& \left\{ \sum_{j=1}^Nc_jT({\bf x}_j):T({\bf x}_j)\in T(B)\ \text{and}\ \sum_{j=1}^N c_j=1\ \text{with} \ c_j\ge 0\right\}\\
&=&\rm{Ch}(T(B))
\end{eqnarray*}
which proves (4).
\end{proof}
\begin{lemma}\label{lem1277}
Let $X=\mathbb{R}^m$ with $S_0\subset \{1,\cdots,m\}$.  Then,
\begin{eqnarray}
&& {\bf N}(P_X(\mathbb{F}_\nu\cap\Lambda_\nu), S_0) =P_X\left({\bf
N}(\mathbb{F}_\nu\cap\Lambda_\nu, S_0)\right),\label{pjkm}\\
&& \mathbb{K}_\nu\in \mathcal{F}({\bf
N}(P_X(\mathbb{F}_\nu\cap\Lambda_\nu), S_0 )) \ \text{if and only
if}\
 P_X^{-1}(\mathbb{K}_\nu)\in\mathcal{F}({\bf N}(
\mathbb{F}_\nu\cap\Lambda_\nu), S_0)),\label{pjkm22}\\
&&    \left( P_X^{-1}(\mathbb{K}_\nu)^{*}\right)^{\circ}|  {\bf
 N}( \mathbb{F}_\nu\cap\Lambda_\nu ,S_0)= (\mathbb{K}_\nu^*)^{\circ}|{\bf
 N}(P_X(\mathbb{F}_\nu\cap\Lambda_\nu),S_0)). \label{seg}
\end{eqnarray}
\end{lemma}
 \begin{proof}
 By (4) of Lemma \ref{59j}  with Definition \ref{ded28} and
 $P_X(B+\mathbb{R}_+^{S_0})= P_X(B)+\mathbb{R}_+^{S_0}$,
 \begin{eqnarray*}
P_X\left({\bf N}(\mathbb{F}_\nu\cap\Lambda_\nu,
S_0)\right)&=&P_X\left(\rm{Ch}\left((\mathbb{F}_\nu\cap\Lambda_\nu)+
 \mathbb{R}_+^{S_0}\right)\right)\\
&=& \rm{Ch}\left(P_X \left((\mathbb{F}_\nu\cap\Lambda_\nu)+
 \mathbb{R}_+^{S_0}\right)\right)\\
 &=&\rm{Ch}\left(P_X \left( \mathbb{F}_\nu\cap\Lambda_\nu \right) +  \mathbb{R}_+^{S_0}
 \right)\\
 &=&{\bf N}\left( P_X \left( \mathbb{F}_\nu\cap\Lambda_\nu \right),S_0 \right)
 \end{eqnarray*}
which yields (\ref{pjkm}). Next (\ref{pjkm22}) follows from
(\ref{pjkm}) and (2) of Lemma \ref{59j}. Lastly, by (\ref{pjkm}) and
(3) of Lemma \ref{59j},
\begin{eqnarray*}
\left( P_X^{-1}(\mathbb{K}_\nu)^{*}\right)^{\circ}|  {\bf
 N}( \mathbb{F}_\nu\cap\Lambda_\nu ,S_0)&=&
 \left( P_X^{-1}(\mathbb{K}_\nu)^{*}\right)^{\circ}|
  P_X^{-1}\left({\bf N}(P_X(\mathbb{F}_\nu\cap\Lambda_\nu)\right)\\
&=&[(P_X^{-1})^{-1}]^{t}(\mathbb{K}_\nu^*)^{\circ}|{\bf
 N}(P_X(\mathbb{F}_\nu\cap\Lambda_\nu),S_0))\\
 &=& (\mathbb{K}_\nu^*)^{\circ}|{\bf
 N}(P_X(\mathbb{F}_\nu\cap\Lambda_\nu),S_0)),
\end{eqnarray*}
which proves (\ref{seg}).
\end{proof}
We continue the roof of Proposition \ref{lem83}.  Let
$\mathbb{K}_\nu\in \mathcal{F}\left({\bf N}(P_X(
\mathbb{F}_\nu\cap\Lambda_\nu ),S_0)\right)$,
 where $\Omega_\nu=
 P_X( \mathbb{F}_\nu\cap\Lambda_\nu )$ as in (\ref{omee}). By (\ref{pjkm22}) of Lemma \ref{lem1277}, there
exists $$\mathbb{G}_{\nu}=P_X^{-1}(\mathbb{K}_\nu)\in
\mathcal{F}\left({\bf N}( F_\nu\cap\Lambda_\nu
 ,S_0)\right).$$ From
  $\bigcup_{\nu=1}^d \mathbb{G}_\nu=P_X^{-1}(\bigcup_{\nu=1}^d \mathbb{K}_\nu),$
\begin{eqnarray}\label{m11}
\text{rank}(\bigcup_{\nu=1}^d  \mathbb{G}_\nu )=\text{rank} \,\big(
\bigcup_{\nu=1}^d  \mathbb{K}_\nu  \big)\le m-1
\end{eqnarray}
because  $P_X$ is an isomorphism. By (\ref{ds41}) and (\ref{seg}),
\begin{eqnarray}
\bigcap_{\nu=1}^d (\mathbb{G}_\nu^*)^{\circ}| \mathbb{F}_\nu  &=&\bigcap_{\nu=1}^d
 (\mathbb{G}_\nu^*)^{\circ}| {\bf
N}(\Lambda_\nu\cap\mathbb{F}_\nu,S_0) \nonumber \\
&=&  \bigcap_{\nu=1}^d\left( P_X^{-1}(\mathbb{K}_\nu)^{*}\right)^{\circ}|  {\bf
 N}( \mathbb{F}_\nu\cap\Lambda_\nu ,S_0) \label{2727w}\\
 &=& \bigcap_{\nu=1}^d (\mathbb{K}_\nu^*)^{\circ}| {\bf
N}(\Omega_\nu,S_0)   \ne \emptyset.\nonumber
\end{eqnarray}
The last line follows from the second condition  defining
$\mathcal{F}_{\rm{lo}}(\vec{{\bf N}}(\Omega,S_0))$ in Proposition
\ref{lem83}.
 By (\ref{913-}),
\begin{eqnarray}\label{o48}
\bigcap_{\nu=1}^d (\mathbb{F}_{\nu}^{*})^{\circ}| {\bf
N}(\Lambda_\nu,S)  \ne \emptyset.
\end{eqnarray}
By applying Lemma \ref{lem15} together with (\ref{2727w}) and
(\ref{o48}),
\begin{eqnarray}\label{m12}
\bigcap_{\nu=1}^d  (\mathbb{G}_\nu^*)^{\circ}|{\bf
N}(\Lambda_\nu,S)\ne \emptyset.
\end{eqnarray}
By (\ref{m11}),(\ref{m12})  and (\ref{96}),
$$\bigcup_{\nu=1}^d \mathbb{G}_\nu\cap\Lambda_\nu \ \ \text{is an even set having no point of
 $(odd,\cdots,odd)$ in $\Sigma\left(\bigcup_{\nu=1}^d \mathbb{G}_\nu\cap\Lambda_\nu\right)$}.$$
By this together with (\ref{mn12004}),
\begin{eqnarray*}
 &&\Sigma\left(\bigcup_{\nu=1}^d \mathbb{G}_\nu\cap\Lambda_\nu \right)\bigcap\left\{\mathfrak{q}\in
 \Sigma\left(\bigcup_{\nu=1}^d \left(\mathbb{F}_\nu\cap\Lambda_\nu\right)\right):
\mathfrak{q}=(\underbrace{odd,\cdots,odd}_{\text{$m$
components}},*,\cdots,*)\right\} \\
&\subset&\Sigma\left(\bigcup_{\nu=1}^d \mathbb{G}_\nu\cap\Lambda_\nu
\right)\bigcap\left\{\mathfrak{q}\in
 \Sigma\left(\bigcup_{\nu=1}^d \left(\mathbb{F}_\nu\cap\Lambda_\nu\right)\right):
\mathfrak{q}=(\underbrace{odd,\cdots,odd}_{\text{$n$
components}})\right\}=\emptyset.
\end{eqnarray*}
Therefore
$$\Sigma\left(\bigcup\mathbb{K}_\nu\cap\Omega_\nu\right)
=P_X\left(\Sigma\left(\bigcup\mathbb{G}_{\nu}\cap\Lambda_\nu\right)
\right)\ \text{contains no point of
$(\underbrace{odd,\cdots,odd}_{\text{$m$ components}})$.}$$ Hence
$\bigcup (\mathbb{K}_\nu\cap\Omega_{\nu})\ \text{ is an even set}$.
Therefore, the proof of  Proposition \ref{lem83} is finished.
\end{proof}

\subsection{Proof of (\ref{nec23}) and (\ref{nec55g})}
  We   prove the independence (\ref{nec23}) and the non-vanishing property (\ref{nec55g})
to finish the necessity proof for Theorems \ref{main18} through
\ref{main4}.
\begin{proof}[Proof of (\ref{nec23})]
Recall (\ref{nec22})
\begin{eqnarray*}
&& \mathcal{J}(P_{\mathbb{F}},\xi,t_Y) =\lim_{a_X\rightarrow
0,b_X\rightarrow
1_X(S_0)}\mathcal{J}(P_{\mathbb{F}},\xi,t_Y,a_X,b_X),
\end{eqnarray*}
with
\begin{eqnarray}\label{sasa}
\qquad\mathcal{J}(P_{\mathbb{F}},\xi,t_Y,a_X,b_X)=\int_{\prod_{j=1}^m\{a_j<t_j<b_j\}}
\sum_{\sigma\in O} (-1)^{|\sigma|} \exp\left(iP_{\mathbb{F}}(\xi,
\sigma t)\right) \frac{dt_1}{t_1}\cdots\frac{dt_{m}}{t_m}
\end{eqnarray}
where
\begin{eqnarray*}
P_\mathbb{F}(\xi,\sigma_1 t_1, \cdots, \sigma_n
t_n)&=&\sum_{\nu=1}^d\left(\sum_{\mathfrak{q}\in \mathbb{F}_\nu}
c_{\mathfrak{q}}^{\nu}\sigma^{\mathfrak{q}}t^{\mathfrak{q}}\right)\xi_\nu.
\end{eqnarray*}
 By (\ref{49k}), we let
 $S_0=\{1,\cdots,k\} \subset \{1,\cdots,m\}.$  In view of
 (\ref{ma1gg}) and (\ref{49k}),
 choose $\{\mathfrak{q}_1 ,\cdots
 ,\mathfrak{q}_m\}\subset\mathbb{R}^n$
with ${\rm Sp}\left(\mathfrak{q}_1 ,\cdots  ,\mathfrak{q}_m
\right)=
 {\rm Sp}\left(\bigcup \mathbb{F}_\nu \right)$ such that
 \begin{itemize}
\item[(i)] For $i=1,\cdots,m$, $\mathfrak{q}_i=(q_{ij})_{j=1}^n$ and  $(\mathfrak{q}_i)_X={\bf
e}_i\in\mathbb{R}^m$,
\item[(ii)] For  $i=1,\cdots,k$, $\mathfrak{q}_i={\bf
e}_i\in\mathbb{R}^n$.
    \end{itemize}
  Fix $t_Y=(t_{m+1},\cdots,t_n)$ and
use the  change of variables:
\begin{eqnarray}\label{9.11}
x_1=t^{\mathfrak{q}_1},\cdots,x_m=t^{\mathfrak{q}_m}.
\end{eqnarray}
As  $\mathfrak{q} \in
\bigcup_{\nu}(\mathbb{F}_\nu\cap\Lambda_\nu)\subset {\rm
Sp}\left(\bigcup \mathbb{F}_\nu \right)$   is expressed as a linear
combination of $\mathfrak{q}_1,\cdots,\mathfrak{q}_m$, there exists
a vector
$\mathfrak{b}(\mathfrak{q})=(b_{1},\cdots,b_{m})\in\mathbb{R}^m$
such that
\begin{eqnarray*}
t^{\mathfrak{q}}=t^{b_{1}\mathfrak{q}_1+\cdots+b_{m}\mathfrak{q}_m}=x_1^{b_1}\cdots x_m^{b_m}=x^{\mathfrak{b}(\mathfrak{q})}.
\end{eqnarray*}
  This implies that the phase function $P_{\mathbb{F}}(\xi,\sigma
t)$ is written as
\begin{eqnarray}
P_{\mathbb{F}}(\xi,\sigma_1 t_1, \cdots, \sigma_n
t_n)&=&\sum_{\nu=1}^d\left(\sum_{\mathfrak{q}\in
\mathbb{F}_\nu\cap\Lambda_\nu}
c^{\nu}_{\mathfrak{q}}\sigma^{\mathfrak{q}}t^{\mathfrak{q}}\right)\xi_\nu\label{oww}\\
& =& \sum_{\nu=1}^d \left(\sum_{\mathfrak{q}\in
\mathbb{F}_\nu\cap\Lambda_\nu}
c^{\nu}_{\mathfrak{q}}\sigma^{\mathfrak{q}}
x^{\mathfrak{b}(\mathfrak{q})}\right)\xi_\nu=Q_{\mathbb{F}}(\xi,\sigma,
x).\nonumber
\end{eqnarray}
By (i) and (ii) above, the $m\times m$ matrix whose $i$-th row given
by
 $(\mathfrak{q}_i)_X={\bf e}_i\in\mathbb{R}^m$,  is the identity
 matrix:
\begin{eqnarray*}
I= \left(
\begin{matrix}
 q_{11},\cdots,q_{1m} \\
 \vdots \\
 q_{m1},\cdots,q_{mm} \\
\end{matrix}
\right).
\end{eqnarray*}
  Then  compute for fixed
$t_Y=(t_{m+1},\cdots,t_n)$,
\begin{eqnarray*}
\frac{\partial (x_1,\cdots,x_m)}{\partial
(t_1,\cdots,t_m)}=\det\left(
\begin{matrix}
\frac{q_{11} t^{\mathfrak{q}_1}}{t_1},\cdots,\frac{q_{1m} t^{\mathfrak{q}_1}}{t_m}\\
\vdots\\
\frac{q_{m1} t^{\mathfrak{q}_m}}{t_1},\cdots,\frac{q_{mm} t^{\mathfrak{q}_m}}{t_m}\\
\end{matrix}
\right)=\det(I)\,\frac{x_1\cdots x_m}{t_1 \cdots t_m}.
\end{eqnarray*}
   Takeing logarithms on
both sides of (\ref{9.11}),
\begin{eqnarray*}
\left(
\begin{matrix}
 q_{11},\cdots,q_{1n} \\
 \vdots \\
 q_{m1},\cdots,q_{mn} \\
\end{matrix}
\right)\left(
\begin{matrix}
 \log t_1 \\
 \vdots \\
 \log t_n \\
\end{matrix}
\right)= \left(
\begin{matrix}
 \log x_1 \\
 \vdots \\
 \log x_m \\
\end{matrix}
\right) .
\end{eqnarray*}
Then
\begin{eqnarray}\label{mat2}
 \left(
\begin{matrix}
 \log t_1 \\
 \vdots \\
 \log t_m \\
\end{matrix}
\right)= \left(
\begin{matrix}
 \log x_1-\left(q_{1, m+1}\log t_{m+1}+\cdots+q_{1,n}\log t_{n}\right)  \\
 \vdots \\
 \log x_m-\left(q_{m,m+1}\log t_{m+1}+\cdots+q_{m,n}\log t_{n}\right)  \\
\end{matrix}
\right) .
\end{eqnarray}
Solve $(t_1,\cdots,t_m)$ in (\ref{mat2}) in terms of
$(x_1,\cdots,x_m)$ and   $t_Y=(t_{m+1},\cdots,t_n)$,
\begin{eqnarray*}
t_i=\frac{x_i}{t_{m+1}^{ q_{i,m+1} }\cdots t_n^{q_{i,n}}} \ \
\text{for}\ \ i=1,\cdots,m.
\end{eqnarray*}
Note from this together with $\mathfrak{q}_1={\bf e}_1,\cdots,
\mathfrak{q}_k={\bf e}_k$ in (ii) above,
  $$t_1=x_1,\cdots,t_k=x_k, t_i=x_i/(t_{m+1}^{ q_{i,m+1} }\cdots
t_n^{q_{i,n}})\ \ \text{for}\ \ i=k+1,\cdots,m.$$ So the region
$\prod_{j=1}^m\{a_j<t_j<b_j\}$ in (\ref{nec22}) is transformed to
the region
\begin{eqnarray}\label{ojs}
U(a_X,b_X,t_Y)=\prod_{i=1}^k\{a_i<x_i<b_i\}\prod_{i=k+1}^m
\left\{a_i<\frac{x_i}{t_{m+1}^{ q_{i,m+1} }\cdots
t_n^{q_{i,n}}}<b_i\right\}
\end{eqnarray}
Thus as $a_X=(a_i)_{i=1}^m\rightarrow 0_X$
 and $b_X=(b_i)_{i=1}^m\rightarrow 1_X(S_0) =(\underbrace{1,\cdots,1}_{\text{$k$
components}},\underbrace{\infty,\cdots,\infty}_{\text{$m-k$
components}})$,
\begin{eqnarray*}
&&\mathcal{J}(P_{\mathbb{F}},\xi,(t_{m+1},\cdots,t_n))\\
&&\quad=\lim_{a_X\rightarrow 0_X,b_X\rightarrow 1_X(S_0)
 }  \int_{U(a_X,b_X,t_Y)}\sum_{\sigma\in
O} (-1)^{|\sigma|}  \exp\left(iQ_{\mathbb{F}}(\xi,x)\right)
\frac{dx_1}{x_1}\cdots\frac{dx_{m}}{x_m}\,,
\end{eqnarray*}
  is independent of $t_Y\in
\prod_{j=m+1}^n\{a_j<t_j<b_j\}$ since $t_{m+1}^{ q_{i,m+1} }\cdots
t_n^{q_{i,n}}$ is absorbed in the limit of $a_X\rightarrow 0_X$
 and $b_X\rightarrow 1_X(S_0)$ in (\ref{ojs}).
\end{proof}
\begin{proof}[Proof of (\ref{nec55g})]
Since $\mathcal{J}(P_{\mathbb{F}},\xi,(t_{m+1},\cdots,t_n))$ is
independent of $t_Y=(t_{m+1},\cdots,t_n)$,  it suffices to show that
for some choices of  $\xi$ and  coefficients in $P_{\mathbb{F}}$,
\begin{equation}\label{nec5}
\mathcal{J}(P_{\mathbb{F}},\xi)=\mathcal{J}(P_{\mathbb{F}},\xi,{\bf
1}_Y) \ne 0 \ \ \text{where}\ \ {\bf
1}_Y=\underbrace{(1,\cdots,1)}_{\text{$n-m$ components}}.
\end{equation}
 Let $\mathbb{Z}_2=\{0,1\}$ be the additive group   and
 let $\mathbb{Z}_2^n=\left\{(v_1,\cdots,v_n): v_i\in \mathbb{Z}_2 \right\}$.
 Define a function $\Gamma: \mathbb{Z}^n\rightarrow \mathbb{Z}_2^n$
by
 $$\Gamma(q_1,\cdots,q_n) =(\gamma(q_1),\cdots,\gamma(q_n))$$ where
\begin{eqnarray*}
\gamma(q_i)=
\begin{cases}0 \ \text{ if }\ \ q_i \ \ \text{ is an even number,}\\
1 \ \text{ if }\ \ q_i \ \ \text{ is an odd number.}
\end{cases}
\end{eqnarray*}
We put
\begin{eqnarray*}
\Gamma\left(\bigcup_{\nu=1}^d(\mathbb{F}_\nu\cap
\Lambda_{\nu})\right)
 &=&\{{\bf z}_1,\cdots,{\bf z}_{L}\}\subset \mathbb{Z}_2^n.
\end{eqnarray*}
For  $\ell=1,\cdots,L,$ let
\begin{eqnarray*}
\Gamma^{-1}\{{\bf z}_\ell\}&=&\left\{\mathfrak{q}\in
\bigcup_{\nu=1}^d(\mathbb{F}_\nu\cap \Lambda_{\nu}):
 \Gamma(\mathfrak{q})={\bf z}_\ell\right\}.
\end{eqnarray*}
We then have
\begin{eqnarray}\label{m11g}
\bigcup_{\nu=1}^d(\mathbb{F}_\nu\cap
\Lambda_{\nu})=\bigcup_{\ell=1}^{L}\Gamma^{-1}\{{\bf z}_\ell\}.
\end{eqnarray}
Since  $\bigcup_{\nu=1}^d(\mathbb{F}_\nu\cap \Lambda_{\nu})$  is an
odd set,
 there exist ${\bf z}_{1},\cdots,{\bf z}_{s}\in
\Gamma\left(\bigcup_{\nu=1}^d(F_\nu\cap \Lambda_{\nu})\right)$ such
that
\begin{eqnarray}\label{9.15}
{\bf z}_{1}\oplus\cdots \oplus {\bf z}_{s}=(1,\cdots,1)\ \
\text{in}\ \mathbb{Z}_2^n.
\end{eqnarray}
 Assume the contrary to (\ref{nec5}). Then from (\ref{nec22}) and (\ref{nec23}), for all $  \xi_1,\cdots,\xi_d$ and all choices of
coefficients $ c_{\mathfrak{q}}^{\nu}\in\mathbb{R}\setminus\{0\}$,
we have in (\ref{sasa}),
  \begin{eqnarray}
 \mathcal{J}(P_{\mathbb{F}},\xi,{\bf 1}_Y)&=&\lim_{a_X\rightarrow
0,b_X\rightarrow
1_X(S_0)}\mathcal{J}(P_{\mathbb{F}},\xi,{\bf 1}_Y,a_X,b_X)\nonumber\\
&=&
 \lim_{a_X\rightarrow
0,b_X\rightarrow 1_X(S_0)}\int_{\prod_{j=1}^m\{a_j<t_j<b_j\}}
\sum_{\sigma\in O} (-1)^{|\sigma|} \exp\left(iP_{\mathbb{F}}(\xi,
\sigma t)\right)
\frac{dt_1}{t_1}\cdots\frac{dt_{m}}{t_m}\label{914}\\
&=&0.\nonumber
\end{eqnarray}
In view of (\ref{oww}),
\begin{eqnarray}\label{9.9}
\qquad P_{\mathbb{F}}(\xi,\sigma t)&=& \sum_{\nu=1}^d
\left(\sum_{\mathfrak{q}\in
 \mathbb{F}_\nu\cap\Lambda_\nu }
c^{\nu}_{\mathfrak{q}}\sigma^{\mathfrak{q}}t^{\mathfrak{q}}\right)\xi_\nu\\
&=&
 \sum_{\mathfrak{q}\in
\bigcup_{\nu=1}^d(\mathbb{F}_\nu\cap\Lambda_\nu)}
\xi_{\nu(\mathfrak{q})}c^{\nu(\mathfrak{q})}_{\mathfrak{q}}\sigma^{\mathfrak{q}}t^{\mathfrak{q}}
\ \  \text{with}\ \  t=(t_1,\cdots,t_m,{\bf 1}_Y).\nonumber
\end{eqnarray}
Rearrange monomials $t^{\mathfrak{q} }$ in (\ref{9.9}) by using
(\ref{m11g}) and reset their coefficients  so that \begin{itemize}
\item[(1)]
$\xi_{\nu(\mathfrak{q})}c^{\nu(\mathfrak{q})}_{\mathfrak{q}}=\zeta_\ell$
if $\mathfrak{q}\in\Gamma^{-1}\{{\bf z}_\ell\}$ for each
$\ell=1,\cdots,s$\,,
\item[(2)]
$\xi_{\nu(\mathfrak{q})}c^{\nu(\mathfrak{q})}_{\mathfrak{q}}=\zeta_{s+1}$
if $\mathfrak{q}\in
E=\bigcup_{\nu=1}^d(\mathbb{F}_\nu\cap\Lambda_\nu)\setminus
\bigcup_{\ell=1}^s \Gamma^{-1}\{{\bf z}_\ell\}$,
\end{itemize}
which is possible because (\ref{914}) holds for all $\xi$ and all
coefficients $c^{\nu}_{\mathfrak{q}}\in\mathbb{R}\setminus\{0\}$.
Then,
\begin{eqnarray}\label{1616m}
P_{\mathbb{F}}(\xi,\sigma t)&=&
\sum_{\ell=1}^{L}\sum_{\mathfrak{q}\in \Gamma^{-1}\{{\bf z}_\ell\} }
\xi_{\nu(\mathfrak{q})}c^{\nu(\mathfrak{q})}_{\mathfrak{q}}\sigma^{\mathfrak{q}}t^{\mathfrak{q}} \nonumber\\
&=&\zeta_1\sum_{\mathfrak{q}\in
\Gamma^{-1}\{{\bf
z}_1\}}\sigma^{\mathfrak{q}}t^{\mathfrak{q}}+\cdots+\zeta_{s}\sum_{\mathfrak{q}\in
\Gamma^{-1}\{{\bf
z}_s\}}\sigma^{\mathfrak{q}}t^{\mathfrak{q}}+\zeta_{s+1}\sum_{\mathfrak{q}\in
E} \sigma^{\mathfrak{q}}t^{\mathfrak{q}}  \\
&=&\sum_{\ell=1}^{s+1} Q_\ell(\sigma,t)\zeta_\ell. \nonumber
\end{eqnarray}
 Rewrite   $\mathcal{J}(P_{\mathbb{F}},\xi,{\bf
1}_Y,a_X,b_X)$ in (\ref{914}) as
\begin{eqnarray*}
\mathcal{L}(\zeta, a_X,b_X) &=&
\int_{\prod_{j=1}^m\{a_j<t_j<b_j\}}\sum_{\sigma\in O}
(-1)^{|\sigma|} \exp\left(i\sum_{\ell=1}^{s+1}
Q_\ell(\sigma,t)\zeta_\ell\right)
\frac{dt_1}{t_1}\cdots\frac{dt_{m}}{t_m}.
\end{eqnarray*}
 Then   we
see in view of (\ref{914}) that for all  $
\zeta=(\zeta_1,\cdots,\zeta_{s+1})\in\mathbb{R}^{s+1}$
\begin{eqnarray}\label{asw}
\mathcal{L}(\zeta)=\lim_{a_X\rightarrow 0,b_X\rightarrow
1_X(S_0)}\mathcal{L}(\zeta, a_X,b_X)\,=\,0.
\end{eqnarray}
On the other hand by Lemma \ref{lem83} and Theorem \ref{th60} with
Remark \ref{rem91},
\begin{equation*}
 \sup_{\xi,a_X,b_X\in I_X(S_0)}\left|\mathcal{J}(P_{\mathbb{F}},\xi,t_Y,a_X,b_X)\right|\le
C_R\prod_{\nu}\prod_{\mathfrak{q}\in\Lambda_\nu}
(|D_{\mathfrak{q}}c^\nu_{\mathfrak{q}}
|+1/|d_{\mathfrak{q}}c^\nu_{\mathfrak{q}} |)^{1/R}.
\end{equation*}
Here $D_{\mathfrak{q}}=\max\{\prod_{j=m+1}^n
|t_{j}^{q_j}|:a_{j}<|t_j|<b_j\}$ and
$d_{\mathfrak{q}}=\min\{\prod_{j=m+1}^n
|t_{j}^{q_j}|:a_{j}<|t_j|<b_j\}$.
 Thus, by simply plug $\xi_{\nu(\mathfrak{q})}=1$ where
  $\xi_{\nu(\mathfrak{q})}c^{\nu(\mathfrak{q})}_{\mathfrak{q}}=\zeta_\ell$ in (\ref{1616m}),
 \begin{eqnarray}\label{kff}
\quad \sup_{ a_X,b_X\in I_X(S_0)}\left|\mathcal{L}(\zeta,
a_X,b_X)\right|\le C_R\prod_{\ell=1}^{s+1}
(|\zeta_\ell|+1/|\zeta_\ell|)^{1/M}\ \text{for some large $M>0$}.
\end{eqnarray}

We now find a contradiction to (\ref{asw}). Let $f$ be a Schwartz
function on $\mathbb{R}^{s+1}$ of the form
$\widehat{f}(\zeta)=\prod_{\ell=1}^{s+1}
\widehat{f}_\ell(\zeta_\ell)$ with $f_\ell$ a Schwartz function on
$\mathbb{R}$. Then   from (\ref{kff}),
$$\sup_{ a_X,b_X\in I_X(S_0)}|\mathcal{L}(\zeta, a_X,b_X)\widehat{f}(\zeta)|\le
C_R\prod_{\ell=1}^{s+1}
(|\zeta_\ell|+1/|\zeta_\ell|)^{M}|\widehat{f}(\zeta)|$$ which is an
integrable function on $\mathbb{R}^{s+1}$. This enables us to use
the dominated convergence theorem for (\ref{asw}) multiplied by
$\widehat{f}(\zeta)$ to obtain that
\begin{eqnarray}\label{kff1}
\quad 0=\int_{\mathbb{R}^{s+1}}
\mathcal{L}(\zeta)\widehat{f}(\zeta)d\zeta&=&\lim_{a_X\rightarrow
0,b_X\rightarrow 1_X(S_0)}\int_{\mathbb{R}^{s+1}}\mathcal{L}(\zeta,
a_X,b_X)\widehat{f}(\zeta)d\zeta.
\end{eqnarray}
Rewrite the integral of the righthand side as follows. Interchange,
by Fubini's theorem, the order of integration and apply the Fourier
inversion formula for the Schwartz functions:
\begin{eqnarray*}
&&\int_{\mathbb{R}^{s+1}}\mathcal{L}(\zeta,
a_X,b_X)\widehat{f}(\zeta)d\zeta\\
&&=
 \int_{\mathbb{R}^{s+1}}\int_{\prod_{j=1}^m\{a_j<t_j<b_j\}}\sum_{\sigma\in
O} (-1)^{|\sigma|} \exp\left(i\sum_{\ell=1}^{s+1}
Q_\ell(\sigma,t)\zeta_\ell\right)\widehat{f}(\zeta)
\frac{dt_1}{t_1}\cdots\frac{dt_{m}}{t_m}d\zeta\nonumber\\
&& =
\int_{\prod_{j=1}^m\{a_j<t_j<b_j\}}\left[\int_{\mathbb{R}^{s+1}}\sum_{\sigma\in
O} (-1)^{|\sigma|} \exp\left(i\sum_{\ell=1}^{s+1}
Q_\ell(\sigma,t)\zeta_\ell\right)\widehat{f}(\zeta)d\zeta\right]
\frac{dt_1}{t_1}\cdots\frac{dt_{m}}{t_m}\\
&& = \int_{\prod_{j=1}^m\{a_j<t_j<b_j\}}\sum_{\sigma\in O}
(-1)^{|\sigma|} \prod_{\ell=1}^{s+1}
f_\ell\big(Q_\ell(\sigma,t)\big)
\frac{dt_1}{t_1}\cdots\frac{dt_{m}}{t_m}\,,\nonumber
\end{eqnarray*}
where $ Q_\ell(\sigma,t)$ is defined in   (\ref{1616m}).  Here we
choose $f_1,\cdots, f_s$ to be odd Schwartz functions on
$\mathbb{R}$ that are positive on $[0,\infty)$, and
$f_{s+1}(x)=e^{-x^2}$. By using
$\sigma^{\mathfrak{q}}=\sigma^{\Gamma(\mathfrak{q})}=\sigma^{{\bf
z}_\ell}$ for all $\mathfrak{q}\in \Gamma^{-1}({\bf z}_\ell)$ and
oddness of $f_\ell$,
\begin{eqnarray*}
 \sum_{\sigma\in O}(-1)^{|\sigma|} \prod_{\ell=1}^{s+1}
f_\ell\big(Q_\ell(\sigma,t)\big)  & =& \sum_{\sigma\in O}
(-1)^{|\sigma|} \prod_{\ell=1}^{s}
f_\ell(\sigma^{\mathfrak{q}}\sum_{\mathfrak{q}\in \Gamma^{-1}({\bf
z}_\ell)}t^{\mathfrak{q}})\exp\big(-|\sum_{\mathfrak{q}\in
E } \sigma^{\mathfrak{q}}t^{\mathfrak{q}}|^2\big)\\
&=&\sum_{\sigma\in O} (-1)^{|\sigma|} \sigma^{{\bf z}_1+\cdots+{\bf
z}_s}\prod_{\ell=1}^{s} f_\ell(\sum_{\mathfrak{q}\in
\Gamma^{-1}({\bf
z}_\ell)}t^{\mathfrak{q}})\exp\big(-|\sum_{\mathfrak{q}\in
E } \sigma^{\mathfrak{q}}t^{\mathfrak{q}}|^2\big)\\
&=&\sum_{\sigma\in O}\prod_{\ell=1}^{s} f_\ell(\sum_{\mathfrak{q}\in
\Gamma^{-1}({\bf
z}_\ell)}t^{\mathfrak{q}})\exp\big(-|\sum_{\mathfrak{q}\in E }
\sigma^{\mathfrak{q}}t^{\mathfrak{q}}|^2\big)>0\,,
\end{eqnarray*}
where the last equality follows from  (\ref{9.15}) and
 $$
\sigma^{{\bf z}_1+\cdots+{\bf z}_s}=\sigma^{{\bf
z}_1\oplus\cdots\oplus{\bf z}_s}=(-1)^{|\sigma|}.$$ So, the limit in
(\ref{kff1}) is positive, which is a contradiction.  Hence
(\ref{nec5}) is proved.
\end{proof}

\section{Proofs of Corollary \ref{ccoo} and Main Theorem \ref{main4}}\label{sec15}
\subsection{Proof of Corollary \ref{ccoo}}
\begin{proof}[Proof of Corollary \ref{ccoo}]
{\bf Sufficiency}. Suppose that
\begin{eqnarray}
 \text{
 $(\mathbb{F}_{n+1} \cap\Lambda_{n+1})\cup A$
is an even set whenever $\text{rank}(\mathbb{F}_{n+1} \cup A)\le
n-1$} \label{ysk}
 \end{eqnarray}
where $\mathbb{F}_{n+1}\in\mathcal{F}({\bf N}(\Lambda_{n+1},S))$ and
$A\subset\{{\bf e}_1,\cdots,{\bf e}_n\}.$  It suffices to deduce from (\ref{ysk})  that the hypothesis of Main
 Theorem \ref{main3} holds for the case $\Lambda=(\{{\bf
e}_1\},\cdots,\{{\bf e}_n\},\Lambda_{n+1})
 $, since we have already proved Main
 Theorem \ref{main3}.
  Let  $\text{rank}\left(
\bigcup_{\nu=1}^{n+1}\mathbb{F}_{\nu} \right)\le n-1 $ and
$\bigcap_{\nu=1}^{n+1}(\mathbb{F}_\nu^*)^{\circ}\ne\emptyset$. We
claim that $\bigcup_{\nu=1}^{n+1}(\mathbb{F}_{\nu}\cap\Lambda_\nu)$
is an even set.  Observe  that for every nonempty face
$\mathbb{F}_\nu\in \mathcal{F}({\bf N}(\{{\bf e}_\nu\},S))$,\,
$\mathbb{F}_{\nu}\cap\Lambda_\nu=\{{\bf e}_\nu\}$. Thus for
$A=\{{\bf e}_\nu:\mathbb{F}_\nu\ne \emptyset\ \  \text{for}\
\nu=1,\cdots,n\}$, we write
$$
 \bigcup_{\nu=1}^{n+1}\mathbb{F}_{\nu}\cap\Lambda_\nu = \left(
\mathbb{F}_{n+1}\cap\Lambda_{n+1}\right)  \cup A.
$$  By
(\ref{ysk}), $\bigcup_{\nu=1}^{n+1}\mathbb{F}_{\nu}\cap\Lambda_\nu$
is an even set.\\
 {\bf Necessity}.
 Suppose that (\ref{ysk}) does not hold.
  Then there exists $A=\{{\bf e}_{\nu}:\nu\in I\}\subset\{{\bf e}_1,\cdots,{\bf e}_n\}$   and $\mathbb{F}_{n+1}$ such that
$$\text{rank}\left(A\cup
\mathbb{F}_{n+1} \right)\le n-1\ \text{and}\ A\cup (\mathbb{F}_{n+1}
\cap\Lambda_{n+1}) \ \text{is odd.}$$  Let
$\text{Sp}(\mathbb{F}_{n+1})\cap \{{\bf e}_\nu:\nu\in S\}=\{{\bf
e}_{\nu_1},\cdots,{\bf e}_{\nu_k}\}$ where  $\{\nu_1,\cdots,\nu_k\}=S_1\subset
S$. Choose
\begin{itemize}
\item  For $\nu\in N_n\setminus I$, let $\mathbb{F}_\nu=\emptyset$ with
$(\mathbb{F}_\nu^*)^{\circ}=Z(S)\setminus\{0\}$.
\item For $\nu \in
I$, let $\mathbb{F}_\nu=\{{\bf e}_\nu\}+\mathbb{R}^{S_1}$ with
$$(\mathbb{F}_{\nu}^*)^{\circ}=\text{CoSp}^{\circ}(\{{\bf e}_j:j\in
S\setminus S_1\}\cup\{\pm{\bf e}_j:j\in N_n\setminus S\}).$$
\end{itemize}
Then we can observe that
$(\mathbb{F}_{n+1}^*)^{\circ}\subset(\mathbb{F}_{\nu}^*)^{\circ}$
for all $\nu=1,\cdots,n$. Therefore,
$$\text{rank}\left(\bigcup_{\nu=1}^{n+1}
 \mathbb{F}_\nu\right)=\text{rank}\left(A\cup
\mathbb{F}_{n+1} \right)  \le n-1\ \text{and}  \
\bigcap_{\nu=1}^{n+1}
\left(\mathbb{F}^*_{\nu}\right)^{\circ}\ne\emptyset, $$ but
$\bigcup_{\nu=1}^{n+1} \mathbb{F}_\nu\cap\Lambda_\nu=A\cup
(\mathbb{F}_{n+1} \cap\Lambda_{n+1}) $  is an odd set, which implies
that
  the hypothesis of Main Theorem \ref{main3}
breaks. Let $
\mathfrak{q}=(q_1,\cdots,q_n)\in\bigcap_{\nu=1}^d(\mathbb{F}_\nu^*)^{\circ}$ with $q_j=0$
for
   $j\in S_0\subset S$. We follow the same argument for the
   necessity proof in Section \ref{sec10}. Then  we  obtain
(\ref{921}) so that there exists  $P_\Lambda \in\mathcal{P}_\Lambda$
such that
\begin{eqnarray*}
  \left\|
  \mathcal{H}^{P_{\mathbb{F}}}_{1_{S_0}}\right\|_{L^2(\mathbb{R}^d)\rightarrow
  L^2(\mathbb{R}^d)}=\infty.
  \end{eqnarray*}
This implies   $
   \left\|\mathcal{H}^{P}_{1_S}
   \right\|_{L^2(\mathbb{R}^d)}=\infty$ by the following standard
   argument:
   For $\delta>0$, define a dilation
   $$f_{\delta}(x_1,\cdots,x_n,x_{n+1})=f(\delta^{-q_1}x_1,\cdots,\delta^{-q_n}x_n,\delta^{-d}x_{n+1})$$  and a
   measure
   \begin{eqnarray*}
   \mu_{\delta}^S(\phi)&=&\int_{I(S)}
   \phi(\delta^{-q_1}t_1,\cdots,\delta^{-q_n}t_n,\delta^{-d}P(t_1,\cdots,t_n))\frac{dt_1}{t_1}\cdots\frac{dt_n}{t_n}
   \end{eqnarray*}
   satisfying $\mathcal{H}^{P}_{1_S}(f)=[\mu_{\delta}^S*f_{\delta^{-1}}]_{\delta}.$
   By using
   \begin{eqnarray*}
  \lim_{\delta\rightarrow 0}\mu_{\delta}^S(\phi)= \int_{I(S_0)}
   \phi( t_1,\cdots, t_n,
   P_\mathbb{F}(t_1,\cdots,t_n))\frac{dt_1}{t_1}\cdots\frac{dt_n}{t_n},
   \end{eqnarray*}
   we conclude that the boundedness of
   $\left\|
  \mathcal{H}^{P}_{1_{S}}\right\|_{L^2(\mathbb{R}^d)\rightarrow
  L^2(\mathbb{R}^d)}$  implies the boundedness of  $\|
  \mathcal{H}^{P_{\mathbb{F}}}_{1_{S_0}} \|_{L^2(\mathbb{R}^d)\rightarrow
  L^2(\mathbb{R}^d)}$.
\end{proof}
\subsection{Proof of Main Theorem \ref{main4}}
We now develop the argument of \cite{MY} for $n\ge 3$, and
 obtain  Main Theorem \ref{main4}.
\begin{definition}
Let $P \in\mathcal{P}_\Lambda$ where $\Lambda=(\Lambda_\nu)$ with
$\Lambda_\nu\subset \mathbb{Z}_+^n$ and  $S\subset \{1,\cdots,n\} $.
Let $A\in GL(d)$.
 We set the collection of all $\vec{{\bf
 N}}(AP,S)$ with $A\in GL(d)$ in  Definition \ref{de0303},
 $$\mathcal{N}(P,S)=\{\vec{{\bf N}}(AP,S): A\in GL(d)\}.$$
Consider the
 collection of
equivalent classes  $\mathcal{A}(P)=\{[A]: A\in GL(d)\}$ with the
equivalence relation for $A,B\in GL(d)$, $$\text{$A \sim B$ if and
only if $\Lambda(AP)=\Lambda(BP)$.}$$ We see that $\mathcal{A}(P)$
is finite and
 write $\mathcal{A}(P)$ as $\{ [A_k]: k=1,\cdots,N\ \}.$
  Thus we can regard $\mathcal{N}(P,S)$ as the ordered $N$-tuples of
$\vec{{\bf N}}(A_kP,S)$
  (indeed, $Nd$ tuples  of Newton Polyhedrons ${\bf N}((A_kP)_\nu,S)$):
  $$\mathcal{N}(P,S)=\left(\vec{{\bf N}}(AP,S)
\right)_{[A]\in\mathcal{A}(P)}=\left( \vec{{\bf N}}(A_kP,S)
\right)_{k=1}^N.$$ So, we define the class of all  combinations of
$Nd$-tuples of faces by
\begin{eqnarray} \label{dqe}
\mathcal{F}(\mathcal{N}(P,S))&=&\left\{(\mathbb{F}_{[A]})_{[A]\in\mathcal{A}(P)}:\mathbb{F}_{[A]}\in
\mathcal{F}\left(\vec{{\bf N}}(AP,S)\right) \right\}\\
&=&\left\{(\mathbb{F}_{A_k})_{k=1}^N:\mathbb{F}_{A_k}\in
\mathcal{F}\left(\vec{{\bf N}}(A_kP,S)\right) \right\},\nonumber
\end{eqnarray}
 where
$\mathbb{F}_{A_k}=((\mathbb{F}_{A_k})_1,\cdots,(\mathbb{F}_{A_k})_d)$
with $ (\mathbb{F}_{A_k})_\nu\in \mathcal{F}\left({\bf
N}((A_kP)_\nu,S)  \right).$
\end{definition}
To prove Main Theorem \ref{main4}, we apply the
Proposition \ref{74sj} for every $\vec{{\bf N}}(A_kP,S)$ with $k=1,\cdots,N$ to obtain the following general form of cone decomposition.
\begin{lemma}\label{lemdal}
Let $P \in\mathcal{P}_\Lambda$ where $\Lambda=(\Lambda_\nu)$ with
$\Lambda_\nu\subset \mathbb{Z}_+^n$ and  $S\subset \{1,\cdots,n\} $.
Then,
$$
\bigcup_{(\mathbb{F}_{[A]})_{[A]\in\mathcal{A}(P)}\in\mathcal{F}(\mathcal{N}(P,S))}
 \left( \bigcap_{[A]\in\mathcal{A}(P)}  \rm{Cap}(\mathbb{F}_{[A]}^*)
 \right)=Z(S).$$
Given $\Lambda$,  there are
finitely many  Newton polyhedrons in
$ \{\vec{{\bf N}}(AP,S): A\in GL(d), P\in
\mathcal{P}_\Lambda\}$.
\end{lemma}
\begin{proof}
By (\ref{dqe}), the left hand side above   is
$$
\bigcup_{(\mathbb{F}_{A_k})_{k=1}^N\in\mathcal{F}(\mathcal{N}(P,S))}
 \left( \bigcap_{k=1}^N \bigcap_{\nu=1}^d(\mathbb{F}_{A_k})^*_\nu \right)=\bigcap_{k=1}^N\bigcup_{
\mathbb{F}_{A_k}\in \mathcal{F}\left( \vec{{\bf
N}} \left(A_kP,S\right)\right)}
\bigcap_{\nu=1}^d(\mathbb{F}_{A_k})^*_\nu.$$ For each fixed $A_k$, Proposition \ref{74sj} yields that
$$\bigcup_{ \mathbb{F}_{A_k}\in \mathcal{F}\left( {\bf
N}\left(A_kP,S\right)\right)}
\bigcap_{\nu=1}^d(\mathbb{F}_{A_k})^*_\nu   =   Z(S),$$ which proves
Lemma \ref{lemdal}.
\end{proof}
 To each $[A]\in\mathcal{A}(P)$, we first assign  a $d$-tuple of faces $\mathbb{F}_{A}\in \mathcal{F}\left( \vec{{\bf
N}}\left(AP,S\right)\right)$. Next, fix
 \begin{equation}\label{fc1}
\left(\mathbb{F}_{A}\right)_{[A]\in \mathcal{A}(P)}.
\end{equation}
To show Main Theorem \ref{main4}, in view of Lemma \ref{lemdal}, it suffices to show that
\begin{eqnarray}
\quad \left\|\sum_{J \in Z}H_J^{P_\Lambda}\right\|\le C\   \text{where $Z =\bigcap_{[A]\in\mathcal{A}(P)} {\rm Cap}(
\mathbb{F}^*_A)$ with $\mathbb{F}_A$ chosen in (\ref{fc1})}.\label{3p3p}
\end{eqnarray}
To show (\ref{3p3p}), we can replace $H_J^{P_\Lambda}$  by $H_J^{UP_\Lambda}$  for some $U\in GL(d)$ and prove that
\begin{eqnarray}
\left\|\sum_{J\in Z  }H_J^{P_\Lambda}\right\|=\left\|\sum_{J\in Z }H_J^{UP_\Lambda}\right\|\le C \label{p5p5}
\end{eqnarray}
where the equality  follows from
 \begin{eqnarray*}
 \int f(x-P(t)) \prod_{\nu=1}^d\frac{\chi(2^{j_\nu}t_\nu)}{t_\nu}dt=
\int f(U^{-1}(Ux-UP(t)) \prod_{\nu=1}^d\frac{\chi(2^{j_\nu}t_\nu)}{t_\nu}dt.
\end{eqnarray*}
Without the disjointness of $\Lambda_\nu$'s, we are lack of
  the decay condition (\ref{nniiii}) in Remark \ref{rem00} and (\ref{nnii}) of Theorem \ref{th60}. In order to recover this,
 we shall modify the proof of \cite{MY} and find an appropriate $U$ to satisfy the desirable  decay estimate
 in Lemma \ref{lemari}.
We work this process for $d=3$. Let $[A_1]\in\mathcal{A}(P)$ with $A_1=I$.
Then
$$ (A_1P)(t) =P(t)=\left( \sum_{\mathfrak{m}\in
\Lambda((A_1P)_\nu)=\Lambda(P_\nu)}c^{\nu}_{\mathfrak{m}}t^{\mathfrak{m}}\right)_{\nu=1}^3.
$$
Take any vector $\mathfrak{m}(A_1,1)\in (\mathbb{F}_{A_1})_1\cap
\Lambda((A_1P)_1)$ where $\mathbb{F}_{A_1}\in\mathcal{F}(\vec{{\bf
N}}(A_1P,S))$ was chosen in (\ref{fc1}) with
$[A_1]\in\mathcal{A}(P)$. Define
\begin{eqnarray*}
A_2=\left(
\begin{matrix}
1&0&0 \\
-\frac{c^2_{\mathfrak{m}(A_1,1)}}{ c^1_{\mathfrak{m}(A_1,1)} } &1&0 \\
-\frac{c^3_{\mathfrak{m}(A_1,1)}}{ c^1_{\mathfrak{m}(A_1,1)} } & 0&1
\end{matrix}
\right)\ \ \text{so that}\ \ A_2A_1P(t)=\left(\begin{matrix}
(A_1P)_1(t)\\
(A_1P)_2(t)-\frac{c^2_{\mathfrak{m}(A_1,1)}}{ c^1_{\mathfrak{m}(A_1,1)} }(A_1P)_1(t) \\
(A_1P)_3(t)-\frac{c^3_{\mathfrak{m}(A_1,1)}}{ c^1_{\mathfrak{m}(A_1,1)} }(A_1P)_1(t)
\end{matrix}\right)
\end{eqnarray*}
where
\begin{itemize}
\item $t^{\mathfrak{m}(A_1,1)}$ does not appear in each of $2^{th}$  and $3^{rd}$
components  of    $ A_2A_1P(t) $.
\end{itemize}
Next choose $\mathfrak{m}(A_2,2)\in
\left(\mathbb{F}_{A_2A_1}\right)_2\cap  \Lambda((A_2A_1P)_2)$ where
$\mathbb{F}_{A_2A_1}\in\mathcal{F}(\vec{{\bf N}}(A_2A_1P,S))$ was chosen
in (\ref{fc1}) with $[A_2A_1]\in\mathcal{A}(P)$. Define a matrix
\begin{eqnarray*}
A_3=\left(
\begin{matrix}
1&0&0 \\
0&1&0  \\
0 & -\frac{c^3_{\mathfrak{m}(A_2,2)}}{ c^2_{\mathfrak{m}(A_2,2)} }&1
\end{matrix}
\right) \text{ so that}\
A_3A_2A_1P(t)=\left(\begin{matrix}
(A_2A_1P)_1(t)=(A_1P)_1(t)\\
(A_2A_1P)_2(t) \\
(A_2A_1P)_3(t)-\frac{c^3_{\mathfrak{m}(A_2,2)}}{ c^2_{\mathfrak{m}(A_2,2)} }(A_2A_1P)_2(t)
\end{matrix}\right)
\end{eqnarray*}
where
\begin{eqnarray}
&&\text{ $t^{\mathfrak{m}(A_1,1)}$ does not appear in each of $2^{th}$ and $3^{rd}$
components   of $ A_3A_2A_1P(t) $},\label{huk1}\\
&&\text{ $t^{\mathfrak{m}(A_2,2)}$ does not appear in the $3^{rd}$   component  of $ A_3A_2A_1P(t) $.}\label{huk2}
\end{eqnarray}
Choose $\mathfrak{m}(A_3,3)\in
\left(\mathbb{F}_{A_3A_2A_1}\right)_3\cap\,
\Lambda((A_3A_2A_1P)_3)$ where
$\mathbb{F}_{A_3A_2A_1}\in\mathcal{F}(\vec{{\bf N}}(A_3A_2A_1P,S))$ was
chosen in (\ref{fc1}) with $[A_3A_2A_1]\in\mathcal{A}(P)$.   Since
$$ \mathfrak{m}(A_k,k)\in (\mathbb{F}_{A_k \cdots A_1})_k\ \ \text{and}\ \   J\in \bigcap_{[A]\in\mathcal{A}(P)} {\rm
Cap}(\mathbb{F}^*_A)\subset
(\mathbb{F}_{A_k \cdots A_1})_k^*\ \ \text{for $k=1,\cdots,3$},$$
we have for each $k=1,2,3$,
\begin{eqnarray}\label{huk3}
2^{-J\cdot \mathfrak{m}(A_k,k)} \ge  2^{-J\cdot \mathfrak{m}}\ \
\text{for} \ \mathfrak{m}\in \Lambda((A_k\cdots
A_1P)_k)
\end{eqnarray}
where $ \Lambda((A_k\cdots
A_1P)_k)=\Lambda((A_3A_2A_1P)_k)$ for each $k$ by construction above.
\begin{lemma}\label{lemari}
Let $U=A_3A_2A_1$ and $\mathbb{F}_U=\mathbb{F}_{A_3A_2A_1}\in\mathcal{F}(\vec{{\bf N}}(UP,S))$
where $A_1,A_2,A_3$ and $\mathbb{F}_{A_3A_2A_1}$ were defined above. For $\mathbb{G}\in \mathcal{F}(\vec{{\bf N}}(UP,S))$ such that $\mathbb{G}\succeq \mathbb{F}_U$, let
$$\mathcal{I}_J([UP]_{\mathbb{G}},\xi)=\int e^{i\left( \sum_{\nu=1}^3 \left(\sum_{\mathfrak{m}\in\mathbb{G}_\nu\cap \Lambda((UP)_\nu)}
2^{-J\cdot\mathfrak{m}}c^\nu_{\mathfrak{m}}t^{\mathfrak{m}}\right)\xi_\nu\right)}\prod_{\ell=1}^n h(t_\ell)dt.$$
Then for $J\in Z=\bigcap_{A\in\mathcal{A}(P)} {\rm
Cap}(\mathbb{F}^*_A)\subset
  {\rm
Cap}\left(\mathbb{F}_{U}^*\right)$, there exists $C>0$ and
$\delta$ that are independent of $J,\xi$ satisfying:
\begin{equation}\label{huk4}
 \left| \mathcal{I}_J([UP]_\mathbb{G},\xi)\right|\le C\min\left\{|2^{-J\cdot\mathfrak{m}}\xi_\nu|^{-\delta}:\mathfrak{m}\in \Lambda((UP)_\nu),\ \nu=1,2,3\right\}.
 \end{equation}
 \end{lemma}
\begin{proof}[Proof of (\ref{huk4})]
By (\ref{huk3}), it suffices to show that
\begin{eqnarray}\label{huk5}
\left| \mathcal{I}_J([UP]_\mathbb{G},\xi)\right|\le C |2^{-J\cdot \mathfrak{m}(A_k,k)}\xi_k|^{-\delta}\ \ \text{for}\ \  k=1,2,3.
\end{eqnarray}
The case $k=1$ follows from (\ref{huk1}). To show (\ref{huk5}) for $k=2$, it suffices to consider
 $|2^{-J\cdot \mathfrak{m}(A_2,2)}\xi_2|\gg |2^{-J\cdot \mathfrak{m}(A_1,1)}\xi_1|$. This and (\ref{huk2}) yield the desired result for $k=2$.
Since  (\ref{huk5}) holds for $k=1,2$,  we may assume that  $|2^{-J\cdot \mathfrak{m}(A_3,3)}\xi_3|\gg |2^{-J\cdot \mathfrak{m}(A_k,k)}\xi_k|$ for $k=1,2$. So,  the case $k=3$ is  obtained by  the Van der Corput lemma.
\end{proof}
\begin{proof}[Proof of (\ref{3p3p})]
The first hypothesis of Theorem  \ref{th60} is satisfied by Lemma
\ref{lemari}. The second hypothesis of Theorem  \ref{th60} is also satisfied by the hypothesis (\ref{enm1}) of Main Theorem
\ref{main4} such that
\begin{eqnarray*}
\bigcup_{\nu=1}^d [\mathbb{K}_U]_\nu\cap [\Lambda(UP_\Lambda)]_\nu \ \text{
 is an even set}
\end{eqnarray*}
whenever
$$\mathbb{K}_U\in \left\{\mathbb{K}_U\in\mathcal{F}(\vec{{\bf
 N}}(UP,S)): \bigcap_{\nu=1}^d([\mathbb{K}_U]_\nu^{*})^{\circ}\ne \emptyset  \ \ \text{and}\ \
 \text{rank}\left(\bigcup_{\nu=1}^d  [\mathbb{K}_U]_\nu \right)\le n-1   \right\}.
$$
Therefore by applying Theorem \ref{th60} for  $Z=\bigcap_{[A]\in\mathcal{A}(P)} {\rm Cap}(
\mathbb{F}^*_A)\subset {\rm
Cap}\mathbb{F}_{U}^*$ with $U=A_3A_2A_1$, we obtain  (\ref{3p3p}).
\end{proof}
Finally, the proof of necessity part of Main Theorem \ref{main4} is
the same as that of Main Theorems 1-3 once  it is assumed that the
evenness hypothesis (\ref{enm1}) is broken with a fixed matrix $A$.


\begin{thebibliography}{99}
\bibitem{CWW3}  A. Carbery, S. Wainger, and J. Wright,
\emph{Triple Hilbert transform along polynomial surfaces in
$\mathbb{R}^4$}, Rev. Mat. Iberoam. \textbf{25} (2009),  471--519.

\bibitem{CWW-elescorial}  A.
Carbery, S. Wainger, and J. Wright, \emph{Singular integrals and the
Newton diagram}, {Collect. Math.} \textbf{Vol. Extra} (2006)
171--194.


\bibitem{CWW} A. Carbery, S. Wainger, and J.  Wright,
\emph{Hilbert transforms along polynomial surfaces in
$\mathbb{R}^{3}$}, {Duke Math. J.} \textbf{101} (2000), 499--513.






\bibitem{CHKY}
Y. Cho, S. Hong, J. Kim, C. W. Yang,  \emph{Triple Hilbert transforms
along polynomial surfaces}, Integral Equations Operator Theory \textbf{65} (2009), 485--528.

\bibitem{CHKY2}
Y. Cho, S. Hong, J. Kim, C. W. Yang,  \emph{Multiparameter singular integrals and maximal operators along flat surfaces}, { Rev. Mat. Iberoam.}
 \textbf{24} (2008),  1047--1073.



\bibitem{CNSW}
M. Christ, A. Nagel, E.M. Stein, and S. Wainger, \emph{Singular and
maximal Radon transforms: Analysis and geometry}, {Ann. of Math.}
\textbf{150} (2000) 489--577.

\bibitem{F}
W. Fulton,  {\rm Introduction to Toric Varieties}, {Ann. of Math.
Stud.} \textbf{131}, Princeton Univ. Press, Princeton, (1993).

\bibitem{FW}
M. Folch-Gabayet,  J. Wright,  \emph{
Singular integral operators associated to curves with rational components},
 {Trans. Amer. Math. Soc.} \textbf{360} (2008),  1661--1679.



\bibitem{NRS}
A. Nagel, F. Ricci,  E.M. Stein,   \emph{Singular integrals with flag kernels and analysis on quadratic CR manifolds},
{J. Funct. Anal.} \textbf{181} (2001),  29--118.


\bibitem{NRSW1}
A. Nagel, F. Ricci,  E.M. Stein,  S. Wainger,  \emph{ Singular integrals with flag kernels on homogeneous groups  I,}
 {Rev. Mat. Iberoam.} \textbf{28} (2012), 631-722.

\bibitem{NP}
 A. Nagel, M. Pramanik,  \emph{ Maximal averages over linear and monomial polyhedra,} {Duke Math. J.} \textbf{149} (2009),  209--277.


 \bibitem{NW} A. Nagel and S. Wainger,
\emph{$L\sp{2}$ boundedness of Hilbert transforms along surfaces and
convolution operators homogeneous with respect to a multiple
parameter group}, {Amer. J. Math.} \textbf{99} (1977), 761--785.

\bibitem{P1} S. Patel,  \emph{ $L^p$   estimates for a double Hilbert transform,} { Acta Sci. Math. (Szeged)} \textbf{75} (2009),  241--264
\bibitem{P2} S. Patel,  \emph{Double Hilbert transforms along polynomial surfaces in $\mathbb{R}^3$},  {Glasg. Math. J.} \textbf{50} (2008),
 395--428.


\bibitem{PS} D. H. Phong and E. M. Stein, \emph{The Newton polyhedron
and oscillatory integrals}, {Acta Math.} \textbf{179} (1997),
105--152.

\bibitem{MY}
M. Pramanik, C. W. Yang, \emph{Double Hilbert transform along real-analytic surfaces in $\mathbb{R}^{d+2}$}, {J. Lond. Math. Soc.}
 \textbf{77} (2008),   363--386.




\bibitem{RS} F. Ricci and E. M. Stein, \emph{Multiparameter singular
integrals and maximal functions}, {Ann. Inn. Fourier (Grenoble)}
\textbf{42} (1992), 637--670.


\bibitem{S}
E. M. Stein, {\rm Harmonic Analysis: Real-Variable Methods,
Orthogonality, and Oscillatory Integrals}, Princeton press (1993).

\bibitem{SS2}
E. M. Stein,  B. Street,   {\rm Multi-parameter singular Radon transforms}, {  Math. Res. Lett.} \textbf{18} (2011),   257--277.

\bibitem{SS3}
E. M. Stein,  B. Street,  {\rm Multi-parameter singular Radon transforms III: Real analytic surfaces}, { Adv. Math.} \textbf{229} (2012),
 2210--2238.

\bibitem{SW}
E. M. Stein and S. Wainger, \emph{Problems in harmonic analysis
related to curvature}, { Bull. Amer. Math. Soc.}  \textbf{84}
 (1978), 1230--1295.






\bibitem{SS1}
B. Street, {\rm Multi-parameter singular radon transforms I: The $L^2$   theory}, {J. Anal. Math.} \textbf{116} (2012), 83--162.






 \bibitem{V}
A. Varchenko,  \emph{Newton polyhedra and estimations of oscillatory integrals}, {Functional
Anal. Appl.}  \textbf{18}
 (1976),   175--196.














\end{thebibliography}
\end{document}